\documentclass[11pt]{article}

\usepackage[utf8]{inputenc} % allow utf-8 input
\usepackage[T1]{fontenc}    % use 8-bit T1 fonts
\usepackage{url}            % simple URL typesetting
\usepackage{booktabs}       % professional-quality tables
\usepackage{nicefrac}       % compact symbols for 1/2, etc.
\usepackage{microtype}      % microtypography
\usepackage{multirow}
\usepackage{array}
\usepackage[noblocks]{authblk}
\usepackage{amssymb,arydshln}
\usepackage[nodayofweek]{datetime}
\usepackage{float}
\usepackage{bm}

\usepackage[shortlabels]{enumitem}
\setlist[itemize]{itemindent=0ex,itemsep=-0.5ex,leftmargin=3ex,topsep=5pt}
\setlist[enumerate]{label={\arabic*)},itemindent=0ex,itemsep=-0.5ex,leftmargin=3.6ex,topsep=5pt,labelwidth=3.6ex,labelsep=1.5ex}

\usepackage[font=small,labelfont=bf]{caption}
\usepackage{amsmath,amsfonts,nccmath,mathtools,mathrsfs}
\usepackage{tcolorbox}
\usepackage{xcolor}
\usepackage[colorlinks = true,
            linkcolor = blue,
            urlcolor  = blue,
            citecolor = blue,
            anchorcolor = blue]{hyperref}

\usepackage[margin=1in]{geometry}

\usepackage{pgfplots}
\usepgfplotslibrary{fillbetween}
\usetikzlibrary{arrows.meta}
\tikzset{>=latex}

\pgfmathdeclarefunction{cubic}{1}{%
  \pgfmathparse{-2*(#1+2)*(#1+2)*(#1-2)+40}%
}

\pgfmathdeclarefunction{bicuadratic}{1}{%
  \pgfmathparse{(#1-1)*(#1-1)*(#1-1)*(#1-1)+10}%
}

\pgfmathdeclarefunction{cuadratic}{1}{%
    \pgfmathparse{(-(2*#1)^2)+70}%
}

\usepackage{caption} 
\usepackage{pifont}% http://ctan.org/pkg/pifont
\newcommand{\cmark}{\ding{51}}%
\newcommand{\xmark}{\ding{55}}%

\usepackage[numbers,merge,sort&compress]{natbib}
\usepackage{amsthm}
\usepackage{graphicx,color}
\usepackage{subcaption,setspace}

\newtheorem{theorem}{Theorem}[section]
\newtheorem{fact}{Fact}[section]
\newtheorem{proposition}{Proposition}[section]
\newtheorem{conjecture}{Conjecture}[section]
\newtheorem{lemma}{Lemma}[section]
\newtheorem{corollary}{Corollary}[section]

\newtheorem{assumption}{Assumption}

\theoremstyle{definition}
\newtheorem{definition}{Definition}[section]
\newtheorem{example}{Example}[section]

\newtheorem{remark}{Remark}[section]

\newcommand{\method}[1]{\texttt{#1}}

\usepackage[nameinlink]{cleveref}
\crefname{equation}{}{}
\crefname{theorem}{Theorem}{Theorems}
\crefname{corollary}{Corollary}{Corollaries}
\crefname{example}{Example}{Examples}
\crefname{assumption}{Assumption}{Assumptions}
\crefname{lemma}{Lemma}{Lemmas}
\crefname{proposition}{Proposition}{Propositions}
\crefname{figure}{Figure}{Figures}
\crefname{table}{Table}{Tables}
\crefname{fact}{Fact}{Facts}
\crefname{conjecture}{Conjecture}{Conjectures}
\crefname{section}{Section}{Sections}
\crefname{appendix}{Appendix}{Appendices}
\Crefname{equation}{}{}
\Crefname{theorem}{Theorem}{Theorems}
\Crefname{corollary}{Corollary}{Corollaries}
\Crefname{example}{Example}{Examples}
\Crefname{lemma}{Lemma}{Lemma}
\Crefname{proposition}{Proposition}{Proposition}
\Crefname{figure}{Figure}{Figures}
\Crefname{table}{Table}{Tables}
\Crefname{section}{Section}{Sections}
\Crefname{appendix}{Appendix}{Appendices}

\newcommand{\tr}{{{\mathsf T}}}
\newcommand{\her}{{{\mathsf H}}}
\newcommand{\mK}{{\mathsf{K}}}

\newcommand{\mV}{{\mathsf{V}}}
\newcommand{\ECL}{$\mathtt{ECL}$}

\newcommand{\LQG}{\mathtt{LQG}}

\usepackage{tikz}

\makeatletter
\newcommand{\removelatexerror}{\let\@latex@error\@gobble}
\makeatother

\title{\bf Benign Nonconvex Landscapes in Optimal and Robust Control, Part I: Global Optimality
\thanks{The work of Yang Zheng and Chih-fan Pai is supported by NSF ECCS-2154650 and NSF CMMI-2320697. Emails: zhengy@ucsd.edu (Yang Zheng); cpai@ucsd.edu (Chih-fan Pai); yujietang@pku.edu.cn (Yujie Tang).}}
\author[1]{Yang Zheng}
\author[1]{Chih-fan Pai}
\author[2]{Yujie Tang}

\affil[1]{\small Department of Electrical and Computer Engineering, University of California San Diego}
\affil[2]{\small Department of Industrial Engineering \& Management, Peking University}
\date{\vspace{-4mm} \small \today}

\begin{document}

\maketitle

\vspace{-8mm}

\begin{abstract}

Direct policy search has achieved great empirical success in reinforcement learning. Many recent studies have revisited its theoretical foundation for continuous control, which reveals elegant nonconvex geometry in various benchmark problems, especially in fully observable state-feedback cases. This paper considers two fundamental optimal and robust control problems with \textit{partial observability}: the Linear Quadratic Gaussian (LQG) control with stochastic noises, and $\mathcal{H}_\infty$ robust control with adversarial noises. In the policy space, the former problem is~smooth~but nonconvex, while the latter one is nonsmooth and nonconvex. We highlight some interesting and surprising ``discontinuity'' of LQG and $\mathcal{H}_\infty$ cost functions around the boundary of their domains. Despite the lack of convexity~(and possibly smoothness), %we show that both LQG control and $\mathcal{H}_\infty$ robust control have
our main results show that for a class of \textit{non-degenerate} policies, all Clarke stationary points are globally optimal and there is no spurious local minimum for both LQG and $\mathcal{H}_\infty$ control. %, revealing hidden convexity.
Our proof techniques rely on a new and unified framework of \texttt{Extended Convex Lifting (ECL)}, which reconciles the gap between nonconvex policy optimization and convex reformulations. This \ECL{} framework is of independent interest, and we will discuss its details in Part II of this paper.

%\vspace{5pt}
%\noindent{\bf Keywords: }
\end{abstract}

% Section 1: introduction
\section{Introduction}

Inspired by the empirical successes of reinforcement learning (RL) in various applications \cite{silver2016mastering,mnih2015human,kober2013reinforcement}, policy optimization techniques have received renewed interest in the control of dynamical systems with continuous action spaces \cite{recht2019tour,hu2023toward}. Significant advances have been established in terms of understanding the~theoretical properties of direct policy search on a range of benchmark control problems, including stabilization~\cite{perdomo2021stabilizing,zhao2022sample}, linear quadratic regulator (LQR) \cite{fazel2018global,malik2019derivative,mohammadi2021convergence,fatkhullin2021optimizing}, linear risk-sensitive control \cite{zhang2021policy},~linear quadratic Gaussian (LQG) control \cite{zheng2021analysis,zheng2022escaping,duan2022optimization,ren2023controller}, dynamic filtering~\cite{umenberger2022globally,zhang2023learning}, and linear distributed control \cite{furieri2020learning,li2022distributed}; see \cite{hu2023toward} for a recent survey.   

All these control problems are known to be nonconvex in the policy space. In classical control theory,  one typical approach to deal with the nonconvexity is to reparameterize the problem (e.g. via a suitable change of variables \cite{scherer1997multiobjective,zheng2021equivalence,zhou1996robust,boyd1991linear,zheng2022system}) into a convex form for which efficient algorithms exist. %%\cite{boyd2004convex}.
Indeed, since the 1980s, convex reformulations or relaxations become very popular and~almost universally accepted as solutions for optimal and robust control problems.
% thanks to the rapid development of interior-point methods \cite{nesterov1994interior}. 
It is now well-known that many  optimal and robust 
control problems can~be~reformulated as convex programs, using the linear matrix inequalities (LMIs) techniques \cite{scherer1997multiobjective,gahinet1994linear,boyd1994linear}, or relaxed via the sums of squares \cite{parrilo2000structured,anderson2015advances}. These techniques can provide rigorous stability and safety certificates in terms of matrix inequalities representing Lyapunov or dissipativity conditions. To achieve such certificates, these convex approaches typically rely on certain controller reparameterizations \cite{scherer1997multiobjective,zheng2021equivalence,zhou1996robust,boyd1991linear,zheng2022system}, which often require an explicit underlying system model and are thus model-based designs.  

When system models are unknown, complex, or poorly parameterized, direct policy optimization is another viable option for controller synthesis. In fact, this approach can date back to the 1970s \cite{levine1970determination,hyland1984optimal,makila1987computational,bernstein1989lqg}. %  
% , which~has~recieved renewed interest due to impressive empirical successes of reinforcement learning % across many applications 
% \cite{silver2016mastering,mnih2015human,ibarz2021train,vinyals2019grandmaster}. 
% 
In principle, policy optimization is conceptually simpler, computationally more flexible, and also more amenable to learning-based control (as evidenced by the aforementioned successes of RL), but it naturally leads to nonconvex optimization (as pointed out above) for which it is harder to derive theoretical guarantees or certificates. Fortunately, despite the nonconvexity, a series of recent findings highlighted above have revealed favorable optimization landscape properties in many benchmark control problems. For example, global convergence of model-free policy gradient methods has been established for both discrete-time~\cite{fazel2018global} and continuous-time LQR~\cite{mohammadi2021convergence} thanks to the \textit{gradient dominance} property of the cost functions; the LQG cost function has no spurious stationary points that correspond to controllable and observable policies~\cite{zheng2021analysis}. Beyond LQR and LQG, global or local convergence results of direct policy search have also been established for linear risk-sensitive control \cite{zhang2021policy} and distributed control problems \cite{furieri2020learning,li2022distributed}. Furthermore, the work \cite{guo2022global} has established a global convergence result of direct policy search for \emph{state-feedback} $\mathcal{H}_\infty$ control. For \emph{output-feedback} $\mathcal{H}_\infty$ control, it has been revealed in \cite{hu2022connectivity} that there always exists a continuous path connecting any initial stabilizing controller to a globally optimal controller.    

These developments have highlighted elegant nonconvex geometry in various benchmark problems, generating momentous excitements at the interface of control theory, RL, and nonconvex optimization. However, so far, most existing results have focused on control problems with \textit{full-state} observation for which the globally optimal policies are \textit{static}. The theoretical understanding of policy optimization for {\textit{partially observed}} systems, where \textit{dynamic policies} are needed, remains limited. The few notable exceptions \cite{zheng2021analysis,zheng2022escaping,umenberger2022globally,duan2022optimization} only deal with linear quadratic (LQ) control with stochastic noises, where their cost functions are nonconvex but smooth over the feasible region. Another fundamental control paradigm, known as \emph{robust control}, addresses the worst-case performance against uncertainties or adversarial noises~\cite{zhou1996robust}. In this case, the performance measure is the $\mathcal{H}_\infty$ norm of a certain closed-loop transfer function, which is known to be nonconvex and nonsmooth in the policy space~\cite{apkarian2006nonsmooth,lewis2007nonsmooth}. %, which requires techniques from nonsmooth analysis to investigate the behavior of direct policy search.
In this paper, we focus on the policy optimization perspective of output-feedback systems in the cases of both stochastic noises (i.e., LQG control) and adversarial noises (i.e., $\mathcal{H}_\infty$ robust control). In the spirit of \cite{doyle1988state}, our results provide a unified treatment to these two classical optimal and robust control problems, shedding new light on their benign yet intricate nonconvex optimization landscapes.

\subsection{Our contributions}

In this paper, we study two fundamental control problems with \textit{partial observability}: the LQG control with stochastic noises, and $\mathcal{H}_\infty$ robust control with adversarial noises, from a modern nonconvex optimization perspective. Despite the nonconvexity, our results characterize the \textit{global optimality} of a large class of stationary points for both LQG and $\mathcal{H}_\infty$ optimization over \textit{dynamic} policies. A key concept underpinning our results is the notion of \textit{non-degenerate dynamic policies}, which requires a Lyapunov variable with a \textit{full-rank} off-diagonal block to certify the corresponding $\mathcal{H}_2$ or $\mathcal{H}_\infty$ norm of the closed-loop system; the precise definition is given in \Cref{section:LMIs-h2-hinf-norm}. Our main contributions towards the benign nonconvex landscapes of LQG and $\mathcal{H}_\infty$ control are summarized as follows:

\begin{enumerate}
\item \textbf{Smooth nonconvex landscape of LQG control.} Despite being analytical over its domain, we reveal that the LQG cost has intricate ``discontinuous'' behavior around the boundary~of~its~domain: the LQG costs can converge to a finite value when dynamic policies go to its boundary, while with the same limiting policy, the LQG costs can also diverge to infinity (\Cref{example:discontinuity-LQG-boundary}). We prove that if the limiting policy corresponds to a controllable and observable controller, the corresponding LQG costs will always diverge to infinity irrespective of the sequence of dynamic policies (\Cref{proposition:divergence-minimial-controllers}).  Our main technical result shows that for the class of non-degenerate policies, any stationary points are globally optimal in LQG control, and there exist no spurious stationary points (\Cref{theorem:LQG-main}). Notably, this result includes the existing characterization of minimality in \cite{zheng2021analysis} as a special case (\Cref{theorem:LQG-global-optimality}). Therefore, all potentially spurious stationary points belong to the class of degenerate policies. For this aspect, we further identify a class of degenerate stationary points that are likely to be sub-optimal (\Cref{lemma:LQG-lower-order-stationary-point-to-high-order}).   

\item\textbf{Nonsmooth and nonconvex landscape of $\mathcal{H}_\infty$ robust control.} 
One key difference compared to LQG control is that $\mathcal{H}_\infty$ robust control has a nonsmooth cost function (this fact is known~and also expected due to the max operation over adversarial noises; see \cref{example:Hinf-nonsmooth}). The non-smoothness indeed complicates analysis machinery. However, not surprisingly in the spirit of \cite{doyle1988state} (or surprisingly), we show that many landscape properties of LQG control have nonsmooth counterparts in $\mathcal{H}_\infty$ control. This is mainly because $\mathcal{H}_2$ and $\mathcal{H}_\infty$ norms have similar LMI~characterizations. Similar to the LQG case, the nonsmooth $\mathcal{H}_\infty$ cost function also exhibits intricate ``discontinuous'' behavior around its boundary (\Cref{example:non-coercivity-of-Hinf-cost,proposition:divergence-minimial-controllers-hinf}).    
One good property of $\mathcal{H}_\infty$ cost is that it is locally Lipschitz and also subdifferentially regular in the sense of Clarke, enabling the use of Clarke subdifferential for nonsmooth analysis \cite{clarke1990optimization,clarke2008nonsmooth} (see \cref{appendix:clarke} for a review). Our main result is that all Clarke stationary points in the set of non-degenerate policies are globally optimal for $\mathcal{H}_\infty$ robust control (\cref{theorem:Hinf-main-global-optimality}), and thus the set of non-degenerate policies has no spurious stationary points. % (local minimum/maximum/saddle). 
%
% The proof also relies on the \texttt{ECL} framework in \cref{section:proofs}. 
% At the end of this section, 
%
Finally, we also highlight a class of Clarke stationary points that is potentially degenerate and sub-optimal (\Cref{theorem:hinf_non_globally_optimal_stationary_point}).  

\end{enumerate}

Our main technical results are established using a new and unified framework of \texttt{Extended Convex Lifting (ECL)}, which reconciles the gap between nonconvex policy optimization and convex reformulations. The remarkable (now classical) results in control theory reveal that many optimal and robust control problems admit ``convex reformulations'' in terms of LMIs, via a suitable change of variables (which often comes from Lyapunov theory) \cite{gahinet1994linear,scherer1997multiobjective,dullerud2013course,boyd1994linear,scherer2000linear}. Under suitable settings, these changes of variables establish a differemorphism that connects nonconvex policy optimization with their convex reformulations, enabling convex analysis for nonconvex problems. The technical details are involved, and we postponed them to Part II of this paper.

\subsection{Related work}

We review some related literature in this section. The review below is not intended to be exhaustive, but rather to outline some of the closely related work. We refer interested to some excellent textbooks \cite{zhou1996robust,dullerud2013course} for classical results on optimal and robust control, the monographs \cite{boyd1994linear,scherer2000linear} for LMIs in control, and excellent surveys \cite{hu2023toward,recht2019tour} for recent advances on policy optimization in controls. The interested reader might also see \cite{zheng2021sample,dean2019sample,tsiamis2022statistical} for statistical learning theory for LQ control. 

\vspace{-2mm}

\paragraph{Policy optimization for smooth LQ control.} As mentioned earlier, policy optimization has a long history dating back to the 1970s and is one of the main workhorses for modern RL. There has been a recent surging interest in applying direct policy search to control complex dynamic systems \cite{recht2019tour,hu2023toward}. One emphasis is to obtain theoretical guarantees in terms of optimality, robustness, and sample complexity. For the classical LQR problem, it has been revealed that the cost function is \textit{coercive} and \textit{smooth} over any sublevel set, and \textit{gradient dominated}~\cite{fazel2018global}. These properties are fundamental~to~establishing global~convergence of direct policy search and their model-free extensions for solving LQR~\cite{malik2019derivative,mohammadi2021convergence,fatkhullin2021optimizing}. 
Similar~favorable nonconvex properties of risk-sensitive state-feedback control are revealed in \cite{zhang2021policy}.
Other recent results include Markovian jump LQR \cite{jansch2022policy}, distributed LQR \cite{li2022distributed,furieri2020learning,fatkhullin2021optimizing}, and finite-horizon LQ control \cite{hambly2021policy}.

All of these studies considered \textit{static} policies with full state observations, i.e., the policy's output only utilizes the current state observation but no historical information. Control problems~with partial output observations typically require \emph{dynamic} policies, e.g, involving a suitable state estimator. The resulting optimization landscape becomes richer and yet much more complicated than~LQR. Our prior work \cite{zheng2021analysis} has revealed favorable landscape properties: 1) the set of stabilizing full-order dynamic policies has at most two path-connected components that are identical in the frequency domain; 2) all stationary points corresponding to controllable and observable controllers in LQG control are globally optimal. However, there exist suboptimal saddle points in the LQG landscape \cite[Examples 5 \& 6]{zheng2021analysis}, which never appear in state feedback LQR. A saddle-escaping algorithm has been proposed in \cite{zheng2022escaping}. A few other recent studies include \cite{duan2022optimization,umenberger2022globally,zhang2023learning,ren2023controller}. In~\cite{umenberger2022globally}, the global convergence of policy search over dynamic filters was proved for a simpler estimation problem.

\vspace{-2mm}

\paragraph{Policy optimization for nonsmooth robust control.}
%The optimal LQG controller has no robustness guarantee~\cite{doyle1978guaranteed}, and thus it is important to incorporate robustness for the search over dynamic controllers. 
Direct policy optimization techniques~for $\mathcal{H}_\infty$ robust control have been extensively studied since the early 2000s \cite{lewis2007nonsmooth,apkarian2006nonsmooth,apkarian2006nonsmooth2,apkarian2008mixed,burke2005robust}. It is known that the closed-loop $\mathcal{H}_\infty$ norm is nonsmooth and not always differentiable in the policy space~\cite{apkarian2006nonsmooth}. Indeed, robust control problems were one of the early motivations and applications for nonsmooth optimization \cite{lewis2007nonsmooth}. Some of the early nonsmooth optimization algorithms for robust control have been implemented as software packages --- \texttt{HIFOO} \cite{burke2006hifoo} and \texttt{HINFSTRUCT} \cite{gahinet2011decentralized}, which have seen successful applications in practice. However, none of these early studies address the global optimality~of policy search for robust control. Global optimality guarantees only appear recently for state-feedback $\mathcal{H}_\infty$ control~\cite{guo2022global}: any Clarke stationary points are globally optimum; see \cite{hu2022connectivity,guo2023complexity} for related discussions. 

The recent studies \cite{guo2022global,hu2022connectivity} utilize convex reparameterization in terms LMIs to analyze the nonconvex policy optimization, which is consistent with the strategies in \cite{zheng2021analysis,mohammadi2021convergence,furieri2020learning} for smooth LQ control. This analysis idea via convex reformulations was also studied in \cite{sun2021learning} for static state feedback policies. A related framework, called \texttt{DCL}, was proposed in \cite{umenberger2022globally}, to study the nonconvex properties of Kalman filter. Despite sharing a similar spirit of exploiting LMI-based parameterization, our framework of \ECL{} applies to a wider range of problems than all these existing studies \cite{guo2022global,hu2022connectivity,zheng2021analysis,mohammadi2021convergence,furieri2020learning,umenberger2022globally} in the sense that \ECL{} works for both smooth and nonsmooth control problems, while naturally differentiating degenerate and non-degenerate policies. %; see Part II of this paper.     

%\paragraph{LQG control with unknown dynamics and controller parameterization.} model-based strategies, online optimization strategies 

\vspace{-2mm}

\paragraph{Benign nonconvex landscapes in machine learning problems} 

Some exciting nonconvex optimization results have recently emerged in machine learning literature, where the underlying geometrical properties (such as rotational~\cite{sun2018geometric,chi2019nonconvex,li2019symmetry} and discrete~\cite{qu2019analysis,ge2017optimization} symmetries) enable identifying the local curvature of stationary points and thus contribute to global optimality of gradient-based methods \cite{zhang2020symmetry,jin2017escape,lee2019first}. Many of these favorable nonconvex properties have been studied in \cite{levin2022effect} under smooth parametrizations of nonconvex problems into convex forms. However, control problems naturally involve extra Lyapunov variables with a unique feature of similarity transformations, which is very different from machine learning problems. Accordingly, our main technical results are established using a new analysis framework of 
%
%Our framework of 
\ECL{} accounting for Lyapunov variables and similarity transformations. We believe this \ECL{} framework is of independent interest, and the details will be presented in Part II of this paper.

\subsection{Paper outline}

The rest of this paper is organized as follows. \Cref{section:problem_statement} presents the problem statements of nonconvex policy optimization for Linear Quadratic Gaussian (LQG) control and $\mathcal{H}_\infty$ robust control. In \Cref{section:LMIs-h2-hinf-norm}, we overview a few LMI characterizations for the LQG cost and $\mathcal{H}_\infty$ cost, and we also introduce a class of non-degenerate stabilizing controllers. We present our main results on the global optimality of smooth and nonconvex LQG control in \Cref{section:LQG} and on the global optimality of nonsmooth and nonconvex $\mathcal{H}_\infty$ robust control in \Cref{section:Hinf-control}. %Our main proof techniques, centered around a new Extended Convex Lifting (\ECL), are introduced in \Cref{section:proofs}, and 
Some numerical experiments are reported in \Cref{section:experiments}. We finally conclude the paper in \Cref{section:conclusion}. Many auxiliary results, additional discussions, and technical proofs are provided in \cref{appendix:preliminaries,app:nonsmooth-optimization,appendix:auxillary-results,appendix:technical-proofs}.

\vspace{-2mm}

\paragraph{Notations.} 
%We use $\mathbb{R}$ and $\mathbb{N}$ to denote the set of real and natural numbers, respectively.  
%
We use lower and upper case letters (\emph{e.g.} $x$ and $A$) to denote vectors and matrices, respectively. Lower and upper case boldface letters (\emph{e.g.} $\mathbf{x}$ and $\mathbf{G}$) are used to denote signals and transfer matrices, respectively.  
The set of  $k\times k$ real symmetric matrices is denoted by $\mathbb{S}^k$, and the determinant of a square matrix $M$ is denoted by $\det M$. We denote the set of $k \times k$ real invertible matrices by $\mathrm{GL}_k=\{T\in\mathbb{R}^{k\times k} \mid \det T\neq 0\}$. Given a matrix $M \in \mathbb{R}^{k_1 \times k_2}$, $M^\tr$ denotes the transpose of $M$, and $\|M\|_F$ denotes the Frobenius norm of $M$. For any $M_1,M_2\in\mathbb{S}^k$, we use $M_1\prec (\preceq) M_2$ and $M_2\succ (\succeq) M_1$ to mean that $M_2-M_1$ is positive (semi)definite. 
We denote the set of real-rational proper stable
% \footnote{Throughout the paper, ``stable'' means ``asymptotically stable'', \emph{i.e.}, all eigenvalues/poles have strictly negative real parts in continuous time (magnitudes less than 1 in discrete time).} 
transfer matrices as $\mathcal{RH}_{\infty}$.  For simplicity, we omit the dimension of transfer matrices, which shall be clear in the context.  Also, we use $I$ (resp. $0$) to denote the identity matrix (resp. zero matrix) of compatible dimensions. 
% The state-space realization of a transfer function $C(sI - A)^{-1}B + D$ is denoted as
%  $ \left[\begin{array}{c|c} A & B \\\hline
%      C & D\end{array}\right].$ 
%We use $I_k$ to denote the $k\times k$ identity matrix, and use $0_{k_1\times k_2}$ to denote the $k_1\times k_2$ zero matrix; we sometimes omit their subscripts if the dimensions can be inferred from the context.   

% Section 2: problem formulation
\section{Motivations and Problem Statement} \label{section:problem_statement}

In this section, we first overview the formulations of optimal control with stochastic noises and robust control with adversarial disturbances, then introduce their policy optimization forms, %review their (sub)optimal solutions from classical control theory, 
and finally present the problem statement of our work. Many %of the 
results introduced in this section are standard in classical control theory, 
and we refer to the book \cite{zhou1996robust} if we do not explicitly mention their sources. %the source of such results.

\subsection{System dynamics and disturbances}
We consider a continuous-time\footnote{Within the scope of this paper, we only consider the continuous-time case. Analog results for discrete-time systems also exist.
} 
linear time-invariant (LTI) dynamical system
\begin{equation}\label{eq:Dynamic}
\begin{aligned}
\dot{x}(t) &= Ax(t)+Bu(t)+ B_w w(t), \\
y(t) &= Cx(t)+ D_v v(t),
\end{aligned}
\end{equation}
where $x(t) \in \mathbb{R}^n$ is the vector of state variables, $u(t)\in \mathbb{R}^m$ the vector of control inputs, $y(t) \in \mathbb{R}^p$ the vector of measured outputs available for feedback, and $w(t) \in \mathbb{R}^n, v(t)  \in \mathbb{R}^p$ are the disturbances on the system process and measurement at time $t$. Here we introduce the matrices $B_w\in\mathbb{R}^{n\times n}$ and $D_v\in\mathbb{R}^{p\times p}$ for a unified treatment of LQG and $\mathcal{H}_\infty$ control problems that will be presented later. For notational simplicity, we will drop the argument $t$ when it is clear in the context. We consider the following performance signal 
\begin{equation} \label{eq:performance-signal}
    z(t) = \begin{bmatrix}
        Q^{1/2} \\ 0
    \end{bmatrix}x(t) + \begin{bmatrix}
        0 \\ R^{1/2}
    \end{bmatrix}u(t),
\end{equation}
where $Q \succeq 0$ and $R\succ 0$ are performance weight matrices. Throughout the paper, we also make the following standard assumption (see \Cref{appendix:preliminaries} for a review of related notions). 
\begin{assumption} \label{assumption:stabilizability}
In \cref{eq:Dynamic}, $(A,B)$ is controllable and $(C,A)$ is observable.
\end{assumption}

In this paper, we consider two different settings with regard to the disturbances $w(t)$ and $v(t)$, and they lead to the LQG optimal control problem and the $\mathcal{H}_\infty$ robust control problem.

%\vspace{-3mm}

\paragraph{LQG optimal control with stochastic noises. }
When $w(t)$ and $v(t)$ are white Gaussian noises with identity intensity matrices, i.e., $\mathbb{E}[w(t)w(\tau)]=\delta(t-\tau)I_n$ and $\mathbb{E}[v(t)v(\tau)]=\delta(t-\tau)I_p$, we consider an averaged mean performance
\begin{equation*} %\label{eq:LQG-cost-def}
\mathfrak{J}_{\LQG}
\coloneqq \lim_{T \rightarrow +\infty }\mathbb{E}\!\left[ \frac{1}{T} \int_{0}^T  z(t)^\tr z(t)\, dt\right]
=\lim_{T \rightarrow +\infty }\mathbb{E} \!\left[\frac{1}{T}\int_{0}^T \left(x^\tr Q x + u^\tr R u\right)dt\right].
\end{equation*}
The classical problem of linear quadratic Gaussian (LQG) optimal control is formulated as
\begin{equation} \label{eq:LQG}
    \begin{aligned}
        \min_{u(t)} \quad & \mathfrak{J}_{\LQG} \\
        \text{subject to} \quad & ~\cref{eq:Dynamic},
    \end{aligned}
\end{equation}
where  the input $u(t)$ is allowed to depend on all past observation $y(\tau)$ with $\tau \leq t$.

\vspace{-3mm}

\paragraph{\texorpdfstring{$\mathcal{H}_\infty$}{Lg} robust control with adversarial disturbances.}
When $w(t)$ and $v(t)$ are deterministic disturbances, we consider the worst-case performance in an adversarial setup. Assume that the system starts from a zero initial state $x(0)=0$. Let $\mathcal{L}_2^{k}[0,\infty)$ be the set of square-integrable (bounded energy) signals of dimension $k$, i.e.,
$$
\mathcal{L}_2^{k}[0,\infty): =\left\{u:[0,+\infty)\rightarrow\mathbb{R}^{k}
\left|\,\|u\|_2 \coloneqq \left(\int_0^\infty u(t)^\tr u(t)dt\right)^{1/2} < \infty \right.\right\},
$$
Denote $
d(t) = \begin{bmatrix}
w(t) \\ v(t)
\end{bmatrix} \in \mathbb{R}^{p + n}
$, and consider the worst-case performance when the disturbance signal $d(t)$ has bounded energy less than or equal to $1$ i.e., $d\in \mathcal{L}_2^{p+n}[0,+\infty)$ with $\|d\|_2 \leq 1$:
\begin{equation*} %\label{eq:Hinf-cost-def}
    \mathfrak{J}_{\infty} := \sup_{\|d\|_2 \leq 1} \; \int_0^\infty z(t)^\tr z(t)\, dt = \sup_{\|d\|_2 \leq 1} \int_{0}^\infty \left(x^\tr Q x + u^\tr R u\right)dt. 
\end{equation*}
The $\mathcal{H}_\infty$ robust control problem is formulated as
\begin{equation} \label{eq:Hinf}
    \begin{aligned}
        \min_{u(t)} \quad & \mathfrak{J}_{\infty} \\
        \text{subject to} \quad & ~\cref{eq:Dynamic}\text{ and }
        x(0)=0,
    \end{aligned}
\end{equation}
where  the input $u(t)$ is allowed to depend on all past observation $y(\tau)$ with $\tau \leq t$.

It can be shown that for both LQG optimal control and $\mathcal{H}_\infty$ robust control, the cost values $\mathfrak{J}_{\LQG}$ and $\mathfrak{J}_\infty$ depend on $B_w$ and $D_v$ only via $B_wB_w^\tr$ and $D_v D_v^\tr$. We therefore define the weight~matrices
$
W \coloneqq B_w B_w^\tr,\, V\coloneqq D_vD_v^\tr,
$
and assume $B_w=W^{1/2}$ and $D_v=V^{1/2}$ without loss of generality. 
We now state the following standard assumptions, which is the counterpart to \Cref{assumption:stabilizability}. 
\begin{assumption} \label{assumption:performance-weights}
In both LQG optimal control and $\mathcal{H}_\infty$ robust control, the weight matrices satisfy $Q \succeq 0, R \succ 0, W \succeq 0, V \succ 0$. Furthermore, $(A,W^{1/2})$ is controllable, and $(Q^{1/2}, A)$ is observable.
\end{assumption}

\subsection{Dynamic feedback policies}
% It is known in classical control theory that, 
When the full state $x(t)$ cannot be directly observed, static feedback policies in which $u(t)$ only depends on the measurement $y(t)$ cannot achieve optimal values of the LQG cost or the $\mathcal{H}_\infty$ cost in general. On the other hand, we also do not need to consider general nonlinear policies with memory \cite{khargonekar1986uniformly}.
Indeed, it is known that the optimal values of both \cref{eq:LQG} and \cref{eq:Hinf} can be achieved or approximated by employing linear dynamic feedback policies of the form
\begin{equation}\label{eq:Dynamic_Controller}
    \begin{aligned}
        \dot \xi(t) &= A_{\mK}\xi(t) + B_{\mK}y(t), \\
        u(t) &= C_{\mK}\xi(t) + D_{\mK}y(t),
    \end{aligned}
\end{equation}
where $\xi(t) \in \mathbb{R}^q$ is the internal state, and $A_{\mK},B_{\mK},C_{\mK}$ and $D_{\mK}$ are matrices of proper dimensions that specify the policy dynamics; see \cite[Chaps 14 \& 16]{zhou1996robust} and also \Cref{section:optimal_policies_classical} for a brief summary of relevant results. We parameterize dynamic feedback policies of the form~\cref{eq:Dynamic_Controller} by%
\footnote{For notational simplicity, we lumped the policy parameters into a single matrix, but it should be interpreted as a dynamic policy represented by~\cref{eq:Dynamic_Controller}.
}
\[
\mK =
\begin{bmatrix}
D_\mK & C_\mK \\ B_\mK & A_\mK
\end{bmatrix} \in \mathbb{R}^{(m+q)\times (p+q)}.
\]
Combining~\cref{eq:Dynamic_Controller} with~\cref{eq:Dynamic} via simple algebra~leads to the closed-loop system
\begin{subequations}\label{eq:closed-loop}
\begin{equation}
\begin{aligned}
\frac{d}{dt} \begin{bmatrix} x \\ \xi \end{bmatrix} &=
A_{\mathrm{cl}}(\mK)
\begin{bmatrix} x \\ \xi \end{bmatrix}
+ B_{\mathrm{cl}}(\mK)
\begin{bmatrix} w \\ v \end{bmatrix}, \\
z &= C_{\mathrm{cl}}(\mK)
\begin{bmatrix} x \\ \xi \end{bmatrix}
+ D_{\mathrm{cl}}(\mK)
\begin{bmatrix} w \\ v \end{bmatrix},
\end{aligned}
\end{equation}
where we denote the closed-loop system matrices by
\begin{equation}\label{eq:closed-loop-matrices}
\begin{aligned}
A_{\mathrm{cl}}(\mK)
\coloneqq\  &
\begin{bmatrix}
A + BD_\mK C & BC_\mK \\ B_\mK C & A_\mK
\end{bmatrix}, 
& B_{\mathrm{cl}}(\mK)
\coloneqq\  & \begin{bmatrix} W^{1/2} & BD_{\mK}V^{1/2} \\ 0 & B_{\mK}V^{1/2}  \end{bmatrix}, \\
C_{\mathrm{cl}}(\mK)
\coloneqq\  &
 \begin{bmatrix}
        Q^{1/2} & 0 \\
        R^{1/2}D_{\mK}C & R^{1/2}C_{\mK}
    \end{bmatrix}, &
D_\mathrm{cl}(\mK) \coloneqq\ &
\begin{bmatrix} 0 & 0\\ 0 & R^{1/2}D_{\mK}V^{1/2}\end{bmatrix}.
\end{aligned}
\end{equation}
\end{subequations}
Then, the transfer matrix from the disturbance $d(t)=\begin{bmatrix}
w(t) \\ v(t)
\end{bmatrix}$ to the performance signal $z(t)$ becomes %given by
\begin{equation} \label{eq:transfer-function-zd}
    \mathbf{T}_{{zd}}(\mK, s) = C_{\mathrm{cl}}(\mK)
\left(sI - A_{\mathrm{cl}}(\mK)
\right)^{-1}
B_{\mathrm{cl}}(\mK)
+D_{\mathrm{cl}}(\mK).
\end{equation}

When the closed-loop system \cref{eq:closed-loop} is internally stable (see \cref{subsection:policy-optimization-internally-stable-systems} for precise definitions), it is a standard result in control theory that the LQG cost $\mathfrak{J}_{\LQG}$ of the closed-loop system~\cref{eq:closed-loop} is related to the $\mathcal{H}_2$ norm of the transfer matrix $\mathbf{T}_{zd}(\mK,s)$ by
\[
\mathfrak{J}_{\LQG}
=\|\mathbf{T}_{zd}(\mK,\cdot)\|_{\mathcal{H}_2}^2
=\frac{1}{2\pi}\int_{-\infty}^{+\infty}
\operatorname{tr}\!
\left(\mathbf{T}_{zd}(\mK,j\omega)^\her\,
\mathbf{T}_{zd}(\mK,j\omega)\right) d\omega,
\]
while the $\mathcal{H}_\infty$ cost of the closed-loop system~\cref{eq:closed-loop} is related to the $\mathcal{H}_\infty$ norm of $\mathbf{T}_{zd}(\mK,s)$ by
\[
\mathfrak{J}_{\infty}
=\|\mathbf{T}_{zd}(\mK,\cdot)\|_{\mathcal{H}_\infty}^2
= \left(\sup_{\omega\in\mathbb{R}}
\sigma_{\max}(\mathbf{T}_{zd}(\mK,j\omega))\right)^{\!2},
\]
where $\sigma_{\max}(\cdot)$ denotes the largest singular value. We note that when $\mathfrak{J}_\infty<+\infty$, the closed-loop LTI system \cref{eq:closed-loop} with $x(0)=0$ can be viewed as a linear mapping that maps any disturbance signal $d(t) \in \mathcal{L}_2^{p+n}[0,+\infty)$ to the performance signal $z(t) \in \mathcal{L}_2^{m+n}[0,+\infty)$, and $\sqrt{\mathfrak{J}_\infty}$ is exactly the $\mathcal{L}_2^{p+n}[0,+\infty) \to \mathcal{L}_2^{m+n}[0,+\infty)$ induced $\mathcal{H}_\infty$ norm of the closed-loop LTI system \cref{eq:closed-loop}, as shown above. 

We refer to the dimension $q \geq 0$ of the internal state variable $\xi$ as the order of the dynamic policy~\cref{eq:Dynamic_Controller}. A dynamic policy is called \emph{full-order} if $q = n$ and is called \emph{reduced-order} if $q < n$. Since the globally optimal LQG or $\mathcal{H}_\infty$ policy can be realized or approximated by a dynamic policy of order at most $n$ \cite[Chaps 14 \& 16]{zhou1996robust} (also see \cref{section:optimal_policies_classical} for a review), it is unnecessary to consider dynamic policies with orders beyond %the system dimension 
$n$. We also note that, any dynamic policy $\mK=\begin{bmatrix}
D_\mK & C_\mK \\ B_\mK & A_\mK
\end{bmatrix}$ of order $q<n$ can be augmented to full-order policies given by
\begin{equation}\label{eq:lifted_controller}
\tilde{\mK}_c =
\left[\begin{array}{c:cc}
    D_\mK & C_{\mK} &  0 \\[2pt]
    \hdashline
    B_{\mK} & A_{\mK} & 0 \\
    \tilde{B} & 0 & \Lambda
    \end{array}\right],
    \qquad
\tilde{\mK}_o =
\left[\begin{array}{c:cc}
    D_\mK & C_{\mK} &  \tilde{C} \\[2pt]
    \hdashline
    B_{\mK} & A_{\mK} & 0 \\
    0 & 0 & \Lambda
    \end{array}\right]
\end{equation}
where $\Lambda\in \mathbb{R}^{(n-q)\times(n-q)}$ is an arbitrary stable matrix, and $\tilde{B}\in \mathbb{R}^{(n-q)\times p}$ and $\tilde{C}\in \mathbb{R}^{m\times(n-q)}$ are arbitrary. It is clear that the augmented policies $\tilde{\mK}_c$ and $\tilde{\mK}_o$ satisfy $\mathbf{T}_{zd}(\tilde{\mK}_c,s)=\mathbf{T}_{zd}(\tilde{\mK}_o,s)=\mathbf{T}_{zd}(\mK,s)$ and thus induce the same LQG cost and $\mathcal{H}_\infty$ cost as $\mK$.

The dynamic policy~\cref{eq:Dynamic_Controller} is said to be \emph{strictly proper}, if there exists no direct feedback term from the measurement $y(t)$, i.e., $D_{\mK} = 0$. The LQG problem requires strictly proper dynamic policies, since the measurement $y(t)$ is corrupted by white Gaussian noise with infinite covariance, and the LQG cost $\mathfrak{J}_{\LQG}$ would be infinite if $D_\mK\neq 0$.

\subsection{Policy optimization over internally stabilizing policies} \label{subsection:policy-optimization-internally-stable-systems}

We now introduce an important notion of \textit{internal stability} that stems from control theory.
A dynamic feedback policy~\cref{eq:Dynamic_Controller} is said to \emph{internally stabilize} the plant~\cref{eq:Dynamic} if the 
the closed-loop matrix $A_{\mathrm{cl}}(\mK)$ 
\begin{comment}
\begin{equation}
\label{eq:closedloopmatrix}
A_{\mathrm{cl}}(\mK)=
\begin{bmatrix}
    A +BD_{\mK}C & BC_{\mK} \\
    B_{\mK}C & A_{\mK}
\end{bmatrix} = \begin{bmatrix}
    A & 0 \\
    0 & 0
\end{bmatrix} + \begin{bmatrix}
    B & 0 \\
    0 & I
\end{bmatrix}\begin{bmatrix}
    D_{\mK} & C_{\mK} \\
    B_{\mK} & A_{\mK}
\end{bmatrix}\begin{bmatrix}
    C & 0 \\
    0 & I
\end{bmatrix}
\end{equation}
\end{comment}
is (Hurwitz) stable, i.e., the real parts of all its eigenvalues are strictly negative. We denote%
\begin{subequations} \label{eq:internallystabilizing}
\begin{align}
    \mathcal{C}_{q} & \coloneqq \left\{
    \left.\mK=\begin{bmatrix}
    D_{\mK} & C_{\mK} \\
    B_{\mK} & A_{\mK}
    \end{bmatrix}
    \in \mathbb{R}^{(m+q) \times (p+q)} \,\right|
    A_{\mathrm{cl}}(\mK)\text{ is stable}\right\}, \label{eq:internallystabilizing-a} \\
    \mathcal{C}_{q,0} & \coloneqq \left\{
    \left.\mK=\begin{bmatrix}
    D_{\mK} & C_{\mK} \\
    B_{\mK} & A_{\mK}
    \end{bmatrix}
    \in \mathcal{C}_q\,\right| D_{\mK} = 0_{m\times p}\right\}. \label{eq:internallystabilizing-b}
\end{align}
\end{subequations}
In other words, $\mathcal{C}_q$ parameterizes the set of internally stabilizing policies, while $\mathcal{C}_{q,0}$ parameterizes the set of \textit{strictly proper} internally stabilizing policies. It can be shown (see \cite{zheng2021analysis}) that $\mathcal{C}_q$ is an open subset of $\mathbb{R}^{(m+q) \times (p+q)}$, and $\mathcal{C}_{q,0}$ is an open subset of the vector space
\[
\mathcal{V}_q \coloneqq \left\{\left.\mK = \begin{bmatrix} D_\mK & C_\mK \\ B_\mK & A_\mK \end{bmatrix}\,\right| D_\mK = 0\right\}.
\]
Moreover, classical control theory establishes the following properties of the sets $\mathcal{C}_q$ and $\mathcal{C}_{q,0}$:
\begin{enumerate}
\item Any dynamic policy in $\mathcal{C}_{q,0}$ (resp. $\mathcal{C}_q$) induces a finite LQG cost $\mathfrak{J}_{\LQG}$ (resp. a finite $\mathcal{H}_\infty$ cost $\mathfrak{J}_\infty$).
\item Any dynamic policy of order $q\leq n$ with a finite LQG cost $\mathfrak{J}_{\LQG}$ (resp. a finite $\mathcal{H}_\infty$ cost $\mathfrak{J}_\infty$), can be augmented to a policy in $\mathcal{C}_{n,0}$ (resp. a policy in $\mathcal{C}_n$) with the same LQG cost (resp. $\mathcal{H}_\infty$ cost).
\item For any reduced-order policy $\mK$ and the associated augmented policies $\tilde{\mK}_c$ and $\tilde{\mK}_o$ in~\cref{eq:lifted_controller}, $\mK\in\mathcal{C}_q$ implies $\tilde{\mK}_c\in\mathcal{C}_n\text{ and }\tilde{\mK}_o\in\mathcal{C}_n$, and $\mK\in\mathcal{C}_{q,0}$ implies $\tilde{\mK}_c\in\mathcal{C}_{n,0}\text{ and }\tilde{\mK}_o\in\mathcal{C}_{n,0}$.
\end{enumerate}
Therefore, we will mainly consider dynamic policies in $\mathcal{C}_{n,0}$ (resp. $\mathcal{C}_n$) for the LQG optimal control (resp. the $\mathcal{H}_\infty$ robust control). 

We are now ready to view LQG optimal control~\cref{eq:LQG} and $\mathcal{H}_\infty$ robust control~\cref{eq:Hinf} from a modern perspective of policy optimization. Specifically, we will introduce cost functions $J_{\LQG,q}:\mathcal{C}_{q,0}\rightarrow\mathbb{R}$ and $J_{\infty,q}:\mathcal{C}_q\rightarrow\mathbb{R}$, and formulate \cref{eq:LQG} and \cref{eq:Hinf} as minimization of these cost functions.

\vspace{-3mm}

\paragraph{LQG policy optimization.}
Given $q\leq n$ and $\mK\in\mathcal{C}_{q,0}$, we define
\begin{equation}\label{eq:H2-norm}
J_{\LQG,q}(\mK)
= \|\mathbf{T}_{zd}(\mK,\cdot)\|_{\mathcal{H}_2},
\end{equation}
which is also equal to $\sqrt{\mathfrak{J}_{\LQG}}$ for the closed-loop system with the dynamic policy $\mK$.%
\footnote{
Here, unlike \cite{zheng2021analysis}, $J_{\LQG,q}(\mK)$ is defined to be the square root of $\mathfrak{J}_{\LQG}$ rather than $\mathfrak{J}_{\LQG}$ itself; the function $J_{\infty,q}(\mK)$ defined later is treated similarly. Minimizing $J_{\LQG,q}(\mK)$ is clearly equivalent to minimizing $\mathfrak{J}_{\LQG}$. These definitions will facilitate our subsequent analysis based on convex reformulations of classical control problems.
}
It can be seen that $J_{\LQG,q}$ gives a function defined over the domain $\mathcal{C}_{q,0}$ with values in $(0,+\infty)$. As a result, we can formulate the LQG optimal control \cref{eq:LQG} as the following policy optimization problem:
\begin{equation}\label{eq:LQG_policy_optimization}
\begin{aligned}
\min_{\mK}\ \ & J_{\LQG,n}(\mK) \\
\text{subject to}\ \ &  \mK \in \mathcal{C}_{n,0}.
\end{aligned}
\end{equation}
%Note that 
Here we only consider optimizing over full-order dynamic policies in $\mathcal{C}_{n,0}$, since any reduced-order policy in $\mathcal{C}_{q,0}$ can be augmented to a full-order policy in $\mathcal{C}_{n,0}$ by~\eqref{eq:lifted_controller} without changing the LQG cost. 
It can be shown that the set of \emph{full-order} strictly proper internally stabilizing policies $\mathcal{C}_{n,0}$ is nonempty as long as~\cref{assumption:stabilizability,assumption:performance-weights} hold; we refer to \cite{zheng2021analysis} for other geometrical properties such as path-connectedness of $\mathcal{C}_{n,0}$. We also define $J_{\LQG,q}$ for $q\leq n$, which will be used for theoretical~analysis.

It is known in classical control that $J_{\LQG,q}(\mK)$ can be computed by solving a Lyapunov equation, which is summarized in the lemma below \cite{basuthakur1975optimal,hyland1984optimal}. 

\begin{lemma}\label{lemma:LQG_cost_formulation1}
Fix $q\in\mathbb{N}$ such that $\mathcal{C}_{q,0}\neq\varnothing$. Given any $\mK\in\mathcal{C}_{q,0}$, we have
\begin{equation}\label{eq:LQG_cost_formulation1}
\begin{aligned}
J_{\LQG,q}(\mK)
={} &
\sqrt{\operatorname{tr}\!
\left(C_{\mathrm{cl}}(\mK)X_\mK C_{\mathrm{cl}}(\mK)^\tr \right)}
=
\sqrt{
\operatorname{tr}\!
\left(B_{\mathrm{cl}}(\mK)^\tr Y_\mK B_{\mathrm{cl}}(\mK)\right)},
\end{aligned}
\end{equation}
where $X_{\mK}$ and $Y_{\mK}$ are the unique positive semidefinite 
solutions to the following Lyapunov equations
\begin{subequations} \label{eq:Lyapunov-LQG}
\begin{align}
\begin{bmatrix} A &  BC_{\mK} \\ B_{\mK} C & A_{\mK} \end{bmatrix}X_{\mK} + X_{\mK}\begin{bmatrix} A &  BC_{\mK} \\ B_{\mK} C & A_{\mK} \end{bmatrix}^\tr +  \begin{bmatrix} W & 0 \\ 0 & B_{\mK}VB_{\mK}^\tr  \end{bmatrix}
& = 0, \label{eq:LyapunovX}
\\
\begin{bmatrix} A &  BC_{\mK} \\ B_{\mK} C & A_{\mK} \end{bmatrix}^\tr Y_{\mK} +  Y_{\mK}\begin{bmatrix} A &  BC_{\mK} \\ B_{\mK} C & A_{\mK} \end{bmatrix} +   \begin{bmatrix} Q & 0 \\ 0 & C_{\mK}^\tr R C_{\mK} \end{bmatrix}
& = 0. \label{eq:LyapunovY}
\end{align}
\end{subequations}
\end{lemma}

\vspace{-3mm}

\paragraph{$\mathcal{H}_\infty$ policy optimization.} 
Given $q\leq n$ and $\mK\in\mathcal{C}_{q}$, we define

\begin{equation}\label{eq:Hinf-norm}
J_{\infty,q}(\mK)
= \|\mathbf{T}_{zd}(\mK,\cdot)\|_{\mathcal{H}_\infty}
= \sup_{\omega\in\mathbb{R}}\sigma_{\max}(\mathbf{T}_{zd}(\mK,j\omega)),
\end{equation}
which is also equal to $\sqrt{\mathfrak{J}_{\infty}}$ for the closed-loop system with the dynamic policy $\mK$. We see that $J_{\infty,q}$ gives a function defined over $\mathcal{C}_q$ with values in $(0,+\infty)$. The $\mathcal{H}_\infty$ robust control~\cref{eq:Hinf} can then be reformulated into a policy optimization problem:
\begin{equation} \label{eq:Hinf_policy_optimization}
    \begin{aligned}
        \min_{\mK} \quad &J_{\infty,n}(\mK) \\
        \text{subject to} \quad& \mK \in {\mathcal{C}}_n. % \;~\cref{eq:LyapunovX}. 
    \end{aligned}
\end{equation}
where we only consider optimizing over $\mathcal{C}_n$ for similar arguments as in LQG policy optimization.

Instead of using the definition~\cref{eq:Hinf-norm}, there are state-space characterizations (e.g., the bounded real lemma in \cref{lemma:bounded_real}) for computing $J_{\infty,q}(\mK)$ for $\mK\in\mathcal{C}_q$. However, unlike the LQG cost that can be evaluated via solving Lyapunov equations \Cref{eq:Lyapunov-LQG}, evaluating $J_{\infty,q}(\mK)$ typically requires solving LMIs that admit no closed-form solutions.
%As we shall see in \Cref{section:LQG,section:Hinf-control}, this seemly minor difference has deep consequences on the optimization landscape of LQG optimal control and $\mathcal{H}_\infty$ robust control (e.g., smooth versus nonsmooth optimization). 
We also note that the classical $\mathcal{H}_2$ and $\mathcal{H}_\infty$ control often considers a general LTI system \cite{zhou1996robust}; see \Cref{appendix:h2_hinf}. It is straightforward to extend our main results to the general situation; %~\cref{eq:Dynamics_general};
see our conference report on $\mathcal{H}_\infty$ control \cite{tang2023global}.

\subsection{Problem statements}

The classical solutions for LQG control \cref{eq:LQG} and $\mathcal{H}_\infty$ robust control \cref{eq:Hinf} explicitly depend on the system model $A, B, C$ in two different ways (see \Cref{theorem:LQG-riccati-solution,theorem:Hinf-riccati}): one in formulating an observer-based controller, % (see \cref{eq:LQGcontroller} or \cref{eq:hinf-suboptimal}), 
and another in solving Riccati equations % (\cref{eq:Riccati} or \cref{eq:riccati-hinf}) 
to obtain feedback and observer~gains.  %Recently, there is a renewed interest in applying model-free policy gradient methods for solving a range of control problems;  %, such as standard LQR \cite{fazel2018global,mohammadi2021convergence}, Markovian jump LQR \cite{jansch2022policy}, distributed LQR \cite{li2022distributed}, finite-horizon linear quadratic control \cite{hambly2021policy,furieri2020learning},  dynamic filter \cite{umenberger2022globally}, and LQG \cite{zheng2021analysis};
% see a recent survey~\cite{hu2023toward}.  
% 
%These methods view classical control problems from a modern optimization perspective
% and aim to directly optimize control policies based on system observations with no explicit knowledge of
% the underlying model. 
% 
To avoid the explicit dependence on model parameters, we consider the policy optimization formulations \cref{eq:LQG_policy_optimization} for LQG optimal control and \cref{eq:Hinf_policy_optimization} for $\mathcal{H}_\infty$ robust control, where the feedback policy is parameterized by $\mK =(A_\mK, B_\mK, C_\mK, D_\mK)$ in \cref{eq:Dynamic_Controller}. 
The idea of direct policy search is to start from an initial policy $\mK_0 \in \mathcal{C}_n$ and conduct the iteration $\mK_{t+1} = \mK_t + \alpha_t F_t, t\geq 0$, where $\alpha_t >0$ is a step size and $F_t$ is a search direction, such that the cost $J_{\LQG,n}(\mK) $ or $ J_{\infty,n}(\mK)$ is gradually improved. 
In addition to being more amenable to model-free %learning-based 
    control, it is observed that policy optimization is also more flexible as evidenced by the advances in modern reinforcement learning, and more scalable as it directly searches policies using computationally cheap iterates.% \cite{hu2023toward}.    
 
However, both policy optimization  \cref{eq:LQG_policy_optimization} and \cref{eq:Hinf_policy_optimization} are naturally nonconvex problems, making theoretical guarantees for direct policy search challenging.  
In this paper, we highlight some benign nonconvex landscape properties of policy optimization \cref{eq:LQG_policy_optimization} and \cref{eq:Hinf_policy_optimization}, which allow us to certify global optimality for a class of stationary points. In particular, we focus on the following two topics:   

% nonconvexity, global optimality, landscape results

\vspace{-1mm}

\begin{enumerate}
    \item \textit{Smooth and nonconvex LQG control}, which will be studied in \Cref{section:LQG}. It is well-known that many linear quadratic control problems (including the LQG control $J_{\LQG,q}(\mK)$) are smooth and infinitely differentiable over their domains \cite{furieri2020learning,fazel2018global,jansch2022policy,fatkhullin2021optimizing}. Furthermore, many state-feedback cases (such as the standard LQR) have no spurious stationary points, and satisfy favorable coercivity and gradient dominance properties. However, landscape results for partially observed cases such as the LQG control are much fewer, with notable exceptions in \cite{zheng2021analysis,zheng2022escaping,mohammadi2021lack}. As our first main contribution, we will reveal the hidden convexity in LQG policy optimization \cref{eq:LQG_policy_optimization}, and certify the global optimality for a class of stationary points (which we call \textit{non-degenerate} policies). 

    \item \textit{Nonsmooth and nonconvex robust control}, which will be studied in \Cref{section:Hinf-control}. Unlike LQR or LQG,  it is known that the $\mathcal{H}_\infty$ cost is nonsmooth with respect to its argument due to two possible sources of non-smoothness: One from taking the largest singular of complex matrices, and the other from maximization over all the frequencies $\omega \in \mathbb{R}$. % (the non-smoothness can also be seen from the maximum operation \cref{eq:Hinf-cost-def} in the time domain). 
    Direct policy search for $\mathcal{H}_\infty$ control has been used in earlier studies~\cite{apkarian2006nonsmooth,apkarian2006nonsmooth2,apkarian2008mixed}, but no optimality~guarantees are given. Only very recently, a global optimality guarantee is given for state-feedback $\mathcal{H}_\infty$ control \cite{guo2022global}. As our second main contribution, we will also reveal hidden convexity in nonsmooth $\mathcal{H}_\infty$ policy optimization \cref{eq:Hinf_policy_optimization}, and show that for a class of \textit{non-degenerate} policies, all Clarke stationary points in \Cref{eq:Hinf_policy_optimization} are globally optimal. 

\end{enumerate}

\vspace{-1mm}

To support our results for both LQG and $\mathcal{H}_\infty$ control, we will formally introduce a notion of non-degenerate stabilizing policies in \Cref{section:LMIs-h2-hinf-norm}. The proofs of our main global optimality characterizations rely on a new \ECL{} framework which is the main topic of Part II of this paper.

% Section 3: LMIs and Non-degenerate controllers
\section{Strict/Nonstrict LMIs and Non-degenerate Policies} \label{section:LMIs-h2-hinf-norm}

In this section, we give an overview of LMI characterizations for both the LQG cost in \cref{eq:LQG_policy_optimization} and $\mathcal{H}_\infty$ cost in \Cref{eq:Hinf_policy_optimization}, which play a key role in our analysis based on convex reformulations \cite{scherer1997multiobjective,scherer2000linear}. We will also highlight some subtleties between strict and nonstrict LMIs, which motivate our definition of a general class of \textit{non-degenerate} policies.

\subsection{LMIs for smooth LQG cost and nonsmooth $\mathcal{H}_\infty$ cost} \label{subsection:smooth-LQG-cost}
As shown in \cref{eq:H2-norm}, the LQG cost for a strictly proper stabilizing policy $\mK \in {\mathcal{C}}_{q,0}$ can be characterized by the $\mathcal{H}_2$ norm of the transfer function $\mathbf{T}_{zd}(\mK, s)$. Although the $\mathcal{H}_2$ (and $\mathcal{H}_\infty$) norm of a stable transfer function can, in principle, be computed from their definitions, we here summarize some alternative characterizations, involving Lyapunov equations and LMIs. %commonly used methods to compute %the $\mathcal{H}_2$ norm of stable transfer functions.

\begin{lemma} \label{lemma:H2norm}
    Consider a transfer function $\mathbf{G}(s) = C(sI - A)^{-1}B$, where $A \in \mathbb{R}^{n\times n}$ is stable and $C \in \mathbb{R}^{p \times n}, B \in \mathbb{R}^{n \times m}$. The following statements hold.
    \begin{enumerate}
        \item (Lyapunov equations). We have $\|\mathbf{G}\|_{\mathcal{H}_2}^2 = \mathrm{trace}(B^\tr L_{\mathrm{o}} B) = \mathrm{trace}(C L_{\mathrm{c}} C^\tr) $, where $L_{\mathrm{o}}$ and $L_{\mathrm{c}}$ are observability and controllability Gramians which can be obtained from the Lyapunov equations
        \begin{subequations} \label{eq:Lyapunov-equations-H2norm}
            \begin{align}
                AL_{\mathrm{c}} + L_{\mathrm{c}}A^\tr + BB^\tr &= 0, \label{eq:Lyapunov-equations-H2norm-a}\\
                A^\tr L_{\mathrm{o}} + L_{\mathrm{o}}A + C^\tr C & = 0.  \label{eq:Lyapunov-equations-H2norm-b}
            \end{align}
        \end{subequations}
        \item (Strict LMI). $\|\mathbf{G}\|_{\mathcal{H}_2} < \gamma$ if and only if there exist $P\in \mathbb{S}^n_{++}$ and $\Gamma \in \mathbb{S}^p_{++}$%$P \succ 0$ and $\Gamma \succ 0$
        such that the following strict LMI is feasible:
        %\begin{subequations}
            \begin{align} \label{eq:strict-LMI-H2norm}
                %AX + XA^\tr + BB^\tr &\prec 0, \quad \mathrm{trace}(CXC^\tr) < \gamma^2\\
                % \begin{bmatrix} AP+PA^\tr & PC^\tr \\ CP & -\gamma I \end{bmatrix}&\prec 0, \;\; \begin{bmatrix} P & B \\ B^\tr & \Gamma \end{bmatrix}\succ 0,\;\; \mathrm{trace}(\Gamma)<\gamma.
                \begin{bmatrix} A^\tr P+PA & PB \\ B^\tr P & -\gamma I \end{bmatrix}&\prec 0, \;\; \begin{bmatrix} P & C^\tr \\ C & \Gamma \end{bmatrix}\succ 0,\;\; \mathrm{trace}(\Gamma)<\gamma.
            \end{align}
        %\end{subequations}
        \item (Nonstrict LMI). $\|\mathbf{G}\|_{\mathcal{H}_2} \leq \gamma$ if there exist $P \in \mathbb{S}^n_{++}$ and $\Gamma \in \mathbb{S}^p_{+}$ %{$P \succ 0$ and $\Gamma \succeq 0$}
        such that the following nonstrict LMI is feasible:
        %\begin{subequations}
            \begin{align} \label{eq:nonstrict-lmi-h2}
                %AX + XA^\tr + BB^\tr &\prec 0, \quad \mathrm{trace}(CXC^\tr) < \gamma^2\\
                % \begin{bmatrix} AP+PA^\tr & PC^\tr \\ CP & -\gamma I \end{bmatrix}& \preceq 0, \;\; \begin{bmatrix} P & B \\ B^\tr & \Gamma \end{bmatrix} \succeq 0,\;\; \mathrm{trace}(\Gamma) \leq \gamma.
                \begin{bmatrix} A^\tr P+PA & PB \\ B^\tr P & -\gamma I \end{bmatrix}&\preceq 0, \;\; \begin{bmatrix} P & C^\tr \\ C & \Gamma \end{bmatrix}\succeq 0,\;\; \mathrm{trace}(\Gamma)\leq\gamma.
            \end{align}
            The converse holds if $(A, B)$ is controllable\footnote{The controllability of $(A, B)$ cannot be removed (\Cref{lemma:controllability-A-B-H2-norm}), but this nonstrict LMI \cref{eq:nonstrict-lmi-h2} does not require the observability of $(C, A)$. If $(C, A)$ is also observable, the solution to \cref{eq:nonstrict-lmi-h2} when $\gamma = \|\mathbf{G}\|_{\mathcal{H}_2}$ is unique (\Cref{lemma:unique-solution-H2}).}.
        %\end{subequations}
    \end{enumerate}
\end{lemma}

It is clear that \Cref{lemma:LQG_cost_formulation1} can be considered as a direct application of the first result in \Cref{lemma:H2norm}. This lemma is classical in control \cite{zhou1996robust,scherer2000linear,dullerud2013course}. To clarify some subtleties between strict and nonstrict LMIs \cref{eq:strict-LMI-H2norm} - \cref{eq:nonstrict-lmi-h2}, we reproduce a short proof of \Cref{lemma:H2norm} in \Cref{app:h2-hinf-norms}.
%
%\begin{remark} \label{remark:non-uniqueness-LMIs-H2}
    The proof in \Cref{app:h2-hinf-norms} is constructive, i.e., we shall explicitly construct solutions satisfying \cref{eq:strict-LMI-H2norm,eq:nonstrict-lmi-h2} from the solution to \cref{eq:Lyapunov-equations-H2norm}. While the solution to \cref{eq:Lyapunov-equations-H2norm} is unique, the solutions to \cref{eq:strict-LMI-H2norm} are not unique given any $\gamma > \|\mathbf{G}\|_{\mathcal{H}_2}$. Indeed, the solution set to \cref{eq:strict-LMI-H2norm} is convex and open; any sufficiently small perturbation to a feasible point remains feasible. This perturbation perspective, while somewhat less explicitly stated in the literature, is one crucial step in many synthesis tasks using strict LMIs (see, e.g., \cite{scherer1997multiobjective,dullerud2013course}). Meanwhile, although not obvious, the solution set to the non-strict LMI \cref{eq:nonstrict-lmi-h2} is in general not a singleton even when setting $\gamma = \|\mathbf{G}\|_{\mathcal{H}_2}$ (see \Cref{example:non-unique-certificate}). % This non-uniqueness also gives us flexibility when performing controller synthesis.
    {If the system $(C, A)$ is also observable, then the solution to \cref{eq:nonstrict-lmi-h2} is unique when setting $\gamma = \|\mathbf{G}\|_{\mathcal{H}_2}$ (see \Cref{lemma:unique-solution-H2}). }  % \hfill \qed
% \end{remark}

We next present the celebrated \textit{bounded real lemma} that gives a state-space characterization for the $\mathcal{H}_\infty$ norm of a stable transfer function using LMIs.

\begin{lemma}[Bounded real lemma]
\label{lemma:bounded_real}
 Consider $\mathbf{G}(s) = C(sI - A)^{-1}B+D$, where $A \in \mathbb{R}^{n\times n}$ is stable and $C \in \mathbb{R}^{p \times n}, B \in \mathbb{R}^{n \times m}, D \in \mathbb{R}^{p \times m}$.
Let $\gamma>0$ be arbitrary. The following statements hold.

\begin{enumerate}[leftmargin=12pt,labelwidth=8pt,labelsep=4pt]
\item (Strict version, \cite[Lemma 7.3]{dullerud2013course}) $\|\mathbf{G}\|_{\mathcal{H}_\infty}< \gamma$ if and only if there exists a symmetric matrix $P$
%$P\succ 0$
such that
\begin{equation} \label{eq:strict-hinf}
\begin{bmatrix}
A^\tr P + P A & PB & C^\tr \\
B^\tr P & -\gamma I & D^\tr \\
C & D & -\gamma I
\end{bmatrix}\prec 0.
\end{equation}

\item (Nonstrict version,
\cite[Section 2.7.3]{boyd1994linear}) $\|\mathbf{G}\|_{\mathcal{H}_\infty}\leq \gamma$ if there exists a symmetric matrix $P$ %$P\succ 0$
such that
\begin{equation} \label{eq:non-strict-hinf}
    \begin{bmatrix}
A^\tr P + P A & PB & C^\tr \\
B^\tr P & -\gamma I & D^\tr \\
C & D & -\gamma I
\end{bmatrix}\preceq 0.
\end{equation}
{The converse holds if $(A,B)$ is controllable\footnote{%Similar to the $\mathcal{H}_2$ norm case,
The controllability of $(A, B)$ cannot be removed; see \Cref{example:hinf-controllability}. This result does not require $(C, A)$ to be observable. If $(C, A)$ is observable and $A$ is stable, any solution $P$ satisfying \cref{eq:strict-hinf,eq:non-strict-hinf} is strictly positive definite. %; thus $P \succ 0$ is redundant.
}.}
\end{enumerate}
\end{lemma}

A more complete list of the strict version of \Cref{lemma:bounded_real} is summarized in \cite[Corollary 12.3]{zhou1998essentials}. This lemma is also classical in control, and one general version -- the Kalman-Yakubovich-Popov (KYP) lemma is arguably one of the most fundamental tools in systems theory \cite{zhou1996robust,scherer2000linear,dullerud2013course}.
The sufficiency of LMIs \cref{eq:strict-hinf,eq:non-strict-hinf} is relatively easy to prove, which involves a control-theoretic interpretation of the variable $P$ to construct a storage function $V(x)=x^\tr P x$ certifying a dissipativity condition. The necessity of LMIs \cref{eq:strict-hinf,eq:non-strict-hinf} is more involved to prove. In \Cref{app:h2-hinf-norms}, we adopt convex duality techniques in \cite{rantzer1996kalman,balakrishnan2003semidefinite,you2016direct} to highlight some subtleties between strict and non-strict LMIs.

\subsection{Non-degenerate policies}

We can use the LMI characterizations in \Cref{lemma:H2norm,lemma:bounded_real} to formulate nonconvex (bilinear) matrix inequalities when fixing the full-order dynamic policies for $\mathcal{H}_2$ and $\mathcal{H}_\infty$ synthesis. It was realized in the early 1990s that those nonconvex matrix inequalities can be convexified using a sophisticated change of variables (see, for example, \cite{scherer2000linear,gahinet1994linear} and the references therein). It is also possible to derive classical Riccatti inequalities from those matrix inequalities \cite{liu2006simple}. All the previous studies \cite{scherer2000linear,gahinet1994linear} only use the strict LMIs \cref{eq:strict-LMI-H2norm,eq:strict-hinf}. As we will see in Part II, strict LMIs can only characterize \textit{strict epigraphs} of $J_{\LQG,n}(\mK)$ or $J_{\infty,n}(\mK)$, but fail to characterize global optimality~directly.

In this work, we aim to analyze the LQG and $\mathcal{H}_\infty$ cost functions instead of their upper bounds, and thus our subsequent results require non-strict LMIs \cref{eq:nonstrict-lmi-h2,eq:non-strict-hinf}. However, the non-strict LMIs \cref{eq:nonstrict-lmi-h2,eq:non-strict-hinf} do not present sufficient and necessary conditions for $\mathcal{H}_2$ and $\mathcal{H}_\infty$ characterizations. This fact will affect our analysis framework based on \ECL{} that will be formally introduced in Part II of our paper.
This subsection introduces a broad class of \textit{non-degenerate} policies that are closely related to the nonstrict LMIs \cref{eq:nonstrict-lmi-h2,eq:non-strict-hinf}. We also present their physical interpretations.

\paragraph{LQG optimal control}

The nonstrict LMI in \Cref{lemma:H2norm} is in general not equivalent to, but only provides a sufficient condition for $\|\mathbf{G}\|_{\mathcal{H}_2} \leq \gamma$ when $A$ is stable. Consequently, given $\mK\in\mathcal{C}_{n,0}$ and $\gamma>0$, it is only a sufficient condition for $J_{\LQG,n}(\mK)\leq\gamma$ that there exist matrices $P$ and $\Gamma$ satisfying the followng matrix inequalities:
\begin{subequations} \label{eq:LQG-Bilinear}
    \begin{align}
        & \begin{bmatrix} A_{\mathrm{cl}}(\mK)^\tr P+PA_{\mathrm{cl}}(\mK) & PB_{\mathrm{cl}}(\mK) \\ B_{\mathrm{cl}}(\mK)^\tr P & -\gamma I \end{bmatrix} \preceq 0, \label{eq:LQG-Bilinear-a}\\
        &\begin{bmatrix} P & C_{\mathrm{cl}}(\mK)^\tr \\ C_{\mathrm{cl}}(\mK) & \Gamma \end{bmatrix} \succeq 0,\;  {P \succ 0}, \; \mathrm{trace}(\Gamma) \leq \gamma,  \label{eq:LQG-Bilinear-b}
    \end{align}
\end{subequations}
where $A_{\mathrm{cl}}(\mK), B_{\mathrm{cl}}(\mK), C_{\mathrm{cl}}(\mK)$ are the closed-loop matrices of the policy $\mK$, defined in~\Cref{eq:transfer-function-zd}. We are particularly interested in those policies $\mK\in\mathcal{C}_{n,0}$ for which~\eqref{eq:LQG-Bilinear} can be satisfied with $\gamma= J_{\LQG,n}(\mK)$, as these values of $J_{\LQG,n}(\mK)$ can then be directly characterized by~\eqref{eq:LQG-Bilinear}, which would further allow us to utilize convex reformulations of optimal control problems for the analysis of $J_{\LQG,n}(\mK)$.

Note that the Lyapunov variable $P$ in \cref{eq:LQG-Bilinear} is of dimension $2n \times 2n$. Upon partitioning
\begin{equation} \label{eq:P-blocks}
P = \begin{bmatrix} P_{11} & P_{12} \\ P_{12}^\tr & P_{22} \end{bmatrix} \in \mathbb{S}^{2n}_{++},
\end{equation}
we will see later that the invertibility of $P_{12}$ plays an important role in ``convexifying'' the bilinear matrix inequalities \cref{eq:LQG-Bilinear}. Motivated by these observations, we define a class of \textit{non-degenerate policies} that admit a Lyapunov variable $P\in \mathbb{S}^{2n}_{++}$ with an off-diagonal block $P_{12}$ having full rank to certify its LQG cost.

\begin{definition} \label{definition:non-degenerate-controller-LQG}
    A full-order dynamic policy $\mK \in {\mathcal{C}}_{n,0}$ is called \textit{non-degenerate for LQG} if there exists a $P \in \mathbb{S}^{2n}_{++}$ with $P_{12} \in \mathrm{GL}_{n}$ such that \Cref{eq:LQG-Bilinear-b,eq:LQG-Bilinear-a} hold with $\gamma = J_{\LQG,n}(\mK)$.
\end{definition}

Precisely, we define the set of non-degenerate policies for LQG as
\begin{equation} \label{def:LQG-Cnd}
    {\mathcal{C}}_{\mathrm{nd},0} :=
  \left\{\mK =\begin{bmatrix}
    D_{\mK} & C_{\mK} \\
    B_{\mK} & A_{\mK}
    \end{bmatrix} \in {\mathcal{C}}_{n,0}
    %\in \mathbb{R}^{(m+n) \times (p+n)}
    \left|
    \begin{aligned} & \;\; \exists P \in \mathbb{S}^{2n}_{++}\; \mathrm{with} \; P_{12} \in \mathrm{GL}_{n},\; \Gamma\; \mathrm{such~that}\\ & \; \Cref{eq:LQG-Bilinear-a}\; \mathrm{and}\; \Cref{eq:LQG-Bilinear-b}\; \textrm{hold with} \; \gamma = J_{\LQG,n}(\mK)  \end{aligned} \right.\right\}.
\end{equation}
%By definition, ${\mathcal{C}}_{\mathrm{nd},0}$ is a subset of all stabilizing policies ${\mathcal{C}}_{n,0}$.
%
% The proof of ${\mathcal{C}}_{\mathrm{nd},0}  \subseteq {\mathcal{C}}_{n,0}$ uses a converse result of \cref{lemma:H2norm}, and
 % We use a numerical example below to illustrate this property.
%
% By definition, it is clear that ${\mathcal{C}}_{\mathrm{nd},0}  \subseteq {\mathcal{C}}_{n,0}$.
%Moreover, {the class of non-degenerate policies includes so-called \textit{informative policies} as special cases.}
%
Non-degeneracy of dynamic policies for LQG has a nice physical interpretation. Given $\mK \in {\mathcal{C}}_{n,0}$, it can be shown that the system and policy state variables $(x(t), \xi(t))$ follow a Gaussian process with mean and covariance
\begin{equation} \label{eq:co-variance}
\begin{aligned}
    \lim_{t \rightarrow \infty} \mathbb{E}\left(\begin{bmatrix} x(t) \\ \xi(t) \end{bmatrix} \right) = 0, \quad  X_{\mK} = \begin{bmatrix}X_{\mK,11} & X_{\mK,12} \\ X_{\mK,12}^\tr & X_{\mK,22} \end{bmatrix}:= \lim_{t \rightarrow \infty} \mathbb{E}\left(\begin{bmatrix} x(t) \\ \xi(t) \end{bmatrix}\begin{bmatrix} x(t) \\ \xi(t) \end{bmatrix}^\tr \right),
    % &= \lim_{t \rightarrow \infty} \int_{0}^t e^{A_{\mathrm{cl},\mK}(t - \tau)}\begin{bmatrix}W & \\ & B_{\mK} V B_{\mK}^\tr  \end{bmatrix}e^{A_{\mathrm{cl},\mK}^\tr (t - \tau)} d\tau\\
    % &= \int_{0}^\infty e^{A_{\mathrm{cl},\mK}t}\begin{bmatrix}W & \\ & B_{\mK} V B_{\mK}^\tr  \end{bmatrix}e^{A_{\mathrm{cl},\mK}^\tr t} dt.
\end{aligned}
\end{equation}
where $X_{\mK}$ is the unique positive semidefinite solution to the Lyapunov equation~\cref{eq:LyapunovX}; see \Cref{appendix:covariance_LQG} for a proof. Following the terminology in \cite{umenberger2022globally} (which only focuses on Kalman filtering), we say a policy $\mK\in\mathcal{C}_{n,0}$ is \emph{informative}
if the steady-state correlation matrix $X_{\mK,12}:=\lim_{t\to\infty}\mathbb{E}\left(x(t)\xi^\tr(t)\right)$
has full rank. Our next result shows that a policy in ${\mathcal{C}}_{n,0}$ is non-degenerate for LQG if and only if it is informative.

\begin{theorem} \label{proposition:informativity-non-degenerate}
Given $\mK \in {\mathcal{C}}_{n,0}$, we have $\mK\in {\mathcal{C}}_{\mathrm{nd},0}$ if and only if $\mK$ is \emph{informative} in the sense that
\[X_{\mK,12}:=\lim_{t\to\infty}\mathbb{E}\left(x(t)\xi^\tr(t)\right)
\]
has full rank.
\end{theorem}

The sufficiency part of \Cref{proposition:informativity-non-degenerate} can be proved by the following idea: We first show that the full rank of $X_{\mK,12}$ leads to the controllability of $(A_\mK,B_{\mK})$, which further implies that $X_{\mK}$ from~\cref{eq:LyapunovX} is positive definite (cf. \cite[Lemma 4.5]{zheng2021analysis}). As a result, $P = \gamma X_{\mK}^{-1}$ with $\gamma = J_{\LQG,n}(\mK)$ is well-defined and we can show it satisfies \Cref{eq:LQG-Bilinear-a,eq:LQG-Bilinear-b}. Further, $P_{12} =\gamma(X_{\mK}^{-1})_{12}$ is invertible since $X_{\mK,12}$ is, which completes the proof of sufficiency. The necessity part is more technically involved, and we provide all the proof details in \Cref{appendix:informative}.

We can further show that non-degenerate policies are ``generic'' in the sense that the complement set ${\mathcal{C}}_{n,0}\backslash{\mathcal{C}}_{\mathrm{nd},0}$ has measure zero. The proof %of the second statement
relies on the fact that the zero-set of a nontrivial real-analytic function has measure zero (cf. \cite{mityagin2015zero}). The details are technical, and we postpone them to \cref{appendix:measure-zero-LQG}. % We summarize these results into the following proposition.
\begin{theorem} \label{proposition:measure-zero-LQG}
    %We have ${\mathcal{C}}_{\mathrm{nd},0}  \subseteq {\mathcal{C}}_{n,0}$, and
    The set difference ${\mathcal{C}}_{n,0}\backslash{\mathcal{C}}_{\mathrm{nd},0}$ has measure zero. % ({\color{red} the measure zero part seems non-trivial}).
\end{theorem}

The following example provides numerical evidence illustrating that ${\mathcal{C}}_{n,0}\backslash{\mathcal{C}}_{\mathrm{nd},0}$ has measure zero.

\begin{figure}[t]
    \centering
    \setlength{\abovecaptionskip}{4pt}
    \hspace{4mm} \begin{subfigure}{0.48\textwidth}
        \includegraphics[width=0.9 \textwidth]{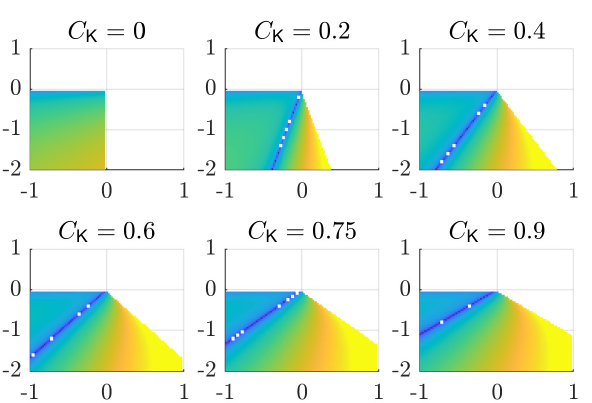}
        \caption{LQG case in \Cref{example:LQG-1}}     \label{fig:LQG-heatmap}
    \end{subfigure}
    %\hfill
    \begin{subfigure}{0.48\textwidth}
        \includegraphics[width=0.9 \textwidth]{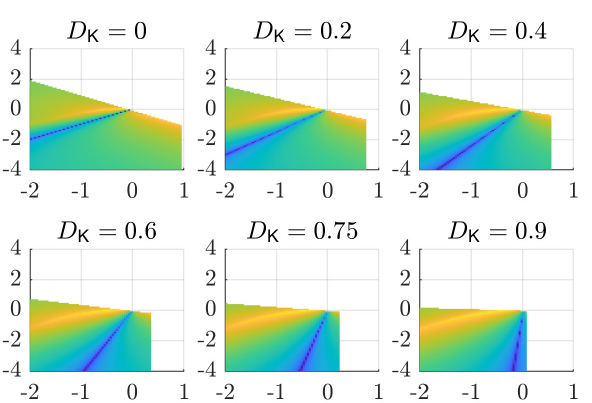}
        \caption{$\mathcal{H}_\infty$ case in \Cref{example:Hinf-1}}     \label{fig:sim_results}
    \end{subfigure}
    \caption{Illustration that the set of degenerate policies has measure zero. (a) The case of LQG control \cref{def:LQG-Cnd}: Heatmaps of $\ln|P_{12}|$ for different values of $C_\mK$. %The $x$-axes and $y$-axes represent $A_\mK$ and $B_\mK$ respectively.
    (b) The case of $\mathcal{H}_\infty$ control \cref{def:Hinf-Cnd}: Heatmaps of $\ln|P_{12}|$ for different values of $D_{\mathsf{K}}$. In both subfigures, the $x$-axes and $y$-axes represent $A_{\mK}$ and $B_{\mK}$ respectively.  Points with very low values of $\ln|P_{12}|$ (i.e., points whose corresponding $P_{12}$ are very close to $0$) are colored in dark blue, and they roughly form a line passing through $(0,0)$ in each sub-figure which has measure zero.
    }
\end{figure}

\begin{example}[${\mathcal{C}}_{n,0}\backslash{\mathcal{C}}_{\mathrm{nd},0}$ has measure zero] \label{example:LQG-1} %{\color{red} To do: this example needs to be updated.}
Consider an open-loop stable dynamic system~\cref{eq:Dynamic} with
$
A=-1,\;  B=1, \; C=1.
$
We choose $Q = R = 1, W = V = 1$ for the LQG in \cref{eq:LQG_policy_optimization}.
Our task is to numerically search for points in ${\mathcal{C}}_{n,0}\backslash{\mathcal{C}}_{\mathrm{nd},0}$, and inspect whether they form a set of measure zero.
% Note that the dynamic controllers with the same value of $B_\mK C_\mK$ will be similarity transformations of each other and hence have the same LQG cost.
We imposes the constraints $A_\mK \in [-1,1], B_\mK \in [-2,2]$.
We first generate a set of points $\{\mK_j\}_{j=1}^N$ by discretizing the region $[-1,1]\times[-2,2]$ into a spacial grid with $N=101\times101$ points that are equally spaced.
Then for each $j=1,\hdots,N,$ we numerically compute %{\color{red}(numerical error?)}
$\gamma_j=J_{\LQG,1}(\mK_j)$ and try to construct $P_j \succeq 0$ by solving a Lyapunov equation (cf. \Cref{proposition:informativity-non-degenerate}). %{\color{red}using the LMI-based approach} such that \Cref{eq:LQG-Bilinear-a} and \Cref{eq:LQG-Bilinear-b} holds with $\mK=\mK_j, P=P_j$ and $\gamma=\gamma_j$.
We then check whether the smallest eigenvalue of $P_j$ is sufficiently bounded away from zero (say greater than or equal to $10^{-5}$), and record the value of $(P_j)_{12}$.
%
%For some $j$, we can not find matrices $P_j$ satisfying $P_j\succeq0$ such that \Cref{eq:LQG-Bilinear-a} and \Cref{eq:LQG-Bilinear-b} holds with $\mK=\mK_j, P=P_j$ and $\gamma=\gamma_j$.
% For the controller corresponding to these $j$, they also belong to ${\mathcal{C}}_{n,0}\backslash{\mathcal{C}}_{\mathrm{nd},0}$.
\Cref{fig:LQG-heatmap} illustrates several heatmaps of $\ln|P_{12}|$ with fixed $C_\mK$ and varying $(A_\mK,B_\mK)$, generated from the recorded values of $\{\ln|(P_j)_{12}|\}_{j=1}^N$. Points with low values of $\ln|P_{12}|$ (i.e., points whose corresponding $P_{12}$ are close to $0$) are colored in dark blue, and we observe that they roughly form a line passing through $(0,0)$ in each sub-figure, which has measure zero.
\hfill \qed
\end{example}

\paragraph{$\mathcal{H}_\infty$ robust control}

Similar to the LQG case, for the $\mathcal{H}_\infty$ policy optimization problem~\cref{eq:Hinf_policy_optimization}, we focus on the policies $\mK\in\mathcal{C}_{n}$ such that the following set of matrix inequalities are satisfied by some matrix $P$ with $\gamma=J_{\infty,n}(\mK)$:
\begin{subequations} \label{eq:Hinf-Bilinear}
    \begin{align}
&     \begin{bmatrix}
A_{\mathrm{cl}}(\mK)^\tr P + P A_{\mathrm{cl}}(\mK) & PB_{\mathrm{cl}}(\mK) & C_{\mathrm{cl}}(\mK)^\tr \\
B_{\mathrm{cl}}(\mK)^\tr P & -\gamma I & D_{\mathrm{cl}}(\mK)^\tr \\
C_{\mathrm{cl}}(\mK) & D_{\mathrm{cl}}(\mK) & -\gamma I
\end{bmatrix}\preceq 0, \label{eq:Hinf-Bilinear-a}\\
        & {P \in \mathbb{S}^{2n}_{++}},  \label{eq:Hinf-Bilinear-b}
    \end{align}
\end{subequations}
where $A_{\mathrm{cl}}(\mK), B_{\mathrm{cl}}(\mK), C_{\mathrm{cl}}(\mK)$ and $D_\mathrm{cl}(\mK)$ are the closed-loop matrices for policy $\mK$, defined in \Cref{eq:transfer-function-zd}. The matrix $P$ has a dimension of $2n$-by-$2n$ and can be partitioned as \cref{eq:P-blocks}.

Similar to \cref{subsection:smooth-LQG-cost}, the invertibility of $P_{12}$ plays an important role in ``convex reformulations'' of \cref{eq:Hinf-Bilinear}. We thus define a class of \textit{non-degenerate policies} for $\mathcal{H}_\infty$ control, which admits a Lyapunov variable $P\in \mathbb{S}^{2n}_{++}$ with an off-diagonal block $P_{12}$ having full rank to certify its $\mathcal{H}_\infty$ cost. % has a full rank .

\begin{definition} \label{definition:Hinf-Cnd}
    A full-order dynamic policy $\mK \in {\mathcal{C}}_n$ is called \textit{non-degenerate} for  $\mathcal{H}_\infty$ control if there exists a $P \in \mathbb{S}^{2n}_{++}$ with $P_{12} \in \mathrm{GL}_{n}$ such that \Cref{eq:Hinf-Bilinear-a} holds with $\gamma = J_{\infty,n}(\mK)$.
\end{definition}

We thus define the set of non-degenerate policies for $\mathcal{H}_\infty$ control as
\begin{equation} \label{def:Hinf-Cnd}
    {\mathcal{C}}_{\mathrm{nd}} :=
  \left\{\mK =\begin{bmatrix}
    D_{\mK} & C_{\mK} \\
    B_{\mK} & A_{\mK}
    \end{bmatrix}
    \in \mathcal{C}_n
    %\mathbb{R}^{(m+n) \times (p+n)}
    \left|
    \begin{aligned}  \;\; \exists P &\in \mathbb{S}^{2n}_{++}\; \mathrm{with} \; P_{12} \in \mathrm{GL}_{n}\;  \\ \; \mathrm{such~that}\; & \Cref{eq:Hinf-Bilinear-a} \; \textrm{holds with} \; \gamma = J_{\infty,n}(\mK)  \end{aligned} \right.\right\}.
\end{equation}
% \begin{equation} \label{def:Hinf-Cnd}
%     {\mathcal{C}}_{\mathrm{nd}} := \big\{\mK \in {\mathcal{C}}_n \mid \exists P \in \mathbb{S}^{2n}_{++}\; \mathrm{with} \; P_{12} \in \mathrm{GL}_{n}\;  \mathrm{such~that}\;  \Cref{eq:Hinf-Bilinear-a} \; \textrm{holds with} \; \gamma = J_{\infty,n}(\mK)  \big\}.
% \end{equation}
%
%{It is not difficult to show that $\mathcal{C}_{\mathrm{nd}} \subseteq \mathcal{C}_n$.
We conjecture that non-degenerate policies are ``generic'' in the sense that the complement set $\mathcal{C}_n\backslash\mathcal{C}_{\mathrm{nd}}$ has measure zero. Unlike the smooth LQG case, a rigorous proof of this conjecture seems challenging, and we leave it to future work. One main difficulty for the proof lies in the non-smoothness of the $\mathcal{H}_\infty$ cost (we cannot use the argument that the zero-set of a nontrivial real-analytic function has measure zero).

\begin{conjecture} \label{conjecture:measure-zero-hinf}
    ${\mathcal{C}}_n \backslash {\mathcal{C}}_{\mathrm{nd}}$ has measure zero in the space $\mathbb{R}^{(m+n) \times (p+n)}$.
\end{conjecture}

%Similar comments to \Cref{remark:strict-non-strict-LMI} also apply to $\mathcal{H}_\infty$ robust control.
We here provide some numerical evidence supporting \Cref{conjecture:measure-zero-hinf}.

% \begin{figure}[t]
% \vspace{3pt}
%     \centering
%     \includegraphics[width=.55\textwidth]{CDC-2023/all_results.pdf}
%     \caption{Heatmaps of $\ln|P_{12}|$ for different values of $D_{\mathsf{K}}$. The $x$-axes and $y$-axes represent $A_{\mK}$ and $B_{\mK}$ respectively. Points with very low values of $\ln|P_{12}|$ (i.e., points whose corresponding $P_{12}$ are very close to $0$) are colored in dark blue, and we can observe that they roughly form a line passing through $(0,0)$ in each sub-figure.}
%     \label{fig:sim_results}
% \end{figure}

\begin{example}[Measure zero of ${\mathcal{C}}_n \backslash {\mathcal{C}}_{\mathrm{nd}}$] \label{example:Hinf-1}
%   We consider the $\mathcal{H}_\infty$ control problem for the LTI system
% \begin{equation}\label{eq:LTI_numerical}
% \begin{aligned}
% \dot{x}(t) =\ & -x(t) + \begin{bmatrix}
%     1 & 0
% \end{bmatrix} w(t) + u(t), \\
% z(t) =\ & \begin{bmatrix} x(t) \\ u(t) \end{bmatrix},
% \qquad y(t) = x(t) + \begin{bmatrix} 0 & 1\end{bmatrix} w(t),
% \end{aligned}
% \end{equation}
% where $x(t),u(t),y(t)\in\mathbb{R}$ and $z(t),w(t)\in\mathbb{R}^2$.
We consider the same problem data as \Cref{example:LQG-1} for the $\mathcal{H}_\infty$ formulation \cref{eq:Hinf_policy_optimization}.
The dynamic feedback policy is parameterized by $\mK={\small \begin{bmatrix}
    D_\mK & C_\mK \\ B_\mK & A_\mK
\end{bmatrix}\in\mathbb{R}^{2\times 2}}$. Our task is to numerically~search for points in $\mathcal{C}_n\backslash\mathcal{C}_{\mathrm{nd}}$, and inspect whether they form a set of measure zero. %Note that dynamic controllers with the same value of $B_\mK C_\mK$ will be similarity transformations of each other. Therefore,
For visualization purposes, we fix $C_\mK=1$ and only examine the set $\{\mK\in\mathcal{C}_n:C_\mK=1\}$. We also restrict $
A_\mK\in[-2,2]$, $B_{\mK}\in$ $[-4,4]$, $D_{\mK}\in[-1.5,1.5]$ when searching over the set $\{\mK\in\mathcal{C}_n:C_\mK=1\}$.
We first generate a set of points $\{\mK_j\}_{j=1}^N$ by discretizing the region $[-2,2]\times[-4,4]\times[-1.5,1.5]$ into a spatial grid with $N=101\times101\times 61$ points that are equally spaced. Then for each $j=1,\ldots,N$, we numerically compute $\gamma_j=J(\mK_j)$, and try to construct $P_j\succeq 0$ based on an LMI approach. %such that $\mathscr{N}(\mK_j,P_j,\gamma_j)\preceq 0$. %\footnote{Due to numerical errors, we can only find an approximate value $\hat{\gamma}_j$ of $J(\mK_j)$. In our numerical experiments, we set the tolerance so that $|\hat{\gamma}_j-J(\mK_j)|/J(\mK_j)<\epsilon$ and find $P_j$ satisfying $\mathscr{N}(\mK_j,P_j,\hat{\gamma}_j/(1-\epsilon))\preceq 0$ instead, where $\epsilon=10^{-9}$. We employ the Riccati-equation-based approach for finding $P_j$ when the associated Riccati equation is well-posed and has a positive definite solution, and turn to the LMI-based approach if the Riccati-equation-based approach does not work.}
We then check whether the minimum eigenvalue of $P_j$ is sufficiently bounded away from zero (say greater than or equal to $10^{-4}$), and record the value of $(P_j)_{12}$.
%
%
%Our numerical experiments show that we can find matrices $P_j$ satisfying $P_j\succ 0$ and $\mathscr{N}(\mK_j,P_j,\gamma_j)\preceq 0$ for all $j$ in the test case.
\Cref{fig:sim_results} illustrates several typical heatmaps of $\ln|P_{12}|$ with fixed $D_\mK$ and varying $(A_{\mK},B_{\mK})$, generated from the recorded values $\{\ln|(P_{j})_{12}|\}_{j=1}^N$. It can be observed from the heatmaps that for each fixed value of $D_{\mK}$, the points with very low values of $\ln|P_{12}|$ seem to lie near a straight line that passes through $(0,0)$, which has indeed measure zero. More computational details can be found in our conference report \cite[Section V]{tang2023global}.
%
% These observations seem to suggest that, for the LTI system~\eqref{eq:LTI_numerical}, the points in $\mathcal{C}\backslash\mathcal{C}_{\mathrm{nd}}$ with $C_\mK=1$ and some fixed $D_\mK$ form a straight line passing through $(0,0)$ with a slope depending on $D_\mK$, and consequently, the set $\mathcal{C}\backslash\mathcal{C}_{\mathrm{nd}}$ could be represented as
% \begin{align*}
% \!\!\left\{
% \!\begin{bmatrix}
% D_\mK & \!\!\!\!C_\mK \\ B_\mK & \!\!\!\!A_\mK
% \end{bmatrix}\!\!\in\!\mathcal{C}\! :
% \cos\theta(D_\mK) \!\cdot\! A_\mK
% +\sin\theta(D_\mK) \!\cdot\! B_\mK C_\mK = 0
% \!\right\}
% \end{align*}
% for some function $\theta(D_\mK)$ of $D_\mK$, which has measure zero.
    \hfill $\square$
\end{example}

\begin{remark}[Physical interpretation in terms of storage functions]
Similar to LQG optimal control, there also exists a physical interpretation for non-degenerate policies in $\mathcal{C}_{\mathrm{nd}}$. In particular, the variable $P$ can be used to define a  storage function $V: \mathbb{R}^{2n} \to \mathbb{R}$, defined by
$$
\begin{aligned}
V(x(t),\xi(t)) &=
\frac{1}{J_{\infty,n}(\mK)}\begin{bmatrix}x(t) \\ \xi(t) \end{bmatrix}^\tr P\begin{bmatrix}x(t) \\ \xi(t) \end{bmatrix} \\
&= \frac{1}{J_{\infty,n}(\mK)}\left(x(t)^\tr P_{11}x(t) + 2 x(t)^\tr P_{12}\xi(t) + \xi(t)^\tr P_{22} \xi(t)\right),
\end{aligned}
$$
which satisfies the following dissipation inequality for the closed-loop system \cref{eq:transfer-function-zd} (see \cref{appendix:proof-Hinf-norm})
$$
\frac{dV(t)}{dt}+ z(t)^\tr z(t) \leq {J_{\infty,n}(\mK)}^2 d(t)^\tr d(t), \quad \forall t \geq 0.
$$
Integrating this inequality over $[0,\infty)$ gives a certification of the worst case perfromance $\|z\|_2^2 \leq {J_{\infty,n}(\mK)}^2 \|d\|^2, \forall d(t)\in \mathcal{L}_2[0,\infty).$ A non-degenerate stabilizing policy $\mK \in \mathcal{C}_{\mathrm{nd}}$ indicates that there exists a storage function $V$ with a full interaction between the system state $x(t)$ and policy state $\xi(t)$ to certify the worst-case $\mathcal{H}_\infty$ performance of the closed-loop system $J_{\infty,n}(\mK)$.  \hfill $\square$
\end{remark}

\begin{remark}[Strict LMIs vs. nonstrict LMIs and internal stability] \label{remark:stability-constraint}
    Our analysis techniques will rely on \Cref{lemma:H2norm,lemma:bounded_real}. We note that the results in \Cref{lemma:H2norm,lemma:bounded_real} assume that the matrix $A$ is stable, while the internal stability \Cref{eq:internallystabilizing} is a constraint when designing policies, i.e. $\mK \in {\mathcal{C}}_{n,0}$.
    % As mentioned in \Cref{footnote:epi-graph}, we have added $P \succ 0$ in \Cref{eq:LQG-Bilinear} to ensure $\mK \in {\mathcal{C}}_{n,0}$.
    There are converse results to guarantee that $A$ is stable (\Cref{lemma:converse-Lyaounov}).
    % Especially, if $(A_{\mathrm{cl}}(\mK), B_{\mathrm{cl}}(\mK))$ is stabilizable, any feasible $P\succ 0$ in \cref{eq:LQG-Bilinear} ensures that $\mK \in {\mathcal{C}}_{n,0}$. However, if $(A_{\mathrm{cl}}(\mK), B_{\mathrm{cl}}(\mK))$ is not stabilizable,
    For example, the feasibility of strict LMI $A^\tr P + P A \prec 0, P \succ 0$ guarantees the stability of~$A$. All the previous results~\cite{gahinet1994linear,scherer1997multiobjective,scherer2000linear,masubuchi1998lmi} rely on strict LMIs and any feasible policy directly guarantees internal stability. Furthermore, any sufficiently small perturbation on a feasible point to \Cref{eq:strict-LMI-H2norm} remains feasible. Thus, if a matrix $P$ satisfies the strict version of \Cref{eq:LQG-Bilinear}, we can always add a small perturbation to $P_{12}$ such that it becomes full rank. This perturbation strategy has been heavily used in classical LMI-based controller synthesis \cite{scherer2000linear}.
    As stated above, we need to use the nonstrict versions \cref{eq:LQG-Bilinear,eq:Hinf-Bilinear}. This brings two difficulties: 1) \cref{eq:LQG-Bilinear} can only guarantee that the real parts of all eigenvalues of $A_{\mathrm{cl}}(\mK)$ are negative or zero; 2) $P_{12}$ of a feasible $P$ may not have full rank. The second difficulty leads to our notion of \textit{non-degenerate policies}. The first point seems minor, but will significantly impact our technical analysis in \ECL; further details can be seen later in Part II of our paper.
    %Therefore, the constraints
    %Moreover, we can show that when $A$ is stable and $(A,B)$ is controllable, all $P$ satisfying \cref{eq:strict-LMI-H2norm,eq:nonstrict-lmi-h2} are strictly positive definite.
    \hfill \qed
    % Some details on the stability constraint $\mK \in {\mathcal{C}}_{n,0}$. The nonstrict inequality in \cref{eq:LQG-Bilinear-a}, but the converse result of Lyapunov equations is the key. Most existing literature uses the strict inequality \cref{eq:strict-LMI-H2norm}, which naturally guarantees the stability of $A$. We need to add some details in the appendix to clarify this seemingly minor point. See \Cref{lemma:converse-Lyaounov}.
\end{remark}

% Section 4: LQG optimal control
\section{Smooth Nonconvex LQG Optimal Control} \label{section:LQG}
In this section, we examine the smooth nonconvex landscape of LQG optimal control and reveal its hidden convexity.
Let us review a known fact that the LQG cost $J_{\LQG,q}(\mK)$ is a real analytical function over its domain ${\mathcal{C}}_{q,0}$, which implies that $J_{\LQG,q}(\mK)$ is infinitely differentiable.
\begin{lemma}[{\cite[Lemma 2.3]{zheng2021analysis}}] \label{lemma:LQG-analytical}
    Fix $q \in \mathbb{N}$ such that $ {\mathcal{C}}_{q,0}\neq\varnothing$. Then, $J_{\LQG,q}(\mK)$ is a real analytic function on ${\mathcal{C}}_{q,0}$.
\end{lemma}

We here recall the notion of coercivity, which is widely used in optimization to ensure that the iterates of first-order algorithms are bounded under mild conditions \cite{beck2017first}.
\begin{definition}[Coercivity] \label{definition:Coercivity}
    A real-valued continuous function $J$ defined on an open set $\mathcal{C}\subseteq\mathbb{R}^d$ is said to be coercive if for any sequence $\{\mK_l\}_{l=1}^\infty\subset \mathcal{C}$, we have $J(\mK_l)\to \infty$ if either $\lim_{l\to\infty}\|{\mK_l}\|=\infty$ or $\lim_{l\to\infty}K_l\in\partial \mathcal{C}$.
\end{definition}

It is known that the LQG cost is nonconvex and not coercive. We summarize these facts below. %; see e.g., \cite{zheng2021analysis} for examples.

\begin{fact} \label{fact:LQG}
The following facts hold.
\begin{enumerate}
    \item The LQG policy optimization problem \Cref{eq:LQG_policy_optimization} is nonconvex;
    \item Its domain ${\mathcal{C}}_{n,0}$ can be disconnected but has at most two path-connected components;
    \item Fix $q \in \mathbb{N}$ such that ${\mathcal{C}}_{q,0}\neq\varnothing$. Then, $J_{\LQG,q}(\mK)$ in \cref{eq:LQG_cost_formulation1} is not coercive.
\end{enumerate}

\end{fact}

The nonconvexity of \Cref{eq:LQG_policy_optimization} is well-known, but the set ${\mathcal{C}}_{n,0}$ being potentially disconnected was only identified recently; we refer the interested reader to \cite[Theorems 3.1-3.3]{zheng2021analysis} for further characterizations and examples.
\Cref{fact:LQG}(3) is not difficult to see since the LQG cost is invariant with respect to similarity transformations. Thus, it is easy to find a sequence of policies $\{\mK_l\}\subset{\mathcal{C}}_{n,0}$ where $\lim_{l\to\infty}\|{\mK_l}\|_F=\infty$ such that $J_{\LQG,n}(\mK_l)< \infty$.

Despite being analytical over its domain, this section first highlights that the LQG cost has complicated ``discontinuous'' behavior around its boundary, which is caused by uncontrollable or unobservable policies. We then present one main technical result that any non-degenerate stationary points are global minimum, which includes existing characterizations in \cite{zheng2021analysis} as a special case. %Our technical proof is based on a new \texttt{ECL} framework, which will be provided in Part II of this paper.   % We finally reveal benign nonconvex landscapes for LQG optimization over ${\mathcal{C}}_{\mathrm{nd},0}$, which is essentially a convex problem under disguise. %{\color{red}We further prove that for any non-degenerate stabilizing policies, approximate stationarity implies approximate global optimality.}
We also discuss a class of degenerate policies for LQG at the end of this section.

\subsection{``Discontinuous'' behavior of the LQG cost} \label{subsection:basic-properties H2}

\Cref{fact:LQG}(3) reveals that the LQG cost is not coercive. Indeed,  \cite[Example 4]{zheng2021analysis} confirms that the LQG cost
can converge to a finite value even when the policy $\mK$ goes to the boundary of ${\mathcal{C}}_{n,0}$.
We here further illustrate that the LQG cost is ``discontinuous'' on some boundary points in {the sense that the LQG function admits different limit values}.
%the function exhibits a wide range of values in the neighborhood of those points.

\begin{example}[``Discontinuous'' behavior of the LQG cost] \label{example:discontinuity-LQG-boundary}
    Consider an open-loop stable dynamic system~\cref{eq:Dynamic} with
$
A=-1,\;  B=1, \; C=1.
$
Using the Routh--Hurwitz stability criterion, it is straightforward to derive that
\begin{equation*}%\label{eq:region_example_connected}
    \begin{aligned}
    {\mathcal{C}}_{1,0} &= \left\{ \left. \mK = \begin{bmatrix} 0 & C_{\mK} \\
                          B_{\mK} & A_{\mK}\end{bmatrix} \in \mathbb{R}^{2 \times 2} \right| A_{\mK} < 1, B_{\mK}C_{\mK} < -A_{\mK} \right\}. \\
\end{aligned}
\end{equation*}
We choose $Q = R = 1, W = V = 1$ for the LQG problem \cref{eq:LQG_policy_optimization}. Using \cref{eq:LQG_cost_formulation1}, we can directly  obtain
\begin{equation} \label{eq:JmK-example}
     J_{\LQG,1}^2\!\left(
     \begin{bmatrix}
     0 & C_{\mK} \\
     B_{\mK} & A_{\mK}
     \end{bmatrix}
     \right)
     =
     \frac{A_{\mK}^2 - A_{\mK} (1 + B_{\mK}^2 C_{\mK}^2) -
 B_{\mK} C_{\mK} (1 - 3 B_{\mK} C_{\mK} + B_{\mK}^2 C_{\mK}^2)}{2 (-1 + A_{\mK}) (A_{\mK} + B_{\mK} C_{\mK})}, \;\; \forall \mK \in \mathcal{C}_{1,0}.
\end{equation}
Consider the boundary point $\mK_0 \in \partial {C}_{1,0}$ with
$
    A_{\mK_0} = 0, \, B_{\mK_0} = 0, \, C_{\mK_0} = 0,
$
and two sequences of policies in $\mathcal{C}_{1,0}$
\begin{subequations}
    \begin{align}
\mK_\epsilon^{(1)}&=\begin{bmatrix}  D_{\mK_\epsilon^{(1)}} & C_{\mK_\epsilon^{(1)}} \\ B_{\mK_\epsilon^{(1)}} & A_{\mK_\epsilon^{(1)}} \end{bmatrix}=\begin{bmatrix}0 & \epsilon \\ -\epsilon & 0 \end{bmatrix},  \qquad \forall  \epsilon > 0. \label{eq:convergence-controllers} \\
\mK_\epsilon^{(2)} &=\begin{bmatrix}  D_{\mK_\epsilon^{(2)}} & C_{\mK_\epsilon^{(2)}} \\ B_{\mK_\epsilon^{(2)}} & A_{\mK_\epsilon^{(2)}} \end{bmatrix}=\begin{bmatrix}0 & \epsilon+\epsilon^4 \\ -\epsilon & \epsilon^2 \end{bmatrix},  \qquad \forall  \epsilon \in(0,1). \label{eq:divergence-controllers}
\end{align}
\end{subequations}
It is clear that both sequences converge to the same boundary point $\lim_{\epsilon \downarrow 0} \mK_\epsilon^{(1)} = \lim_{\epsilon \downarrow 0} \mK_\epsilon^{(2)} = \mK_0$. From \cref{eq:JmK-example}, we can further compute that
\begin{equation*}
\begin{aligned}
    J_{\LQG,1}^2(\mK_\epsilon^{(1)})=\frac{1-3\epsilon^2 + \epsilon^4}{2}, \;\;
     J_{\LQG,1}^2(\mK_\epsilon^{(2)})=\frac{4 - 3\epsilon + 3\epsilon^2 + 3\epsilon^3 - 3\epsilon^4 + 4\epsilon^5- \epsilon^6 + \epsilon^7  + \epsilon^8 - \epsilon^9 + \epsilon^{10}}{2\epsilon(1-\epsilon)}.
\end{aligned}
\end{equation*}
Then, we have
$$
\lim_{\epsilon\downarrow0} J_{\LQG,1}(\mK_\epsilon^{(1)}) = \frac{\sqrt{2}}{2}, \qquad \lim_{\epsilon\downarrow0} J_{\LQG,1}(\mK_\epsilon^{(2)}) = \infty.
$$
Thus, $J_{\LQG,n}$ cannot be defined or extended to be continuous at the boundary point $\mK_0$. It is clear that $J_{\LQG,1}(\mK)$ in \Cref{eq:JmK-example} depends on the values of $A_\mK$ and $B_\mK C_\mK$, i.e., the individual values of $B_\mK$ and $C_\mK$ affects the LQG cost via their product. With a slight abuse of the notation, we plot  $J_{\LQG,1}(\mK)$ in \Cref{fig:LQG-boundary} where the $x$-axis represents $A_\mK$ and the $y$-axis denotes $B_\mK C_\mK$. We can clearly observe the converged sequence of policies \cref{eq:convergence-controllers} (blue curve) and divergence behavior of \cref{eq:divergence-controllers} (red curve). \hfill \qed
\end{example}

In \Cref{example:discontinuity-LQG-boundary}, the second sequence in \cref{eq:divergence-controllers} also converges to another boundary point with $A_{\mK} = 1, B_{\mK} = -1, C_{\mK} = 2$ when $\epsilon \uparrow 1$, and we have $\lim_{\epsilon\uparrow 1} J_{\LQG,1}(\mK_\epsilon^{(2)}) = \infty$. It is easy to see that this new boundary point corresponds to a minimal policy (i.e., $(A_\mK,B_\mK)$ is controllable and $(C_\mK,A_\mK)$ is observable; see~\Cref{appendix:preliminaries:ctrl_obsr} for the definition of minimality). In fact, we can prove that this divergence behavior is always guaranteed, as shown in the result below.

\begin{theorem} \label{proposition:divergence-minimial-controllers}
    Consider a sequence of policies $\{\mK_t\}_{t=1}^\infty\subset {\mathcal{C}}_{n,0}$ that converges to a boundary point $\mK_\infty=\begin{bmatrix}
    0 & C_{\mK_\infty} \\
    B_{\mK_\infty} & A_{\mK_\infty}
    \end{bmatrix}\in \partial{\mathcal{C}}_{n,0}$. Suppose $\mK_\infty$ corresponds to a minimal policy, i.e., $(A_{\mK_\infty},B_{\mK\infty})$ is controllable and $(C_{\mK_\infty},A_{\mK_\infty})$ is observable. Then $\lim_{t \to \infty} J_{\LQG,n}(\mK_t) = \infty$.
\end{theorem}

{This result is a consequence of \Cref{lemma:Boundary-minimal-LTI-systems}, since a minimal policy $\mK$ also leads~to~a~minimal closed-loop system \cref{eq:transfer-function-zd} (see \Cref{lemma:minimal-closed-loop-systems}). } % ({\color{red} We need to double-check}).
Indeed, the sequence \cref{eq:divergence-controllers} was motivated from~\cref{proposition:divergence-minimial-controllers}. The policy $A_{\mK} = \epsilon^2$, $B_\mK = -\epsilon$, $C_\mK = \epsilon, \forall \epsilon < 1$ is minimal and on the boundary~$\partial {C}_{1,0}$, thus any sequence converging to this point will diverge to infinity by \Cref{proposition:divergence-minimial-controllers}. While the sequence \cref{eq:divergence-controllers} converges to $(0,0,0)$, they stay close enough to minimal boundary points $A_{\mK} = \epsilon^2$, $B_\mK = -\epsilon$, $C_\mK = \epsilon$, thus we expect that the corresponding LQG cost would diverge, which is confirmed by our calculation. %We provide more discussions on the boundary behavior of LQG cost in \Cref{appendix:boundary-LQG}.
Finally, \Cref{example:discontinuity-LQG-boundary} also implies that the sublevel set
$
\{\mK \in {\mathcal{C}}_{n,0}\mid J_{\LQG,n}(\mK) \leq \gamma\}
$
is not bounded and might not be closed in $\mathcal{V}_n$ despite $J_{\LQG,n}$ being continuous over its domain (indicating that $J_{\LQG,n}(\mK)$, viewed as an extended real-valued function, is not lower semicontinuous over ${\mathcal{V}}_{n}$). %({\color{red} @Rich, Present a 3-D plot using \Cref{example:discontinuity-LQG-boundary} to illustrate this fact; similar to \cite[Figure 1]{zheng2021analysis}}.).

\begin{figure}
\centering

\begin{subfigure}{0.4\textwidth}
\setlength{\abovecaptionskip}{0pt}
    \includegraphics[width=0.85\textwidth]{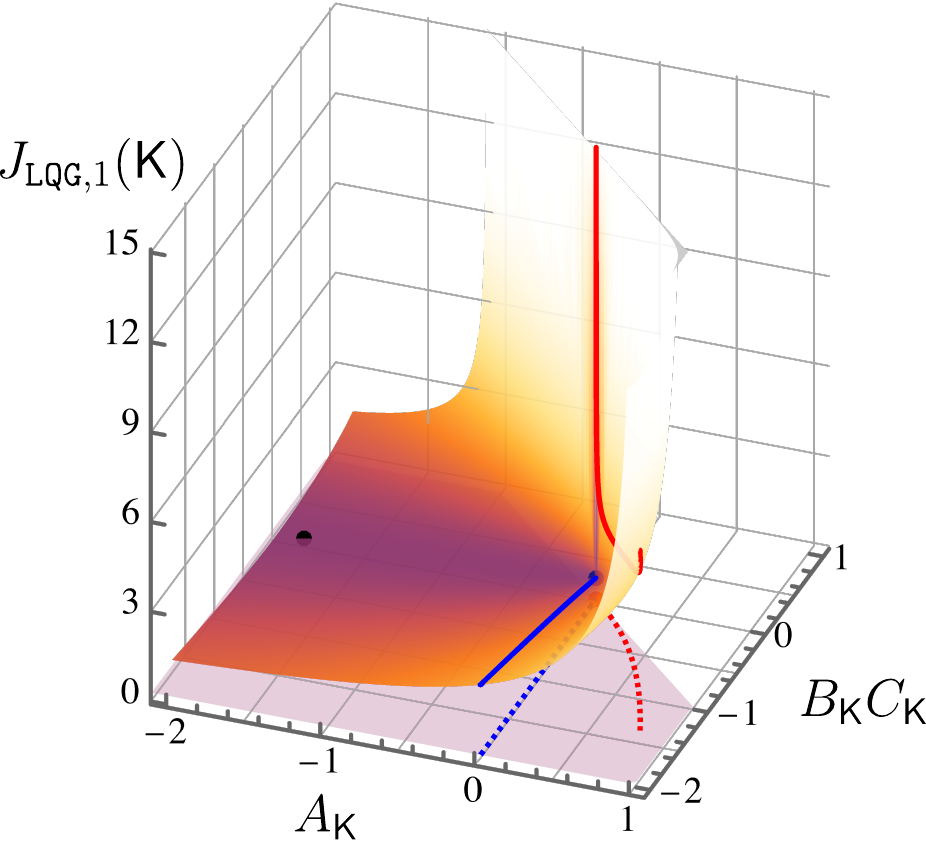}
    \caption{}
    \label{fig:LQG-boundary}
\end{subfigure}
 \hspace{10mm}
\begin{subfigure}{0.4\textwidth}
    \includegraphics[width=0.75\textwidth]{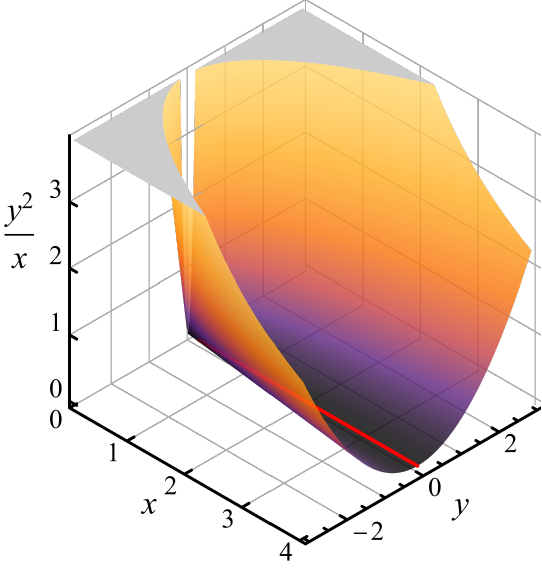}
    \caption{}
    \label{fig:rational-function-example}
\end{subfigure}
\vspace{-1mm}
    \caption{(a) Discontinuous behavior of LQG cost in \Cref{eq:JmK-example}. The blue and red (dotted) curves correspond to the behaviors for polices in \cref{eq:convergence-controllers} and \cref{eq:divergence-controllers}, respectively, as  $\epsilon \downarrow 0$. The blue curve approaches a constant while the red one goes to infinity. The black dot represents the globally optimal policy $A_\mK^\star = 1 -2\sqrt{2}, B_\mK^\star C_\mK^\star = -(1 - \sqrt{2})^2 $ (see \Cref{example:LQG-minimal-non-degenerate}). (b) Shape of $f(x,y)$ in \Cref{eq:rational-function-f}, where the red line marks the globally optimal points $\{(x,y):x>0,y=0\},$ which can be arbitrarily close to the boundary point $(0,0)$.
    }  % {\color{red} it is better to show an LQG instance as well}}
    \label{fig:LQG-hidden-convexity}
    \vspace{-1mm}
\end{figure}

\begin{remark}[Hidden convexity of LQG] \label{remark:Hidden-convexity-LQG}
    The discontinuous behavior of the LQG cost around its boundary is somewhat surprising to us at first glance, but in retrospect, it also becomes expected. In principle, the LQG cost $J_{\LQG,n}(\mK)$ behaves similarly to rational functions like %{\color{red}(@Rich, could you plot this function using Mathematica?)}
    \begin{equation} \label{eq:rational-function-f}
    f(x,y) = \frac{y^2}{x},  \qquad \text{dom}(f) = \{(x,y) \in \mathbb{R}^2\mid x> 0\}.
    \end{equation}
    This function is ``discontinuous'' at the boundary point $(0,0)$ and will grow unbounded as it approaches other boundary points $(0,y)$ with $y \neq 0$ (see \Cref{fig:rational-function-example}). We note that the Hessian of this function is
    $$
    \nabla^2 f(x,y) = \frac{2}{x^3}\begin{bmatrix}
        y^2 &  \displaystyle  -xy \\
         \displaystyle  -xy &  \displaystyle  x^2
    \end{bmatrix} \succeq 0, \qquad \forall x > 0, y \in \mathbb{R}.
    $$
    This function is convex over its whole domain, thus all stationary points (i.e., $y = 0, x >0$) are globally optimal. Note that the convexity of $f(x,y) $ can also be seen from the fact that its epigraph
    $$
    \{(x,y,\gamma) \in \mathbb{R}^3 \mid \gamma \geq f(x,y), x > 0\} = \left\{(x,y,\gamma) \in \mathbb{R}^3 \mid \begin{bmatrix}
        \gamma & y \\ y & x
    \end{bmatrix} \succeq 0, x > 0\right\},
    $$
    is convex.
    Indeed, we will show that a certain form of the epigraph of the LQG control \cref{eq:LQG_policy_optimization} admits a convex representation, and thus LQG optimization is ``almost a convex problem'' in disguise. The global optimality characterizations in \Cref{theorem:LQG-main,theorem:LQG-global-optimality} confirm this nice property. Our \texttt{ECL} framework in Part II of this paper will more precisely reveal the hidden convexity of many policy optimization problems in control. \Cref{fig:LQG-hidden-convexity} illustrates the similar landscape for \Cref{eq:JmK-example,eq:rational-function-f}.
    \hfill \qed %This example also
%
    % {\color{blue}
    % \vspace{5ex}
    % 3 ways to show the convexity of $f(x,y) = \frac{y^2}{x}$ in $(y,x)$ for $x>0$:
    % \begin{itemize}
    %     \item perspective of $$y^\tr y$$
    %     \item sum of quadratic over linear $$\frac{y_i^2}{x}$$
    %     \item matrix fractional function $$y^\tr(xI)^{-1}y$$
    % \end{itemize}}
\end{remark}

\subsection{Main technical results}  \label{subsection:main-results-H2}

Before presenting our main technical results, we first recall the closed-form expression to compute the gradient of $J_{\LQG,q}(\mK)$. If desired, one can further compute its Hessian (and any arbitrary high-order derivatives) since $J_{\LQG,q}(\mK)$ is real analytical; see \cite[Lemma 4.3 \& Appendix B.4]{zheng2021analysis}.
\begin{lemma}[{\cite[Lemma 4.2]{zheng2021analysis}}] \label{lemma:gradient_LQG_Jn}
    {Fix $q \geq 1$ such that ${\mathcal{C}}_{q,0} \neq \varnothing$.}
    For every $\mK = \begin{bmatrix} 0 & C_{\mK} \\B_{\mK} & A_{\mK} \end{bmatrix} \in {\mathcal{C}}_{q,0}$, the gradient of $J_{\LQG,q}(\mK)$ is given by\footnote{Note that \cite[Lemma 4.2]{zheng2021analysis} defines the LQG cost as $\|\mathbf{T}_{zd}(\mK,\cdot)\|_{\mathcal{H}_2}^2$, while this paper uses the definition $\|\mathbf{T}_{zd}(\mK,\cdot)\|_{\mathcal{H}_2}$. } %{\color{red} this lemma could be moved to the appendix.}
    % $$
    % \nabla J_{\texttt{LQG},q}(\mK)
    % =\left[\!\begin{array}{cc}
    % 0 & \displaystyle\frac{\partial J_{\texttt{LQG},q}(\mK)}{\partial C_{\mK}} \\[6pt]
    % \displaystyle\frac{\partial J_{\texttt{LQG},q}(\mK)}{\partial B_{\mK}} & \displaystyle\frac{\partial J_{\texttt{LQG},q}(\mK)}{\partial A_{\mK}}
    % \end{array}\!\right],
    % $$
    % with
    \begin{subequations} \label{eq:gradient_Jn}
        \begin{align}
         \frac{\partial J_{\LQG,q}(\mK)}{\partial A_{\mK}} &= \frac{1}{J_{\LQG,q}(\mK)}\left(Y_{12}^\tr X_{12} + Y_{22}X_{22}\right), \label{eq:partial_Ak}\\
        \frac{\partial J_{\LQG,q}(\mK)}{\partial B_{\mK}} &= \frac{1}{J_{\LQG,q}(\mK)}\left(Y_{22}B_{\mK}V + Y_{22}X_{12}^\tr C^\tr + Y_{12}^\tr X_{11} C^\tr\right), \label{eq:partial_Bk}\\
        \frac{\partial J_{\LQG,q}(\mK)}{\partial C_{\mK}} &= \frac{1}{J_{\LQG,q}(\mK)}\left(RC_{\mK}X_{22} + B^\tr Y_{11}X_{12} + B^\tr Y_{12} X_{22}\right), \label{eq:partial_Ck}
        \end{align}
    \end{subequations}
    where $X_{\mK}$ and $Y_{\mK}$, partitioned as
    \begin{equation} \label{eq:LyapunovXY_block}
        X_{\mK} = \begin{bmatrix}
        X_{11} & X_{12} \\ X_{12}^\tr & X_{22}
        \end{bmatrix},  \qquad Y_{\mK} = \begin{bmatrix}
        Y_{11} & Y_{12} \\ Y_{12}^\tr & Y_{22}
        \end{bmatrix}
    \end{equation}
    are the unique positive semidefinite
    solutions to~\cref{eq:LyapunovX} and~\cref{eq:LyapunovY}, respectively.
\end{lemma}

It has been recently revealed in \cite[Theorem 4.2]{zheng2021analysis} that the set of stationary points $\{\mK \in {\mathcal{C}}_{n,0} \mid \nabla J_{\LQG,q}(\mK) = 0\}$ contains strictly suboptimal saddle points for LQG control \cref{eq:LQG_policy_optimization}. A positive result is that if a stationary point corresponds to a minimal policy, then it is globally optimal for LQG \cite[Theorem 4.3]{zheng2021analysis}. However, the analysis in \cite[Theorem 4.3]{zheng2021analysis} strongly depends on the assumption of minimality, which fails to deal with non-minimal globally optimal policies for LQG. Additionally, it may be difficult to enforce minimality during the policy optimization procedure as argued in \cite{umenberger2022globally}. %, as illustrated in \cite[Example 7]{zheng2021analysis}.

We present our next main technical result of this section, which utilizes the notion of non-degenerate policies for LQG defined in \cref{def:LQG-Cnd} without relying on the minimality assumption.
\begin{theorem}\label{theorem:LQG-main}
Let   $\mK^\ast\in {\mathcal{C}}_{n,0}$ be a stationary point of $J_{\LQG,n}$ (i.e., $\nabla J_{\LQG,n}(\mK^\ast)=0$). If $\mK^\ast$ is a non-degenerate policy for LQG, then it is a global minimum of $J_{\LQG,n}$ over ${\mathcal{C}}_{n,0}$.

%Let   $\mK^\ast\in {\mathcal{C}}_{\mathrm{nd},0}$ be a non-degenerate policy for LQG. If $\mK^\ast$ is a stationary point, i.e., $\nabla J_{\LQG,n}(\mK^\ast)=0$, then it is a global minimum of $J_{\LQG,n}$ over ${\mathcal{C}}_{n,0}$.
\end{theorem}

Our proof is based on a new \texttt{ECL} framework for epigraphs of nonconvex functions, which will be the focus of Part II of this paper. Essentially, we show that the LQG optimization ``almost'' behaves as a convex problem. This proof strategy is much more general than \cite[Theorem 4.3]{zheng2021analysis} that relies on the minimality assumption and Riccati equations. We will use the same proof strategy to establish similar guarantees for nonsmooth and nonconvex $\mathcal{H}_\infty$ control in \Cref{subsection:Hinf-main-results}.
\Cref{theorem:LQG-main} directly implies that there exist no spurious (local minimum/maximum/saddle points) stationary points of $J_{\LQG,n}(\mK)$ in ${\mathcal{C}}_{\mathrm{nd},0}$. Indeed, we show that if a policy $\mK \in {\mathcal{C}}_{\mathrm{nd},0}$ is not globally optimal, then there exists a direction $\mV$ such that the directional derivative is negative, i.e.,
$$
  \lim_{t \downarrow 0} \frac{J_{\LQG,n}(\mK + t\mV) - J_{\LQG,n}(\mK)}{t} < 0,
$$
thus it is always possible to improve over this point via suitable local search algorithms.

Our next result shows that all minimal stationary points must be non-degenerate. In other words, the global optimality in \cite[Theorem 4.3]{zheng2021analysis} is a special case of \cref{theorem:LQG-main}.
% non-degeneracy characterizes at least as many globally optimal controllers as minimality.

\begin{theorem} \label{theorem:LQG-global-optimality}
    Let $\mK^\ast\in {\mathcal{C}}_{n,0}$ be a stationary point of $J_{\LQG,n}$ (i.e., $\nabla J_{\LQG,n}(\mK^\ast)=0$). If $\mK^\ast$ corresponds to a minimal policy,  then it is a non-degenerate policy for LQG, and thus is a global minimum of $J_{\LQG,n}$ over ${\mathcal{C}}_{n,0}$.
\end{theorem}

\begin{proof}
    Since $\mK^\ast$ is a minimal policy,  the solutions $X_{\mK^\ast}$ and $Y_{\mK^\ast}$ to~\cref{eq:LyapunovX} and~\cref{eq:LyapunovY} are strictly positive definite (cf. \cite[Lemma 4.5]{zheng2021analysis}). We partition the solutions as \cref{eq:LyapunovXY_block}, and we have $X_{22} \succ 0, Y_{22} \succ 0$.
    Further, since $\mK^\ast$ is a stationary point, from \cref{eq:partial_Ak}, we have
    \[
    \frac{\partial{J_{\LQG,n}(\mK^\ast)}}{\partial{A_\mK}} =\frac{1}{J_{\LQG,n}(\mK^\ast)}(Y_{12}^\tr X_{12}+Y_{22}X_{22})=0 \quad \Rightarrow \quad  Y_{12}^\tr X_{12}=-Y_{22}X_{22},
    \]
    which implies $X_{12}, Y_{12}\in \mathrm{GL}_n$.

   We next construct matrices $\Gamma$ and $P \in \mathbb{S}^{2n}_{++}$ with $ P_{12} \in \mathrm{GL}_{n}$  such that  \Cref{eq:LQG-Bilinear-a} {and}  \Cref{eq:LQG-Bilinear-b} {hold with} $ \gamma = J_{\LQG,n}(\mK^\ast)$. This will complete our proof of $\mK^\ast \in {\mathcal{C}}_{\mathrm{nd},0}$ by the definition \cref{def:LQG-Cnd}.
Indeed, we choose
\begin{equation} \label{eq:construction-P-Gamma}
P = {J_{\LQG,n}(\mK^\ast)}X_{\mK^\ast}^{-1}, \qquad \Gamma = C_{\mathrm{cl}}(\mK^\ast)P^{-1}C_{\mathrm{cl}}(\mK^\ast)^\tr.
\end{equation}
It is not difficulty to verify that they satisfy \Cref{eq:LQG-Bilinear-a} {and}  \Cref{eq:LQG-Bilinear-b} {with} $ \gamma = J_{\LQG,n}(\mK^\ast)$ (see the constructions in the proof of \Cref{lemma:H2norm} in \Cref{app:h2-hinf-norms}). It remains to show $P_{12}\in \mathrm{GL}_n$. For this, we infer from $X_{\mK^\ast} X_{\mK^\ast}^{-1}=I$ that
\[
    X_{11}(X_{\mK^\ast}^{-1})_{12}+X_{12}(X_{\mK^\ast}^{-1})_{22}=0,
\]
which implies $(X_{\mK^\ast}^{-1})_{12}\in \mathrm{GL}_n$ since $X_{11},X_{12},(X_{\mK^\ast}^{-1})_{22}\in \mathrm{GL}_n$.
    Thus $P_{12}\in \mathrm{GL}_n$. % This completes our proof. % and hence $\mK\in\mathcal{C}_{\mathrm{nd}}$.
\end{proof}

%add another example

\begin{table}[t]
\setlength{\abovecaptionskip}{3pt}
    \caption{Cases of stationary points in LQG control \cref{eq:LQG_policy_optimization}.}
    \label{tab:stationary-points-LQG}
    \centering
    \begin{tabular}{ccc}
    \toprule
        & Minimal & Non-minimal  \\
   Non-degenerate    &  \cmark (\Cref{example:LQG-minimal-non-degenerate}) & \cmark (\cref{example:non-minimal-non-degenerate})\\
   Degenerate & \xmark (cannot exist) & \cmark (\Cref{remark:degenerate-LQG-controller}) \\
    \bottomrule
    \end{tabular}

    \vspace{6pt}
    \parbox{3.6in}
    {\footnotesize Any non-degenerate stationary points are globally optimal (\Cref{theorem:LQG-main}), but there may exist degenerate globally optimal policies (\Cref{remark:degenerate-LQG-controller}). }
\end{table}

In the following, we present three examples illustrating \Cref{theorem:LQG-main,theorem:LQG-global-optimality}. \Cref{tab:stationary-points-LQG} summarizes different cases of stationary points in LQG control.  Our first example below is the case where the globally optimal policy for LQG from the Riccati equation in \Cref{theorem:LQG-riccati-solution} is minimal and non-degenerate.

\begin{example}[A globally optimal policy that is minimal and non-degenerate] \label{example:LQG-minimal-non-degenerate}
Consider the same LQG instance in \Cref{example:discontinuity-LQG-boundary}.
It is easy to see that the unique positive semidefinite solution to Riccati equation \Cref{eq:Riccati_P} is
    $
    P = \sqrt{2}-1,
    $
    and the \textit{Kalman gain} is $L = \sqrt{2}-1$. Similarly, the \textit{Feedback gain} is $K = \sqrt{2}-1$. Then, the globally optimal policy for LQG is given by
\begin{equation*} %\label{eq:LQG-example-controller}
A_\mK^\star = 1 -2\sqrt{2}, \qquad B_\mK^\star = -1 + \sqrt{2}, \qquad C_\mK^\star = 1 - \sqrt{2}.
\end{equation*}
This policy is minimal and thus non-degenerate (cf. \Cref{theorem:LQG-global-optimality}), i.e., $\mK^\star \in {\mathcal{C}}_{\mathrm{nd},0}$. Therefore, despite the discontinuous boundary behavior in \Cref{example:discontinuity-LQG-boundary}, this LQG instance is well-conditioned and can be solved easily via Riccati equations or LMIs. %Therefore, the LQG optimization over ${\mathcal{C}}_{\mathrm{nd},0}$ \Cref{eq:LQG_policy_optimization-nd-1} is solvable and the solution to its equivalent LMI will return a globally optimal LQG controller.
We can further verify that the globally optimal policies for a class of LQG instances with $
A = a,\,  B = 1,\, C = 1,\, Q = 1,\, R = 1,\, W = 1,\, V =1, $
are all non-degenerate for any $a \in \mathbb{R}$; see \Cref{appendix:globally-optimal-controllers} for computational details. \hfill \qed
\end{example}

The following example shows that the notion of non-degenerate policies indeed characterizes a larger class of stationary points, beyond minimal policies, that are globally optimal to \cref{eq:LQG_policy_optimization}.

\begin{example}[A globally optimal policy for LQG that is non-minimal and non-degenerate] \label{example:non-minimal-non-degenerate}
We consider the same example in \cite[Example 7]{zheng2021analysis}, which originates from \cite{yousuff1984note}. In this example, the policy in ${\mathcal{C}}_{n,0}$ is not minimal, but we show it is indeed non-degenerate.

Specifically, we consider the LTI system \cref{eq:Dynamic} with
$$
A=\begin{bmatrix}
0 & -1 \\ 1 & 0
\end{bmatrix},
\qquad
B = \begin{bmatrix}
1 \\ 0
\end{bmatrix},
\qquad
C = \begin{bmatrix}
1 & -1
\end{bmatrix},
\qquad
W = \begin{bmatrix}
1 & -1 \\ -1 & 16
\end{bmatrix},
\qquad
V = 1,
$$
and let the LQG cost be defined by
$
Q=\begin{bmatrix}
4 & 0 \\ 0 & 0
\end{bmatrix},
\,
R = 1.
$
This LQG problem satisfies \cref{assumption:stabilizability}. The positive definite solutions to the Riccati equations \cref{eq:Riccati} are
$
P=\begin{bmatrix}
1 & 0 \\ 0 & 4
\end{bmatrix},
\,
S
=\begin{bmatrix}
2 & 0 \\ 0 & 2
\end{bmatrix},
$
and the globally optimal policy is given by
\begin{equation} \label{eq:LQGcontroller_Nonminimal_case}
A_{\mK} = \begin{bmatrix}
-3 & 0 \\
5 & -4
\end{bmatrix},
\quad
B_{\mK}=L
=
\begin{bmatrix}
1 \\ -4
\end{bmatrix},
\quad
C_{\mK}=-K =
\begin{bmatrix}
-2 & 0
\end{bmatrix}.
\end{equation}
It is straightforward to verify that $(C_{\mK},A_{\mK})$ is not observable. Therefore, the optimal policy for LQG obtained from the Riccati equations is not minimal in this example. % Consequently, by \cref{corollary:non-minimal-globally-optimal}, all stationary points of $J_n$ are not minimal for this example.

This globally optimal policy \cref{eq:LQGcontroller_Nonminimal_case} turns out to be non-degenerate. Indeed, for the closed-loop system \cref{eq:transfer-function-zd} with this optimal policy $\mK$, the solution to the Lyapunov equation \cref{eq:LyapunovX} reads as
\[
X_\mK=\begin{bmatrix} 5.25 & -8 & 4.25 & -8 \\ -8 & 20.25 & -8 & 16.25 \\ 4.25 & -8 & 4.25 & -8 \\ -8 & 16.25 & -8 & 16.25 \end{bmatrix}.
\]
It is easy to verify that $X_\mK \succ 0$ (this is expected since the closed-loop system is controllable thanks to the controllability of $(A_\mK, B_\mK)$) and $(X_\mK^{-1})_{12}\in \mathrm{GL}_n$. Then, %following the same argument,
the construction in \cref{eq:construction-P-Gamma} satisfies \Cref{eq:LQG-Bilinear-a} {and}  \Cref{eq:LQG-Bilinear-b} {with} $ \gamma = J_{\LQG,n}(\mK)$. Thus, the non-minimal policy \cref{eq:LQGcontroller_Nonminimal_case} is non-degenerate.  \hfill \qed
%
% Further, we can verify that \[X_\mK'=\begin{bmatrix} 5.25 & -8 & 4.25 & -8 \\ -8 & 20.25 & -8 & 16.25 \\ 4.25 & -8 & 4.25 & -8 \\ -8 & 16.25 & -8 & 16.25 \end{bmatrix}+\begin{bmatrix} 0 & 0 & 0 & 0 \\ 0 & 0 & 0 & 0 \\ 0 & 0 & 0 & 0 \\ 0 & 0 & 0 & 1 \end{bmatrix}\] is another certificate.
%
\end{example}

%{\color{red}revision up to here.}

\begin{example}[A globally optimal policy for LQG that is degenerate] \label{remark:degenerate-LQG-controller}
    %{\color{red}Are there globally optimal LQG controllers that are degenerate? The answer seems to be NO under suitable assumption, i.e., controllability and observability of the plant!}
    Similar to minimal policies, our notion of non-degenerate policies depends on state-space realizations.
    In \Cref{example:non-minimal-non-degenerate}, some globally optimal policies in ${\mathcal{C}}_{n,0}$ are not connected by similarity transformations. The following two non-minimal policies are both globally optimal:% in the frequency domain:
     \begin{equation*}
         \mK_1 = \left[\begin{array}{c:cc}
    0 & -2 &  0 \\[2pt]
    \hdashline
    1 & -3 & 0 \\[-2pt]
    -4 & 5 & -4
    \end{array}\right],
         % \begin{bmatrix}
         %    0 & -2 & 0\\
         %    1 & -3& 0\\
         %    -4 & 5 & -4
         % \end{bmatrix},
         \qquad \mK_2 = \left[\begin{array}{c:cc}
    0 & -2 &  0 \\[2pt]
    \hdashline
    1 & -3 & 0 \\[-2pt]
    0 & 0 & -1
    \end{array}\right],
         % \begin{bmatrix}
         %         0 & -2 & 0\\
         %    1 & -3& 0\\
         %    0 & 0 & -1
         % \end{bmatrix},
     \end{equation*}
      (since they correspond to the same transfer~function), but there exists no similarity transformation between $\mK_1$ and $\mK_2$. We have numerically verified in \Cref{example:non-minimal-non-degenerate} that $\mK_1$ is non-degenerate; on the other hand, we can further prove that $\mK_2$ is degenerate; see \Cref{lemma:LQG-lower-order-stationary-point-to-high-order} below. %\Cref{appendix:degenerate-controllers}.
     This observation raises an interesting question: if a globally optimal policy for LQG is minimal, it must be non-degenerate (cf. \Cref{theorem:LQG-global-optimality}); for a globally optimal policy for LQG that is not minimal, itself might be degenerate (see $\mK_2$), but does it admit another state-space realization that is non-degenerate (such as $\mK_1$)? Another related question is whether all policies obtained from Riccati equations (i.e., those from \Cref{theorem:LQG-riccati-solution}) are non-degenerate. Detailed investigations are left for future work.
     \hfill \qed
\end{example}

\subsection{A class of degenerate policies}

In \Cref{remark:degenerate-LQG-controller}, we have identified one degenerate policy that is globally optimal. Due to the existence of saddle points \cite[Theorem 4.2]{zheng2021analysis}, it is expected that there exist sub-optimal degenerate stationary points $\mK \in {\mathcal{C}}_{n,0}\backslash {\mathcal{C}}_{\mathrm{nd},0}$ for LQG control \cref{eq:LQG_policy_optimization}.

We here identify a class of degenerate policies. % This class include
In particular, the following \Cref{lemma:LQG-lower-order-stationary-point-to-high-order} shows that any stationary point of $J_{\LQG,q}$ can be augmented to a stationary point of $J_{\LQG,q+q'}$ for any $q'\geq 1$. Furthermore, any reduced-order minimal policy in ${\mathcal{C}}_{q,0}$ can be augmented to a full-order policy in ${\mathcal{C}}_{n,0}$ that is a degenerate policy. %This characterizes a large class of degenerate controllers that are suboptimal.

\begin{theorem} \label{lemma:LQG-lower-order-stationary-point-to-high-order}
Let $0 \leq q < n$, and consider a reduced-order policy $\mK=\begin{bmatrix}
0 & C_{\mK} \\ B_{\mK} & A_{\mK}
\end{bmatrix}
\in {\mathcal{C}}_{q,0}$. The following statement holds.
\begin{enumerate}
    \item For any $q'\geq 1$ and any stable $\Lambda\in\mathbb{R}^{q'\times q'}$, the following augmented policy
\begin{equation} \label{eq:gradient_nonglobally_K}
        \tilde{\mK}
    =\left[\begin{array}{c:cc}
    0 & C_{\mK} &  0 \\[2pt]
    \hdashline
    B_{\mK} & A_{\mK} & 0 \\[-2pt]
    0 & 0 & \Lambda
    \end{array}\right] \in {\mathcal{C}}_{q+q',0}
\end{equation}
    satisfies $J_{\LQG,q+q'}\big(\tilde{\mK}\big)
=J_{\LQG,q}({\mK})$.
    \item If the policy $\mK$ is a stationary point of $J_{\LQG,q}$, i.e., $\nabla J_{\LQG,q}(\mK)=0$, then the augmented policy \cref{eq:gradient_nonglobally_K} with any $q'\geq 1$ and any stable $\Lambda\in\mathbb{R}^{q'\times q'}$ is a stationary point of $J_{\LQG,q+q'}$ over ${\mathcal{C}}_{q+q',0}$. % satisfying $J_{\texttt{LQG},q+q'}\big(\tilde{\mK}\big)=J_{\texttt{LQG},q}(\tilde{\mK})$.
\item Let $q' = n - q$. If $\mK\in {\mathcal{C}}_{q,0}$ is minimal, the augmented policy \cref{eq:gradient_nonglobally_K} with any stable $\Lambda\in\mathbb{R}^{q'\times q'}$ is degenerate, i.e., $ \tilde{\mK} \in {\mathcal{C}}_{n,0} \backslash {\mathcal{C}}_{\mathrm{nd},0}$.
\end{enumerate}
\end{theorem}

The first statement is straightforward since $\mK$ and $\tilde{\mK}$ correspond to the same transfer function in the frequency domain. The second statement has recently been established in \cite[Theorem 4.1]{zheng2021analysis} which relies on direct gradient computation. % We provide a new proof based on perturbation analysis of transfer functions in \Cref{appendix:lower-order-stationary-point-to-high-order}.
The third statement requires some computations around Lyapunov equations and inequalities, which guarantees that the off-diagonal block $P_{12}$ satisfying \Cref{eq:LQG-Bilinear-a} {and}  \Cref{eq:LQG-Bilinear-b} {with} $ \gamma = J_{\LQG,n}(\tilde{\mK})$ always has low rank. The details are technically involved. We postpone them to \Cref{appendix:degenerate-controllers} (see \Cref{theorem:degenerate-controllers-LQG}).

We finally present a simple corollary that introduces a class of sub-optimal stationary points that correspond to degenerate policies.

\begin{corollary}
Suppose the plant \cref{eq:Dynamic} is open-loop stable. Let $\Lambda\in\mathbb{R}^{n\times n}$ be stable. Then $
\mK = \left[\begin{array}{c:c}
    0 &   0 \\[2pt]
    \hdashline
     0 & \Lambda
    \end{array}\right] .
$
is a stationary point of $J_{\LQG,n}(\mK)$ over ${\mathcal{C}}_{n,0}$, and it is degenerate, i.e., $\mK\in {\mathcal{C}}_{\mathrm{nd},0}$.
\end{corollary}

This corollary is an immediate consequence from \Cref{lemma:LQG-lower-order-stationary-point-to-high-order}, since a zero policy is always a stationary point to LQG control with open-loop stable systems \cite[Theorem 4.2]{zheng2021analysis}. This zero policy is degenerate, and it will be suboptimal if the globally optimal policy is minimal (since the optimal LQG controller is unique in the frequency domain).  % of is
We conclude this section with a degenerate policy that corresponds to a strict saddle point.

\begin{example}[A degenerate policy that corresponds to a strict saddle]\label{example:LQG_strict_saddle}
    Consider the same LQG instance in \Cref{example:discontinuity-LQG-boundary}. It is straightforward to verify that
    $
    \mK = \begin{bmatrix}
        0 & 0 \\ 0 & -1
    \end{bmatrix} \in {\mathcal{C}}_{1,0}
    $
    is a stationary point. It is also not difficult to verify that this policy is degenerate. % ({\color{red}@Rich, you should calculate this as a good exercise for you}).
    Furthermore, by \cite[Theorem 4.3]{zheng2021analysis} and \cite[Theorem 2]{zheng2022escaping}, this policy is a strict saddle point and its corresponding Hessian has one negative eigenvalue (indeed, the three eigenvalues are $0$ and $\pm0.25$). We illustrate this strict saddle in \Cref{fig:LQG_strict_saddle}. \hfill \qed
    % with the corresponding eigenvectors (1,0,0), (0,1,1), (0,-1,1) respectively
\end{example}

\begin{figure}
\centering

    \includegraphics[width =0.4\textwidth]{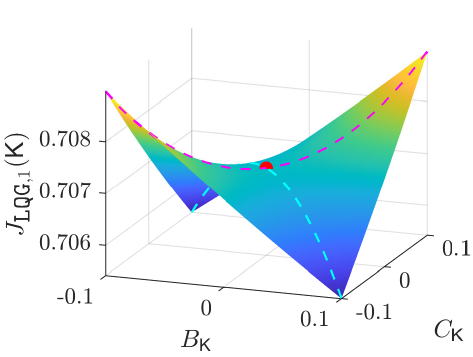}
    \caption{A strict saddle (highlighted as the red point) of LQG optimization landscape in \Cref{example:LQG_strict_saddle}. We fixed $A_\mK = -1$ and change $B_\mK \in (-0.1, 0.1)$ and $C_\mK \in (-0.1, 0.1)$.
    }
\label{fig:LQG_strict_saddle}
\end{figure}

% Section 5: Hinf robust control
\section{Nonsmooth and Nonconvex \texorpdfstring{$\mathcal{H}_\infty$}{Lg} Robust Control}
\label{section:Hinf-control}

In this section, we turn our focus to the policy optimization for $\mathcal{H}_\infty$ control \cref{eq:Hinf_policy_optimization}. One key difference compared to the LQG control is that the $\mathcal{H}_\infty$ cost function $J_{\infty,n}(\mK)$ is nonsmooth. We thus need to pay special care to the non-smoothness. However, not surprisingly (or surprisingly) in the spirit of \cite{doyle1988state}, we will show many properties in \cref{section:LQG} have nonsmooth counterparts for $\mathcal{H}_\infty$ control \cref{eq:Hinf_policy_optimization}, since both $\mathcal{H}_2$ and $\mathcal{H}_\infty$ norms have similar LMI characterizations (cf. \Cref{lemma:H2norm,lemma:bounded_real}). We will also highlight some difficulties in the analysis brought by the unique feature of the bounded real lemma (\cref{lemma:bounded_real}). % also poses certain difficulties.

We first summarize some basic properties of $\mathcal{H}_\infty$ cost function $J_{\infty,n}(\mK)$, and demonstrate its non-smoothness using simple examples. Similar to the LQG case, the $\mathcal{H}_\infty$ cost function $J_{\infty,n}(\mK)$ also exhibits complicated ``discontinuous'' behavior around its boundary.
%We then summarize some useful properties of the $\mathcal{H}_\infty$ cost function $J(\mK)$ in Section \ref{subsection:basic-properties}.
We then present our main technical result, proving that all Clarke stationary points in the set of non-degenerate policies are globally optimal. The proof also relies on the \texttt{ECL} framework.
At the end of this section, we discuss a class of Clarke stationary points that is potentially degenerate and sub-optimal.
%The rest of this section presents its proof.

\subsection{Basic properties and ``discontinuous'' behavior of the \texorpdfstring{$\mathcal{H}_\infty$}{Lg} cost} \label{subsection:basic-properties}

The $\mathcal{H}_\infty$ cost function $J_{\infty,q}(\mK)$ in \cref{eq:Hinf-norm} is known to be nonconvex and also nonsmooth with two possible sources of non-smoothness: One from taking the largest singular value of complex matrices, and the other from maximization over all the frequencies $\omega \in \mathbb{R}$. One could also see the non-smoothness from its \texttt{max} operation in the time domain. % \cref{eq:Hinf-cost-def}.
Since this is a unique feature in $\mathcal{H}_\infty$ control, we highlight it as a fact below.

\begin{fact} \label{fact:hinf-nonconvex}
    Fix $q \in \mathbb{N}$ such that ${\mathcal{C}}_q\neq\varnothing$. Then, $J_{\infty,q}(\mK)$ in \cref{eq:Hinf-norm} is continuous, nonconvex and nonsmooth on ${\mathcal{C}}_q$.
\end{fact}

We note that the continuity of $J_{\infty,q}(\mK)$ in \cref{eq:Hinf-norm} can be viewed from the composition of a convex mapping $\|\cdot\|_\infty$ (which is naturally continuous) and the continuous mapping $\mK\mapsto \mathbf{T}_{zd}$. Indeed, this perspective can ensure strong properties as locally Lipschitz continuity (see \Cref{lemma:H_inf_some_local_Lipschitz} below).
%
%{\color{red} give some motivation as to why it is continuous since singular values are continuous functions as its matrix elements and so on or any norm is continuous with respect to its argument. }
%
We here illustrate \Cref{fact:hinf-nonconvex} using a simple example. %{\color{red}@Rich, please update \Cref{example: non-smooth SF,example:non-smooth SF2,example:non-smooth OF} and give more details. See the examples in \Cref{section:LQG}. }

\begin{example} \label{example:Hinf-nonsmooth}

Consider a single-input single-output system with the same dynamics in \Cref{example:discontinuity-LQG-boundary}. In particular, the dynamics in \cref{eq:Dynamic} and the performance signal in \cref{eq:performance-signal} are given by
\[
A=-1, \quad  B = 1, \quad C = 1, \quad
Q = 1 \quad R = 1,\quad W = 1,\quad V =1.
\]
We first consider a static output feedback $u(t) = D_\mK y(t)$ where $D_\mK \in \mathbb{R}$. It is easy to see that the set of stabilizing static policies is $\mathcal{C}_0 = \{D_\mK \in \mathbb{R} \mid D_\mK < 1\}$. The corresponding $\mathcal{H}_\infty$ cost $J_{\infty,0}(\mK)$ is shown in \cref{fig:Hinf-SF-nonsmooth-a}, which shows a nonsmooth point at $D_\mK = 1-\sqrt{3}$ (highlighted by a red point). This policy turns out to be globally optimal for all dynamic policies (see \Cref{example:Hinf-optimal}).

We then consider a full-order dynamic output feedback policy. Using the Routh--Hurwitz stability criterion, it is straightforward to derive that
\begin{equation}\label{eq:region_example_hinf}
    \begin{aligned}
    {\mathcal{C}}_1 &= \left\{ \left. \mK = \begin{bmatrix} D_{\mK} & C_{\mK} \\
                          B_{\mK} & A_{\mK}\end{bmatrix} \in \mathbb{R}^{2 \times 2} \right| A_{\mK} + D_{\mK}< 1, \;\; B_{\mK}C_{\mK} < -A_{\mK} + A_{\mK}D_{\mK} \right\}. \\
\end{aligned}
\end{equation}
The corresponding $\mathcal{H}_\infty$ cost around the dynamic policy % {\color{red} (@Rich, please update \Cref{fig:Hinf-SF-nonsmooth} and include other points)}
\begingroup
    \setlength\arraycolsep{2pt}
\def\arraystretch{0.9}
$
\mK_1
= \begin{bmatrix}
    1-\sqrt{3} & 0 \\ 0 & -1
    \end{bmatrix}
$
\endgroup
is shown in \cref{fig:Hinf-SF-nonsmooth-b}, which shows a set of nonsmooth points (highlighted by red lines). \hfill \qed  % at $D_\mK = 1-\sqrt{3}$.
%
%\end{itemize}
\end{example}

\begin{figure}
\centering
\begin{subfigure}{0.32\textwidth}
    \includegraphics[width =0.85\textwidth]{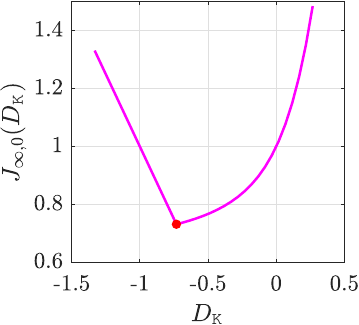}
    \caption{Static output feedback}
    \label{fig:Hinf-SF-nonsmooth-a}
\end{subfigure}
    \hspace{5mm}
\begin{subfigure}{0.32\textwidth}
    \includegraphics[width =\textwidth]{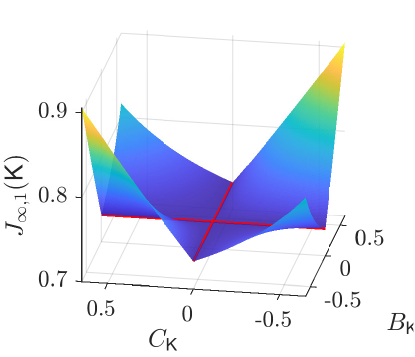}
    \caption{Dynamic output feedback}
    \label{fig:Hinf-SF-nonsmooth-b}
\end{subfigure}

    \caption{Non-smoothness of the $\mathcal{H}_\infty$ function in \Cref{example:Hinf-nonsmooth}. (a) Static output feedback $u(t) = D_\mK y(t)$; (b) dynamic output feedback, where we fix $D_\mK = 1 - \sqrt{3}, A_{\mK} = -1$ and change $B_{\mK} \in (-0.6,0.6)$, $C_{\mK} \in (-0.6,0.6)$. % {\color{red} @Rich, Update (c):, we need to show the effect of similarity transformations. }.% {\color{red}To update; plot its gradient flow as well}
    }
    \label{fig:Hinf-SF-nonsmooth}
\end{figure}

Despite the non-smoothness, it is known that the $\mathcal{H}_\infty$ cost function $J_{\infty,q}(\mK)$ is locally Lipschitz and thus differentiable almost everywhere (\cref{fig:Hinf-SF-nonsmooth} also suggests this). We can further show $J_{\infty,q}(\mK)$ is subdifferentially regular in the sense of Clarke (for self-completeness, we review some fundamentals in nonsmooth optimization in \Cref{app:nonsmooth-optimization}). We summarize this property below.

\begin{lemma}[{{\cite[Proposition 3.1]{apkarian2006nonsmooth2}}}]\label{lemma:H_inf_some_local_Lipschitz}
%For the $\mathcal{H}_\infty$ policy optimization problem~\cref{eq:Hinf_policy_optimization}, the following statements hold.
Fix $q \in \mathbb{N}$ such that ${\mathcal{C}}_q\neq\varnothing$. The function $J_{\infty,q}(\mK)$ in \cref{eq:Hinf-norm} is locally Lipschitz and also subdifferentially regular over $\mathcal{C}_q$.
\end{lemma}
 The proof idea in {\cite[Proposition 3.1]{apkarian2006nonsmooth2}} is to view $J_{\infty,q}(\mK)$ as a composition of a convex mapping $\|\cdot\|_\infty$ and the mapping $\mK\mapsto \mathbf{T}_{zd}$ that is continuously differentiable over $\mK$. Then, the subdifferential regularity of $J(\mK)$ follows from \cite{clarke1990optimization}. We provide some missing details in \Cref{appendix:Hinf-cost-function}.
Similar to the smooth LQG case, $J_{\infty,q}(\mK)$ is not coercive due to similarity transformations: it is easy to see that there exists a sequence of policies $\{\mK_l\}\subset\mathcal{C}_q$ such that
$\lim_{l\to\infty}J_{\infty,q}(\mK_l)< \infty$ as $\lim_{l\to\infty}\|{\mK_l}\|_F=\infty$.  We also summarize this as a fact.

\begin{fact}
    Fix $q \in \mathbb{N}$ such that ${\mathcal{C}}_q\neq\varnothing$. The function $J_{\infty,q}(\mK)$ in \cref{eq:Hinf-norm} is not coercive.
\end{fact}

The following example further shows that {even when the policies $\mK$ go to the boundary of $\mathcal{C}_q$, the corresponding $\mathcal{H}_\infty$ cost
might converge to a finite value }.  This example also shows that the $\mathcal{H}_\infty$ cost has ``discontinuous'' behavior around some boundary points. % in {the sense that the LQG function admits different limit values}.

\begin{figure}
\centering
\setlength{\abovecaptionskip}{3pt}
\begin{subfigure}{.3\textwidth}
    \includegraphics[width =\textwidth]{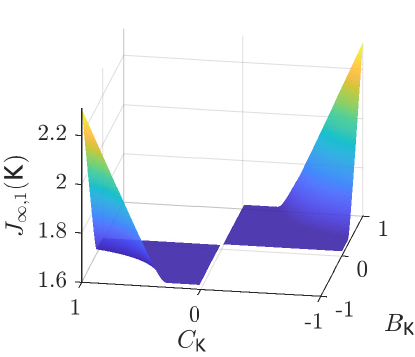}
    \caption{}
\end{subfigure}
    \hspace{10mm}
\begin{subfigure}{.25\textwidth}
    \includegraphics[width=\textwidth]{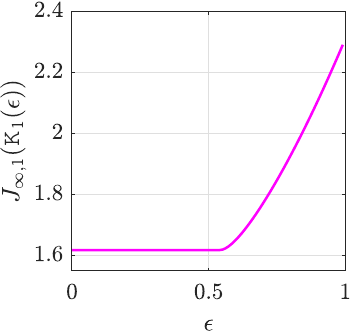}
    \caption{}
\end{subfigure}
        \hspace{10mm}
\begin{subfigure}{.25\textwidth}
    \includegraphics[width=\textwidth]{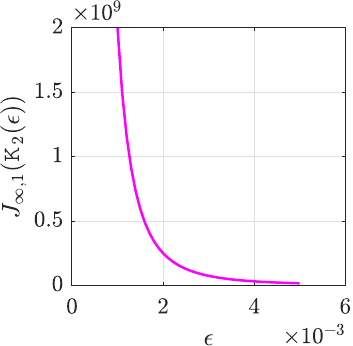}
    \caption{}
\end{subfigure}
    \caption{Boundary behavior of the $\mathcal{H}_\infty$ cost in \Cref{example:non-coercivity-of-Hinf-cost}. (a) Shape of $J_{\infty,1}(\mK)$, which is fairly flat around $(B_\mK,C_\mK)=(0,0)$, where we fix $A_\mK=D_\mK=0$ and change $B_\mK,C_\mK \in(-1,1)$; (b) Shape of $J_{\infty,1}(\mK_1(\epsilon))$ in \cref{eq:hinf_ex_a}, which converges to a finite value as $\epsilon\downarrow0$; (c) Shape of $J_{\infty,1}(\mK_2(\epsilon))$ in \cref{eq:hinf-boundary-2-maintext}, which diverges to infinity as $\epsilon\downarrow0$. %{\color{red}To update}
    }
    \label{figure:hinf-boundary}
\end{figure}

\begin{example}[Non-coercivity and ``discontinuity'' of the $\mathcal{H}_\infty$ cost] \label{example:non-coercivity-of-Hinf-cost}
Consider the same $\mathcal{H}_\infty$ instance in \Cref{example:Hinf-nonsmooth}. Let $\mK_0=(0,0,0,0) \in \partial \mathcal{C}_1$.
From \cref{eq:region_example_hinf}, the following parameterized policies
\begin{subequations}
\begin{equation} \label{eq:hinf_ex_a}
\mK_1(\epsilon) = \begin{bmatrix}0 & \epsilon \\ -\epsilon & 0 \end{bmatrix}, \;\;  \forall \epsilon \neq 0,
\end{equation}
are all stabilizing, which converges to the boundary point $\mK_0$ as $\epsilon\downarrow0$, i.e., $\lim_{\epsilon \downarrow 0}\mK_1(\epsilon) = \mK_0$.
We will show that $J_{\infty,1}(\mK_1(\epsilon))$ converges to a finite value as $\epsilon\downarrow0$ and hence $J_{\infty,1}(\cdot)$ is not coercive; see \Cref{figure:hinf-boundary}.
Indeed, the closed-loop matrices are
\begin{align*}
A_{\mathrm{cl}}(\mK_1(\epsilon))
=  &
\begin{bmatrix}
-1 & \epsilon \\ -\epsilon & 0
\end{bmatrix}, \;\;  B_{\mathrm{cl}}(\mK_1(\epsilon))
= \begin{bmatrix} 1 & 0 \\ 0 & -\epsilon  \end{bmatrix}, \;\;
C_{\mathrm{cl}}(\mK_1(\epsilon))
=
 \begin{bmatrix}
        1 & 0 \\
        0 & \epsilon
    \end{bmatrix}, \;\;
D_\mathrm{cl}(\mK_1(\epsilon)) =0,
\end{align*}
and the closed-loop transfer function is
$$
\mathbf{T}_{zd}(\mK_1(\epsilon)) = \frac{1}{s^2 + s + \epsilon^2}\begin{bmatrix}
    s & -\epsilon^2 \\ -\epsilon^2 & -\epsilon^2 (s + 1)
\end{bmatrix} \quad \Rightarrow \quad \lim_{\epsilon \downarrow 0} \mathbf{T}_{zd}(\mK_1(\epsilon)) = \frac{1}{s+1}\begin{bmatrix}
    1 & 0 \\ 0 & 0
\end{bmatrix}.
$$
However, the convergence of these transfer functions is not uniform. After some tedious computation (see \cref{subsection:hinf-boundary-behavior}), we verify that
% Suppose we have an approximation
% $$f_l(\omega) \leq f(\omega) \leq f_u(\omega).$$
% We still cannot say they have the same maximizes.
% Have checked $f'(0)=0$ for all $\omega$. even a positive second derivative at $\omega=0$ can only say it is a local maximizer.
% } Therefore, we have
\begin{align*}
    \lim_{\epsilon\downarrow 0}J_{\infty,1}(\mK_1(\epsilon)) &=\lim_{\epsilon\downarrow  0} \sup_{\omega\in\mathbb{R}}\sigma_{\mathrm{max}}\left(\mathbf{T}_{zd}(j\omega,\mK_1(\epsilon))\right) = \frac{1 + \sqrt{5}}{2}.
\end{align*}

Consider another parameterized stabilizing policies
\begin{equation} \label{eq:hinf-boundary-2-maintext}
\mK_2(\epsilon) %=\begin{bmatrix}  D_\mK & C_\mK \\ B_\mK & A_\mK \end{bmatrix}
=\begin{bmatrix}0 & \epsilon+\epsilon^4 \\ -\epsilon & \epsilon^2 \end{bmatrix} \in \mathcal{C}_1, \qquad \forall  1 > \epsilon > 0.
\end{equation}
which converges to the same boundary point $\mK_0$ as $\epsilon \downarrow 0$. The closed-loop transfer function is
$$
    \mathbf{T}_{zd}(\mK_2(\epsilon)) = \frac{1}{ s^2 + (1- \epsilon^2)s + \epsilon^5}\begin{bmatrix}
    s-\epsilon^2 & -\epsilon^2 -\epsilon^5 \\ -\epsilon^2 -\epsilon^5  &  -(\epsilon^2 +\epsilon^5) (s + 1)
\end{bmatrix} \; \Rightarrow \; \lim_{\epsilon \downarrow 0}    \mathbf{T}_{zd}(\mK_2(\epsilon)) = \frac{1}{s+1}\begin{bmatrix}
    1 & 0 \\ 0 & 0
\end{bmatrix}.
$$
Again, the convergence of these transfer functions is not uniform. We find that {it is tedious} to get an analytical expression for their $\mathcal{H}_\infty$ norms, but numerical simulation suggests that
$$
    J_{\infty,1}(\mK_2(\epsilon)) = \sup_{\omega \in \mathbb{R}} \, \sigma_{\mathrm{max}}(\mathbf{T}_{zd}(j\omega,\mK_2(\epsilon))) \approx \frac{2}{\epsilon^3} \qquad \text{if}\;\; 0 < \epsilon \leq 10^{-1}.
$$
This analytical expression agrees with numerical values very well (see \cref{subsection:hinf-boundary-behavior}). For $\mK_2(\epsilon)$, we have %$J_{\infty,1}(\mK_2(\epsilon))$ blow up as
%$\epsilon\downarrow0$, i.e.,
$\lim_{\epsilon \downarrow 0}J_{\infty,1}(\mK_2(\epsilon))\to\infty$.
Therefore, $J_{\infty,1}(\cdot)$ cannot be defined or extended to be continuous at the boundary point $\mK_0$. We illustrate these behaviors in \Cref{figure:hinf-boundary}. \hfill \qed
\end{subequations}
\end{example}

Similar to \Cref{proposition:divergence-minimial-controllers}, the $\mathcal{H}_\infty$ cost will diverge to infinity if the sequence of policies converges to a boundary point corresponding to a strictly proper minimal policy. %This result also motivated us to find
The sequence of policies in \cref{eq:hinf-boundary-2-maintext} was indeed motivated by this result.

\begin{theorem} \label{proposition:divergence-minimial-controllers-hinf}
    Consider a sequence of policies $\{\mK_t\}_{t=1}^\infty\subset {\mathcal{C}}_n$ that approaches a boundary point $\mK_\infty \in \partial{\mathcal{C}}_n$. If $\mK_\infty$ is minimal, then $\lim_{t \to \infty} J_{\infty,n}(\mK_t) = \infty$.
\end{theorem}
This is a consequence of \Cref{lemma:Boundary-minimal-LTI-systems} since a minimal policy $\mK$ with $D_\mK = 0$ also leads to a minimal closed-loop system \cref{eq:transfer-function-zd} (see \Cref{lemma:minimal-closed-loop-systems}).
Finally, similar to the LQG case, \Cref{example:non-coercivity-of-Hinf-cost} also implies that the sublevel set of the $\mathcal{H}_\infty$ cost function
$
\{\mK \in {\mathcal{C}}_n\mid J_{\infty,n}(\mK) \leq \gamma\}
$
is not bounded and might not even be closed despite $J_{\infty,n}$ being continuous over its domain. This fact becomes obvious in retrospect since the domain is open and the function is not coercive.  %({\color{red} @Rich: Present a 3-D plot using \Cref{example:non-coercivity-of-Hinf-cost} to illustrate this fact; similar to the figure in \cite{hu2022connectivity}}).

\begin{remark}[Non-smoothness and hidden convexity in $\mathcal{H}_\infty$ control]
Both the smooth LQG cost \cref{eq:H2-norm} and nonsmooth $\mathcal{H}_\infty$ cost \cref{eq:Hinf-norm} are not coercive and exhibit discontinuity around some~non-minimal boundary points. The nonsmooth $J_{\infty,0}(\mK)$ in \Cref{fig:Hinf-SF-nonsmooth-a} appears to be convex (one can indeed verify its convexity), and $J_{\infty,1}(\mK)$ has rich geometric symmetry in \Cref{fig:Hinf-SF-nonsmooth-b}. Indeed, our \texttt{ECL} framework  will confirm that a certain form of the epigraph of the $\mathcal{H}_\infty$ control \cref{eq:Hinf_policy_optimization} admits a convex representation, and it is ``almost a convex problem'' in disguise (see \Cref{theorem:Hinf-main-global-optimality}). This level of similarity between LQG control \cref{eq:LQG_policy_optimization} and $\mathcal{H}_\infty$ control \cref{eq:Hinf_policy_optimization} originates the similarity in the LMI characterizations of $\mathcal{H}_2$ and $\mathcal{H}_\infty$ norms (cf. \cref{lemma:bounded_real,lemma:H2norm}). In classical literature~\cite{doyle1988state,glover2005state,zhou1998essentials,dullerud2013course}, it is known that the (suboptimal) solutions to LQG control \cref{eq:LQG_policy_optimization} and $\mathcal{H}_\infty$ control \cref{eq:Hinf_policy_optimization} share similar structures. % (also see textbooks \cite{zhou1998essentials,dullerud2013course}). % indeed share similar structures.
%This can also be viewed as reminiscent of the seminal work \cite{doyle1988state} which shows that the (suboptimal) solutions to \cref{eq:LQG_policy_optimization} and \cref{eq:Hinf_policy_optimization} indeed share similar structures.

Despite their similarities, we here emphasize some key (also well-known) differences. The LQG cost or $\mathcal{H}_2$ norm \cref{eq:H2-norm} admits a closed-form algebraic expression (i.e., a rational function from the Lyapunov equation). On the other hand, the $\mathcal{H}_\infty$ norm \cref{eq:Hinf-norm} depends on the largest singular value over the entire frequency domain, which has no closed-form algebraic expression; %(cf. Abel-Ruffini theorem \cite{fraleigh2003first}). % in terms of the matrix elements (Abel-Ruffini theorem \cite{fraleigh2003first})
even evaluating the $\mathcal{H}_\infty$ norm is not a simple task and requires an iterative algorithm (e.g., solving an LMI or using the bisection in \cite[Chapter 4.4]{zhou1998essentials}). This is already reflected in our computation for the arguably simplest $\mathcal{H}_\infty$ case in \Cref{example:non-coercivity-of-Hinf-cost}, which is much more involved than the LQG counterpart in \Cref{example:discontinuity-LQG-boundary}. Consequently, our results below are similar to but less complete than the LQG case. \hfill \qed
\end{remark}

\subsection{Main technical results} \label{subsection:Hinf-main-results}

%{\color{red}updates to here}

{\Cref{lemma:H_inf_some_local_Lipschitz} justifies that $J(\mK)$ is Clarke subdifferentiable. It is now clear from \Cref{lemma:clarke_stationary} that if a dynamic policy $\mK \in \mathcal{C}_n$ is a local minimum of $J_{\infty,n}(\mK)$, then $\mK$ is a Clarke stationary point. Our main goal of this section is to establish a class of Clarke stationary points that are globally optimal to the $\mathcal{H}_\infty$  robust control~\cref{eq:Hinf_policy_optimization}, which is a nonsmooth counterpart to \Cref{theorem:LQG-main}.}

We can indeed compute the Clarke subdifferential of $J_{\infty,q}(\mK)$ at any feasible point $\mK \in \mathcal{C}_q$. Let $j\mathbb{R}$ denote the imaginary axis in $\mathbb{C}$. Fix $\mK\in\mathcal{C}_q$, and define
\[
\mathcal{Z} = \{s\in j\mathbb{\mathbb{R}}\cup \{\infty\}
\mid \sigma_{\max}(\mathbf{T}_{zd}(\mK,s)) = J_{\infty,q}(\mK)\}.
\]
For each $s\in\mathcal{Z}$, let $Q_s$ be a complex matrix whose columns form an orthonormal basis of the eigenspace of $\mathbf{T}_{zd}(\mK,s)\mathbf{T}_{zd}(\mK,s)^\her$ associated with its maximal eigenvalue $J_{\infty,q}^2(\mK)$. The following lemma characterizes the subdifferential $\partial J_{\infty,q}(\mK)$, whose proof is technically invovled.%
\footnote{
We note that existing literature~\cite{apkarian2006nonsmooth} has also discussed the calculation of the subdifferential of $J_{\infty,q}$. Our proof follows a different approach and provides stronger results than~\cite{apkarian2006nonsmooth}.
} We present the details in \Cref{appendix:clarke-hinf}.

\begin{lemma}\label{lemma:subdifferential-Hinf}
A matrix $\Phi \in \mathbb{R}^{(m+q)\times(p+q)}$ is a member of $\partial J_{\infty,q}(\mK)$ if and only if there exist finitely many $s_1,\ldots,s_K\in \mathcal{Z}$ and associated positive semidefinite Hermitian matrices $Y_1,\ldots,Y_K$ with $\sum_{\kappa=1}^K\operatorname{tr}Y_\kappa=1$ such that
\begin{align*}
\Phi = \frac{1}{J_{\infty,q}(\mK)}
\sum_{\kappa=1}^K
\operatorname{Re}\bigg\{\!
& \left(
\begin{bmatrix}
0 & V^{1/2} \\ 0 & 0
\end{bmatrix}
+
\begin{bmatrix}
C & 0 \\
0 & I
\end{bmatrix}
(s_\kappa I-A_{\mathrm{cl}}(\mK))^{-1}B_{\mathrm{cl}}(\mK)
\right)
\mathbf{T}_{zd}(\mK,s_\kappa)^\her Q_{s_\kappa} Y_\kappa Q_{s_\kappa}^\her \\
&\quad \cdot
\left(
\begin{bmatrix}
0 & 0 \\ R^{1/2} & 0
\end{bmatrix}
+C_{\mathrm{cl}}(\mK)(s_\kappa I-A_{\mathrm{cl}}(\mK))^{-1}
\begin{bmatrix}
B & 0 \\ 0 & I
\end{bmatrix}
\right)
\!\bigg\}^{\!\tr}.
\end{align*}
\end{lemma}

%\vspace{30mm}
%{\color{red}a few sentences of the proof ideas for \cref{lemma:subdifferential-Hinf}. }
Note that the subdifferential computation in \cref{lemma:subdifferential-Hinf} is much more involved than gradient computations in \cref{lemma:gradient_LQG_Jn}. Still, similar to the LQG case, it is likely that the $\mathcal{H}_\infty$ control \cref{eq:Hinf_policy_optimization} has many strictly suboptimal Clarke stationary points. %We put this as a fact below.
\begin{remark}[Suboptimal Clarke stationary points in $\mathcal{H}_\infty$ control] \label{fact:suboptimal-stationary-point-Hinf}
    If the globally optimal $\mathcal{H}_\infty$ policies to \cref{eq:Hinf_policy_optimization} are all controllable and observable in $\mathcal{C}_n$, and the problem %if finding an optimal reduced-order controller
$
\min_{\mK\in\mathcal{C}_q} J_{\infty,q}(\mK)
$
has a solution for some $0 \leq q < n$, then there are infinitely many {strictly suboptimal Clarke stationary points} of $J_{\infty,n}(\mK)$ over $\mathcal{C}_n$. This result is based on the following characterization: any Clarke stationary point in the space of $\mathcal{C}_q$ can be augmented to a full-order Clarke stationary point in $\mathcal{C}_n$. We will discuss this lifting characterization in \Cref{subsection:Hinf-degenerate-lower-order}.
\end{remark}

Despite the undesirable result in \Cref{fact:suboptimal-stationary-point-Hinf}, our main technical result in this section confirms that if a Clarke stationary point corresponds to a non-degenerate policy in $\mathcal{C}_{\mathrm{nd}}$, then it is globally optimal to \cref{eq:Hinf_policy_optimization}. This is the nonsmooth counterpart to \Cref{theorem:LQG-main}.

\begin{theorem}[Global optimality]\label{theorem:Hinf-main-global-optimality}
    Let $\mK\in\mathcal{C}_{\mathrm{nd}}$ be any non-degenerate policy. If $\mK$ is a Clarke stationary point, i.e., $0 \in \partial J_{\infty,n}(\mK)$, then it is a global minimizer of $J_{\infty,n}(\mK)$ over $\mathcal{C}_n$.
\end{theorem}

{Our key proof strategy is based on the same \texttt{ECL} framework for epigraphs of nonconvex functions. Essentially, we show that the $\mathcal{H}_\infty$ optimization ``almost'' behaves as a convex problem.
We note that \Cref{theorem:Hinf-main-global-optimality} directly implies that there exist no spurious stationary points (local minimum/maximum/saddle) of $J_{\infty,n}(\mK)$ in ${\mathcal{C}}_{\mathrm{nd}}$. Indeed, we will show that if a policy $\mK \in {\mathcal{C}}_{\mathrm{nd}}$ is not globally optimal, then there exists a direction $\mV$ such that the directional derivative is negative (since $J_{\infty,n}(\mK)$ is subdifferential regular, its directional derivative always exists), i.e.,
$$
  \lim_{t \downarrow 0} \frac{J_{\infty,n}(\mK + t\mV) - J_{\infty,n}(\mK)}{t} < 0,
$$
thus it is always possible to improve over this point via suitable local search algorithms.
%}
}

In the LQG case (\cref{theorem:LQG-global-optimality}), we have shown that any minimal stationary point is non-degenerate, and thus is globally optimal. Motivated by this, we make the following conjecture.

% It looks likely we can establish a similar result to \Cref{theorem:LQG-global-optimality}. Looking interesting; needs to be checked in detail!! -- but not trivial
\begin{conjecture} \label{theorem:Hinf-global-optimality-minimal}
    Let $\mK\in {\mathcal{C}}_{n}$ be a Clarke stationary point  (i.e., $0 \in \partial J_{\infty,n}(\mK)$). If $\mK$ is a minimal policy, then it is non-degenerate, and thus is a global minimum of $J_{\infty,n}(\mK)$ over ${\mathcal{C}}_n$.
\end{conjecture}

Unlike the smooth LQG case, a rigorous proof of this conjecture (or identifying a counterexample) seems challenging, and we leave it to future work. One main difficulty lies in the non-smoothness of the $\mathcal{H}_\infty$ cost and the fact that the computation of the Clarke subdifferential (see \Cref{lemma:subdifferential-Hinf}) is much more involved than computing the gradients for LQG (see \Cref{lemma:gradient_LQG_Jn}). This difficulty also appears in the proof of \Cref{conjecture:measure-zero-hinf}.
% {\color{red} how about we assume that it is further differentiable?}

We here present an example, adapted from \cite[Example 14.3]{zhou1998essentials}, where the globally optimal $\mathcal{H}_\infty$ policy can be computed analytically\footnote{Unlike the LQG case, it is difficult to get globally optimal $\mathcal{H}_\infty$ policies in an analytical form even for very simple instances. We were struggling to find examples like those in \Cref{tab:stationary-points-LQG}.}. In the example below, the globally optimal $\mathcal{H}_\infty$ policy from \Cref{theorem:Hinf-riccati} turns out to be degenerate.

\begin{example}[Globally optimal $\mathcal{H}_\infty$ policies that are degenerate] \label{example:Hinf-optimal}
Consider the same $\mathcal{H}_\infty$ control instance in \Cref{example:Hinf-nonsmooth}.
%
    % Let us consider a simple $\mathcal{H}_\infty$ instance with
    % $$
    % A = -1,\, B = 1, \,C = 1,\, Q = 1,\, R = 1,\, W = 1, \, V = 1.
    % $$
  Via analyzing the limiting behavior in \Cref{theorem:Hinf-riccati} (see \cite[Example 14.3]{zhou1998essentials}), we can show that a globally optimal $\mathcal{H}_\infty$ policy in this instance is achieved by static output feedback %{\color{red}@Rich: Could you compute the subdifferential for this point; put it in the appendix \Cref{example:Hinf_degenerate_global_opt}}
  $$
  u(t) = D_\mK^\star y(t), \quad \text{with} \quad D_\mK^\star = 1 - \sqrt{3},
  $$
  and the globally optimal $\mathcal{H}_\infty$ cost is $J^\star_{\infty,1} = \sqrt{3}-1$. Even for this simple example, the computations for global optimality are quite involved, and we put some details in \Cref{appendix:globally-optimal-controllers}.

  It is easy to see that for any $a < 0$, the following policy
  \begin{equation} \label{eq:example-global-Hinf-1}
    A_\mK = a, \quad B_\mK = 0, \quad C_\mK = 0, \quad D_\mK^\star = 1 - \sqrt{3},
  \end{equation}
  is also globally optimal. % (since they all corresponds to same )
  Via some tedious computations, we can verify that the matrix $P\succ 0$ satisfying the non-strict LMI \Cref{eq:Hinf-Bilinear-a} with $\gamma = \sqrt{3}-1$ must be in the form of
  $$
  P = \begin{bmatrix}
      1 & 0 \\ 0 & p_3
  \end{bmatrix} \succeq 0,
  $$
  where $p_3>0$ is any positive value (computational details are presented in \Cref{appendix:globally-optimal-controllers}). Thus, this class of state-space realizations of the globally optimal policy in \cref{eq:example-global-Hinf-1} is degenerate.
In \Cref{appendix:globally-optimal-controllers}, via some straightforward yet very tedious computations, we have further verified that the other state-space realizations
$$
\begin{aligned}
    A_\mK &= a, \quad B_\mK = b, \quad  C_\mK = 0, \quad  D_\mK = 1 - \sqrt{3} \\
    A_\mK &= a, \quad B_\mK = 0, \quad  C_\mK = c, \quad  D_\mK = 1 - \sqrt{3}
\end{aligned}
$$
with $a < 0, b \neq 0, c\neq 0$ are all degenerate. %We remark that these computations are .
\hfill \qed
\end{example}

In classical control, it is known that the globally optimal $\mathcal{H}_\infty$ controller might not be unique even in the frequency domain \cite[Page 406]{zhou1996robust}, and this is different from the LQG control which has a unique globally optimal solution in the frequency domain \cite[Theorem 14.7]{zhou1996robust}. Here, we use an example from \cite{GLOVER201715} to illustrate this fact.
\begin{example}[Non-uniqueness of globally optimal $\mathcal{H}_\infty$ policies]  Consider the general dynamics~\cref{eq:plant} with problem data
    $$
    \begin{aligned}
    A &= \begin{bmatrix}
        -1 & 0 \\ 0 & -2
    \end{bmatrix}, \; B_1 = \begin{bmatrix}
        1 \\0
    \end{bmatrix}, \; B_2 = \begin{bmatrix}
        0 \\ -2 + a
    \end{bmatrix}, \\
    C_1 &= \begin{bmatrix}
        1 & 1
    \end{bmatrix}, \; C_2 = \begin{bmatrix}
        -2 & 0
    \end{bmatrix}, \; D_{11} = 0, \; D_{12} = 1, \; D_{21} = 1,
    \end{aligned}
    $$
    where $a \in \mathbb{R}$ is a parameter. We note that this system is only stabilizable and detectable (violating \Cref{assumption:stabilizability}). We consider an $\mathcal{H}_\infty$ problem by designing a dynamic controller $\mathbf{u} = \mathbf{K} \mathbf{y}$ such that the $\mathcal{H}_\infty$ norm of the closed-loop transfer function from $w(t)$ to $z(t)$ is minimized. Via simple frequency-domain algebra, this problem reads as
\begin{equation} \label{eq:Hinf-example-non-unique}
    \gamma^* = \inf_{\mathbf{K}(s) \in \mathcal{RH}_\infty}\;\; \left\|\frac{1}{s+1} + \frac{(s-1)(s-a)}{(s+1)(s+2)}\mathbf{K}(s)\right\|_{\mathcal{H}_\infty},
\end{equation}
where $\mathbf{K}(s)$ is a stable dynamic controller.
This formulation \cref{eq:Hinf-example-non-unique} is a standard model-matching problem (one potential solution is via the Nevanlinna–Pick strategy) \cite{doyle2013feedback}. When $a = 0$, it is clear that $\gamma^* \geq \|\frac{1}{0+1}\|_\infty = 1$ no matter which controller is used. Further, it is not difficult to check that the following two controllers
$$
\mathbf{K}_1 = 0, \qquad \mathbf{K}_2 = \frac{1}{s+3}-1
$$
archive the optimal value $\gamma^* = 1$. Thus, both of them are globally optimal $\mathcal{H}_\infty$ controllers to \cref{eq:Hinf-example-non-unique}. Indeed, there is a family of $\mathbf{K}(s)$ achieving $\gamma^\star = 1$; see \cite{GLOVER201715} for more details.
\hfill \qed
\end{example}

\subsection{A class of potentially suboptimal Clarke stationary points} \label{subsection:Hinf-degenerate-lower-order}

%{\color{red}probably we cannot prove anything like the LQG case, since the solution is not unique. The KYP LMI is quite different. }

As highlighted in \Cref{fact:suboptimal-stationary-point-Hinf}, it is likely that the $\mathcal{H}_\infty$ robust control \cref{eq:Hinf_policy_optimization} has many suboptimal Clarke stationary points. We here establish an interesting result that any Clarke stationary point of $J_{\infty,q}(\mK)$ can be transferred to Clarke stationary points of $J_{\infty, q+q'}$ for any $q'>0$ with the same $\mathcal{H}_\infty$ value. Therefore, it is likely that these new Clarke stationary points are suboptimal over $\mathcal{C}_{q+q'}$.

\begin{theorem} \label{theorem:hinf_non_globally_optimal_stationary_point}
Let $q\geq 0$ be arbitrary. Suppose there exists $\mK=\begin{bmatrix}
D_{\mK} & C_{\mK} \\ B_{\mK} & A_{\mK}
\end{bmatrix}
\in \mathcal{C}_{q}$ such that $0 \in \partial J_{\infty, q}(\mK)$.
Then for any $q'\geq 1$ and any stable $\Lambda\in\mathbb{R}^{q'\times q'}$, the following policy
\begin{equation} \label{eq:gradient_nonglobally_K-hinf}
        \tilde{\mK}
    =\left[\begin{array}{c:cc}
    D_\mK & C_{\mK} &  0 \\[2pt]
    \hdashline
    B_{\mK} & A_{\mK} & 0 \\[-2pt]
    0 & 0 & \Lambda
    \end{array}\right] \in \mathcal{C}_{q+q'}
\end{equation}
is a Clarke stationary point of $J_{\infty, q+q'}$ over $\mathcal{C}_{q+q'}$ satisfying $J_{\infty, q+q'}\big(\tilde{\mK}\big)
=J_{\infty, q}(\tilde{\mK})$.
\end{theorem}

It is easy to see that $J_{\infty, q+q'}\big(\tilde{\mK}\big)
=J_{\infty, q}(\tilde{\mK})$ since they correspond to the same transfer function. For the proof of Clarke stationary points, we can directly use the computations of Clarke subdifferential in \cref{lemma:subdifferential-Hinf}.  %or a perturbation strategy in the frequency domain. We note that the perturbation strategy presents a unified proof for \Cref{lemma:LQG-lower-order-stationary-point-to-high-order} (LQG control) and \cref{theorem:hinf_non_globally_optimal_stationary_point} ($\mathcal{H}_\infty$ control).
The proof details are postponed in \Cref{appendix:lower-order-stationary-point-to-high-order}.
We conclude this section by highlighting some degenerate and suboptimal Clarke stationary points in $\mathcal{H}_\infty$ control.

\begin{remark}[Degenerate and suboptimal Clarke stationary points in $\mathcal{H}_\infty$ control]
Let $q' = n - q$ in \Cref{theorem:hinf_non_globally_optimal_stationary_point}. The augmented policy \cref{eq:gradient_nonglobally_K-hinf} with any stable $\Lambda\in\mathbb{R}^{q'\times q'}$ is likely to be degenerate, i.e., $ \tilde{\mK} \in {\mathcal{C}}_{n} \backslash {\mathcal{C}}_{\mathrm{nd}}$; otherwise, \Cref{theorem:Hinf-main-global-optimality} ensures that the augmented policy \cref{eq:gradient_nonglobally_K-hinf} will be globally optimal over the entire ${\mathcal{C}}_{n}$ even $\mK$ is just a stationary point of $J_{\infty, q}$ over ${\mathcal{C}}_{q}$ ($0\leq  q < n$). In other words, it is likely that the augmented policy \cref{eq:gradient_nonglobally_K-hinf} is a suboptimal Clarke stationary point. We have indeed proved in \Cref{lemma:LQG-lower-order-stationary-point-to-high-order} that the augmented policy \cref{eq:gradient_nonglobally_K} is degenerate for LQG control under a mild condition. However, a proof for the case of $\mathcal{H}_\infty$ control seems challenging. This difficulty can again be traced back to the non-smoothness of $\mathcal{H}_\infty$ norm: we can use the smooth Lyapunov equation \cref{eq:Lyapunov-equations-H2norm} to compute the $\mathcal{H}_2$ norm,  while the bounded real lemma (\cref{lemma:bounded_real}) for the $\mathcal{H}_\infty$ norm is non-trivial to manipulate analytically. We also expect that strict saddle points exist for $\mathcal{H}_\infty$ cost  $J_{\infty, n}(\mK)$ over $\mathcal{C}_n$, but an explicit example (similar to \Cref{example:LQG_strict_saddle}) is yet to be found.  \hfill \qed
\end{remark}

%\clearpage

% Section 6
\section{Numerical Experiments} %Experiments with Policy Gradient Methods}

\label{section:experiments}

In this section, we present a few numerical experiments to showcase the effectiveness of various methods to approach the global minimum of LQG control \cref{eq:LQG_policy_optimization} and $\mathcal{H}_\infty$ robust control \cref{eq:Hinf_policy_optimization}.

Among different methods, we particularly illustrate the numerical behavior of an existing first-order policy optimization package -- \method{HIFOO} \cite{arzelier2010h2,burke2006hifoo}. \method{HIFOO} employs a two-stage approach involving stabilization (spectral abscissa minimization) followed by $\mathcal{H}_2$ or $\mathcal{H}_\infty$ performance optimization. Both stages utilize nonsmooth and nonconvex optimization techniques with the following components: an initial quasi-Newton algorithm phase to approach a local minimizer, and a subsequent phase including local bundle and gradient sampling methods to verify local optimality (see  \cite[Section 3]{arzelier2010h2} and \cite[Section 2]{gumussoy2008fixed} for algorithm details).

The code for our experiments (as well as all the figures in this paper) is available at

\begin{center}
\url{https://github.com/soc-ucsd/nonconvex_landscape_in_control}.
\end{center}

\subsection{LQG optimal control}

We here consider four different methods to solve the smooth nonconvex LQG control \cref{eq:LQG_policy_optimization}:
\begin{enumerate}
    \item Analytical solution via solving two Riccati equations (see \Cref{theorem:LQG-riccati-solution}).
    \item An LMI approach via a change of variables (see \cite{scherer2000linear} or Part II of this paper). We solve the resulting LMI via an interior-point conic solver -- \method{Mosek} \cite{aps2019mosek}.
    \item A naive gradient descent policy optimization approach, implemented in \cite{zheng2021analysis} (we ran 500 iterations).
    \item A sophisticate implementation of first-order policy optimization --  \method{HIFOO} \cite{arzelier2010h2,burke2006hifoo}.
\end{enumerate}

% It turns out that compared with $\mathcal{H}_\infty$ control, \Cref{eq: min_Fcvx} for LQG is easier to solve for obtaining the globally optimal cost.
We note that the first two methods are model-based, and the last two approaches can in principle be made mode-free via zeroth-order techniques (although we use the model information to compute gradients in our experiments). With known problem data, the globally LQG policy can be very easily computed via two Riccati equations (\Cref{theorem:LQG-riccati-solution}), and we use it to benchmark the performance of the~other three methods. For our numerical experiments, we consider three LQG instances: 1)~a problem with a scalar state variable in \Cref{example:discontinuity-LQG-boundary}; 2) a problem of two-dimensional state with data
 $$A=\begin{bmatrix} 0 & -1 \\ 1 & 0 \end{bmatrix},\ B = \begin{bmatrix} 1 \\ 0 \end{bmatrix},\ C = \begin{bmatrix} 1 & -1 \end{bmatrix},\ W = \begin{bmatrix} 1 & -1 \\ -1 & 16 \end{bmatrix},\ V = 1,\ Q=\begin{bmatrix} 4 & 0 \\ 0 & 0 \end{bmatrix},\ R = 1$$
and 3) another  problem of three-dimensional state with data
    $$A = \begin{bmatrix} 1 & 1 & 1 \\ 0 & 1 & 0 \\ 1 & 0 & 0 \end{bmatrix}, \ B = \begin{bmatrix} 1 & 0 \\ 0 & 1 \\ 0 & 0 \end{bmatrix},\ C = \begin{bmatrix} 0 & 0 & 1 \\ 1 & 0 & 0 \\ 0 & 1 & 1 \end{bmatrix},\ W = I, \ V = I, \ Q = I, \ R=I.$$

\begin{table}[t]
\setlength{\belowcaptionskip}{0pt}
  \begin{center}
    \caption{The best LQG costs of three instances returned by different methods.}
    \label{table:Jopt_LQG_1}
    \begin{tabular}{l*{3}{l}}
    \toprule
     & Instance 1 & Instance 2 & Instance 3  \\
    \hline
    Analytical & $0.6966$ & $6.1644$ & $10.3566$ \\
%    LMI(mincx) & $0.6966$ & $6.1644$ & $10.3566$ \\
    LMI (\method{Mosek}) & $0.6966$ & $6.1644$ & $10.3566$ \\
    Gradient descent & $0.6970$ & $6.1644$ & $10.5397$ \\
    \method{HIFOO} & $0.6967$ & $6.1644$ & $10.3566$ \\
    \toprule
    \end{tabular}
  \end{center}
\end{table}

\Cref{table:Jopt_LQG_1} lists the final LQG costs obtained through different methods. We observe that the global minimum of LQG costs (as computed analytically in the first row of \Cref{table:Jopt_LQG_1}) can be easily attained via solving the corresponding LMI (theoretically, Ricatti-based and LMI-based approaches should return solutions with the same performance; but numerically they may return different solutions; as we will see for the $\mathcal{H}_\infty$ case in \Cref{subsection:hinf-experiments}, the corresponding LMIs appear more difficult to solve numerically). The naive gradient descent implementation in \cite{zheng2021analysis} converges for the second LQG instance, but fails to return a high-quality solution for the other two instances within 500 iterations.

Interestingly, the \method{HIFOO} package has remarkable empirical performance for these three instances, and it returns a globally optimal solution for each of the LQG instances within 120 iterations. \Cref{figure:LQG_HIFOO_conv_1} illustrates their convergence performance. We note that while the \method{HIFOO} package can only certify local optimality, our theoretical guarantees in \Cref{theorem:LQG-main,theorem:LQG-global-optimality} allow for the certification of global optimality for the resulting policy.

\begin{figure}
    \centering
        \includegraphics[width=0.55\textwidth]{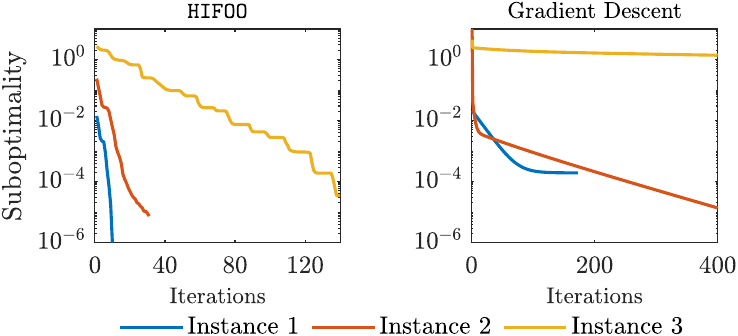}
        \caption{The empirical convergence performance for solving LQG instances. Left: \method{HIFOO} \cite{arzelier2010h2,burke2006hifoo}, where the suboptimality is measured by the LQG cost value gap; Right: Naive gradient descent implementation from \cite{zheng2021analysis}, where the suboptimality is measured by the gradient norms.}
        \label{figure:LQG_HIFOO_conv_1}
\end{figure}

\subsection{$\mathcal{H}_\infty$ robust control} \label{subsection:hinf-experiments}

We next consider the nonsmooth and nonconvex $\mathcal{H}_\infty$ robust control \cref{eq:Hinf_policy_optimization} using four different methods:

\begin{enumerate}
    \item Analytical solution via analyzing the limiting solutions from Riccati equations (see \Cref{theorem:Hinf-riccati}).
    \item An LMI approach via a change of variables (see \cite{scherer2000linear} or Part II of this paper). We solve the resulting LMI via two routines i) a general-purpose interior-point conic solver -- \method{Mosek} \cite{aps2019mosek}; ii) a specialized solver for LMIs -- \method{mincx} \cite{nemirovskii1994projective}.
    \item The MATLAB function \method{hinfsyn}: there are two options in this function using either i) Riccati iterations \cite{doyle1988state} or ii) LMIs \cite{gahinet1994linear}. Our numerical experiments used their default settings in MATLAB.
    \item A sophisticate implementation of first-order policy optimization --  \method{HIFOO} \cite{arzelier2010h2,burke2006hifoo}.
\end{enumerate}

As discussed in the main text, obtaining the global minimum of \cref{eq:Hinf_policy_optimization} via the first method is often impossible. The limiting analysis is very involved even for very simple instances such as \Cref{example:Hinf-optimal} (see our computational details in \Cref{appendix:globally-optimal-controllers}.). Our numerical experience suggests that the LMIs from $\mathcal{H}_\infty$ optimization are often harder to solve than those from $\mathcal{H}_2$ optimization; thus we use two different solvers (\method{Mosek} and \method{mincx}) in the second method. We also report the performance using the built-in function \method{hinfsyn} in MATLAB with its default setting. We will see that this implementation may not give the true globally optimal $\mathcal{H}_\infty$ performance due to the nature of Riccati iterations (although its performance can be improved by properly tuning its setting in MATLAB). Finally, all these methods above are model-based, while the last method via \method{HIFOO} might be adapted to the model-free setting (the details are non-trivial and left for future work).

%We compare three different methods to approach the global minimum $\mathcal{H}_\infty$ cost. The first method entails numerically solving LMIs in \Cref{eq:min_Fcvx} with different solvers, the second employs hinfsyn from the MATLAB Robust Control Toolbox (including both Riccati-based and LMI-based methods), and the last is by applying \method{HIFOO} for $\mathcal{H}_\infty$ control.

We test the aforementioned methods to six $\mathcal{H}_\infty$ control instances. Three of them are academic examples: 1) the first instance is the same as \Cref{example:Hinf-optimal}, for which we have derived its global optimal policy analytically; 2) the second instance has problem data
$$A = \begin{bmatrix}
        1 & 1 \\ 0 & 1
    \end{bmatrix},\ B = \begin{bmatrix}
        0 \\ 1
    \end{bmatrix},\ C = \begin{bmatrix}
        1 & 1 \\ 1 & 0
    \end{bmatrix},\ W = I, \ V = I, \ Q = I,\ R=I,$$
and the third instance is
   $$
   A = \begin{bmatrix}
        1 & 1 & 1 \\ 0 & 1 & 0 \\ 1 & 0 & 0
    \end{bmatrix}, \ B = \begin{bmatrix}
        1 & 0 \\ 0 & 1 \\ 0 & 0
    \end{bmatrix},\ C = \begin{bmatrix}
        0 & 0 & 1 \\ 1 & 0 & 0 \\ 0 & 1 & 1
    \end{bmatrix},\ W = I, \ V = I, \ Q = I, \ R=I.
    $$
The remaining three instances are benchmark examples from real applications, taken from  the \textit{COMPl\textsubscript{e}ib} library \cite{leibfritz2004compleib}:
\begin{itemize}
    \item AC8 \cite{gangsaas1986application}: A 9th-order state-space model of the linearized vertical plane dynamics of an aircraft;
    \item HE1 \cite{keel1988robust}: A 4th-order model of the longitudinal motion of a VTOL helicopter for typical loading and flight condition at the speed of 135 knots;
    \item REA2 \cite{hung1982multivariable}: A 4th-order chemical reactor model.
\end{itemize}
% The problem data are obtained from the COMPLeIB library \cite{leibfritz2004compleib}.

\begin{table}[t]
\setlength{\belowcaptionskip}{0pt}
\renewcommand{\arraystretch}{1.1}
      \begin{center}
        \caption{The $\mathcal{H}_\infty$ costs returned by different methods. N/A means the analytical solution is unavailable. The numbers highlighted in \textbf{bold type} denote the best value among all methods up to four significant digits. }
        \label{table:Jopt_Hinf_1}
        \begin{tabular}{ll*{6}{l}}
        \toprule
        Method & & $1$-dim & $2$-dim & $3$-dim & AC8 & HE1 & REA2 \\
        \hline
       Analytical &  & $\sqrt{3}-1$ & N/A & N/A & N/A & N/A & N/A \\ \cline{2-8}
       \multirow{2}{*}{LMI} & \method{mincx} & $\textbf{0.7321}$ & $5.4260$ & $\textbf{5.0829}$ & $\textbf{1.6165} $ & $\textbf{0.0736}$ & $\textbf{1.1341}$ \\
        & \method{Mosek} & $\textbf{0.7321}$ & $5.4267$ & $5.0834$ & $1.6169$ & $0.0738$ & $\textbf{1.1341}$ \\\cline{2-8}
       \multirow{2}{*}{\method{hinfsyn}} & {LMI} & $0.7322$ & $5.4288$ & $5.0844$ & $1.6176$ & $0.0739$ & $\textbf{1.1341}$ \\
        &  {Riccati} & $0.7381$ & $5.4748$ & $5.1117$ & $1.6258$ & $0.0740$ & $1.1362$ \\\cline{2-8}
       \method{HIFOO} &  & $\textbf{0.7321}$ & $\textbf{5.4259}$ & $5.0831$ & $1.6830$ & $0.0827$ & $\textbf{1.1341}$\\
        \toprule
        \end{tabular}
      \end{center}
    \end{table}

\Cref{table:Jopt_Hinf_1} lists the $\mathcal{H}_\infty$ costs returned by the methods we tested. As we can observe in the second row of \Cref{table:Jopt_Hinf_1}, the solutions returned by \method{Mosek} for the LMIs from $\mathcal{H}_\infty$ robust control is often worse than those returned by the specialized solver \method{mincx}. We have used the default setting of high-accuracy as the stopping criteria in \method{Mosek}, and the solver also reported success for the returned solutions. Similar numerical phenomena for the LMIs from $\mathcal{H}_2$ optimization was rare in our experience, indicating that LMIs from $\mathcal{H}_\infty$ robust control seem harder to solve numerically using general-purpose conic solvers. This may be related to smooth versus nonsmooth features between $\mathcal{H}_2$ and $\mathcal{H}_\infty$ optimization, and its detailed analysis is interesting but beyond the scope of this work.

From the third row of \cref{table:Jopt_Hinf_1}, it is interesting to observe that the solutions from the built-in Matlab routine \method{hinfsyn} (especially for the {Riccati} option) are often not as good as the solutions from LMI (\method{mincx}). This is expected since the Riccati option in \method{hinfsyn} is based on a bisection strategy, which can only approach the infimum in its limit but the corresponding Riccati equations may not be well-behaved asymptotically. %Indeed, we often observe that \method{hinfsyn} (Riccati) return t
Interestingly again, the \method{HIFOO} package has remarkable empirical performance for these $\mathcal{H}_\infty$ instances, and it returns high-quality solutions that are often as good as (very close to) those from LMIs. \Cref{figure:HIFOO_conv} illustrates the convergence behavior of \method{HIFOO} for these $\mathcal{H}_\infty$ instances.

\begin{figure}
    \centering
    \includegraphics[width=0.45\textwidth]{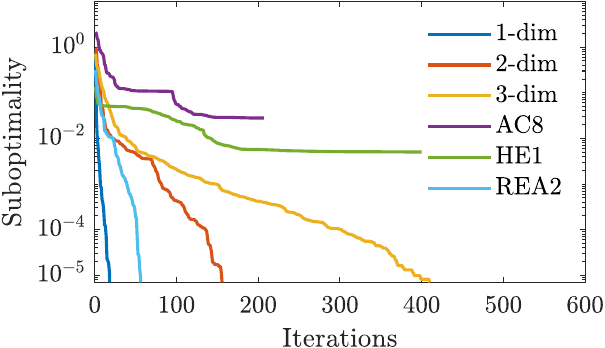}
    \caption{Empirical convergence performance of policy search methods via \method{HIFOO} \cite{arzelier2010h2,burke2006hifoo} for solving $\mathcal{H}_\infty$ instances, where the (local) suboptimality is measured by $\mathcal{H}_\infty$ cost gap compared to the best $\mathcal{H}_\infty$ value obtained from all methods.}.
    \label{figure:HIFOO_conv}
    \vspace{-3mm}
\end{figure}
From our numerical experiments, we observe another advantage of direct policy optimization like \method{HIFOO}: it directly optimizes over a policy class and the resulting policy is often ``well-behaved'' numerically. For example, for the first $\mathcal{H}_\infty$ instance, the solution returned by \method{HIFOO} reads
\begin{subequations}
\begin{align} \label{eq:example1-hifoo}
    J_{\infty,1}(\mK_{\text{\method{HIFOO}}}) = 0.73205, \quad \mK_{\text{\method{HIFOO}}} = \begin{bmatrix}
        -0.73205 & -0.2123\\
        0 & -1.6348
    \end{bmatrix},
\end{align}
while the solution returned by \method{hinfsyn} (via the default Riccati option) %without specifying the target performance
reads
\begin{align} \label{eq:example1-riccati}
    J_{\infty,1}(\mK_{\text{riccati}}) = 0.73814,\quad \mK_{\text{riccati}} = \begin{bmatrix}
        0 & -0.7111 \\
        9.8676 & -10.2738
    \end{bmatrix},
\end{align}
and the solution returned by \method{hinfsyn} (via the default LMI option) reads
\begin{align} \label{eq:example-lmi}
    J_{\infty,1}(\mK_{\text{lmi}}) = 0.73221,\quad \mK_{\text{lmi}} = \begin{bmatrix}
        0 & -21.65687 \\
        21.6569 & -641.8895
    \end{bmatrix}.
\end{align}
\end{subequations}
The solutions in \cref{eq:example1-riccati,eq:example-lmi} from \method{hinfsyn} correspond to strictly proper dynamic policies, while the solution \cref{eq:example1-hifoo} \method{HIFOO} is in fact a static output feedback $u(t) = -0.73205  y(t) \approx  (1-\sqrt{3})y(t)$ which is almost the same as the analytical solution in \Cref{example:Hinf-optimal}. Theoretically, one can use strictly proper dynamic policies to approach the global infimum of $\mathcal{H}_\infty$ control (guaranteed by \Cref{theorem:Hinf-riccati}), but these strictly proper dynamic policies may have large numerical values in their state-space realizations, as shown in \Cref{eq:example-lmi,eq:example1-riccati}.

% Section 7

\section{Conclusions} \label{section:conclusion}

%{\color{red}To do}

% \clearpage

In this paper, we have examined the nonconvex optimization landscapes of two fundamental control problems: the LQG control with stochastic noises, and $\mathcal{H}_\infty$ robust control with adversarial noises. Despite the nonconvexity, our results have characterized the global optimality of non-degenerate stationary points for both LQG and $\mathcal{H}_\infty$ optimization over {dynamic} policies. These results reveal the hidden convexity in LQG control and $\mathcal{H}_\infty$ robust control. We have postponed our main proof techniques via a new analysis framework \ECL{} to Part II of this paper. 

There are a few research topics that are worth further investigation. First, we are interested in establishing convergence conditions for local search algorithms to stationary points; this seemingly simple question is nontrivial to address especially due to the non-coercivity of LQG and $\mathcal{H}_\infty$ cost functions and their ``discontinuity'' behavior around the boundary, and thus the existing results \cite{absil2005convergence,burke2005robust} cannot be applied immediately. Second, it would be of great interest to design data-driven approaches for checking whether a policy is non-degenerate and avoiding degenerate policies via suitable regularization. Our definition of non-degenerate policies depends on the solution of a matrix inequality. It will be interesting to investigate its further system interpretations and its relationship with classical $\mathcal{H}_2$ and $\mathcal{H}_\infty$ theories. Finally, the classical literature on controls typically focuses on suboptimal $\mathcal{H}_\infty$ policies, and there are good reasons for this (some are practical and some are technical \cite[Section 5.1]{glover2005state}). From a nonsmooth optimization perspective, it is still extremely interesting to further characterize the global optimality of $\mathcal{H}_\infty$ optimization. Following the pioneer work \cite{apkarian2006nonsmooth,burke2005robust}, we believe this nonsmooth optimization perspective for $\mathcal{H}_\infty$ synthesis will continue to produce fruitful results, especially considering the recent context of model-free data-driven control.

% establishing convergence conditions for gradient descent algorithms

% convergence to stationary points (this is non-trivial because of open domain and potentially bad behaviors around the boundary points; some existing proofs are not complete)

% \clearpage 

\bibliographystyle{unsrt}
\bibliography{ref.bib}

\newpage
\appendix

\numberwithin{equation}{section}
\numberwithin{example}{section}
\numberwithin{remark}{section}
\numberwithin{assumption}{section}
\vspace{10mm}
\noindent\textbf{\Large Appendix}

\vspace{5mm}
This appendix contains many auxiliary results, additional discussions, and technical proofs. We divide it into four parts:
\begin{itemize}
    \item \cref{appendix:preliminaries} presents some preliminaries in control theory; % and nonsmooth optimization;
 \item \cref{app:nonsmooth-optimization} presents some preliminaries in nonsmooth optimization; %
 \item \Cref{appendix:auxillary-results} presents auxiliary results for LQG and $\mathcal{H}_\infty$ control;
 \item \Cref{appendix:technical-proofs} presents technical proofs that are omitted in the main text. % for continuous-time systems;
% \item \Cref{appendix:ECL} presents additional discussions and proofs for our \texttt{ECL} framework;

\end{itemize}

\section{Fundamentals of Control Theory} \label{appendix:preliminaries}

For self-completeness, we here review some fundamentals on controllability, observability,  Lyapunov equations, and LMI characterizations of $\mathcal{H}_2$/$\mathcal{H}_\infty$ norms.
Most of the results below can be found in standard textbooks in control; see \cite[Chap 3 \& 4 ]{zhou1996robust}, \cite{dullerud2013course,zhou1998essentials,scherer2000linear,boyd1994linear}. Minor pedagogical contributions might be our explicit emphasis on some subtleties between strict and nonstrict LMIs.

\subsection{Controllability and observability}
\label{appendix:preliminaries:ctrl_obsr}
Consider an LTI system parameterized by $(A, B, C, D) \in \mathbb{R}^{n \times n} \times \mathbb{R}^{n \times m} \times \mathbb{R}^{p \times n} \times \mathbb{R}^{p \times m}$,
\begin{equation} \label{eq:App_dynamics}
    \begin{aligned}
        \dot{x} &= A x + Bu,\\
        y &= Cx + Du.
    \end{aligned}
\end{equation}
The pair $(A,B)$ is called \emph{controllable} if the following controllability matrix
$$
\begin{bmatrix} B & AB & \ldots & A^{n-1}B \end{bmatrix}
$$
has full row rank. The pair $(C,A)$ is called \emph{observable} if the following observability matrix
$$
\begin{bmatrix} C \\ CA \\ \vdots \\ CA^{n-1}\end{bmatrix}
$$
has full column rank. It is easy to verify that $(A,B)$ is controllable if and only if $(B^\tr, A^\tr)$ is observable. We also say that the system~\cref{eq:App_dynamics} is controllable (resp. observable) if the associated pair $(A,B)$ is controllable (resp. $(C,A)$ is observable).

The input-output behavior of~\cref{eq:App_dynamics} can also be equivalently described in the frequency domain by the \emph{transfer matrix} $\mathbf{G}(s) = C(sI - A)^{-1}B + D$.
It is easy to verify that the transfer matrix $\mathbf{G}(s)$ is invariant under any \emph{similarity transformation} on the state-space model
\[
(A,B,C,D)\mapsto (TAT^{-1}, TB, CT^{-1}, D),
\]
where $T$ is any invertible $n\times n$ real matrix. The system~\cref{eq:App_dynamics} is called \emph{minimal} if $(A,B)$ is controllable and $(C,A)$ is observable.
This terminology is justified by the following interpretation: if the system~\cref{eq:App_dynamics} is not minimal, then there exists another state-space model with a smaller state dimension $\hat{n} < n$
$$
    \begin{aligned}
        \dot{\hat{x}} &= \hat{A} \hat{x} + \hat{B} u \\
        y &= \hat{C} \hat{x} + D u,
    \end{aligned}
$$
such that the input-output behavior is the same as~\cref{eq:App_dynamics}, \emph{i.e.},
$
    \mathbf{G}(s) = \hat{C} (sI - \hat{A})^{-1}\hat{B} + D.
$
%In this paper, we use the notions of \emph{``minimal controller''} and \emph{``controllable and observable controller''} in an interchangeable way.

In addition to the (Kalman) rank test, there are other equivalent conditions for controllability. %; see \cite[Chapter 3]{zhou1996robust} for further details.
\begin{lemma}[{\cite[Theorem 3.1]{zhou1996robust}}] \label{lemma:controllability}
    Given a matrix pair $(A, B) \in \mathbb{R}^{n \times n} \times \mathbb{R}^{n \times m}$, the following statements are equivalent.
    \begin{enumerate}
        \item $(A,B)$ is controllable.
        \item (Pole placement) The eigenvalues of $(A+BK)$ can be arbitrarily assigned\footnote{Complex eigenvalues need to come as conjugate pairs.} by choosing $K \in \mathbb{R}^{m \times n}$.
        \item (Invariance under state feedback) $(A+BK,B)$ is controllable for any $K\in \mathbb{R}^{m \times n}$.
    \end{enumerate}
\end{lemma}
%Statement 3 is straightforward from the equivalence between (1) and (2), since the eigenvalues of $A + BK + BF = A+ B(K+F)$ can be freely assigned. A direct proof from the controllability rank test can be found in \cite{viswanadham1975invariance}; see \cite[Chapter 5.7]{wonham2013linear} for further details on the invariance of controllability.

In addition to controllability and observability, we will also occasionally encounter the notions of stabilizability and detectability. We say that the pair $(A,B)$ is \emph{stabilizable}, if there exists a matrix $F$ such that $A+BF$ is stable. We say that $(C,A)$ is \emph{detectable}, if there exists a matrix $L$ such that $A+LC$ is stable.

\subsection{Lyapunov equations} \label{appendix:Lyapunov-equation}
We here review some classical results on Lyapunov equations.
%which can be used to compute $\mathcal{H}_2$ norms of stable LTI systems (\cref{lemma:H2norm}).
%
Given a real matrix $A \in \mathbb{R}^{n\times n}$ and a symmetric matrix $Q \in \mathbb{S}^n$, a standard Lyapunov equation is
\begin{equation} \label{eq:Lyapunov_equation}
    AX + X A^\tr + Q = 0,
\end{equation}
This is a linear equation in the variable $X$.
%, and its vectorized version is  $   ( I_n \otimes A^\tr  + A^\tr \otimes I_n)\operatorname{vec}(X) = -\operatorname{vec}(Q)$, where we use $\otimes$ to denote Kronecker product. It is known that when $A$ is stable, the matrix $( I_n \otimes A^\tr  + A^\tr \otimes I_n)$ is invertible, and consequently the Lyapunov equation~\cref{eq:Lyapunov_equation} admits a unique symmetric solution for any matrix $Q \in \mathbb{S}^n$.
When the matrix $A$ is stable, we have the following lemma that characterizes the solution to the Lyapunov equation.

\begin{lemma}[{\cite[Lemma 3.18(i)(ii)]{zhou1996robust}}]
\label{lemma:Lyapunov}
Let $A\in\mathbb{R}^{n\times n}$ be stable, i.e., all eigenvalues of $A$ have negative real parts. Then the Lyapunov equation~\cref{eq:Lyapunov_equation} has a unique solution given by
\[
X = \int_0^{+\infty} e^{A t} Q e^{A^\tr t}\,dt.
\]
In addition, $X\succeq 0$ if $Q\succeq 0$, and $X\succ 0$ if $Q\succ 0$.
\end{lemma}
Another related Lyapunov equation to \cref{eq:Lyapunov_equation} is
\begin{equation*} %\label{eq:Lyapunov_equation-v2}
    A^\tr Y + Y A + Q = 0,
\end{equation*}
which replaces $A$ with $A^\tr$. Thus, similar results to \cref{lemma:Lyapunov} hold. Furthermore, we can also relate the controllability and observability of a system with the solutions to certain Lyapunov equations.
\begin{lemma}[{\cite[Lemma 3.18(iii)]{zhou1996robust}}]
\label{lemma:Lyapunov_controllability_observability}
Let $A\in\mathbb{R}^{n\times n}$ be stable, and let $L_c$ and $L_o$ be the unique solutions to the following Lyapunov equations
\begin{align}
AL_{\mathrm{c}} + L_{\mathrm{c}}A^\tr + BB^\tr & = 0,
\label{eq:controllability_Gramian} \\
A^\tr L_{\mathrm{o}} + L_{\mathrm{o}}A + C^\tr C & = 0.
\label{eq:observability_Gramian}
\end{align}
Then $(A,B)$ is controllable if and only if $L_{\mathrm{c}}\succ 0$. Similarly, $(C,A)$ is observable if and only if $L_{\mathrm{o}}\succ 0$.
\end{lemma}
In control theory, the matrices $L_{\mathrm{c}}$ and $L_{\mathrm{o}}$ defined in \Cref{lemma:Lyapunov_controllability_observability} above are called \emph{controllability} and \textit{observability} Gramians for the (stable) system \cref{eq:App_dynamics}, respectively.

Note that in \Cref{lemma:Lyapunov,{lemma:Lyapunov_controllability_observability}}, we have assumed that $A$ is stable. Given the solution to~\cref{eq:Lyapunov_equation}, converse results also exist to establish the stability of $A$. We present one version that is used in our context. These converse results are important when we aim to design a controller that stabilizes a plant (see \Cref{remark:stability-constraint}).

\begin{lemma}[{\cite[Lemma 3.19]{zhou1996robust}}] \label{lemma:converse-Lyapunov}
    Let $X$ be a solution to the Lyapunov equation \cref{eq:Lyapunov_equation}.  %with $Q = BB^\tr \succeq 0$, where $B \in \mathbb{R}^{n \times m}$. If $(A,B)$ is stabilizable and $Y \succeq 0$, then $A$ is stable.
    \begin{enumerate}
    \item If $Q\succeq 0$ and $X \succ 0$, then $A$ is at least marginally stable, i.e., all eigenvalues of $A$ have non-positive real parts.
    \item If $Q \succ 0$ and $X \succ 0$, then $A$ is stable.
        \item If $Q = BB^\tr \succeq 0$ with $B \in \mathbb{R}^{n \times m}$, $(A,B)$ is stabilizable and $X \succeq 0$, then $A$ is stable.
    \end{enumerate}
\end{lemma}
\begin{comment}
\begin{proof}
    Let $\lambda$ be an eigenvalue of $A$ and $v\neq 0$ be a corresponding eigenvector. Pre-multiplying \cref{eq:Lyapunov_equation-v2} by $v^\her$ and postmultiplying \cref{eq:Lyapunov_equation-v2} by $v$ lead to
    $$
    (\bar{\lambda} + \lambda) v^\her Y v + v^\her BB^\tr v = 0 \; \Rightarrow \; 2 \operatorname{Re}[\lambda] v^\her Y v + v^\her BB^\tr v = 0.
    $$
    If $\operatorname{Re}[\lambda] \geq 0$, then $v^\her BB^\tr v \leq 0$. Thus, we must have $v^\her BB^\tr v = 0$, i.e., $B^\tr v = 0$. This implies that $\lambda$ is an unstable and unobservable mode \cite[Theorem 3.4]{zhou1996robust}, contradicting the assumption that $(A,B)$ is stabilizable.
\end{proof}
\end{comment}

The second statement of \Cref{lemma:converse-Lyapunov} combined with \Cref{lemma:Lyapunov} implies that the system of linear matrix inequalities $A X + X A^\tr \prec 0, X \succ 0$ is feasible if and only if $A$ is stable (which is the famous canonical Lyapunov inequality). Also, \Cref{lemma:converse-Lyapunov} immediately leads to a useful alternative characterization.
\begin{lemma} \label{lemma:converse-Lyaounov}
\begin{enumerate}
    \item If $(A,B)$ is stabilizable, and the following system of LMIs is feasible,
    \begin{equation} \label{eq:reverse-Lyapunov-inequality}
        A X + X A^\tr + BB^\tr \preceq 0, \qquad Y \succeq 0,
    \end{equation}
    then $A$ is stable.
    \item If \cref{eq:reverse-Lyapunov-inequality} is feasible with $X \succ 0$, then $A$ is at least marginally stable, i.e., all eigenvalues of $A$ have non-positive real parts.
\end{enumerate}
\end{lemma}
\begin{proof}
    It is clear that $X \succeq 0$ is a solution to the Lyapunov equation
    $
    A X + X A^\tr + BB^\tr + B_0B_0^\tr = 0
    $
    for some matrix $B_0$. Since $(A,B)$ is stabilizable, we have that $\left(A,\begin{bmatrix}
        B & B_0
    \end{bmatrix}\right)$ is~also~stabilizable. The stability of $A$ directly follows from \cref{lemma:converse-Lyapunov}. \end{proof}

\subsection{Characterizations of $\mathcal{H}_2$ and $\mathcal{H}_\infty$ norms} \label{app:h2-hinf-norms}

We here discuss the proofs in \Cref{lemma:H2norm,lemma:bounded_real}, and highlight some subtleties between strict and non-strict LMIs.

\subsubsection{Characterizations of $\mathcal{H}_2$ norms} \label{appendix:proof-H2-lemma}

%Proof of \Cref{lemma:H2norm}

The results in \Cref{lemma:H2norm} are classical in control, which can be found in many textbooks. For example, the equivalence in \cref{eq:Lyapunov-equations-H2norm} can be found on \cite[Page 76]{scherer2000linear}, \cite[Page 203]{dullerud2013course} and in \cite[Lemma 4.6]{zhou1996robust}; more extensive strict LMI characterizations than \cref{eq:strict-LMI-H2norm} is summarized in \cite[Proposition 3.13]{scherer2000linear} (also see \cite[Proposition 6.13]{dullerud2013course}). However, it seems that the nonstrict LMI characterizations in \cref{eq:nonstrict-lmi-h2} have been less emphasized in the literature.
%(no discussions are provided in \cite{dullerud2013course,scherer2000linear}).}

Here, we reproduce their proofs as they help clarify certain subtleties between strict and nonstrict LMIs for characterizing $\mathcal{H}_2$ norms. As previewed in \Cref{subsection:smooth-LQG-cost}, the proofs below are constructive: the computation of $\|\mathbf{G}\|_{\mathcal{H}_2}^2$ via \cref{eq:Lyapunov-equations-H2norm} is a consequence of the Parseval theorem and \cref{lemma:Lyapunov}; we then explicitly construct solutions to \cref{eq:Lyapunov-equations-H2norm} to prove \cref{eq:strict-LMI-H2norm,eq:nonstrict-lmi-h2}. Special care needs to be put in when dealing with the converse direction of \cref{eq:nonstrict-lmi-h2} which requires the controllability of $(A, B)$.

%\vspace{30mm}

%\begin{proof}%[\unskip\nopunct]

%We establish the proofs for the three statements in \Cref{lemma:H2norm} as follows.

\vspace{3mm}
\noindent \textbf{Proof of Statement 1 in \Cref{lemma:H2norm}.}
    The computation of $\|\mathbf{G}\|_{\mathcal{H}_2}^2$ via \cref{eq:Lyapunov-equations-H2norm} is a consequence of the Parseval theorem and \cref{lemma:Lyapunov}. By definition, we have
    $$
    \|\mathbf{G}\|_{\mathcal{H}_2}^2 =  \frac{1}{2\pi} \int_{-\infty}^{\infty}\operatorname{tr}(\mathbf{G}(j\omega)^\her \mathbf{G}(j\omega))d\omega = \int_{0}^\infty \mathrm{trace}(B^\tr e^{A^\tr t}C^\tr Ce^{At}B)dt
    $$
    where the last equality follows the Parseval theorem. Using linearity of the trace, we have
    $$
    \|\mathbf{G}\|_{\mathcal{H}_2}^2 =  \operatorname{tr}\!\left(B^\tr \int_{0}^\infty e^{A^\tr t}C^\tr Ce^{At}dt B\right) = \mathrm{trace}(B^\tr L_{\mathrm{o}} B ),
    $$
    where $L_{\mathrm{o}}$ is the observability Gramian defined in \cref{lemma:Lyapunov_controllability_observability}. From the cyclic property of the trace, we also have
    $$
    \|\mathbf{G}\|_{\mathcal{H}_2}^2 =  \operatorname{tr}\!\left(C \int_{0}^\infty e^{A t}BB^\tr e^{A^\tr t}dt C^\tr \right) = \mathrm{trace}(C L_{\mathrm{c}} C^\tr ),
    $$
    where $L_{\mathrm{c}}$ is the controllability Gramian defined in \cref{lemma:Lyapunov_controllability_observability}.

\vspace{3mm}

\noindent \textbf{Proof of Statement 2 in \Cref{lemma:H2norm}.} This directly follows from a more complete result. % {(\color{red} needs to fix $P \succ 0$)}.

\begin{lemma} \label{lemma:H2-norm-strict}
Consider a transfer matrix $\mathbf{G}(s) = C(sI - A)^{-1}B$, where $A \in \mathbb{R}^{n\times n}$ is stable and $C \in \mathbb{R}^{p \times n}, B \in \mathbb{R}^{n \times m}$. The following statements are equivalent.
    \begin{enumerate} %\setlength{\itemsep}{2pt}
        \item $\|\mathbf{G}\|_{\mathcal{H}_2}<\gamma$.

        \item There exists a symmetric matrix $X$ such that\footnote{Any $X$ satisfying \cref{eq:strict-LMI-h2-norm} is strictly positive definite, which is from \Cref{lemma:Lyapunov}. It is also obvious that any solutions $P$ and $\Gamma$ satisfying \cref{eq:strict-LMI-1} are strictly positive definite.}
        \begin{equation} \label{eq:strict-LMI-h2-norm}
        A X+XA^\tr+B B^\tr \! \prec0, \quad \operatorname{tr}(C XC^\tr)< \gamma^2.
        \end{equation}
        \item There exist symmetric matrices $P$ and $\Gamma$  such that
        %\begin{subequations}
            \begin{align} \label{eq:strict-LMI-1}
                %AX + XA^\tr + BB^\tr &\prec 0, \quad \mathrm{trace}(CXC^\tr) < \gamma^2\\
                \begin{bmatrix} A^\tr P+PA & PB \\ B^\tr P & -\gamma I \end{bmatrix}&\prec 0, \;\; \begin{bmatrix} P & C^\tr \\ C & \Gamma \end{bmatrix}\succ 0,\;\; \operatorname{tr}(\Gamma)<\gamma.
            \end{align}
    \end{enumerate}
\end{lemma}

  \begin{proof}
    \begin{itemize} \setlength{\itemsep}{2pt}
        \item
     We first consider the direction 1) $\Rightarrow$ 2). By Statement 1 in \Cref{lemma:H2norm}, we know that $\|\mathbf{G}\|_{\mathcal{H}_2}^2 = \mathrm{tr}(C L_\mathrm{c}C^\tr) < \gamma^2$, where $L_\mathrm{c} \succeq 0$ is the controllability Gramian. Since $A$ is stable, there exists $\tilde{X}\succ 0$ such that $A\tilde{X}+\tilde{X}A^\tr\prec0$. We can then find some $\epsilon>0$ such that
    \begin{equation} \label{eq:H2-equation-to-LMI}
    X = L_\mathrm{c}+\epsilon \tilde{X}
    \end{equation}
    satisfies $X \succ 0$ and $\mathrm{tr}(C X C^\tr) < \gamma^2$. Furthermore, we have
    \[
    A X+XA^\tr+B B^\tr = AL_\mathrm{c}+L_\mathrm{c}A^\tr +B B^\tr+\epsilon(A\tilde{X}+\tilde{X}A^\tr)\prec0,
    \]
    which completes the proof of 1) $\Rightarrow$ 2).

    The direction 2 $\Rightarrow$ 1) follows a simple fact: \cref{eq:strict-LMI-h2-norm} implies the existence of a nonzero matrix $B_0$ such that
    $$
    A X+XA^\tr+B B^\tr +B_0 B_0^\tr=0, \quad  \mathrm{tr}(CXC^\tr)<\gamma^2.
    $$
    By Statement 1, we have  $\|\mathbf{G}_e\|_{\mathcal{H}_2} < \gamma$ where $\mathbf{G}_e=C(sI-A)^{-1}\begin{bmatrix}
            B & B_0 \end{bmatrix}$. Consequently, we have
            $$\|\mathbf{G}\|_{\mathcal{H}_2} \leq \|\mathbf{G}_e\|_{\mathcal{H}_2}<\gamma.$$

  \item  Now we consider the direction 2) $\Rightarrow$ 3). From \Cref{lemma:Lyapunov}, any feasible $X$ in \cref{eq:strict-LMI-h2-norm} is strictly positive definite. Then, from a feasible $X \succ 0$, we can find some $\epsilon>0$ such that
  \begin{equation} \label{eq:construction-Gamma-P}
   \Gamma =\frac{1}{\gamma}CXC^\tr+\epsilon I
   \qquad\text{and}\qquad
   P=\gamma X^{-1}
  \end{equation}
  satisfy $\Gamma\succ 0$, $P\succ 0$ and $\operatorname{tr}(\Gamma)<\gamma$.
   We further infer that   $\Gamma\succ\frac{1}{\gamma}CXC^\tr=CP^{-1}C^\tr$ and
    % and let $\epsilon>0$ be s.t. \[\Gamma=\frac{1}{\gamma}CXC^\tr+\epsilon I\] satisfies $\mathrm{tr}(\Gamma)<\gamma$.
    % Such $\epsilon$ exists since from $(ii)$ its second condition is equivalent to \[\mathrm{tr}\left(\frac{1}{\gamma}CXC^\tr\right)<\gamma.\]
    %Then obviously $P\succ0$ and we infer that
    \begin{align*}
        %AX + XA^\tr + BB^\tr &\prec 0, \quad \mathrm{trace}(CXC^\tr) < \gamma^2\\
        &A^\tr P+PA +\frac{1}{\gamma}PBB^\tr P=\gamma X^{-1}\left(AX+XA^\tr+B B^\tr\right) X^{-1} \! \prec0, %\\
       % &\Gamma\succ\frac{1}{\gamma}CXC^\tr=CP^{-1}C^\tr. %\\
      %  & \mathrm{tr}(\Gamma)<\gamma.
    \end{align*}
    By using Schur complements, these two inequalities are equivalent to those in \cref{eq:strict-LMI-1}.

    %\item
    The direction 3) $\Rightarrow$ 2) follows directly if we notice that, given $P$ and $\Gamma$ satisfying~\cref{eq:strict-LMI-1}, the matrix $X = \gamma P^{-1}$ satisfies \cref{eq:strict-LMI-h2-norm}. The proof is now complete. \qedhere

    % {\color{red}The proof of \Cref{lemma:H2norm} 3) is similar to that of \Cref{lemma:H2norm} 2). We need $BB^\tr\succ0$ so that $X\succ0$? }

        \end{itemize}
\end{proof}

  \vspace{3mm}
\noindent \textbf{Proof of Statement 3 in \Cref{lemma:H2norm}.} For this, we establish the following result. % we first show that the following two statements are equivalent:
   \begin{lemma} \label{lemma:H2-norm-non-strict}
   Consider a transfer matrix $\mathbf{G}(s) = C(sI - A)^{-1}B$, where $A \in \mathbb{R}^{n\times n}$ is stable and $C \in \mathbb{R}^{p \times n}, B \in \mathbb{R}^{n \times m}$. The following two statements are equivalent.
       \begin{enumerate}
        \item There exists $X\succ0$ such that\footnote{Unlike \Cref{eq:strict-LMI-h2-norm}  , the solution $X$ satisfying \cref{eq:nonstrict-LMI-h2-norm-1} is not automatically guaranteed to be positive definite unless $(A,B)$ is controllable. We thus need to explicitly require $X\succ0$ in the statement of our lemma.}
        \begin{equation}\label{eq:nonstrict-LMI-h2-norm-1}
        A X+XA^\tr+B B^\tr \! \preceq 0,\quad  \mathrm{tr}(C XC^\tr)\leq \gamma^2.
        \end{equation}

        \item There exist $P \succ 0$ and $\Gamma \succeq 0$ such that %({\color{red} still missing $P \succ 0$})
        %\begin{subequations}
            \begin{align} \label{eq:nonstrict-LMI-h2-norm}
                \begin{bmatrix} A^\tr P+PA & PB \\ B^\tr P & -\gamma I \end{bmatrix}&\preceq 0, \;\; \begin{bmatrix} P & C^\tr \\ C & \Gamma \end{bmatrix}\succeq 0,\;\; \mathrm{tr}(\Gamma)\leq\gamma.
            \end{align}
       \end{enumerate}
The two matrices $X$ and $P$ can be related by $X = \gamma P^{-1}$. Furthermore, either one (and thus both) of \Cref{eq:nonstrict-LMI-h2-norm-1} and \cref{eq:nonstrict-LMI-h2-norm} guarantees $\|\mathbf{G}\|_{\mathcal{H}_2}\leq \gamma$; the converse is true when $(A,B)$ is controllable.
   \end{lemma}

    \begin{proof}

    We use similar constructions as those in \Cref{lemma:H2-norm-strict}.

    \begin{itemize}
        \item  1) $\Rightarrow$ 2): From the solution $X \succ 0$, we use the same construction in \cref{eq:construction-Gamma-P} with $\epsilon = 0$. Then, we have $\mathrm{tr}(\Gamma)\leq\gamma$. We further infer that $\Gamma= CP^{-1}C^\tr \Rightarrow \Gamma-CP^{-1}C^\tr=0\succeq0 $ and
    \begin{align*}
        A^\tr P+PA +\frac{1}{\gamma}PBB^\tr P=\gamma X^{-1}\left(AX+XA^\tr+B B^\tr\right)X^{-1} \! \preceq0.
    \end{align*}
     The Schur complement confirms that $P$ and $\Gamma$ satisfy the nonstrict LMI \cref{eq:nonstrict-LMI-h2-norm}.
     % in   the first two inequalities are equivalent to
    % \begin{align*}
    %     %AX + XA^\tr + BB^\tr &\prec 0, \quad \mathrm{trace}(CXC^\tr) < \gamma^2\\
    %     \begin{bmatrix} A^\tr P+PA & PB \\ B^\tr P & -\gamma I \end{bmatrix}&\preceq 0, \;\; \begin{bmatrix} P & C^\tr \\ C & \Gamma \end{bmatrix}\succeq 0.
    % \end{align*}

    \item The direction 2) $\Rightarrow$ 1) follows the same construction $X =\gamma P^{-1} \succ 0$. Simple calculations confirm that $\mathrm{tr}(CXC^\tr)=\gamma\mathrm{tr}(CP^{-1}C^\tr)\leq\gamma\mathrm{tr}(\Gamma)\leq \gamma^2$ and
    % by Schur complements we have
    % \begin{align*}
    %     %AX + XA^\tr + BB^\tr &\prec 0, \quad \mathrm{trace}(CXC^\tr) < \gamma^2\\
    %     &A^\tr P+PA +\frac{1}{\gamma}PBB^\tr P\preceq0 \\
    %     &\Gamma-CP^{-1}C^\tr=0\succeq0.
    % \end{align*}
    % Defining $X:=\gamma P^{-1}\succ0$, we have
    \begin{align*}
        %AX + XA^\tr + BB^\tr &\prec 0, \quad \mathrm{trace}(CXC^\tr) < \gamma^2\\
        AX+XA^\tr+B B^\tr = \frac{1}{\gamma}X\left(A^\tr P+PA +\frac{1}{\gamma}PBB^\tr P\right)X \preceq 0. %\\
      %  & \mathrm{tr}(CXC^\tr)=\gamma\mathrm{tr}(CP^{-1}C^\tr)\leq\gamma\mathrm{tr}(\Gamma)=\gamma^2.
    \end{align*}
    \end{itemize}

We next show that \Cref{eq:nonstrict-LMI-h2-norm-1} implies  $\|\mathbf{G}\|_{\mathcal{H}_2}\leq \gamma$. From \cref{eq:nonstrict-LMI-h2-norm-1}, we can find a (possibly zero) matrix $B_0$ such~that
\[
A X+XA^\tr+B B^\tr +B_0 B_0^\tr=0, \;\; \operatorname{tr}(CXC^\tr)\leq\gamma^2.
\]
   By Statement 1 in \Cref{lemma:H2norm}, we know that $\|\mathbf{G}_e\|_{\mathcal{H}_2}\leq\gamma$ where $\mathbf{G}_e=C(sI-A)^{-1}\begin{bmatrix}
            B & B_0 \end{bmatrix}$. Consequently, we have $\|\mathbf{G}\|_{\mathcal{H}_2} \leq \|\mathbf{G}_e\|_{\mathcal{H}_2} \leq\gamma.$

    We finally show that if $\|\mathbf{G}\|_{\mathcal{H}_2}\leq \gamma$ and $(A,B)$ is controllable, then \cref{eq:nonstrict-LMI-h2-norm-1} holds with $X \succ 0$.
    By Statement 1 in \Cref{lemma:H2norm}, we know  $\|\mathbf{G}\|_{\mathcal{H}_2}^2 = \mathrm{tr}(C L_\mathrm{c}C^\tr) \leq \gamma^2$, where $L_\mathrm{c}$ satisfies  \cref{eq:controllability_Gramian}. Since $(A,B)$ is controllable, we know $L_\mathrm{c} \succ 0$ by \Cref{lemma:Lyapunov_controllability_observability}.
    The proof is finished by choosing
    \begin{equation} \label{eq:construction-X-L0}
        X=L_\mathrm{c}\succ0
    \end{equation}
which satisfies the nonstrict LMI \cref{eq:nonstrict-LMI-h2-norm-1}, i.e., $A X+XA^\tr+B B^\tr \! =0 \preceq 0,\ \mathrm{tr}(C XC^\tr) \leq \gamma^2.$
\end{proof}

%\end{itemize}

{
\begin{remark}[Strict vs nonstrict LMIs in $\mathcal{H}_2$ norm]
From the constructions in \Cref{lemma:H2-norm-strict,lemma:H2-norm-non-strict}, it is clear that $X\succ 0$ and $P\succ 0$ are very important; otherwise their inverses may not exist. In the strict LMI \cref{eq:strict-LMI-h2-norm}, we know from \cref{lemma:Lyapunov} that all feasible solutions $X$ must be positive definite. Although the solution to \cref{eq:controllability_Gramian} may be only positive semidefinite, the perturbed matrix~\cref{eq:H2-equation-to-LMI} is guaranteed to be strictly positive definite while also satisfying the strict LMI. On the other hand, if $(A, B)$ is controllable, the solution to \cref{eq:controllability_Gramian} (i.e., the controllability Gramian) is guaranteed to be positive definite. In this case, no perturbation is required in \cref{eq:construction-X-L0} to construct the solutions to \Cref{eq:nonstrict-LMI-h2-norm-1,eq:nonstrict-LMI-h2-norm}, which allows us to obtain non-strict LMIs.
\hfill $\square$
% clarify some subtleties during the construction proofs. It could be possible to use strong duality to justify the need for the controllability of $(A,B)$?
%     In the proof of \Cref{lemma:H2norm} 3), one of the conditions to ensure $L_\mathrm{c}\succ0$ is that $BB^\tr\succ0$.
%     Note that if $L_\mathrm{c}\succeq0$ is not positive definite, then only one direction $(iii)\Rightarrow(ii)$ is valid as in the case of non-strict LMI for $\mathcal{H}_\infty$ cost.
\end{remark}
}

%To ensure the feasibility of \cref{eq:nonstrict-LMI-h2-norm}, the controllability of $(A, B)$ cannot be removed in general. Indeed, the following converse result holds.

Throughout \Cref{lemma:H2-norm-strict,lemma:H2-norm-non-strict}, the observability of $(C,A)$ is not required. If $(C, A)$ is observable, we have stronger results that will be presented in \Cref{lemma:unique-solution-H2,lemma:controllability-A-B-H2-norm}.
\begin{lemma} \label{lemma:controllability-A-B-H2-norm}
    Consider a transfer matrix $\mathbf{G}(s) = C(sI - A)^{-1}B$ where $A$ is stable. {Suppose $(C,A)$ is observable.} If \cref{eq:nonstrict-LMI-h2-norm} holds with $\gamma=\|\mathbf{G}\|_{\mathcal{H}_2}$, $P \succ 0$ and $\Gamma \succeq 0$, then $(A, B)$ is controllable.
\end{lemma}
\begin{proof}
    From \Cref{lemma:H2-norm-non-strict}, we see that \cref{eq:nonstrict-LMI-h2-norm-1} holds with $X=\gamma P^{-1}\succ 0$.
    % \begin{equation} \label{eq:converse-step-1}
    % A X+XA^\tr+B B^\tr \preceq 0,\ \mathrm{tr}(C XC^\tr)\leq \gamma^2.
    % \end{equation}
    Then, there exists a matrix $B_0$ such that $A X+XA^\tr+B B^\tr + B_0 B_0^\tr = 0$. By Statement 1 in \Cref{lemma:H2norm}, we have
    $$
    \gamma^2=\left\|\mathbf{G}\|_{\mathcal{H}_2}^2  \leq \|C(sI -A)^{-1}\begin{bmatrix}
B, B_0
    \end{bmatrix}\right\|_{\mathcal{H}_2}^2 = \mathrm{tr}(C XC^\tr)\leq \gamma^2.
    $$
    We must have $\|C(sI -A)^{-1}B_0\|_{\mathcal{H}_2}^2 = 0$, i.e., $C(sI -A)^{-1}B_0 = 0$ is a zero transfer matrix. Since $(C, A)$ is observable, we get $B_0 = 0$.  Therefore, both inequalities in \cref{eq:nonstrict-LMI-h2-norm-1} hold with equalities for $X =\gamma P^{-1}\succ 0$. We see that $X$ is just the controllability Gramian $L_{\mathrm{c}}$, and \Cref{lemma:Lyapunov_controllability_observability} ensures that $(A, B)$ is controllable.
\end{proof}
 %then we must have $(A,B)$ being controllable.

\begin{lemma} \label{lemma:unique-solution-H2}
    Consider a transfer matrix $\mathbf{G}(s) = C(sI - A)^{-1}B$ where $A$ is stable. If $(A,B,C)$ is controllable and observable, then there exist unique $P \succ 0$ and $\Gamma \succeq 0$ satisfying \cref{eq:nonstrict-LMI-h2-norm} with $\gamma = \|\mathbf{G}\|_{\mathcal{H}_2}$.  % {\color{red} this statement could probably be refined as iff.}
\end{lemma}
\begin{proof}
Let $P\succ 0$ and $\Gamma\succeq 0$ satisfy~\cref{eq:nonstrict-LMI-h2-norm}, and let $X=\gamma P^{-1}$ so that it satisfies~\cref{eq:nonstrict-LMI-h2-norm-1}.
From the proof of \Cref{lemma:controllability-A-B-H2-norm}, any $X \succ 0$ satisfying \cref{eq:nonstrict-LMI-h2-norm-1} will necessarily be equal to the controllability Gramian $L_{\mathrm{c}}$. Consequently, the matrix $P = \gamma L_{\mathrm{c}}^{-1}$ is unique. To show the uniqueness of $\Gamma$, we note that \cref{eq:nonstrict-LMI-h2-norm} implies $\Gamma-CP^{-1}C^\tr\succeq 0$, which further leads to
\[
\gamma = \gamma^{-1}\operatorname{tr}(CL_{\mathrm{c}}C^\tr)
=\operatorname{tr}(CPC^\tr)\leq \operatorname{tr}(\Gamma)\leq \gamma.
\]
As a result, $\operatorname{tr}(\Gamma-CP^{-1}C^\tr)=0$, which implies $\Gamma$ is equal to $CP^{-1}C^\tr$ and therefore is unique.
\end{proof}

We present an example below showing that if $(C, A)$ is not observable, then the solution set to the nonstrict LMI \cref{eq:nonstrict-LMI-h2-norm} is not a singleton with $\gamma = \|\mathbf{G}\|_{\mathcal{H}_2}$.
\begin{example} \label{example:non-unique-certificate}
    Consider a second-order system with matrices
    $$
    A = \begin{bmatrix}
        -1 & 0 \\
        2 & -2
    \end{bmatrix}, \qquad B = \begin{bmatrix}
        1 \\ 0
    \end{bmatrix}, \qquad C = \begin{bmatrix}
        1 & 0
    \end{bmatrix}.
    $$
    This system is controllable but not observable (only detectable). The unique solution to \Cref{eq:Lyapunov-equations-H2norm-a} is
    $$
    L_{\mathrm{c}} = \frac{1}{6} \begin{bmatrix}
        3 & 2 \\ 2 & 2
    \end{bmatrix},
    $$
    which is positive definite as expected. The corresponding $\mathcal{H}_2$ norm is $\gamma = \|\mathbf{G}\|_{\mathcal{H}_2} = \frac{\sqrt{2}}{2}$. It is straightforward to verify that $X = L_{c} \succ 0$ satisfies \cref{eq:nonstrict-LMI-h2-norm-1} with $\gamma = \frac{\sqrt{2}}{2}$. To see the extra flexibility, we compute
    $$
    A X+XA^\tr = \begin{bmatrix}
        -1 & 0 \\ 0 & 0
    \end{bmatrix}, \quad BB^\tr  = \begin{bmatrix}
        1 & 0 \\ 0 & 0
    \end{bmatrix}, \quad C^\tr C = \begin{bmatrix}
        1 & 0 \\ 0 & 0
    \end{bmatrix}.
    $$
    The above equalities suggest that
    the bottom right element has extra freedom for the inequality in \cref{eq:nonstrict-LMI-h2-norm-1}. In particular, let us choose
    $$
    X_t = X + \frac{1}{6}\begin{bmatrix}
        0 & 0 \\ 0 & t
    \end{bmatrix} = \frac{1}{6} \begin{bmatrix}
        3 & 2 \\ 2 & 2+t
    \end{bmatrix} \succ 0, \quad \forall t \geq 0.
    $$
    It is straightforward to verify that
    $$
    A X_t+X_tA^\tr + BB^\tr = -\frac{2}{3}\begin{bmatrix}
        0 & 0 \\ 0 & t
    \end{bmatrix} \preceq 0, \quad \mathrm{tr}(C X_tC^\tr) = \frac{1}{2}, \quad \forall t \geq 0.
    $$
    Thus, $X_t \succ 0, \forall t \geq 0$ is feasible to \cref{eq:nonstrict-LMI-h2-norm-1}.
    For each $X_t$, the construction \cref{eq:construction-Gamma-P} with $\epsilon = 0$ is feasible to the nonstrict LMI \cref{eq:nonstrict-LMI-h2-norm} with $\gamma = \frac{\sqrt{2}}{2}$. \hfill \qed
 \end{example}

% The non-uniqueness in \cref{example:non-unique-certificate} is due to the fact that $(C, A)$ is not observable. If the system $(A,B,C)$ is minimal, then the solution set to is unique when setting $\gamma = \|\mathbf{G}\|_{\mathcal{H}_2}$.

\subsubsection{Characterizations of $\mathcal{H}_\infty$ norm} \label{appendix:proof-Hinf-norm}

The bounded real lemma in \Cref{lemma:bounded_real} is a classical result in control, which establishes equivalence between a frequency-domain inequality (i.e., $\mathcal{H}_\infty$ norm) and
the algebraic feasibility of LMIs in the state-space domain.
This result is covered in many modern control textbooks; see, for example, \cite[Corollary 12.3]{zhou1998essentials} for strict versions and \cite[Lemma 2.11]{scherer2000linear} for nonstrict versions. % However, its proof is not very accessible to readers beyond control

There are different proofs for \Cref{lemma:bounded_real}. The constructive proof in \cite[Lemma 7.3]{dullerud2013course} and \cite[Corollary 12.3]{zhou1998essentials} for strict versions reveals some of the deepest relationships in linear systems theory (such as Riccati operator and Hamiltonian matrix), which is not easily accessible for readers outside control. An alternative elegant proof based on convex duality first appeared in \cite{rantzer1996kalman}, and this duality perspective has been further developed in \cite[Lemma 2.11]{scherer2000linear}, \cite[Section 6]{balakrishnan2003semidefinite} and \cite{you2016direct,you2015primal,Gattami2016Hinf}. We adopt the duality techniques in \cite{rantzer1996kalman,balakrishnan2003semidefinite,you2016direct} to highlight the subtleties between strict and non-strict LMIs.

%\vspace{10mm}

% {\color{red} Some comments about the bounded real lemma; Comments on the \cite[Corollary 12.3]{zhou1998essentials}, \cite{rantzer1996kalman,you2016direct,you2015primal}}

Before presenting the proof of \Cref{lemma:bounded_real}, we introduce a few equivalent LMI formulations to \cref{eq:strict-hinf} and \cref{eq:non-strict-hinf}. First, via Schur complement,  \cref{eq:strict-hinf}  is equivalent to the following LMI
\begin{subequations}
\begin{equation} \label{eq:HinfLMI-1}
\begin{bmatrix}
A^\tr P + P A & PB\\
B^\tr P & -\gamma I
\end{bmatrix} + \frac{1}{\gamma} \begin{bmatrix}
    C^\tr \\ D^\tr
\end{bmatrix}\begin{bmatrix}
    C & D
\end{bmatrix}\prec 0.
\end{equation}
Via a scaling argument $\gamma P \to P$, \cref{eq:HinfLMI-1} is equivalent to the following LMI
% it is easy to see that there exists a $P \succ 0$ such that \cref{eq:HinfLMI-1} holds if and only if there exists a $P \succ 0$ such that the following LMI holds
\begin{equation} \label{eq:HinfLMI-2}
\begin{bmatrix}
A^\tr P + P A & PB\\
B^\tr P & -\gamma^2 I
\end{bmatrix} +  \begin{bmatrix}
    C^\tr \\ D^\tr
\end{bmatrix}\begin{bmatrix}
    C & D
\end{bmatrix}\prec 0,
\end{equation}
which is further equivalent to the following Riccati inequality (again via Schur complement)
\begin{equation}
    A^\tr P + P A + C^\tr C - (PB + C^\tr D)(D^\tr D - \gamma^2 I)^{-1}(PB + C^\tr D)^\tr \prec 0, \; D^\tr D \prec \gamma^2 I.
\end{equation}
%admits a positive definite solution $P \succ 0$.
\end{subequations}
Similarly, the nonstrict LMI \cref{eq:non-strict-hinf} is equivalent to
\begin{equation} \label{eq:HinfLMI-3}
\begin{bmatrix}
A^\tr P + P A & PB\\
B^\tr P & -\gamma^2 I
\end{bmatrix} +  \begin{bmatrix}
    C^\tr \\ D^\tr
\end{bmatrix}\begin{bmatrix}
    C & D
\end{bmatrix}\preceq 0.
\end{equation}

We are now ready to prove \Cref{lemma:bounded_real}.

\vspace{3mm}

\textbf{Proof of the sufficient conditions in \Cref{lemma:bounded_real}}:
    For both the strict LMI \cref{eq:strict-hinf} and nonstrict LMI \cref{eq:non-strict-hinf}, the ``if '' direction is relatively straightforward and involves a control-theoretic interpretation for the variable $P$.

    Consider the LTI system \cref{eq:App_dynamics} with $x(0) = 0$. When $A$ is stable, its $\mathcal{H}_\infty$ norm $\|C(sI - A)^{-1}B + D\|_{\mathcal{H}_\infty} \leq \gamma$ if and only if
    \begin{equation} \label{eq:Hinf-output-bound}
        \|y\|_2^2 \leq \gamma^2 \|u\|_2^2, \qquad \forall u \in \mathcal{L}_2[0,\infty).
    \end{equation}
    We next show \cref{eq:Hinf-output-bound} is true if there exists a $P \succ 0$ such that \cref{eq:HinfLMI-3} holds. Pre-multiplying \cref{eq:HinfLMI-3} by $\begin{bmatrix}
        x(t)^\tr, u(t)^\tr
    \end{bmatrix}$ and postmultiplying \cref{eq:HinfLMI-3} by $\begin{bmatrix}
        x(t)^\tr, u(t)^\tr
    \end{bmatrix}^\tr$ lead to
    \begin{equation} \label{eq:HinfLMI-dissipation-1}
        x^\tr P(Ax + Bu) + (Ax + Bu)^\tr Px - \gamma^2 u^\tr u + y^\tr y \leq 0, \qquad \forall t \geq 0.
    \end{equation}
We introduce the storage function $V: \mathbb{R}^n \to \mathbb{R}$ defined by $V(x) =
x^\tr Px$. Then, \cref{eq:HinfLMI-dissipation-1} implies the so-called dissipation inequality
\begin{equation} \label{eq:HinfLMI-dissipation-2}
\frac{d}{dt}V(x(t)) + y^\tr y \leq \gamma^2 u^\tr u , \qquad \forall t \geq 0.
    \end{equation}
Integrating over an interval $[0,T]$, and recalling that $x(0) = 0$, we have
$$
V(x(T)) + \int_0^T  y^\tr y\,dt \leq  \gamma^2\int_0^T u^\tr u\,dt.
$$
Note that $P\succ 0$ implies $V(x(T))\geq 0$, which further leads to $\int_0^T  y^\tr y\,dt \leq  \gamma^2\int_0^T u^\tr u\,dt$. Letting $T\rightarrow+\infty$ proves \cref{eq:Hinf-output-bound}.

Suppose there exists a $P\succ 0$ such that \cref{eq:HinfLMI-2} holds. Then we can find some $\epsilon>0$ such that
\[
\begin{bmatrix}
A^\tr P + P A & PB\\
B^\tr P & -(\gamma-\epsilon)^2 I
\end{bmatrix} +  \begin{bmatrix}
    C^\tr \\ D^\tr
\end{bmatrix}\begin{bmatrix}
    C & D
\end{bmatrix}\preceq 0.
\]
By the same construction in \Cref{eq:Hinf-output-bound,eq:HinfLMI-dissipation-1}, we get
\[
\|y\|_2^2\leq (\gamma-\epsilon)^2 \|u\|_2^2,
\qquad\forall u\in\mathcal{L}_2[0,+\infty),
\]
which implies $\|C(sI-A)^{-1}B+D\|_{\mathcal{H}_\infty}\leq \gamma-\epsilon<\gamma$.

\vspace{3mm}

\textbf{Proof of the necessary conditions in \Cref{lemma:bounded_real}}:
The ``only if'' direction in \cref{eq:strict-hinf} and \cref{eq:non-strict-hinf} are more involved to prove. We here adopt the convex duality techniques in \cite{rantzer1996kalman,balakrishnan2003semidefinite,you2016direct} to highlight some subtleties between strict and non-strict LMIs.

Let us construct a primal and dual perspective of computing $\mathcal{H}_\infty$ norms \cite{balakrishnan2003semidefinite}.
Consider the LTI system \cref{eq:App_dynamics} with $x(0) = 0$. Let $A$ be stable and $\gamma\geq\|C(sI - A)^{-1}B + D\|_{\mathcal{H}_\infty}$.
For any $u \in \mathcal{L}_2[0,+\infty)$, both the state $x(t)$ and output $y(t)$ have finite energy. Therefore, the following quantities are well-defined
$$
Z_{11} = \int_0^\infty x(t)x(t)^\tr dt, \quad  Z_{12} = \int_0^\infty x(t)u(t)^\tr dt, \quad Z_{22} = \int_0^\infty u(t)u(t)^\tr dt.
$$
It is straightforward to verify that
$$
A Z_{11} + BZ_{12}^\tr + Z_{11}A^\tr + Z_{12}B^\tr = \int_0^\infty \frac{d}{dt}(x(t)x(t)^\tr) = 0
$$
and
$$
\begin{bmatrix}
    Z_{11} & Z_{12} \\ Z_{12}^\tr & Z_{22}
\end{bmatrix} = \int_0^\infty \begin{bmatrix}
    x(t) \\ u(t)
\end{bmatrix}\begin{bmatrix}
    x(t) \\ u(t)
\end{bmatrix}^\tr dt \succeq 0, \qquad \operatorname{tr}\left(\begin{bmatrix}
    C & D
\end{bmatrix} \begin{bmatrix}
    Z_{11} & Z_{12} \\ Z_{12}^\tr & Z_{22}
\end{bmatrix}\begin{bmatrix}
    C^\tr \\ D^\tr
\end{bmatrix}\right) = \int_0^\infty y(t)^\tr y(t)dt.
$$
We now define two sets
$$
\begin{aligned}
\mathcal{V} &= \left\{Z \succeq 0 \mid Z = \int_0^\infty \begin{bmatrix}
    x(t) \\ u(t)
\end{bmatrix}\begin{bmatrix}
    x(t) \\ u(t)
\end{bmatrix}^\tr dt \;\; \text{for some}\;\; u \in \mathcal{L}_2[0,\infty) \; \text{to \cref{eq:App_dynamics} with $x(0) = 0$}  \right\},  \\
\mathcal{V}_{\mathrm{sdp}} &=\left\{Z \succeq 0 \mid A Z_{11} + BZ_{12}^\tr + Z_{11}A^\tr + Z_{12}B^\tr = 0 \right\}.
\end{aligned}
$$
It is clear that $\mathcal{V} \subseteq \mathcal{V}_{\mathrm{sdp}}$. In fact, it is proved in \cite[Lemma 3.4]{you2016direct} that $\mathcal{V}_{\mathrm{sdp}} = \mathrm{cl}\,\mathcal{V}$.

We then consider the following semidefinite optimization problem
\begin{equation} \label{eq:primal-Hinf-norm}
    \begin{aligned}
        p^\star = \sup_{Z} \quad &\operatorname{tr}\left(\begin{bmatrix}
    C & D
\end{bmatrix} \begin{bmatrix}
    Z_{11} & Z_{12} \\ Z_{12}^\tr & Z_{22}
\end{bmatrix}\begin{bmatrix}
    C^\tr \\ D^\tr
\end{bmatrix}\right) \\
        \text{subject to}\quad & A Z_{11} + BZ_{12}^\tr + Z_{11}A^\tr + Z_{12}B^\tr = 0,\\
        & \operatorname{tr}(Z_{22}) = 1, \quad \begin{bmatrix}
    Z_{11} & Z_{12} \\ Z_{12}^\tr & Z_{22}
\end{bmatrix} \succeq 0.
    \end{aligned}
\end{equation}
The fact $\mathcal{V}_{\mathrm{sdp}} = \mathrm{cl}\,\mathcal{V}$ implies that
$
p^\star
$
is equal to the squared $\mathcal{H}_\infty$ norm of the transfer matrix $C(sI-A)^{-1}B + D$.
It can be verified that the Lagrangian dual of \cref{eq:primal-Hinf-norm} is
\begin{equation} \label{eq:dual-Hinf-norm}
    \begin{aligned}
        d^\star = \inf_{\beta,P} \quad &\beta \\
        \text{subject to}\quad &  \begin{bmatrix}
    A^\tr P + PA  & PB\\ B^\tr P & -\beta I
\end{bmatrix} + \begin{bmatrix}
    C^\tr \\ D^\tr
\end{bmatrix}\begin{bmatrix}
    C & D
\end{bmatrix} \preceq 0.
    \end{aligned}
\end{equation}
% According to the ``if'' direction, we have
% $$
%  d^\star \geq \gamma^2.
% $$
%
Since $A$ is stable, we have a positive definite solution $P \succ 0$ to $A^\tr P + PA = -t I, \forall t > 0$. Therefore, we can construct a strict feasible solution to the dual SDP \cref{eq:dual-Hinf-norm} by letting $\beta>0$ and $t>0$ be sufficiently large and letting $P$ be the positive definite solution to $A^\tr P + PA = -t I$. %(which implies that \cref{eq:primal-Hinf-norm} is always solvable [?]).
\cite[Theorem 2.4.1]{ben2001lectures} then justifies that strong duality holds for \cref{eq:dual-Hinf-norm,eq:primal-Hinf-norm} ,
i.e.,
$$
p^\star = d^\star.
$$
To show the ``only if'' direction of the non-strict LMI \cref{eq:non-strict-hinf} in \Cref{lemma:bounded_real}, we employ the result proved in \cite[Proposition 3.3]{you2016direct} that the primal SDP \cref{eq:primal-Hinf-norm} is strictly feasible if and only if $(A, B)$ is controllable. Then, since the strict feasibility of the primal SDP \cref{eq:primal-Hinf-norm} implies the existence of an optimal solution to the dual SDP~\cref{eq:dual-Hinf-norm} \cite[Theorem 2.4.1]{ben2001lectures}, we see that the controllability of $(A, B)$ implies the existence of a matrix $P$ such that \cref{eq:dual-Hinf-norm} is feasible with $\beta=p^\star=d^\star$. Then it is not hard to see that this matrix $P$ satisfies
\[
\begin{bmatrix}
    A^\tr P + PA  & PB\\ B^\tr P & -\gamma^2 I
\end{bmatrix} + \begin{bmatrix}
    C^\tr \\ D^\tr
\end{bmatrix}\begin{bmatrix}
    C & D
\end{bmatrix} \preceq
\begin{bmatrix}
    A^\tr P + PA  & PB\\ B^\tr P & -\beta I
\end{bmatrix} + \begin{bmatrix}
    C^\tr \\ D^\tr
\end{bmatrix}\begin{bmatrix}
    C & D
\end{bmatrix}
\preceq
0,
\]
which proves the ``only if'' direction of the non-strict LMI \cref{eq:non-strict-hinf} in \Cref{lemma:bounded_real}.

We now finalize the proof for the ``only if'' direction of the strict LMI \cref{eq:strict-hinf} in \Cref{lemma:bounded_real}. Let $\gamma > \|C(sI - A)^{-1}B + D\|_{\mathcal{H}_\infty}$, and let $(\tilde{\beta},\tilde{P})$ be an arbitrary strictly feasible solution to the dual SDP~\cref{eq:dual-Hinf-norm}. It is clear that
$$
\gamma^2 > p^\star = d^\star,
$$
and so we can find $\epsilon>0$ such that $\frac{\gamma^2-\epsilon\tilde{\beta}}{1-\epsilon}>d^\star$. By the definition of $d^\star$, there exists a feasible solution $(\hat{\beta},\hat{P})$ to the dual SDP~\cref{eq:dual-Hinf-norm} such that $\hat{\beta}\leq \frac{\gamma^2-\epsilon\tilde{\beta}}{1-\epsilon}$. Uppon defining $P=(1-\epsilon)\hat{P}+\epsilon\tilde{P}$, we see that
\begin{align*}
& \begin{bmatrix}
    A^\tr P + PA  & PB\\ B^\tr P & -\gamma^2 I
\end{bmatrix} + \begin{bmatrix}
    C^\tr \\ D^\tr
\end{bmatrix}\begin{bmatrix}
    C & D
\end{bmatrix}
\preceq
\begin{bmatrix}
    A^\tr P + PA  & PB\\ B^\tr P & -[(1-\epsilon)\hat{\beta}+\epsilon\tilde{\beta}] I
\end{bmatrix} + \begin{bmatrix}
    C^\tr \\ D^\tr
\end{bmatrix}\begin{bmatrix}
    C & D
\end{bmatrix} \\
={} &
(1-\epsilon)
\left(
\begin{bmatrix}
    A^\tr \hat{P} + \hat{P}A  & \hat{P}B\\ B^\tr \hat{P} & -\hat{\beta} I
\end{bmatrix}
+\begin{bmatrix}
    C^\tr \\ D^\tr
\end{bmatrix}\begin{bmatrix}
    C & D
\end{bmatrix}
\right)
+
\epsilon
\left(
\begin{bmatrix}
    A^\tr \tilde{P} + \tilde{P}A  & \tilde{P}B\\ B^\tr \tilde{P} & -\tilde{\beta} I
\end{bmatrix}
+ \begin{bmatrix}
    C^\tr \\ D^\tr
\end{bmatrix}\begin{bmatrix}
    C & D
\end{bmatrix}
\right) \\
\prec{} & 0.
\end{align*}
The proof is now complete.
 \hfill \qed

\vspace{3mm}

From the duality-based proof, it is clear that the assumption for controllability of $(A, B)$ in the nonstrict version \cref{eq:non-strict-hinf} cannot be removed (otherwise the infimum of the dual SDP \cref{eq:dual-Hinf-norm} may not be attainable). We illustrate this unsolvable case of  \cref{eq:dual-Hinf-norm} in the following example.

\begin{example} \label{example:hinf-controllability}
    We here present three simple LTI systems to illustrate the positive (semi)definiteness of the solution $P$ in \Cref{lemma:bounded_real}.

\begin{enumerate}
    \item System 1: $A = -1, \, B = 0, \, C = 1, \, D = 1$, which is not controllable.
    \item System 2: $
    A = -1, \, B = 1, \, C = 0, \, D = 1
    $, which is not observable.
    \item System 3: $
    A = -1, \, B = 1, \, C = 1, \, D = 1
    $, which is controllable and observable (thus minimal).
\end{enumerate}

    For the first system with matrices
    $
    A = -1, \, B = 0, \, C = 1, \, D = 1,
    $
    it is straightforward to check $\|\mathbf{G}(s)\|_\infty = \|1\|_\infty = 1$.  The non-strict LMI \cref{eq:HinfLMI-3} reads as
    $$
    \begin{bmatrix}
        -2p+1 & 1 \\ 1 & -\gamma^2 +1
    \end{bmatrix}\preceq 0,
    $$
    which cannot be feasible if $\gamma =1$ (in this case, the LMI above requires $p \to \infty$). Indeed, the dual SDP \cref{eq:dual-Hinf-norm} is not solvable (i.e., its infimum can not be achieved). % Thus, the controllability assumption in \cref{eq:non-strict-hinf} cannot be removed.

    For the second system  with matrices
    $
    A = -1, \, B = 1, \, C = 0, \, D = 1,
    $
    the non-strict LMI \cref{eq:HinfLMI-3} reads as
    $$
    \begin{bmatrix}
        -2p & p \\ p & -\gamma^2 +1
    \end{bmatrix}\preceq 0,
    $$
    which has a feasible point $p = 0$ if $\gamma =1$ (this solution is only positive semidefinite since $(C,A)$ is not observable).

    For the third system with matrices
    $
    A = -1, \, B = 1, \, C = 1, \, D = 1,
    $
    it is straightforward to verify that $\|\mathbf{G}(s)\|_\infty = 2$.
    The non-strict LMI \cref{eq:HinfLMI-3} reads as
    $$
    \begin{bmatrix}
        -2p+1 & p+1 \\ p+1 & -\gamma^2 +1
    \end{bmatrix}\preceq 0.
    $$
    With $\gamma = 2$, this non-strict LMI requires
    $$
    \begin{cases}
    -2p+1 - 3 \leq 0, \\
    3(2p-1) - (p+1)^2 \geq 0, %\qquad p = 2.
    \end{cases} \qquad \Rightarrow \qquad \begin{cases}
    p\geq 1, \\
    (p-2)^2 \leq 0, %\qquad p = 2.
    \end{cases}
    $$
    which has a unique positive definite solution $p=2$ (note that System 3 is minimal).
    \hfill \qed.
\end{example}

%\vspace{20mm}

% {\color{red}
% % The direction of ``if'' is relatively simple. This is like the S lemma; one direction is almost immediate.

% Clarify the subtitles in strict vs nonstrict LIMIs.  To do}

\subsection{State covariance of an LTI system driven by white Gaussian noise}
\label{appendix:covariance_LQG}
Here, we provide detailed discussions on the statement in \Cref{eq:co-variance}. In particular, let us consider the following LTI system driven by white Gaussian noise, which we model by a stochastic differential equation:
\[
dx(t) = Ax(t)\,dt + B\,dw(t),
\]
where $x(t)\in\mathbb{R}^n$ for each $t\geq 0$, $A\in\mathbb{R}^{n\times n}$, $B\in\mathbb{R}^{n\times m}$, and $w(t)$ is an $m$-dimensional Brownian motion. We also assume that $\mathbb{E}[x(0)]=\mu_0$ and $\mathbb{E}[x(0)x(0)^\tr]=\Sigma_0$ for some $\mu_0\in \mathbb{R}^n$ and  positive semidefinite $\Sigma_0$. For each $t\geq 0$, we let $\mathcal{F}_t$ denote the $\sigma$-algebra generated by $(w(\tau):\tau\leq t)$ and $x(0)$.

We first note that, by It\^{o}'s formula \cite[Theorem 4.2.1]{oksendal2003stochastic},
\begin{align*}
d(e^{-At}x(t)) &= -e^{-At}Ax(t)\,dt
+e^{-At}dx(t) \\
& = -e^{-At}Ax(t)\,dt
+e^{-At}Ax(t)\,dt + e^{-At}B\,dw(t) \\
& = e^{-At}B\,dw(t),
\end{align*}
i.e.,
\[
e^{-At}x(t) - x(0) =
\int_0^t e^{-A\tau}B\,dw(\tau).
\]
By taking the expectation, we see that
\[
\mathbb{E}[x(t)] = e^{At}\,\mathbb{E}[x(0)]
=e^{At}\mu_0,
\]
where we used $\mathbb{E}[\int_0^t Z(\tau)dw(\tau)]=0$ for any $\mathcal{F}_t$-adapted $\mathbb{R}^{n\times m}$-valued process $Z(t)$ with $\mathbb{E}[\int_0^t \|z(\tau)\|_F^2\,d\tau]<+\infty$ (see \cite[Theorem 3.2.1]{oksendal2003stochastic}).
If the matrix $A$ is stable, we then get $\lim_{t\rightarrow\infty}\mathbb{E}[x(t)]=0$.

To analyze the covariance of $x(t)$, we again apply It\^{o}'s formula to obtain
\begin{align*}
d\!\left(e^{-At}x(t)(e^{-At}x(t))^\tr\right)
={} & e^{-At}x(t)
\left(e^{-At}B\,dw(t)\right)^\tr
+ \left(e^{-At}B\,dw(t)\right)(e^{-At}x(t))^\tr \\
& + e^{-At}B(e^{-At}B)^\tr dt,
\end{align*}
By taking the expectation and noting that $x(t)$ is adapted to the filtration $\mathcal{F}_t$, we get
\[
\frac{d}{dt}
\mathbb{E}\!\left[
e^{-At}x(t)x(t)^\tr e^{-A^\tr t}
\right]
=e^{-At}BB^\tr e^{-A^\tr t},
\]
or equivalently,
\[
\mathbb{E}\!\left[e^{-At}x(t)x(t)^\tr e^{-A^\tr t}\right]
-\mathbb{E}\!\left[x(0)x(0)^\tr\right]
=\int_0^t e^{-A\tau}BB^\tr e^{-A^\tr\tau }\,d\tau.
\]
We then get
\begin{align*}
\mathbb{E}\!\left[x(t)x(t)^\tr\right]
& =e^{At}\, \mathbb{E}\!\left[x(0)x(0)^\tr \right] e^{A^\tr t}
+ \int_0^t e^{A(t-\tau)}BB^\tr e^{-A^\tr(t-\tau)}\,d\tau \\
& = e^{At}\Sigma_0 e^{A^\tr t}
+ \int_0^t e^{A\tau}BB^\tr e^{A^\tr\tau}\,d\tau.
\end{align*}
If the matrix $A$ is stable, we can let $t\rightarrow\infty$ and get
\[
\lim_{t\rightarrow\infty}
\mathbb{E}\!\left[x(t)x(t)^\tr\right]
=\int_0^{+\infty} e^{A\tau}BB^\tr e^{A^\tr \tau}\,d\tau.
\]
By \Cref{lemma:Lyapunov}, the right-hand side is just the solution to the Lyapunov equation
\[
AX+XA^\tr + BB^\tr = 0.
\]

\subsection{Controllability and observability of the closed-loop system}

In this work, we heavily rely on Lyapunov equations, their related LMIs, and LMIs from the bounded real lemma to compute $\mathcal{H}_2$ and $\mathcal{H}_\infty$ norms of the closed-loop systems. In many cases, we need to examine the positive definiteness of the solutions from those LMIs, and this property depends on the controllability and observability of the closed-loop system. We here summarize these properties.

% For convenience, we recall the following assumption throughout the paper.
% \begin{assumption}
%     A
% \end{assumption}
% Recall that we make \Cref{assumption:dynamics} throughout the paper. We first state the following fact.
% \begin{fact}
%    We choose the system matrices as \cref{eq:system-matrices-conversion}.
%    \begin{enumerate}
%        \item If $(A,W^{1/2})$ is controllable, then  $(A,B_1)$ is controllable.
%        \item If $(Q^{1/2}, A)$ is observable, then $(C_1, A)$ is observable.
%    \end{enumerate}
% \end{fact}
Consider a dynamic feedback policy of the form~\cref{eq:Dynamic_Controller} with internal state $\xi(t) \in\mathbb{R}^{q}$. The closed-loop transfer function from the exogenous disturbance $d$ to $z$ is given by \cref{eq:transfer-function-zd}. For convenience, we copy the equations here
\begin{equation*}
    \mathbf{T}_{{zd}}(s) = C_{\mathrm{cl}}(\mK)
\left(sI - A_{\mathrm{cl}}(\mK)
\right)^{-1}
B_{\mathrm{cl}}(\mK)
+D_{\mathrm{cl}}(\mK),
\end{equation*}
where we denote
\begin{align*}
A_{\mathrm{cl}}(\mK)
\coloneqq\  &
\begin{bmatrix}
A + BD_\mK C & BC_\mK \\ B_\mK C & A_\mK
\end{bmatrix},
& B_{\mathrm{cl}}(\mK)
\coloneqq\  & \begin{bmatrix} W^{1/2} & BD_{\mK}V^{1/2} \\ 0 & B_{\mK}V^{1/2}  \end{bmatrix}, \\
C_{\mathrm{cl}}(\mK)
\coloneqq\  &
 \begin{bmatrix}
        Q^{1/2} & 0 \\
        R^{1/2}D_{\mK}C & R^{1/2}C_{\mK}
    \end{bmatrix}, &
D_\mathrm{cl}(\mK) \coloneqq\ &
\begin{bmatrix} 0 & 0\\ 0 & R^{1/2}D_{\mK}V^{1/2}\end{bmatrix}.
\end{align*}
Recall that we make the following assumption, which is the same as \cref{assumption:performance-weights,assumption:stabilizability} in the main text.
\begin{assumption} \label{assumption:dynamics-weights}
    $(A, B)$ is controllable and $(C,A)$ is observable. The weight matrices satisfy $R \succ 0, V \succ 0$, and $(A,W^{1/2})$ is controllable and $(Q^{1/2},A)$ is observable.
\end{assumption}
Then, the closed-loop system is controllable and observable if the feedback policy $\mK \in \mathcal{C}_q$, viewed as a linear dynamical system, is controllable and observable. This result is known when $D_\mK = 0$, and more complete characterizations appeared in \cite[Theorem 1]{zheng2022system}.
\begin{lemma} \label{lemma:minimal-closed-loop-systems}
Under \Cref{assumption:dynamics-weights}, let $\mK \in \mathcal{C}_q$. The following statements hold.
\begin{itemize}
    \item If $(A_\mK, B_\mK)$ is controllable, then $(A_{\mathrm{cl}}(\mK),B_{\mathrm{cl}}(\mK))$ is controllable.
    \item If $(C_\mK, A_\mK)$ is observable, then $(C_{\mathrm{cl}}(\mK),A_{\mathrm{cl}}(\mK))$ is observable.
\end{itemize}
\end{lemma}
\begin{proof}
    By \Cref{lemma:controllability}, we only need to prove that the eigenvalues of the following matrix
    $$
    \begin{aligned}
    &\begin{bmatrix}
A + BD_\mK C & BC_\mK \\ B_\mK C & A_\mK
\end{bmatrix} + \begin{bmatrix} W^{1/2} & BD_{\mK}V^{1/2} \\ 0 & B_{\mK}V^{1/2}  \end{bmatrix}\begin{bmatrix} M_{11} & M_{12}  \\ M_{21}  & M_{22}   \end{bmatrix} \\
= &\begin{bmatrix}
A + BD_\mK C +  W^{1/2}M_{11} + BD_{\mK}V^{1/2}M_{21}  & BC_\mK +W^{1/2}M_{21} + BD_{\mK}V^{1/2}M_{22}\\ B_\mK C + B_{\mK}V^{1/2}M_{21} & A_\mK+B_{\mK}V^{1/2}M_{22}
\end{bmatrix}
    \end{aligned}
    $$
    can be arbitrarily assigned by choosing $M_{11}, M_{12}, M_{21}, M_{22}$ properly. This is true by choosing
    $$
    M_{21} = -V^{-1/2}C,
    $$
    and observing that $(A_\mK, B_\mK)$ and $(A, W^{1/2})$ are both controllable.

    The proof for observability is very similar. This completes our proof.
\end{proof}

%\begin{comment}
\subsection{$\mathcal{H}_2$ and $\mathcal{H}_\infty$ control with general plant dynamics} \label{appendix:h2_hinf}
Consider a general LTI system of the form
\begin{equation}\label{eq:plant}
\begin{aligned}
\dot{x}(t) =\ & Ax(t) + B_1w(t) + B_2 u(t), \\
z(t) =\ & C_1 x(t) + D_{11}w(t) + D_{12} u(t), \\
y(t) =\ & C_2x(t) + D_{21} w(t),
\end{aligned}
\end{equation}
where $x(t)\in\mathbb{R}^{n_x}$ is the state of the plant, $w(t)\in\mathbb{R}^{n_w}$ represents exogenous disturbance, $u(t)\in\mathbb{R}^{n_u}$ is the control input, $z(t)\in\mathbb{R}^{n_z}$ represents the regulated performance signal, and $y(t)\in\mathbb{R}^{n_y}$ is the measured output. We make the following standard assumption.
\begin{assumption} \label{assumption:dynamics}
$(A,B_1)$ and $(A,B_2)$ are controllable, and  $(C_1,A)$ and $(C_2,A)$ are observable.
\end{assumption}
 A typical control task is to synthesize a feedback controller (policy) that maps the output $y(t)$ to the control input $u(t)$, which stabilizes the plant and minimizes a certain performance metric. % When we only have access to the output signal $y(t)$, a static feedback policy is typically not sufficient to ensure good closed-loop performance \cite{zhou1996robust}.
Consider a  dynamic feedback policy of the form \cref{eq:Dynamic_Controller} with $\xi(t) \in\mathbb{R}^{n_x}$ being full-order. The closed-loop transfer function from the exogenous disturbance $w$ to the output $z$ is given by
\begin{equation} \label{eq:transfer-function-zw-appendix}
    \mathbf{T}_{{zw}}(s) = C_{\mathrm{cl}}(\mK)
\left(sI - A_{\mathrm{cl}}(\mK)
\right)^{-1}
B_{\mathrm{cl}}(\mK)
+D_{\mathrm{cl}}(\mK),
\end{equation}
where we denote
\begin{align*}
A_{\mathrm{cl}}(\mK)
\coloneqq\  &
\begin{bmatrix}
A + B_2D_\mK C_2 & B_2C_\mK \\ B_\mK C_2 & A_\mK
\end{bmatrix}, &
B_{\mathrm{cl}}(\mK)
\coloneqq\  & \begin{bmatrix}
B_1 + B_2D_\mK D_{21} \\
B_\mK D_{21}
\end{bmatrix}, \\
C_{\mathrm{cl}}(\mK)
\coloneqq\  &
\begin{bmatrix}
C_1 + D_{12}D_\mK C_2 & D_{12}C_\mK
\end{bmatrix}, &
D_\mathrm{cl}(\mK) \coloneqq\ &
D_{11} + D_{12}D_\mK D_{21}.
\end{align*}

The setup of LQG optimal control and $\mathcal{H}_\infty$ robust control in \Cref{section:problem_statement} corresponds to the general dynamics \cref{eq:plant} with
\begin{equation*} %\label{eq:system-matrices-conversion}
\begin{aligned}
 B_1=\begin{bmatrix} W^{\frac{1}{2}} & 0 \end{bmatrix}, \; B_2=B,  \; C_1=\begin{bmatrix} Q^{\frac{1}{2}} \\ 0 \end{bmatrix},\;  C_2=C,  \; %\\
 D_{11}=0, \;D_{12}=\begin{bmatrix} 0 \\ R^{\frac{1}{2}} \end{bmatrix}, \;D_{21}=\begin{bmatrix} 0 & V^{\frac{1}{2}} \end{bmatrix}.
\end{aligned}
\end{equation*}
With these matrices, the LQG optimal control \cref{eq:LQG_policy_optimization} is the same as
\begin{equation*} %\label{eq:LQG_policy_optimization}
    \begin{aligned}
        \min_{\mK} \quad &J_{\texttt{LQG},n}(\mK):=\|\mathbf{T}_{{zw}}(s)\|_{\mathcal{H}_2} \\
        \text{subject to} \quad& \mK \in {\mathcal{C}}_{n,0}. % \;~\cref{eq:LyapunovX}.
    \end{aligned}
\end{equation*}
and the $\mathcal{H}_\infty$ robust control \cref{eq:Hinf_policy_optimization} is the same as
\begin{equation*} %\label{eq:LQG_policy_optimization}
    \begin{aligned}
        \min_{\mK} \quad &J_{{\infty},n}(\mK):=\|\mathbf{T}_{{zw}}(s)\|_{\mathcal{H}_\infty} \\
        \text{subject to} \quad& \mK \in {\mathcal{C}}_n. % \;~\cref{eq:LyapunovX}.
    \end{aligned}
\end{equation*}
Finally, it is easy to see that \Cref{assumption:performance-weights,assumption:stabilizability} implies \Cref{assumption:dynamics}.
%\end{comment}

\subsection{Globally (sub)optimal policies for LQG and $\mathcal{H}_\infty$ control}
\label{section:optimal_policies_classical}

Both LQG optimal control \cref{eq:LQG} and $\mathcal{H}_\infty$ robust control \cref{eq:Hinf} have been extensively studied before. For completeness, we here briefly review their classical solutions.

In particular, the LQG problem~\cref{eq:LQG} admits the celebrated separable principle and has a \textit{closed-form} optimal solution by solving two algebraic Riccati equations. The optimal solution to~\cref{eq:LQG} is in the form of $u(t) = -K \xi(t)$ with a fixed $p \times n$ matrix $K$ and $\xi(t)$ is the state estimation based on the Kalman filter. We summarize the result in the following theorem.

\begin{theorem}[{\cite[Theorem 14.7]{zhou1996robust}}] \label{theorem:LQG-riccati-solution}
    A globally optimal dynamic policy to the LQG optimal control~\cref{eq:LQG} is given by an observer-based controller of the form
\begin{equation} \label{eq:LQGcontroller}
\begin{aligned}
    \dot \xi(t) &= A \xi(t) + Bu(t) + L(y(t) - C\xi(t)), \\
    u(t) &= -K \xi(t),
\end{aligned}
\end{equation}
where the matrix $L$ is called the \textit{Kalman gain}, given by $L = PC^\tr V^{-1}$ with $P$ being the unique {positive semidefinite} solution to
\begin{subequations}\label{eq:Riccati}
\begin{equation} \label{eq:Riccati_P}
    AP + PA^\tr - PC^\tr V^{-1}CP + W = 0,
\end{equation}
and the matrix $K$ is called the \textit{feedback gain}, given by $K = R^{-1}B^\tr S$ with $S$ being the unique {positive semidefinite} solution to
\begin{equation} \label{eq:Riccati_S}
    A^\tr S + SA - SB R^{-1}B^\tr S + Q = 0.
\end{equation}
\end{subequations}
\end{theorem}

It is easy to see that the optimal LQG policy~\cref{eq:LQGcontroller} can be written into the form of~\cref{eq:Dynamic_Controller} with
    %\begin{equation} \label{eq:LQGstatespace}
     $A_{\mK} = A - BK - LC,\,  B_{\mK} = L, \, C_{\mK} = -K$.
    %\end{equation}
     Thus, the solution from Ricatti equations~\cref{eq:Riccati} is always full-order, i.e., $q = n$.
     % We note that two dynamical controllers with the same transfer function
     % $
     % \mathbf{K}(s)=
     % C_{\mK}(sI-A_{\mK})^{-1}B_{\mK}
     % $
     % lead to the same LQG cost.
     We note that the optimal LQG policy is unique in the frequency domain~\cite[Theorem 14.7]{zhou1996robust}, but it has non-unique state-space representations~\cref{eq:Dynamic_Controller} due to similarity transformations.
     %; any similarity transformation on~\cref{eq:LQGstatespace} leads to another optimal solution that achieves the global minimum cost\footnote{This is a well-known fact and can be verified easily; see ?}%~\cref{lemma:Jn_invariance}.}.

Unlike LQG control, the optimal $\mathcal{H}_\infty$ policy in general does not admit a closed-form solution. The optimal $\mathcal{H}_\infty$ policy might not be unique even in the frequency domain, %(note that the optimal LQG controller is uni)
and finding an optimal $\mathcal{H}_\infty$ policy can be theoretically and numerically complicated; see \cite[Section 5.1]{glover2005state}. This is in contrast with the standard LQG theory.
%, in which the optimal policy is unique in the f and can be obtained by solving two Riccati equations (cf. \cref{theorem:LQG-riccati-solution}).
Instead, most existing classical results focus on suboptimal $\mathcal{H}_\infty$ control. One standard result is summarized in the following theorem.

\begin{theorem}[{\cite[Theorem 16.4]{zhou1996robust}, \cite[Theorem 14.1]{zhou1998essentials}}] \label{theorem:Hinf-riccati}
    Consider the $\mathcal{H}_\infty$ robust control problem~\cref{eq:Hinf}. Given $\gamma > 0$, there exists a suboptimal policy such that $\mathfrak{J}_{\infty} < \gamma^2$ if and only if the following two Riccati equations
    \begin{subequations} \label{eq:riccati-hinf}
        \begin{align}
            A^\tr X_{\infty} + X_{\infty}A - X_{\infty}(B R^{-1}B^\tr - \gamma^{-2} W)X_{\infty}+ Q&=0 \label{eq:riccati-hinf-a}\\
           AY_{\infty} + Y_{\infty}A^\tr - Y_{\infty}(C^\tr V^{-1} C - \gamma^{-2} Q)Y_{\infty}+ W&=0 \label{eq:riccati-hinf-b}
        \end{align}
    \end{subequations}
    admit positive definite solutions $X_{\infty} \succ 0$ and $Y_{\infty} \succ 0$ such that {$A - (B R^{-1}B^\tr - \gamma^{-2} W)X_\infty$ and $A - Y_\infty (C^\tr V^{-1} C - \gamma^{-2} Q)$} are stable,
    %
    %{\color{red} stability property},
    and they further satisfy $\rho(X_{\infty}Y_{\infty}) < \gamma^2$ (where $\rho(\cdot)$ denotes the spectral radius of a matrix).

    When these conditions hold, a suboptimal $\mathcal{H}_\infty$ policy satisfying $\mathfrak{J}_{\infty} < \gamma^2$ is given by
    \begin{equation} \label{eq:hinf-suboptimal}
        \begin{aligned}
            \dot{\xi}(t) &= (A+\gamma^{-2}WX_{\infty})\xi(t) + Bu(t) + Z_{\infty}L_{\infty}(y(t) - C\xi(t)) \\
            u(t) &= -F_{\infty}\xi(t),
        \end{aligned}
    \end{equation}
    where the feedback gain is $F_{\infty} = R^{-1}B^\tr X_{\infty}$ and the observer gain is given by $L_{\infty} = Y_{\infty}CV^{-1}$ and $Z_{\infty} = (I - \gamma^{-2}Y_{\infty}X_{\infty})^{-1}$.
\end{theorem}

The dynamic policy \cref{eq:hinf-suboptimal} is often called the \textit{central
controller} or \textit{minimum entropy controller}; see \cite[Chapter 16]{zhou1996robust} for more discussions. When $\gamma \to \infty$, the Riccati equations \cref{eq:riccati-hinf} are reduced to those for LQG control \cref{eq:Riccati}, and the suboptimal $\mathcal{H}_\infty$ policy \cref{eq:hinf-suboptimal} is reduced to the LQG optimal policy \cref{eq:LQGcontroller}. \Cref{theorem:Hinf-riccati} does not give an explicit formula to compute the globally optimal $\mathcal{H}_\infty$ policy. In principle, it can be computed as closely as desired via bisection, but the limiting behavior is subtle and involved, as discussed in \cite[Chapter 16.9]{zhou1996robust}, \cite[Theorem 14.6]{zhou1998essentials} and \cite[Section 5.1]{glover2005state}. % for discussions.

\begin{remark}
    One important difference between the Riccati equations \cref{eq:Riccati} and \cref{eq:riccati-hinf} is that $C^\tr V^{-1}C \succeq 0$ and $B R^{-1}B^\tr  \succeq 0$ in \cref{eq:Riccati} while the corresponding terms in \cref{eq:riccati-hinf} are sign indefinite. Consequently, we can guarantee the solution to \cref{eq:Riccati} exists and is unique and positive semidefinite under \Cref{assumption:performance-weights,assumption:stabilizability} (e.g., \cite[Theorem 12.2]{zhou1998essentials}), while explicit conditions for the solutions to \cref{eq:riccati-hinf} are not easily known in general (as stated in \Cref{theorem:Hinf-riccati}, the existence of solutions to \cref{eq:riccati-hinf} is part of the sufficient and necessary conditions for a suboptimal $\mathcal{H}_\infty$ controller). We also note that the suboptimal policy from \cref{eq:hinf-suboptimal} is of full order but always strictly proper with $D_\mK = 0$. When letting $\gamma \to \gamma^\star \coloneqq \inf_{\mK \in \mathcal{C}_n}\,J_{\infty,n}(\mK)$, the policy from \cref{eq:hinf-suboptimal} can be numerically strange due to $D_\mK = 0$. We discuss some numerical behavior in \Cref{section:experiments}.
    % {\color{red} one difference is that the Riccati equation for LQG is always well-defined while the Hinf case requires further care: $B R^{-1}B^\tr - \gamma^{-2} W$ may not be PSD. The controller is always strictly proper; there is no direct term $D_\mK$ which may make the controller look strange numerically. }
    \hfill $\square$
\end{remark}

\subsection{Similarity transformations}

We conclude this section by briefly introducing the notion of similarity transformation that has been widely used in control theory. Similarity transformations bring nice symmetry in the policy optimization of LQG, as recently highlighted in \cite{zheng2021analysis,hu2022connectivity,zheng2022escaping}.

For any fixed $T\in\mathrm{GL}_n$, we let $\mathscr{T}_T:\mathcal{C}_n\rightarrow\mathcal{C}_n$ denote the mapping given by
\begin{equation} \label{eq:similarity-transformation}
\mathscr{T}_T(\mK)
\coloneqq
 \begin{bmatrix}
I_m & 0 \\
0 & T
\end{bmatrix}\begin{bmatrix}
D_{\mK} & C_{\mK} \\
B_{\mK} & A_{\mK}
\end{bmatrix}\begin{bmatrix}
I_p & 0 \\
0 & T
\end{bmatrix}^{-1}
=\begin{bmatrix}
D_{\mK} & C_{\mK}T^{-1} \\
TB_{\mK} & TA_{\mK}T^{-1}
\end{bmatrix}.
\end{equation}
It is easy to verify that for any $T\in\mathrm{GL}_n$ and $\mK\in\mathcal{C}_n$, $\mathscr{T}_T(\mK)$ is an internally stabilizing policy of order $n$ and thus is in $\mathcal{C}_n$. In fact, the mapping $\mathscr{T}_T(\mK)$ is a diffeomorphism from $\mathcal{C}_n$ to itself  \cite[Lemma 3.2]{zheng2021analysis}.
One can also verify that for any $T\in\mathrm{GL}_n$, the similarity transformation is invariant in the set of non-degenerate stabilizing policies for both LQG and $\mathcal{H}_\infty$ control (see \Cref{lemma:invariance-non-degenerate-controllers,lemma:Hinf_invariance-non-degenerate-controllers}), i.e.,
$$
\mathscr{T}_T(\mK) \in {\mathcal{C}}_{\mathrm{nd},0}, \; \forall \mK \in {\mathcal{C}}_{\mathrm{nd},0}, \qquad {\mathscr{T}_T(\mK) \in {\mathcal{C}}_{\mathrm{nd}},\; \forall \mK \in {\mathcal{C}}_{\mathrm{nd}}.}
$$
Finally, it is well-known that similarity transformations do not change the input-output behavior of the dynamic policy \cref{eq:Dynamic_Controller}. They correspond to the same transfer function in the frequency domain. Thus, both LQG cost \cref{eq:H2-norm} and $\mathcal{H}_\infty$ cost \cref{eq:Hinf-norm} are invariant with respect to similarity transformations.

\section{Fundamentals of Nonsmooth optimization}
\label{app:nonsmooth-optimization}

In this section, we review some fundamentals of nonsmooth optimization, especially Clarke subdifferential \cite{clarke1990optimization}, which enables the analysis of a large class of nonsmooth functions.

\subsection{Clarke subdifferential} \label{appendix:clarke}

Let $f(x): C\rightarrow\mathbb{R}$ be a function defined on an open subset $C\subseteq\mathbb{R}^n$. We say that $f$ is \textit{locally Lipschitz near $x\in C$}, if there exists $\epsilon>0$ and $L>0$ such that for any $y_1,y_2\in C$ satisfying $\|y_1-x\|<\epsilon$ and $\|y_2-x\|<\epsilon$, we have
$$|f(y_1) - f(y_2)| \leq L \|y_1 - y_2\|.
$$
The function $f$ is said to be \textit{locally Lipschitz over $C$} if it is locally Lipschitz near any $x\in C$. The Rademacher theorem \cite[Theorem 3.2]{evans2015measure} guarantees that a locally Lipschitz function is differentiable almost everywhere in the domain.

Let $f(x)$ be locally Lipschitz over $C$. We define its Clarke directional derivative at $x\in C$ in the direction $v\in\mathbb{R}^n$ by
\[
f^\circ(x;v)
=\limsup_{x'\rightarrow x,t\downarrow 0}
\frac{f(x'+tv)-f(x')}{t}.
\]
The local Lipschitz continuity of $f$ guarantees that $f^\circ(x;v)$ is finite for all $x\in C$ and $v\in\mathbb{R}^n$. It can be shown that for any fixed $x\in C$, $f^\circ(x;\cdot)$ is a convex function and satisfies $f^\circ(x;\lambda v)=\lambda f^\circ(x;v)$ for any $\lambda>0$. We then define the Clarke subdifferential of $f$ at $x\in C$ as the set
\[
\partial f(x) \coloneqq
\left\{g\in\mathbb{R}^n: f^\circ(x;v)\geq\langle g,v\rangle\text{ for all }v\in\mathbb{R}^n\right\},
\]
which is nonempty for any $x\in C$.
It is shown that $f^\circ(x,\cdot)$ is the support function of $\partial f(x)$  \cite[Proposition 2.1.2]{clarke1990optimization}:
\begin{equation} \label{eq:support-function}
f^\circ(x,v) = \max_{g \in \partial f(x)}\; \langle g,v\rangle.
\end{equation}
Moreover, the following equality holds~\cite[Theorem 2.5.1]{clarke1990optimization}:
\begin{equation*}
    \partial f(x) = \operatorname{conv}\left\{\lim_{ x_k \to x} \nabla f(x_k) \,\Big|\, \nabla f(x_k) \text{ exists}, x_k \in C\right\},
\end{equation*}
where $\mathrm{conv}$ denotes the convex hull of a set.

We call $x \in C$ a \textit{Clarke stationary point} if $0 \in \partial f(x)$. The following result relates local minima and local maxima with Clarke stationary points \cite[Proposition 2.3.2]{clarke1990optimization}.
\begin{lemma}\label{lemma:clarke_stationary}
Let $f(x)$ be locally Lipschitz over $C$. If $x\in C$ is a local minimum or maximum of $f(x)$, then $x$ is a Clarke stationary point, i.e., $0 \in \partial f(x)$.
\end{lemma}
Note that the converse of \Cref{lemma:clarke_stationary} does not hold in general.
The function $f$ is called \textit{subdifferntial regular}, if for any $x \in C$, the ordinary directional derivative exists and coincides with the Clarke directional~derivative for all directions, i.e.,
\[
\lim_{t\downarrow 0}
\frac{f(x+tv)-f(x)}{t}
=
f^\circ(x,v),\quad\forall v \in \mathbb{R}^n,x\in C.
\]
We denote the ordinary directional derivative by $f'(x;v)$ whenever it exists.

From~\cref{eq:support-function}, the following result is clear.
\begin{lemma}\label{lemma:regular_Clarke_stationary}
    Let $f(x): \mathbb{R}^n \to \bar{\mathbb{R}}$ be a subdifferential regular function. The point $x \in \mathbb{R}^n$ is a Clarke stationary point if and only if $f'(x,v) \geq 0$ for all $v\in \mathbb{R}^n$.
\end{lemma} %(iff?)

\subsection{Clarke subdifferential of the $\mathcal{H}_\infty$ cost function} \label{appendix:clarke-hinf}

Here we provide a detailed proof for \Cref{lemma:subdifferential-Hinf} on the computation of Clarke subdifferential of the $\mathcal{H}_\infty$ cost \cref{eq:Hinf-norm} for each $\mK \in \mathcal{C}_q$.

For convenience, we copy \cref{eq:Hinf-norm} below
\begin{equation} \label{eq:Hinf-norm-appendix}
J_{\infty,q}(\mK) = \|\mathbf{T}_{zd}(\mK, s)\|_{\mathcal{H}_\infty} : = \sup_{\omega\in\mathbb{R}}\sigma_{\max}(\mathbf{T}_{zd}(\mK, j\omega)), \quad \forall \mK \in \mathcal{C}_q,
\end{equation}
where $\mathbf{T}_{zd}(\mK, s)$ denotes the transfer function from $d(t) = \begin{bmatrix} w(t) \\ v(t) \end{bmatrix}$ to $z(t)$, given by
\begin{equation*} %\label{eq:transfer-function-zd}
    \mathbf{T}_{{zd}}(\mK, s) = C_{\mathrm{cl}}(\mK)
\left(sI - A_{\mathrm{cl}}(\mK)
\right)^{-1}
B_{\mathrm{cl}}(\mK)
+D_{\mathrm{cl}}(\mK),
\end{equation*}
with $A_{\mathrm{cl}}(\mK), B_{\mathrm{cl}}(\mK), C_{\mathrm{cl}}(\mK), D_{\mathrm{cl}}(\mK)$ being the closed-loop matrices in \cref{eq:closed-loop-matrices}. Recall that the set of internally stabilizing policies with order $q$ is defined as
\begin{equation} %\label{eq:internallystabilizing}
    \mathcal{C}_{q} := \left\{
    \left.\mK=\begin{bmatrix}
    D_{\mK} & C_{\mK} \\
    B_{\mK} & A_{\mK}
    \end{bmatrix}
    \in \mathbb{R}^{(m+q) \times (p+q)} \right|\; % D_{\mK} = 0_{m\times p},
    \begin{bmatrix}
    A +BD_{\mK}C & BC_{\mK} \\
    B_{\mK}C & A_{\mK}
\end{bmatrix}~\text{is stable}\right\}.
\end{equation}
We consider $J_{\infty,q}(\mK)$ as a function from the open set $\mathcal{C}_{q}$ in $\mathbb{R}^{(m+q) \times (p+q)}$ to $\mathbb{R}$.

Our proof heavily relies on the following technical lemma.\footnote{
This lemma is essentially Corollary 2 of \cite[Theorem 2.8.2]{clarke1990optimization}, but we represent the subdifferential in the form of weak${}^\ast$-closed convex hull as in \cite[Section 4.2.2, Theorem 3]{ioffe1979theory} instead of vector-valued integration.
}
\begin{lemma}\label{lemma:supremum_subdiff}
Let $\Theta$ be a compact metrizable space, and let $\{f_\theta:\theta\in\Theta\}$ be a family of continuously (Fr\'{e}chet) differentiable functions defined on some open subset $U$ of a Banach space $X$. Suppose the following conditions hold:
\begin{enumerate}
\item The mapping $\theta\mapsto f_\theta(x)$ is continuous on $\Theta$ for all $x\in U$.
\item The mapping $(\theta,x)\mapsto Df_\theta(x)$ is continuous on $\Theta\times U$.
\end{enumerate}
Then the function $f:U\rightarrow\mathbb{R}$ defined by
\[
f(x) = \sup_{\theta\in\Theta} f_\theta(x)
\]
is subdifferentially regular, and
\[
\partial f(x) = \operatorname{cl}^\ast\operatorname{conv}
\{Df_\theta(x)\mid f_\theta(x)=f(x)\},
\]
where $\operatorname{cl}^\ast$ denotes closure in the weak${}^\ast$-topology of $X^\ast$.
\end{lemma}

Applying \cref{lemma:supremum_subdiff} to \cref{eq:Hinf-norm-appendix} requires some preparations. Let $\mathsf{S}_d(\mathbb{C}) = \{u\in\mathbb{C}^d\mid u^\her u = 1\}$ denote the unit sphere in $\mathbb{C}^d$. Let $\mathbb{C}^\ast = \mathbb{C}\cup\{\infty\}$ and equip it with the topology generated by all open sets in $\mathbb{C}$ and all sets of the form $(\mathbb{C}\backslash K)\cup\{\infty\}$ with $K\subset\mathbb{C}$ compact. It can be shown that $\mathbb{C}^\ast$ is then a compact metrizable space, and all transfer matrices in $\mathcal{RH}_\infty$ can be extended to be continuous mappings on
\[
\mathbb{C}^\ast_+ = \{z\in\mathbb{C}^\ast\mid \operatorname{Re}[z]\geq 0\text{ or }z=\infty\}.
\]
We denote $j\mathbb{R}^\ast = \{z\in\mathbb{C}^\ast\mid \operatorname{Re}[z]=0\text{ or }z=\infty\}$, which can be shown to be compact and metrizable. Finally, denote
\[
\mathcal{S} = \left\{
\left.
\begin{bmatrix}
D & C \\
B & A
\end{bmatrix}\in \mathbb{R}^{(p+n)\times(m+n)}
\right|
A\text{ is stable}
\right\}.
\]
It can be seen that $\mathcal{S}$ is an open set in $\mathbb{R}^{(p+n)\times (m+n)}$.

We first consider finding the subdifferential of the function $f: \mathcal{S} \to \mathbb{R}$
\[
f(P) = \|\mathcal{G}_P\|_{\mathcal{H}_\infty},
\quad
\mathcal{G}_P(s) = C(sI-A)^{-1}B+D,
\qquad\forall P=\begin{bmatrix}
D & C \\
B & A
\end{bmatrix}\in\mathcal{S}.
\]
We note that this function $f$ is defined in the finite-dimensional space $\mathbb{R}^{(p+n)\times (m+n)}$, which is different from the treatment in \cite{apkarian2006nonsmooth}. Then, the subdifferential of \cref{eq:Hinf-norm-appendix} follows from the chain rule.

\begin{proposition}
Let $P\in\mathcal{S}\backslash\{0\}$ be arbitrary, and let
\[
\mathcal{Z} = \{s\in j\mathbb{R}^\ast: \sigma_{\max}(\mathcal{G}_P(s)) = f(P)\},
\]
represent the frequencies that achieve its $\mathcal{H}_\infty$ norm.
For each $s\in\mathcal{Z}$, we let $Q_s$ be a matrix whose column vectors form an orthonormal basis of the eigenspace of $\mathcal{G}_P(s)\mathcal{G}_P(s)^\her$ associated with the eigenvalue $\sigma_{\max}(\mathcal{G}_P(s))^2 = f(P)^2$.

Then $\Phi\in\mathbb{R}^{(p+n)\times(m+n)}$ is a subdifferential of $f$ at $P$ if and only if there exist finitely many $s_1,\ldots,s_K \in \mathcal{Z}$ and corresponding positive semidefinite Hermitian matrices $Y_1,\ldots,Y_K$ of appropriate dimensions satisfying $\sum_{\kappa=1}^K \operatorname{tr}Y_\kappa = 1$ such that
\begin{equation}\label{eq:subdifferential_hinf_1}
\Phi = \frac{1}{f(P)}
\sum_{\kappa=1}^K
\operatorname{Re}\!\left\{
\begin{bmatrix}
I \\ (s_\kappa I-A)^{-1}B
\end{bmatrix}
\mathcal{G}_P(s_\kappa)^\her Q_{s_\kappa} Y_\kappa Q_{s_\kappa}^\her
\begin{bmatrix}
I & C(s_\kappa I-A)^{-1}
\end{bmatrix}
\right\}^{\!\tr}.
\end{equation}
\end{proposition}
\begin{proof}
Denote
\[
\Theta = \mathsf{S}_m(\mathbb{C})\times\mathsf{S}_p(\mathbb{C})\times j\mathbb{R}^\ast.
\]
It is not hard to see that $\Theta$ is compact and metrizable. Now for each $\theta=(u,v,s)\in\Theta$, define
\[
f_\theta(P)
=\operatorname{Re}\!\left[v^\her (C(sI-A)^{-1}B+D)u\right],
\qquad \forall P = \begin{bmatrix}
D & C \\
B & A
\end{bmatrix}\in \mathcal{S}.
\]
We can check that $f_\theta$ is continuously differentiable over $\mathcal{S}$ for any fixed $\theta\in\Theta$. Moreover, by definition, we have
\[
\sup_{\theta\in\Theta}
f_\theta(P)
= f(P),\qquad\forall P\in\mathcal{S}.
\]
To find the derivative of $f_\theta$ at $P=\begin{bmatrix}
D & C \\ B & A
\end{bmatrix}\in\mathcal{S}$, we let $\Delta=\begin{bmatrix}
\Delta_D & \Delta_C \\
\Delta_B & \Delta_A
\end{bmatrix}\in \mathbb{R}^{(p+n)\times(m+n)}$ be arbitrary and $t>0$ be sufficiently small. Then
\begin{align*}
& f_\theta(P+t\Delta) \\
={} &
\operatorname{Re}\!\left[
v^\her ((C+t\Delta_C)(sI-(A+t\Delta_A))^{-1}(B+t\Delta_B)+D+t\Delta_D)u
\right]\\
={} &
\operatorname{Re}\!\left[
v^\her ((C+t\Delta_C)(I-t(sI-A)^{-1}\Delta_A)^{-1}(sI-A)^{-1}(B+t\Delta_B)+D+t\Delta_D)u\right] \\
={} &
\operatorname{Re}\!\left[
v^\her \left((C+t\Delta_C)
\left(I + t(sI-A)^{-1}\Delta_A + o(t)\right)(sI-A)^{-1}(B+t\Delta_B)+D+t\Delta_D\right) u\right] \\
={} & f_\theta(P)
+t \operatorname{Re}\!\big[v^\her\big(
\Delta_D
+\Delta_C(sI-A)^{-1}B
+C(sI-A)^{-1}\Delta_B \\
&
+C(sI-A)^{-1}\Delta_A(sI-A)^{-1}B
\big)u\big] + o(t) \\
={} & f_\theta(P)
+ t\cdot\operatorname{tr}
\!\left(
\operatorname{Re}\!\left\{
\begin{bmatrix}
I \\ (sI-A)^{-1}B
\end{bmatrix}
uv^\her
\begin{bmatrix}
I & C(sI-A)^{-1}
\end{bmatrix}
\right\}
\begin{bmatrix}
\Delta_D & \Delta_C \\
\Delta_B & \Delta_A
\end{bmatrix}
\right) + o(t),
\end{align*}
which shows that
\[
Df_\theta(P)
=\operatorname{Re}\!\left\{
\begin{bmatrix}
I \\ (sI-A)^{-1}B
\end{bmatrix}
uv^\her
\begin{bmatrix}
I & C(sI-A)^{-1}
\end{bmatrix}
\right\}^{\!\tr}
\]
It is not hard to see that the mapping $(\theta,P)\mapsto Df_\theta(P)$ is continuous over $\Theta\times\mathcal{S}$.

We can now apply \Cref{lemma:supremum_subdiff} and obtain
\[
\partial f(P)
=\operatorname{cl}\operatorname{conv}
\{Df_\theta(P)\mid f_\theta(P)= f(P)\},
\]
where we replace $\operatorname{cl}^\ast$ by $\operatorname{cl}$ since the weak${}^\ast$-topology coincides with the standard topology in finite-dimensional vector spaces. We also note that the set
\[
\Theta_0(P) = \{\theta\in\Theta\mid f_\theta(P) = f(P) \}
\]
is a nonempty compact set since $\Theta$ is compact and the mapping $\theta\mapsto f_\theta(P)$ is continuous for a given $P\in \mathcal{S}$. As a result, the set $\{Df_\theta(P)\mid f_\theta(P)=f(P)\}$, which is the image of $\Theta_0(P)$ under the continuous mapping $\theta\mapsto Df_\theta(P)$, is compact. Since the convex hull of a compact set in a finite-dimensional space is compact (see \cite[Theorem 17.2]{rockafellar1970convex}), we get
\[
\partial f(P)
= \operatorname{conv}
\{
Df_\theta(P) \mid f_\theta(P) = f(P)
\}.
\]

To derive~\cref{eq:subdifferential_hinf_1}, for each $s\in\mathcal{Z}$, we let $d_s$ be the number of columns of $Q_s$, and let $Q_{s,\perp}$, $O_{s}$, $O_{s,\perp}$ be matrices such that
\[
\mathcal{G}_P(s) = \begin{bmatrix}
Q_{s} & Q_{s,\perp}
\end{bmatrix}
\begin{bmatrix}
f(P) I_{d_s} \\
& *
\end{bmatrix}
\begin{bmatrix}
O_{s}^\her \\ O_{s,\perp}^\her
\end{bmatrix}
\]
gives a singular value decomposition of $\mathcal{G}_P(s)$. We then have
\begin{align*}
\Theta_0(P) ={}  &
\{(u,v,s)\in \mathsf{S}_m(\mathbb{C})\times\mathsf{S}_p(\mathbb{C})\times j\mathbb{R}^\ast \mid
f_{(u,v,s)}(P) = f(P)
\} \\
={} &
\{(u,v,s)\in \mathsf{S}_m(\mathbb{C})\times\mathsf{S}_p(\mathbb{C})\times \mathcal{Z} \mid
\operatorname{Re}[v^\her \mathcal{G}_P(s)u] = \sigma_{\max}(\mathcal{G}_P(s))\} \\
={} &
\!\left\{
(O_s\xi,Q_s\xi, s)\mid
s\in \mathcal{Z}, \xi \in \mathsf{S}_{d_s}(\mathbb{C})
\right\}.
\end{align*}
Consequently,
\begin{align*}
& \partial f(P) \\
={} &
\operatorname{conv}\left\{
\left.\operatorname{Re}\!\left\{
\begin{bmatrix}
I \\ (sI-A)^{-1}B
\end{bmatrix}
uv^\her
\begin{bmatrix}
I & C(sI-A)^{-1}
\end{bmatrix}
\right\}^{\!\tr}
\right| (u,v,s)\in\Theta_0(P)
\right\} \\
={} &
\operatorname{conv}\left\{
\left.\operatorname{Re}\!\left\{
\begin{bmatrix}
I \\ (sI-A)^{-1}B
\end{bmatrix}
O_s\xi\xi^\her Q_s
\begin{bmatrix}
I & C(sI-A)^{-1}
\end{bmatrix}
\right\}^{\!\tr}
\right| s\in\mathcal{Z},\xi\in \mathsf{S}_{d_s}(\mathbb{C})
\right\} \\
={} &
\operatorname{conv}\left\{
\frac{1}{f(P)}\left.\operatorname{Re}\!\left\{
\begin{bmatrix}
I \\ (sI-A)^{-1}B
\end{bmatrix}
\mathcal{G}_P(s)^\her Q_s\xi\xi^\her Q_s^\her
\begin{bmatrix}
I & C(sI-A)^{-1}
\end{bmatrix}
\right\}^{\!\tr}
\right| s\in\mathcal{Z},\xi\in \mathsf{S}_{d_s}(\mathbb{C})
\right\} \\
={} &
\operatorname{conv}\left\{
\left.
\frac{1}{f(P)}
\operatorname{Re}
\!\left\{
\begin{bmatrix}
I \\ (sI-A)^{-1}B
\end{bmatrix}
\mathcal{G}_P(s)^\her Q_s Y Q_s^\her
\begin{bmatrix}
I & C(sI-A)^{-1}
\end{bmatrix}
\right\}^{\!\tr}
\right|
\begin{aligned}
& s\in\mathcal{Z}, Y\in\mathbb{C}^{d_s\times d_s}, \\
& Y=Y^\her,
\operatorname{tr}Y = 1
\end{aligned}
\right\},
\end{align*}
where in the third step we used $\mathcal{G}_P(s)^\her Q_s = f(P) O_s$ that follows from the singular value decomposition. The proof is now complete.
\end{proof}

Finally, let us consider the closed-loop system $\begin{bmatrix}
D_{\mathrm{cl}}(\mK) & C_{\mathrm{cl}}(\mK) \\
B_{\mathrm{cl}}(\mK) & A_{\mathrm{cl}}(\mK)
\end{bmatrix}$ and its associated $\mathcal{H}_\infty$ cost $J_{\infty,q}(\mK)$. We have
\[
\begin{bmatrix}
D_{\mathrm{cl}}(\mK) & C_{\mathrm{cl}}(\mK) \\
B_{\mathrm{cl}}(\mK) & A_{\mathrm{cl}}(\mK)
\end{bmatrix}
=\begin{bmatrix}
0 & 0 & Q^{1/2} & 0 \\
0 & 0 & 0 & 0 \\
W^{1/2} & 0 & A & 0 \\
0 & 0 & 0 & 0
\end{bmatrix}
+
\begin{bmatrix}
0 & 0 \\
R^{1/2} & 0 \\
B & 0 \\
0 & I
\end{bmatrix}
\begin{bmatrix}
D_\mK & C_\mK \\
B_\mK & A_\mK
\end{bmatrix}
\begin{bmatrix}
0 & V^{1/2} & C & 0 \\
0 & 0 & 0 & I
\end{bmatrix}.
\]
This shows that $\begin{bmatrix}
D_{\mathrm{cl}}(\mK) & C_{\mathrm{cl}}(\mK) \\
B_{\mathrm{cl}}(\mK) & A_{\mathrm{cl}}(\mK)
\end{bmatrix}$ is an affine function of $\mK\in \mathcal{C}_q$. By employing the chain rule for Clarke subdifferential~\cite[Theorem 2.3.10]{clarke1990optimization}, we see that $\partial J_{\infty, q}(\mK)$ is the convex hull of all matrices of the form
\begin{align*}
\frac{1}{J_{\infty,q}(\mK)}
\operatorname{Re}\bigg\{\!
\left[\!\begin{array}{cc:cc}
0 & V^{1/2} & C & 0 \\
0 & 0 & 0 & I
\end{array}\!\right]
& \begin{bmatrix}
I \\ (sI-A_{\mathrm{cl}}(\mK))^{-1}B_{\mathrm{cl}}(\mK)
\end{bmatrix}
\mathbf{T}_{zd}(\mK,s)^\her Q_s Y Q_s^\her \\
& \cdot \begin{bmatrix}
I & C_{\mathrm{cl}}(\mK)(sI-A_{\mathrm{cl}}(\mK))^{-1}
\end{bmatrix}
\left[\!\begin{array}{cc}
0 & 0 \\[-2pt]
R^{1/2} & 0 \\[-2pt]
\hdashline
B & 0 \\[-2pt]
0 & I
\end{array}\!\right]
\!\bigg\}^{\!\tr}
\end{align*}
or
\begin{align*}
\frac{1}{J_{\infty,q}(\mK)}
\operatorname{Re}\bigg\{\!
& \left(
\begin{bmatrix}
0 & V^{1/2} \\ 0 & 0
\end{bmatrix}
+
\begin{bmatrix}
C & 0 \\
0 & I
\end{bmatrix}
(sI-A_{\mathrm{cl}}(\mK))^{-1}B_{\mathrm{cl}}(\mK)
\right)
\mathbf{T}_{zd}(\mK,s)^\her Q_s Y Q_s^\her \\
& \cdot
\left(
\begin{bmatrix}
0 & 0 \\ R^{1/2} & 0
\end{bmatrix}
+C_{\mathrm{cl}}(\mK)(sI-A_{\mathrm{cl}}(\mK))^{-1}
\begin{bmatrix}
B & 0 \\ 0 & I
\end{bmatrix}
\right)
\!\bigg\}^{\!\tr},
\end{align*}
where $s\in j\mathbb{R}^\ast$ satisfies $\sigma_{\max}(\mathbf{T}_{zd}(\mK,s)) = J_{\infty,q}(\mK)$, the columns of $Q_s$ form an orthonormal basis of the eigenspace of $\mathbf{T}_{zd}(\mK,s)\mathbf{T}_{zd}(\mK,s)^\her$ with the eigenvalue $J_{\infty,q}(\mK)^2$, and $Y$ is a positive semidefinite Hermitian matrix with $\operatorname{tr}Y=1$. Note that in the above expression, the dimension of $\begin{bmatrix}
0 & V^{1/2} \\ 0 & 0
\end{bmatrix}$ is $(p+q)\times (n+p)$, the dimension of $\begin{bmatrix}
C & 0 \\ 0 & I
\end{bmatrix}$ is $(p+q)\times (n+q)$, the dimension of $\begin{bmatrix}
0 & 0 \\ R^{1/2} & 0
\end{bmatrix}$ is $(n+m)\times (m+q)$, and the dimension of $\begin{bmatrix}
B & 0 \\ 0 & I
\end{bmatrix}$ is $(n+q)\times (m+q)$. The proof of \Cref{lemma:subdifferential-Hinf} is now complete.

\begin{remark}
Existing literature~\cite{apkarian2006nonsmooth} has also discussed how to calculate the subdifferential of $J_{\infty,q}$. However, our proof idea is different from \cite{apkarian2006nonsmooth}: While both studies use the chain rule for Clarke subdifferential, the calculation in \cite{apkarian2006nonsmooth} is carried out via the subdifferential of $\|\cdot\|_{\mathcal{H}_\infty}:\mathcal{RH}_\infty\rightarrow\mathbb{R}$ that is defined over the infinite-dimensional space $\mathcal{RH}_\infty$, while our proof stays in finite-dimensional spaces (we view $J_{\infty, n}(\mK)$ as a composition: $\mathcal{C}_n \to \mathcal{S} \to \mathbb{R}$, while \cite{apkarian2006nonsmooth} views $J_{\infty, n}(\mK)$ as a composition: $\mathcal{C}_n \to \mathcal{RH}_\infty \to \mathbb{R}$). Staying in finite-dimensional spaces simplify many analysis details: it reduces the weak${}^\ast$-topology to the standard topology and further allows us to use \cite[Theorem 17.2]{rockafellar1970convex}. As a result, we can obtain stronger results: \cref{lemma:subdifferential-Hinf} provides sufficient and necessary conditions for the subdifferential even when the $\mathcal{H}_\infty$ norm is attained at infinitely many frequencies, while the results in \cite{apkarian2006nonsmooth} only consider the situation where the $\mathcal{H}_\infty$ norm is attained at finitely many frequencies. \hfill \qed
\end{remark}

\section{Auxiliary results for LQG and $\mathcal{H}_\infty$ control} \label{appendix:auxillary-results}

In this section, we provide some auxiliary results for LQG and $\mathcal{H}_\infty$ control, including further discussions on minimal/informative/non-degenerate policies in the LQG case, examples of global optimal policies, and boundary behaviors of $\mathcal{H}_\infty$ control.

\subsection{Examples on minimal/informative/non-degenerate policies} \label{appendix:minimal-informative-non-degenerate}

In the main text, \Cref{proposition:informativity-non-degenerate} reveals that the set of non-degenerate policies is the same as the set of informative policies for LQG control. We here give an example to show the relationship among minimal, informative, and non-degenerate policies for LQG control.

\begin{example}[Venn diagram of different sets of policies]
\label{example:non-degenerate-minimal-informative}
Consider the LQG problem with the system dynamics and the performance signal as in \Cref{example:discontinuity-LQG-boundary}, i.e.,
$$
A=-1,\;  B=1, \; C=1, \; Q =R =1, \; V= W =1.
$$
Consider a full-order policy $$\mK = \begin{bmatrix} 0 & C_{\mK} \\
                          B_{\mK} & A_{\mK}\end{bmatrix} \in \mathbb{R}^{2 \times 2}. $$
Following the definitions, we define 1) stabilizing policies ${\mathcal{C}}_{1,0}$, 2) minimal policies ${\mathcal{C}}_{\mathrm{min},0}$, and 3) informative policies ${\mathcal{C}}_{\mathrm{info}}$:
\[
\begin{aligned}
{\mathcal{C}}_{1,0} &= \{ \mK \in \mathbb{R}^{2 \times 2}: A_{\mK} < 1, B_{\mK}C_{\mK}<-A_{\mK}, D_\mK=0 \},\\
{\mathcal{C}}_{\mathrm{min},0}&=\{\mK\in \mathbb{R}^{2 \times 2}:B_\mK\neq0, C_\mK\neq0,D_\mK=0\},\\
{\mathcal{C}}_{\mathrm{info}}&=\{\mK\in{\mathcal{C}}_{1,0}: X_{\mK,12} \neq 0, D_\mK=0\},
\end{aligned}
\]
%{\color{red} should we have a definition of informative in the main text?}
where
\[
X_\mK=\frac{1}{2(A_\mK + B_\mK C_\mK)(A_\mK - 1)}\begin{bmatrix}
    A_\mK^2 - A_\mK + B_\mK^2 C_\mK^2 - B_\mK C_\mK
 & -B_\mK (A_\mK - B_\mK C_\mK) \\ -B_\mK (A_\mK - B_\mK C_\mK) & -B_\mK^2 (A_\mK + B\mK C_\mK - 2)
\end{bmatrix}
\]
is the solution to the Lyapunov equation
\[
\begin{bmatrix}
    -1 & C_\mK \\ B_\mK & A_\mK
\end{bmatrix}X_\mK+X_\mK\begin{bmatrix}
    -1 & C_\mK \\ B_\mK & A_\mK
\end{bmatrix}^\tr+\begin{bmatrix}
    1 & 0 \\ 0 & B_\mK^2
\end{bmatrix}=0, \ \ \ X_\mK=\begin{bmatrix}
    X_{\mK,11} & X_{\mK,12} \\ X_{\mK,12}^\tr & X_{\mK,22}
\end{bmatrix}.
\]
Finally, the set of non-degenerate policies ${\mathcal{C}}_{\mathrm{nd},0}$ is given in \Cref{def:LQG-Cnd}.

Consider four dynamic policies:
\begin{equation} \label{eq:policies-appendix}
\mK_1=\begin{bmatrix}0 & 2 \\ 2 & -1\end{bmatrix}, \ \ \mK_2=\begin{bmatrix}0 & 1 \\ -1 & -5\end{bmatrix},\ \ \mK_3=\begin{bmatrix}0 & 0 \\ 1 & -1\end{bmatrix},\ \ \mK_4=\begin{bmatrix}0 & -1 \\ 0 & -1\end{bmatrix}.
\end{equation}
We can easily verify $\mK_1\notin{\mathcal{C}}_{1,0}, \mK_2\in{\mathcal{C}}_{1,0},  \mK_3\in{\mathcal{C}}_{1,0},
\mK_4\in{\mathcal{C}}_{1,0},$
and
$\mK_1\in{\mathcal{C}}_{\mathrm{min},0}, \mK_2\in{\mathcal{C}}_{\mathrm{min},0}, \mK_3\notin{\mathcal{C}}_{\mathrm{min},0}, \mK_4\notin{\mathcal{C}}_{\mathrm{min},0}.$

For $\mK_2$, the solution to the Lyapunov equation \Cref{eq:LyapunovX}
is
\[
X_{\mK_2}=\begin{bmatrix} \frac{4}{9} & -\frac{1}{18} \\ -\frac{1}{18} & \frac{1}{9} \end{bmatrix}\succ0 \ \ \ \text{with } \ \  X_{\mK_2,12}\neq0,
\]
and thus $\mK_2\in {\mathcal{C}}_{\mathrm{info}}$. Since $X_{\mK_2}$ is invertible, we can choose
\[
P=J_{\texttt{LQG},1}(\mK_2)X_{\mK_2}^{-1}\succ0, \ \ \ \Gamma=\frac{1}{J_{\texttt{LQG},1}(\mK_2)}C_{\mathrm{cl}}(\mK_2)X_{\mK_2}C_{\mathrm{cl}}(\mK_2)^\tr\geq0
\]
to certify \Cref{eq:LQG-Bilinear-a} {and} \Cref{eq:LQG-Bilinear-b} with $\gamma = J_{\LQG,1}(\mK_2)$ and hence $\mK_2\in {\mathcal{C}}_{\mathrm{nd},0}$.

For $\mK_3$, the solution to the Lyapunov equation \Cref{eq:LyapunovX} is
\[
X_{\mK_3}=\begin{bmatrix} 0.5 & 0.25 \\ 0.25 & 0.75 \end{bmatrix}\succ0 \ \ \ \text{with } \ \  X_{\mK_3,12}\neq0,
\] and thus $\mK_3\in {\mathcal{C}}_{\mathrm{info}}$. Similarly, one can verify that  $\mK_3\in {\mathcal{C}}_{\mathrm{nd},0}$.

For $\mK_4$, the solution to \Cref{eq:LyapunovX} is
\[
X_{\mK_4}=\begin{bmatrix} 0.5 & 0 \\ 0 & 0 \end{bmatrix}\ \ \text{with}\ \ X_{\mK_3,12}=0,
\]
and thus $\mK_4 \notin {\mathcal{C}}_{\mathrm{info}}$. %Since $(C_{\mK_4},A_{\mK_4})$ is observable, by \Cref{lemma:controllability-A-B-H2-norm}, that $(A_{\mK_4},B_{\mK_4})$ is not controllable implies $\mK\notin{\mathcal{C}}_{\mathrm{nd},0}$.
\Cref{fig:Venn_diagram} gives an illustration of the Venn diagram for these instances.

    % \item $\mK_2=\begin{bmatrix}
    %     0 & 1 \\ -1 & 0
    % \end{bmatrix}\in\mathcal{C}_{n}$:
\end{example}

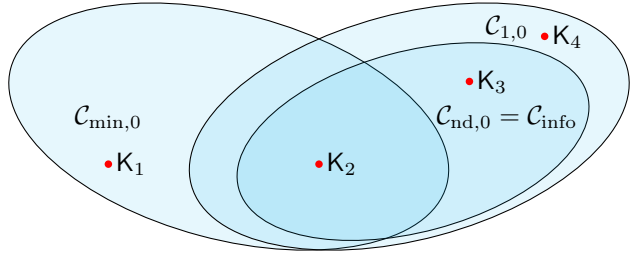
\begin{figure}
    \centering
    \setlength{\abovecaptionskip}{0pt}
    \setlength{\belowcaptionskip}{0pt}
    \begin{tikzpicture}[set/.style={fill=cyan,fill opacity=0.1}]
% Ellipse 1 minimal policies
\draw[set,
     xshift=-1.2cm,
    yshift=1.5cm,
    rotate =-15] (0,0) ellipse (3.cm and 1.5cm);

% Ellipse 2 informative policies
\draw[set,
     xshift=1.25cm,
    yshift=1.3cm,
    rotate =15] (0,0) ellipse (2.4cm and 1.2cm);

% Ellipse 3 stabilizing policies
\draw[set,
     xshift=1.2cm,
    yshift=1.5cm,
    rotate =15] (0,0) ellipse (3.cm and 1.5cm);

\node at (-2.8,1) [circle,fill,red,inner sep=1pt]{};
\node at (-2.5,1) {\small $\mK_1$};

\node at (0,1.) [circle,fill,red,inner sep=1pt]{};
\node at (0.3,1.) {\small $\mK_2$};

\node at (2,2.1) [circle,fill,red,inner sep=1pt]{};
\node at (2.3,2.1) {\small $\mK_3$};

\node at (3,2.7) [circle,fill,red,inner sep=1pt]{};
\node at (3.3,2.7) {\small $\mK_4$};

 \node at (-2.8,1.6) {\small ${{\mathcal{C}}_{\text{min},0}}$};

 \node at (2.5,1.6) {\small ${{\mathcal{C}}_{\text{nd},0}} = \mathcal{C}_{\mathrm{info}}$};
 \node at (2.5,2.8) {\small ${{\mathcal{C}}_{1,0}}$};
 %\node at (1.5,-1.8) {\small (a)};

\end{tikzpicture}

 \vspace{3mm}
    \caption{Venn diagram of \Cref{example:non-degenerate-minimal-informative} for the policies in \cref{eq:policies-appendix}. }
    \label{fig:Venn_diagram}
\end{figure}

\subsection{Examples on globally optimal LQG and $\mathcal{H}_\infty$ policies} \label{appendix:globally-optimal-controllers}

We here discuss a class of examples where we find the globally optimal LQG and $\mathcal{H}_\infty$ policies analytically from \Cref{theorem:LQG-riccati-solution,theorem:Hinf-riccati}. As we can see from the examples below, the computation in \Cref{theorem:LQG-riccati-solution} (LQG control) is much simpler than that in \Cref{theorem:Hinf-riccati} ($\mathcal{H}_\infty$ robust control).

In this section, we consider a single input and single output system \cref{eq:Dynamic} with
\begin{equation} \label{eq:appendix_example}
A = a, \quad B = 1, \quad C = 1,
\end{equation}
which is parameterized by a scalar $a \in \mathbb{R}$.
The performance measures are $Q = 1, R = 1$, and disturbance characterizations are $W = 1, V = 1$.

% {\color{red} we can further talk about the stabilizing controller ${\mathcal{C}}_{n,0}, {\mathcal{C}}_n$, non-degenerate controllers ${\mathcal{C}}_{\mathrm{nd},0},  {\mathcal{C}}_{\mathrm{nd}}$.}

\subsubsection{Globally optimal LQG policies}
    For the LQG problem \cref{eq:LQG_policy_optimization}, the Riccati equation \Cref{eq:Riccati_P} reads as
    $$
    P^2 - 2aP - 1 = 0,
    $$
    which has a unique positive semidefinite solution
    $
    P = a + \sqrt{a^2 + 1}.
    $
    Then, the \textit{Kalman gain} is $L = a + \sqrt{a^2 + 1}$. Similarly, the \textit{Feedback gain} is $K = a + \sqrt{a^2 + 1}$.

    Thus, a globally optimal LQG policy is in the form of
    \begin{equation*}
\begin{aligned}
    \dot \xi(t) &= a \xi(t) + u(t) + (a + \sqrt{a^2 + 1})(y(t) - \xi(t)), \\
    u(t) &= -(a + \sqrt{a^2 + 1}) \xi(t),
\end{aligned}
\end{equation*}
which can also be written as
\begin{equation} \label{eq:LQG-example-controller}
A_\mK^\star = -a -2\sqrt{a^2 + 1}, \qquad B_\mK^\star = a + \sqrt{a^2 + 1}, \qquad C_\mK^\star = -a - \sqrt{a^2 + 1}.
\end{equation}

The globally optimal LQG cost is
$$
J^2_{\texttt{LQG},1}(\mK^\star) = \frac{4a(a^2 + 1)^{3/2} + 6a^2 + 4a^4 + 2}{(a^2 + 1)^{1/2}}.
$$
Note that \cref{eq:LQG-example-controller} is stationary (it is also easy to verify its gradient to be zero). It is clear that the optimal LQG policy \cref{eq:LQG-example-controller} is minimal $\forall a \in \mathbb{R}$, and thus non-degenerate (cf. \Cref{theorem:LQG-global-optimality}), i.e., $\mK^\star \in {\mathcal{C}}_{\mathrm{nd},0}$. Therefore, the LQG optimization for this instance is solvable and the solution to its equivalent LMI will return a globally optimal LQG policy.  %\hfill $\square$

\subsubsection{Globally optimal $\mathcal{H}_\infty$ policies}

Here, we discuss the globally optimal $\mathcal{H}_\infty$ policies for a class of $\mathcal{H}_\infty$ instances with data \cref{eq:appendix_example}. These globally optimal $\mathcal{H}_\infty$ policies turn out to be all degenerate.

\begin{example}[Globally optimal $\mathcal{H}_\infty$ policies] \label{example:optimal-Hinf-policy}
    For the $\mathcal{H}_\infty$ robust control \cref{eq:Hinf_policy_optimization} with data \cref{eq:appendix_example}, given $\gamma > 0$, the Riccati equation \cref{eq:riccati-hinf-a} reads as
    $$
    X_{\infty}^2(1-\gamma^{-2})- 2aX_{\infty} - 1 = 0,
    $$
    which has two real solutions %(for simplicity, we assume $\gamma > 1$ throughout this example; see \cite[Example 14.3]{zhou1998essentials} for further discussions)
    (when $\gamma \neq 1$)
    \begin{subequations} \label{eq:example-Hinf-Riccati}
    \begin{align}
    X_{\infty,1} &= \frac{a + \sqrt{a^2 + 1 - \gamma^{-2}}}{1- \gamma^{-2}}, \label{eq:example-Hinf-Riccati-a}\\ %\qquad
    X_{\infty,2} &= \frac{a - \sqrt{a^2 + 1 - \gamma^{-2}}}{1-\gamma^{-2}}. \label{eq:example-Hinf-Riccati-b}
    \end{align}
    \end{subequations}
    We need to ensure
    $$
    \gamma \geq \frac{1}{\sqrt{a^2 + 1}},
    $$
    otherwise, \Cref{eq:example-Hinf-Riccati} are not real numbers.
    Via simple calculations, we verify that \cref{eq:example-Hinf-Riccati-b} makes the following matrix (scalar) stable.
$$
A - (B R^{-1}B^\tr - \gamma^{-2} W) X_{\infty,1} = a - (1 - \gamma^{-2})X_{\infty,1} = - \sqrt{a^2 +1 - \gamma^{-2}} < 0.
$$
Similarly, the stabilizing solution to \cref{eq:riccati-hinf-b} is
\begin{equation} \label{eq:example-Hinf-Riccati-Y}
Y_{\infty} = \frac{a + \sqrt{a^2 + 1 - \gamma^{-2}}}{1- \gamma^{-2}}.
\end{equation}
Next, we can verify that
$$
\rho(X_{\infty}Y_\infty) = \left(\frac{a + \sqrt{a^2 + 1 - \gamma^{-2}}}{1- \gamma^{-2}}\right)^2 < \gamma^2 \quad \Leftrightarrow \quad \gamma > \sqrt{a^2 + 2} + a \;\left(\geq \frac{1}{\sqrt{a^2 + 1}}\right).
$$

\begin{figure}
\centering
\begin{subfigure}{0.32\textwidth}
    \includegraphics[width =0.85\textwidth]{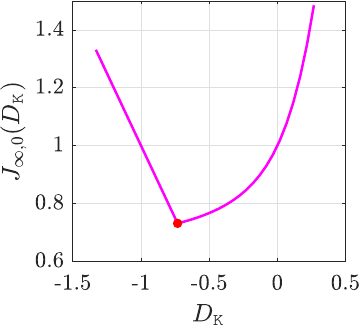}
    \caption{$a=-1, D_\mK^\star=1-\sqrt{3}$}
    %\label{fig:}
\end{subfigure}
\hspace{1mm}
\begin{subfigure}{0.32\textwidth}
    \includegraphics[width =0.82\textwidth]{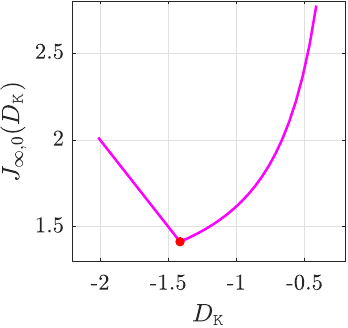}
    \caption{$a=0,D_\mK^\star=-\sqrt{2}$}
    %\label{fig:}
\end{subfigure}
\hspace{1mm}
\begin{subfigure}{0.32\textwidth}
    \includegraphics[width =0.85\textwidth]{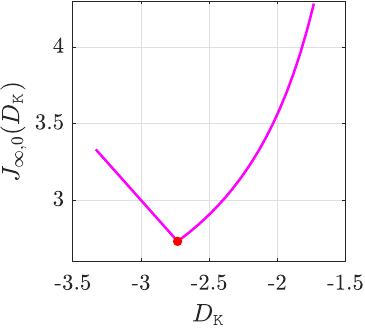}
    \caption{$a=1,D_\mK^\star=-1-\sqrt{3}$}
    %\label{fig:}
\end{subfigure}

    \caption{Landscape of ${J}_{\infty,0}(D_\mK)$ for static output feedback $u(t)=D_\mK y(t)$ in \Cref{example:optimal-Hinf-policy}, where the red points highlight the globally optimal policy $D_\mK^\star$ that is a nonsmooth stationary point.}
    \label{fig:Hinf-3plot-SF-non-smooth}
\end{figure}

To summarize, for any $\gamma > \sqrt{a^2 + 2} + a$, the Riccati equation \Cref{eq:riccati-hinf} has stabilizing solutions \Cref{eq:example-Hinf-Riccati-a,eq:example-Hinf-Riccati-Y} that satisfies $\rho(X_{\infty}Y_\infty) < \gamma^2$. Therefore, a suboptimal $\mathcal{H}_\infty$ policy that ensure $J_{\infty,n}(\mK)$ is in the form of
\begin{equation} \label{eq:hinf-suboptimal-example}
        \begin{aligned}
             \dot{\xi}(t) &= (a+\gamma^{-2}X_{\infty})\xi(t) +  u(t) + (1 - \gamma^{-2}Y_{\infty}X_{\infty})^{-1}Y_{\infty}(y(t) - \xi(t)) \\
            u(t) &= -X_{\infty}\xi(t).
        \end{aligned}
    \end{equation}
%    where the feedback gain is $F_{\infty} = R^{-1}B^\tr X_{\infty}$ and the observer gain is $L_{\infty} = Y_{\infty}CV^{-1}$ and $Z_{\infty} = (I - \gamma^{-2}Y_{\infty}X_{\infty})^{-1}$.
Let $\gamma \to \sqrt{a^2 + 2} + a$, we have the globally optimal $\mathcal{H}_\infty$ cost is
$
\gamma_{\mathrm{opt}} = \sqrt{a^2 + 2} + a,
$
for which the optimal controller becomes a static output feedback $D_\mK^\star=-X_{\infty}$ (see \cite[Example 16.1]{zhou1996robust}):
$$
u(t) = -X_{\infty}y(t) = -\frac{a + \sqrt{a^2 + 1 - \gamma_{\mathrm{opt}}^{-2}}}{1- \gamma_{\mathrm{opt}}^{-2}} y(t)=-\gamma_{\mathrm{opt}}y(t).
$$
Thus, an optimal policy of the $\mathcal{H}_\infty$ robust control \cref{eq:Hinf_policy_optimization} with data \cref{eq:appendix_example} is static, for all $a \in \mathbb{R}$.
 \hfill $\square$
\end{example}

Let us consider three different values $a =0, a = -1$ and $a = 1$ in \Cref{example:optimal-Hinf-policy}. The corresponding optimal cost $\gamma_{\mathrm{opt}}$ and the optimal policies $D_\mK^\star$ are listed as follows:
\begin{itemize}
    \item If $a = 0$, we have $\gamma_{\mathrm{opt}}=\sqrt{2}$ and $D_\mK^\star=-\sqrt{2}$.
    \item If $a = -1$, we have $\gamma_{\mathrm{opt}}=\sqrt{3}-1$ and $D_\mK^\star=1-\sqrt{3}$.
    \item If $a = 1$, we have $\gamma_{\mathrm{opt}}=\sqrt{3}+1$ and $D_\mK^\star=-1-\sqrt{3}$.
\end{itemize}
\Cref{fig:Hinf-3plot-SF-non-smooth} and \Cref{fig:Hinf-3plot-OF-non-smooth} illustrate the landscape of $\mathcal{H}_\infty$ cost for the cases of static output feedback and dynamic output feedback, i.e., $J_{\infty,0}(D_\mK)$ and $J_{\infty,1}(\mK)$, respectively.
In \Cref{fig:Hinf-3plot-SF-non-smooth}, $J_{\infty,0}(D_\mK)$ is plotted as a function of $D_\mK$ around the globally optimal policy $D_\mK^\star$; in \Cref{fig:Hinf-3plot-OF-non-smooth}, $J_{\infty,1}(\mK)$ is plotted as a function of $(B_\mK,C_\mK)$ around $(0,0)$ by fixing $D_\mK=D_\mK^\star$ and $A_\mK=-1$. It is clear that these $\mathcal{H}_\infty$ cost functions are all non-smooth (the nonsmooth points are highlighted in red).

%{\color{blue}It is likely that the controller is degenerate. }

\begin{figure}
\centering
\begin{subfigure}{0.32\textwidth}
    \includegraphics[width =\textwidth]{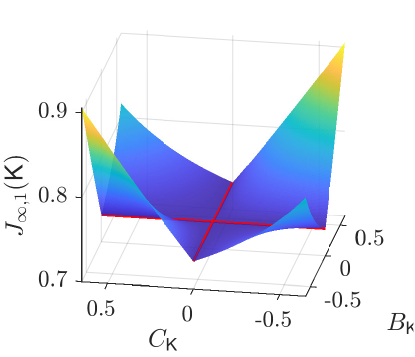}
    \caption{$a=-1,\gamma_\mathrm{opt}=\sqrt{3}-1$}
    %\label{fig:}
\end{subfigure}
    \hspace{1mm}
\begin{subfigure}{0.32\textwidth}
    \includegraphics[width =\textwidth]{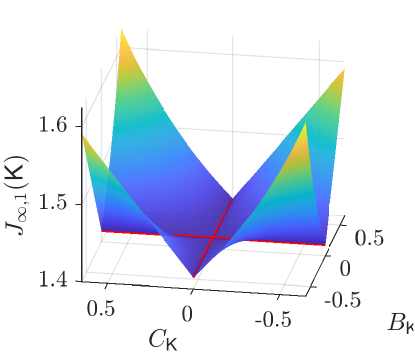}
    \caption{$a=0,\gamma_\mathrm{opt}=\sqrt{2}$}
    %\label{fig:}
\end{subfigure}
\hspace{1mm}
\begin{subfigure}{0.32\textwidth}
    \includegraphics[width =\textwidth]{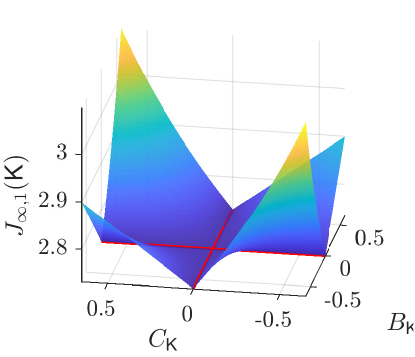}
    \caption{$a=1,\gamma_\mathrm{opt}=\sqrt{3}+1$}
    %\label{fig:}
\end{subfigure}

    \caption{Landscape of ${J}_{\infty,1}(\mK)$ for dynamic output feedback in \Cref{example:optimal-Hinf-policy}, where we fix $D_\mK=D_\mK^\star$ and $A_\mK=-1$ and change $B_{\mK} \in (-0.6,0.6)$, $C_{\mK} \in (-0.6,0.6)$. The red line highlights the nonsmooth points. }
    \label{fig:Hinf-3plot-OF-non-smooth}
\end{figure}

\begin{example}[Globally optimal $\mathcal{H}_\infty$ policies that are degenerate] \label{example:Hinf_degenerate_global_opt}
Consider the same $\mathcal{H}_\infty$ control instance in \Cref{example:Hinf-nonsmooth}.
%
    % Let us consider a simple $\mathcal{H}_\infty$ instance with
    % $$
    % A = -1,\, B = 1, \,C = 1,\, Q = 1,\, R = 1,\, W = 1, \, V = 1.
    % $$
  As discussed in \Cref{example:optimal-Hinf-policy}, via analyzing the limiting behavior in \Cref{theorem:Hinf-riccati}, we can show that a globally optimal $\mathcal{H}_\infty$ controller in this instance is achieved by a static output feedback policy
  \begin{equation} \label{eq:example-static-hinf}
  u(t) = D_\mK^\star y(t), \quad \text{with} \quad D_\mK^\star = 1 - \sqrt{3},
  \end{equation}
  and the globally optimal $\mathcal{H}_\infty$ cost is $J^\star_{\infty,1} = \sqrt{3}-1$.
  {One can also verify zero is an element of its Clarke subdifferential at $D_\mK^\star$, i.e., $0\in\partial J(D_\mK^\star)$. Indeed, the closed-loop transfer function reads as
  \[
  \mathbf{T}_{zd}(D_\mK^\star,j\omega)=\frac{1}{\sqrt{3} + j\omega}\begin{bmatrix}
      1 & 1-\sqrt{3}\\
      1-\sqrt{3} & (1 - \sqrt{3})(1+j\omega)
  \end{bmatrix}
  \]
  with its singular values given by
  \[
  \sqrt{3}-1, \quad \text{and} \quad \sqrt{\frac{1}{\omega^2 + 3}}.
  \]
  Thus we have $\sigma_\mathrm{max}(\mathbf{T}_{zd}(D_\mK^\star,j\omega))=\sqrt{3}-1, \forall \omega \in \mathbb{R}.$
  Using the subdifferential formula given in \Cref{appendix:clarke-hinf}, one can verify that $\partial J(D_\mK^\star)$ contains a positive subgradient $\approx 0.0359$ for $\omega=0.3$, and a negative subgradient $\approx -0.0838$ for $\omega=0.5$.
  The Clarke stationarity of $D_\mK^\star$ follows from the convexity of the subdifferential.}

  We now discuss different state-space realizations of this globally optimal controller \cref{eq:example-static-hinf}.

  \begin{itemize}
      \item \textbf{Uncontrollable and unobservable cases:}
For any $a < 0$, the following uncontrollable and unobservable policy
  \begin{equation} \label{eq:example-global-Hinf}
    A_\mK = a, \quad B_\mK = 0, \quad C_\mK = 0, \quad D_\mK^\star = 1 - \sqrt{3},
  \end{equation}
  is also globally optimal. % (since they all corresponds to same )
  % We here verify that the matrix $P\succ 0$ satisfying the non-strict LMI \Cref{eq:Hinf-Bilinear-a} with $\gamma = \sqrt{3}-1$ must have $P_{12} = 0$, thus this globally optimal controller \cref{eq:example-global-Hinf} is degenerate.
%
The closed-loop system matrices with \cref{eq:example-global-Hinf} read as
\begin{equation*}
    A_{\mathrm{cl}} = \begin{bmatrix}
        -\sqrt{3} & 0 \\ 0 & a
    \end{bmatrix}, \quad
    B_{\mathrm{cl}} = \begin{bmatrix}
        1 & 1 - \sqrt{3} \\ 0 & 0
    \end{bmatrix}, \quad
    C_{\mathrm{cl}} = \begin{bmatrix}
        1 & 0 \\ 1 - \sqrt{3} & 0
    \end{bmatrix}, \quad
    D_{\mathrm{cl}} = \begin{bmatrix}
        0 & 0 \\ 0 & 1 - \sqrt{3}
    \end{bmatrix}.
\end{equation*}
Then, \Cref{eq:Hinf-Bilinear-a} with $\gamma = \sqrt{3}-1$ is equivalent to
\begingroup
    \setlength\arraycolsep{2pt}
\def\arraystretch{0.95}
$$
\begin{bmatrix}
    2\begin{bmatrix}
        -\sqrt{3} & 0 \\ 0 & a
    \end{bmatrix}\begin{bmatrix}
        p_1 & p_2 \\ p_2 & p_3
    \end{bmatrix} & \begin{bmatrix}
        p_1 & p_2 \\ p_2 & p_3
    \end{bmatrix}\begin{bmatrix}
        1 & 1 - \sqrt{3} \\ 0 & 0
    \end{bmatrix} \\ * & -(\sqrt{3}-1)^2 \begin{bmatrix}
        1 & 0 \\ 0 & 1
    \end{bmatrix}
\end{bmatrix} + \begin{bmatrix}
        \begin{bmatrix}
        1 + (1 - \sqrt{3})^2 & 0 \\ 0 & 0
    \end{bmatrix} & \begin{bmatrix}
        0 & (1 - \sqrt{3})^2 \\ 0 & 0
    \end{bmatrix} \\ * & \begin{bmatrix}
        0 & 0 \\ 0 & (1 - \sqrt{3})^2
    \end{bmatrix}
    \end{bmatrix} \preceq 0,
$$
    \endgroup
which is the same as
\begingroup
    \setlength\arraycolsep{2pt}
\def\arraystretch{0.95}
$$
\begin{bmatrix}
  1+(1 - \sqrt{3})^2 - 2\sqrt{3}p_1  & ap_2 - \sqrt{3}p_2 & p_1  & p_1(1 - \sqrt{3}) + (1 - \sqrt{3})^2 \\
   * & 2ap_3 & p_2 & p_2(1 - \sqrt{3})  \\
   * & *& -(\sqrt{3}-1)^2 &  0 \\
   * & * & * & 0
    \end{bmatrix} \preceq 0.
$$
    \endgroup
Since the (4,4) element is zero, the entire column must be zero, i.e.,
$
p_2 = 0, \, p_1 = \sqrt{3}-1.
$
Then the LMI is simplified as
$$
\begin{bmatrix}
  -1  & 0 & \sqrt{3}-1 & 0 \\
   * & 2ap_3 & 0 & 0  \\
   * & *& -(\sqrt{3}-1)^2 &  0 \\
   * & * & * & 0
    \end{bmatrix} \preceq 0,
$$
which is satisfied for any $p_3 \geq 0$ since $a <0$.

Since we always have $p_2 = 0$, this policy \cref{eq:example-global-Hinf} is degenerate.

\item  \textbf{Controllable but unobservable cases:} For any $a < 0, b \neq 0$, the following controllable but unobservable policy
  \begin{equation} \label{eq:example-global-Hinf-con-unobs}
    A_\mK = a, \quad B_\mK = b, \quad C_\mK = 0, \quad D_\mK^\star = 1 - \sqrt{3},
  \end{equation}
  is also globally optimal. The closed-loop system matrices with \cref{eq:example-global-Hinf} read as
\begin{equation*}
    A_{\mathrm{cl}} = \begin{bmatrix}
        -\sqrt{3} & 0 \\ b & a
    \end{bmatrix}, \quad
    B_{\mathrm{cl}} = \begin{bmatrix}
        1 & 1 - \sqrt{3} \\ 0 & b
    \end{bmatrix}, \quad
    C_{\mathrm{cl}} = \begin{bmatrix}
        1 & 0 \\ 1 - \sqrt{3} & 0
    \end{bmatrix}, \quad
    D_{\mathrm{cl}} = \begin{bmatrix}
        0 & 0 \\ 0 & 1 - \sqrt{3}
    \end{bmatrix}.
\end{equation*}
Then, \Cref{eq:Hinf-Bilinear-a} with $\gamma = \sqrt{3}-1$ is equivalent to
$$
\begin{bmatrix}
    \begin{bmatrix}
        -\sqrt{3} & 0 \\ b & a
    \end{bmatrix}\begin{bmatrix}
        p_1 & p_2 \\ p_2 & p_3
    \end{bmatrix} + * & \begin{bmatrix}
        p_1 & p_2 \\ p_2 & p_3
    \end{bmatrix}\begin{bmatrix}
        1 & 1 - \sqrt{3} \\ 0 & b
    \end{bmatrix} \\ * & -(\sqrt{3}-1)^2 \begin{bmatrix}
        1 & 0 \\ 0 & 1
    \end{bmatrix}
\end{bmatrix} + \begin{bmatrix}
        \begin{bmatrix}
        1 + (1 - \sqrt{3})^2 & 0 \\ 0 & 0
    \end{bmatrix} & \begin{bmatrix}
        0 & (1 - \sqrt{3})^2 \\ 0 & 0
    \end{bmatrix} \\ * & \begin{bmatrix}
        0 & 0 \\ 0 & (1 - \sqrt{3})^2
    \end{bmatrix}
    \end{bmatrix} \preceq 0,
$$
which is the same as
\begingroup
    \setlength\arraycolsep{2pt}
\def\arraystretch{0.95}
$$
\begin{bmatrix}
  1+(1 - \sqrt{3})^2 - 2\sqrt{3}p_1 +2bp_2  & ap_2 - \sqrt{3}p_2+bp_3 & p_1 & p_1(1 - \sqrt{3}) + bp_2+ (1 - \sqrt{3})^2 \\
   * & 2ap_3 & p_2 & p_2(1 - \sqrt{3})+bp_3  \\
   * & *& -(\sqrt{3}-1)^2 &  0 \\
   * & * & * & 0
    \end{bmatrix} \preceq 0.
$$
\endgroup
Since the (4,4) element is zero, we must have
$$
\begin{cases}
    p_2(1 - \sqrt{3})+bp_3 &= 0, \\
p_1(1 - \sqrt{3}) + bp_2+ (1 - \sqrt{3})^2 &= 0,
\end{cases} \quad \Rightarrow \quad \begin{cases}
     p_3 &= \displaystyle \frac{\sqrt{3}-1}{b}p_2, \\
p_1 &= \displaystyle  \frac{b}{\sqrt{3}-1}p_2 + \sqrt{3}-1.
\end{cases}
$$
Then, the LMI is simplified as
% $$
% \begin{bmatrix}
%   1-(1 - \sqrt{3})^2 - 2p_1   & (a-1)p_2 & p_1 & 0 \\
%    * & \frac{2a(\sqrt{3}-1)}{b}p_2 & p_2 & 0  \\
%    * & *& -(\sqrt{3}-1)^2 &  0 \\
%    * & * & * & 0
%     \end{bmatrix} \preceq 0.
% $$
% which is also the same as
$$
\begin{bmatrix}
  - 1 + \displaystyle \frac{2bp_2}{1-\sqrt{3}}   & (a-1)p_2 & \displaystyle  \frac{b}{\sqrt{3}-1}p_2 + \sqrt{3}-1 & 0 \\
   * & \displaystyle  \frac{2a(\sqrt{3}-1)}{b}p_2 & p_2 & 0  \\
   * & *& -(\sqrt{3}-1)^2 &  0 \\
   * & * & * & 0
    \end{bmatrix} \preceq 0.
$$
The determinant of the 3-by-3 non-zero block is
$$
\begin{aligned}
  & p_2^2\times \big((4- 2\sqrt{3})a^2 + (10 - 6\sqrt{3})a +7 - 4\sqrt{3}\big) \\
= \; & p_2^2\times (4- 2\sqrt{3})\left(a-\frac{\sqrt{3}-1}{2}\right)^2  \\
\leq \; &0.
\end{aligned}
$$
Therefore, we must have $p_2 = 0$ since $a < 0$.

In summary, for the controllable but unobservable policy \cref{eq:example-global-Hinf-con-unobs}, the LMI \Cref{eq:Hinf-Bilinear-a} with $\gamma = \sqrt{3}-1$ admits a unique positive semidefinite solution
$$
P = \begin{bmatrix}
    \sqrt{3}-1&0\\ 0&0
\end{bmatrix} \succeq 0.
$$
Thus, this policy \cref{eq:example-global-Hinf-con-unobs} is degenerate.

\item  \textbf{Uncontrollable but observable cases:} For any $a < 0, c \neq 0$, the following uncontrollable but observable policy
  \begin{equation} \label{eq:example-global-Hinf-uncon-obs}
    A_\mK = a, \quad B_\mK = 0, \quad C_\mK = c, \quad D_\mK^\star = 1 - \sqrt{3},
  \end{equation}
  is also globally optimal. The closed-loop system matrices with \cref{eq:example-global-Hinf-uncon-obs} read as
\begin{equation*}
    A_{\mathrm{cl}} = \begin{bmatrix}
        -\sqrt{3} & c \\ 0 & a
    \end{bmatrix}, \quad
    B_{\mathrm{cl}} = \begin{bmatrix}
        1 & 1 - \sqrt{3} \\ 0 & 0
    \end{bmatrix}, \quad
    C_{\mathrm{cl}} = \begin{bmatrix}
        1 & 0 \\ 1 - \sqrt{3} & c
    \end{bmatrix}, \quad
    D_{\mathrm{cl}} = \begin{bmatrix}
        0 & 0 \\ 0 & 1 - \sqrt{3}
    \end{bmatrix}.
\end{equation*}
Then, \Cref{eq:Hinf-Bilinear-a} with $\gamma = \sqrt{3}-1$ is equivalent to
\begingroup
    \setlength\arraycolsep{2pt}
\def\arraystretch{0.95}
$$
\begin{bmatrix}
    \begin{bmatrix}
        -\sqrt{3} & c \\ 0 & a
    \end{bmatrix}\begin{bmatrix}
        p_1 & p_2 \\ p_2 & p_3
    \end{bmatrix} + * & \begin{bmatrix}
        p_1 & p_2 \\ p_2 & p_3
    \end{bmatrix}\begin{bmatrix}
        1 & 1 - \sqrt{3} \\ 0 & 0
    \end{bmatrix} \\ * & -(\sqrt{3}-1)^2 \begin{bmatrix}
        1 & 0 \\ 0 & 1
    \end{bmatrix}
\end{bmatrix} + \begin{bmatrix}
        \begin{bmatrix}
        1 + (1 - \sqrt{3})^2 & c(1-\sqrt{3}) \\ * & c^2
    \end{bmatrix} & \begin{bmatrix}
        0 & (1 - \sqrt{3})^2 \\ 0 & c(1-\sqrt{3}
    \end{bmatrix} \\ * & \begin{bmatrix}
        0 & 0 \\ 0 & (1 - \sqrt{3})^2
    \end{bmatrix}
    \end{bmatrix} \preceq 0,
$$
\endgroup
which is the same as
\begingroup
    \setlength\arraycolsep{2pt}
\def\arraystretch{0.95}
$$
\begin{bmatrix}
  1+(1 - \sqrt{3})^2 - 2\sqrt{3}p_1  & (a - \sqrt{3})p_2 + c(p_1 + 1- \sqrt{3}) & p_1 & (\sqrt{3} - 1)^2 - p_1(\sqrt{3} - 1) \\
   * & c^2 + 2(ap_3+cp_2) & p_2 & (p_2+c)(1 - \sqrt{3})  \\
   * & *& -(\sqrt{3}-1)^2 &  0 \\
   * & * & * & 0
    \end{bmatrix} \preceq 0.
$$
\endgroup
Similarly, we must have
$
p_2 = -c, \quad p_1 = \sqrt{3}-1.
$
The LMI is simplified as
$$
\begin{bmatrix}
  -1  & (\sqrt{3}-a)c  & \sqrt{3}-1 & 0 \\
   * & 2ap_3-c^2 & -c & 0  \\
   * & *& -(\sqrt{3}-1)^2 &  0 \\
   * & * & * & 0
    \end{bmatrix} \preceq 0.
$$
The determinant of the 3-by-3 nonzero block is (recall that $a < 0$)
$$
\begin{aligned}
&c^2((4- 2\sqrt{3})a^2 + (10 - 6\sqrt{3})a  + 7 - 4\sqrt{3}) \\
= \; & c^2\times (4- 2\sqrt{3})\left(a-\frac{\sqrt{3}-1}{2}\right)^2 \leq 0
\end{aligned}
$$
which can not be feasible unless $c = 0$.
% Then, we must have
% $$
% \begin{aligned}
%     2ap_3&\leq c^2,\\
%      2ap_3 & \leq - (a^2 -2\sqrt{3}a + 2)c^2  \\
%     (8 - 4\sqrt{3}) ap_3 &\leq (3 - 2\sqrt{3})c^2
% \end{aligned}
% $$
%
In summary, for the uncontrollable but observable policy \cref{eq:example-global-Hinf-uncon-obs}, the LMI \Cref{eq:Hinf-Bilinear-a} with $\gamma = \sqrt{3}-1$ is infeasible. Thus, this policy \cref{eq:example-global-Hinf-uncon-obs} is also degenerate. %\hfill $\square$

%{\color{red} @Rich, the LMI formulation should be unsolvable! The infimum cannot be achieved, unless there is another globally optimal controller!}

  \end{itemize}
  To conclude this example, we have verified that all the following  state-space realizations of the static optimal policy \cref{eq:example-static-hinf}
$$
\begin{aligned}
    A_\mK &= a, \quad B_\mK = b, \quad  C_\mK = 0, \quad  D_\mK = 1 - \sqrt{3} \\
    A_\mK &= a, \quad B_\mK = 0, \quad  C_\mK = c, \quad  D_\mK = 1 - \sqrt{3}
\end{aligned}
$$
with $a < 0, b \neq 0, c\neq 0$ are all globally optimal and degenerate for the $\mathcal{H}_\infty$ instance \Cref{example:Hinf-nonsmooth}.\hfill \qed
\end{example}
%
%
%
% {\color{red}
% Consider another globally optimal $\mathcal{H}_\infty$ controller
% $$
%     A_\mK = -1, \quad B_\mK = 1, \quad  C_\mK = 0, \quad  D_\mK = 1 - \sqrt{3}.
% $$
% check whether this one is non-degenerate; it looks like that they are all degenerate

% Consider another globally optimal $\mathcal{H}_\infty$ controller
% $$
%     A_\mK = -1, \quad B_\mK = 0, \quad  C_\mK = 1, \quad  D_\mK = 1 - \sqrt{3}.
% $$
% This controller makes the bounded real LMI infeasible. Check this.
% }

\subsection{Boundary behavior of $\mathcal{H}_\infty$ cost} %\label{appendix:boundary-LQG}
\label{subsection:hinf-boundary-behavior}

%\subsubsection{$\mathcal{H}_\infty$ case in \Cref{example:non-coercivity-of-Hinf-cost}}
We here give a few more details regarding the boundary behavior of the $\mathcal{H}_\infty$ case in \Cref{example:non-coercivity-of-Hinf-cost}. In particular, we consider two cases: 1) case 1: we have a sequence of stabilizing policies that approach to a boundary point, but the corresponding $\mathcal{H}_\infty$ costs converge to a finite value; 2) case 2: we have another sequence of stabilizing policies that approach to the same boundary point, but the corresponding $\mathcal{H}_\infty$ values diverge to infinity. % to a finite value

\begin{figure}
\centering
\setlength{\abovecaptionskip}{0pt}
    \includegraphics[width =0.5\textwidth]{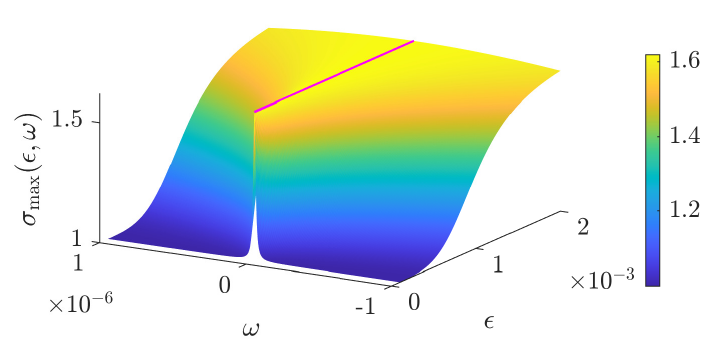}
    \caption{The shape of $\sigma_\mathrm{max}(\mathbf{T}_{zd}(\mK_\epsilon^{(1)},j\omega))$ in \Cref{eq:hinf-ex-singular-value}. The red line denotes the maximum for small $\epsilon$, which is achieved in $\omega = 0$.
    }
\label{fig:Hinf_boundary_ex1}
\end{figure}

\textbf{Case 1}: Let us first consider a sequence of parameterized stabilizing policies of the form
\begin{equation} \label{eq:appendix_ex_hinf_a}
\mK_\epsilon^{(1)} = \begin{bmatrix}0 & \epsilon \\ -\epsilon & 0 \end{bmatrix}\in \mathcal{C}_1, \quad  \forall \epsilon \neq 0,
\end{equation}
which is the same as \cref{eq:hinf_ex_a} in the main text. The corresponding closed-loop transfer function from $d(t)$ to $z(t)$ become
$$
\mathbf{T}_{zd}(\mK_\epsilon^{(1)},s) = \frac{1}{s^2+s+\epsilon}\begin{bmatrix}
    s & -\epsilon^2 \\ -\epsilon^2 & -\epsilon^2 (s + 1)
\end{bmatrix}.
$$
Using the following closed-form formula for the maximum singular value of a $2\times2$ matrix $\mathbf{T}_{zd}$:
\begin{equation}\label{eq:sigma_max}
    \sigma_\mathrm{max}(\mathbf{T}_{zd})
    %=\left(\lambda_\mathrm{max}(\mathbf{T}_{zd}^\her\mathbf{T}_{zd})\right)^{\frac{1}{2}}
    =\left(\frac{\mathrm{tr}(\mathbf{T}_{zd}^\her\mathbf{T}_{zd})+\sqrt{\mathrm{tr}(\mathbf{T}_{zd}^\her\mathbf{T}_{zd})^2-4\mathrm{det}(\mathbf{T}_{zd}^\her\mathbf{T}_{zd})}}{2}\right)^{1/2},
\end{equation}
the maximum singular value of $\mathbf{T}_{zd}(\mK_\epsilon^{(1)},j\omega)$ reads as
\begin{equation} \label{eq:hinf-ex-singular-value}
\sigma^2_\mathrm{max}(\mathbf{T}_{zd}(\mK_\epsilon^{(1)},j\omega))
=%\left(
\frac{(\epsilon^4+1)\omega^2+3\epsilon^4+\sqrt{(\epsilon^8-2\epsilon^4+1)\omega^4+(6\epsilon^8+8\epsilon^6+2\epsilon^4)\omega^2+5\epsilon^8}}{2\omega^4+(2-4\epsilon^2)\omega^2+2\epsilon^4}. %\right)^{1/2}.
\end{equation}

For this function $\sigma_\mathrm{max}(\mathbf{T}_{zd}(\mK_\epsilon^{(1)},j\omega))$, its behavior for small $\epsilon$ is is shown in \Cref{fig:Hinf_boundary_ex1}. Indeed, we claim the result in \Cref{fact:hinf-example}.  Therefore, we have
\begin{align*}
    \lim_{\epsilon\downarrow 0}J_{\infty,1}(\mK_\epsilon^{(1)}) &=
    \lim_{\epsilon\downarrow  0, \epsilon\leq\epsilon_0} \left(\frac{3\epsilon^4+\sqrt{5\epsilon^8}}{2\epsilon^4}\right)^{1/2} = \left(\frac{3+\sqrt{5}}{2}\right)^{1/2}= \frac{1 + \sqrt{5}}{2},
\end{align*}
which confirms the statement of finite value in  \Cref{example:non-coercivity-of-Hinf-cost}.

\begin{fact} \label{fact:hinf-example}
    Consider the sequence of stabilizing policies \cref{eq:appendix_ex_hinf_a}. The $\mathcal{H}_\infty$ cost of $\mK_\epsilon^{(1)}$, defined as
$$J_{\infty,1} (\mK_\epsilon^{(1)})=\sup_{\omega\in\mathbb{R}}\sigma_{\mathrm{max}}(\mathbf{T}_{zd}(\mK_\epsilon^{(1)},j\omega))$$
is achieved at $\omega=0$ when $0<\epsilon<\epsilon_0$ for sufficiently small $\epsilon_0>0$. %(for $\epsilon\leq0.5$).
\end{fact}

\noindent \textbf{Proof}:
From \cref{eq:hinf-ex-singular-value}, for notational simplicity, let us choose $x = \omega^2$ and define
$$
\begin{aligned}
f(x)=&\frac{(\epsilon^4+1)x+3\epsilon^4+\sqrt{(\epsilon^8-2\epsilon^4+1)x^2+2\epsilon^4(3\epsilon^4+4\epsilon^2+1)x+5\epsilon^8}}{2x^2+(2-4\epsilon^2)x+2\epsilon^4}=\sigma_\mathrm{max}^2\left(\mathbf{T}_{zd}(\mK_\epsilon^{(1)},j\omega)\right).
\end{aligned}
$$
We can prove that when $\epsilon>0$ is small enough, $f(x)$ is a monotonically non-increasing function, i.e.
\begin{equation} \label{eq:decreasing-function}
f'(x) \leq 0, \qquad \forall x \geq 0.
\end{equation}
Indeed, to verify \cref{eq:decreasing-function}, let us write
$
f(x) = f_1(x) + f_2(x)
$ with
$$
\begin{aligned}
    f_1(x) &= \frac{(\epsilon^4+1)x+3\epsilon^4}{2x^2+(2-4\epsilon^2)x+2\epsilon^4} = \frac{p_1(x)}{q_1(x)}, \\
    f_2(x)& = \frac{\sqrt{(\epsilon^8-2\epsilon^4+1)x^2+2\epsilon^4(3\epsilon^4+4\epsilon^2+1)x+5\epsilon^8}}{2x^2+(2-4\epsilon^2)x+2\epsilon^4}  = \frac{p_2(x)}{q_2(x)}.
\end{aligned}
$$
Then, there exists an $\epsilon_0>0$ such that  $f_1(x)$ is a monotonically decreasing function, i.e.,
$$
\begin{aligned}
p_1'(x)q_1(x) - p_1(x)q_1'(x) &= 2\epsilon^8 + 12\epsilon^6 - 2\epsilon^4(x^2 + 6x +2) - x^2 \\
&\leq 2\epsilon^8 + 12\epsilon^6 - 4\epsilon^4 - x^2 \\
& < 0, \qquad  \forall x \geq 0, \epsilon \leq \epsilon_0.
%-2(1+\epsilon^4)x^2 - 12\epsilon^4 x + 2\epsilon^8 + 12\epsilon^6 %\\
%&=-2x^2 -2\epsilon^4(x^2 + 6x + 9)
\end{aligned}
$$
Similarly, we can verify that $f_2(x)$ is also monotonically decreasing when $\epsilon$ is sufficiently small.
%
% $$
% -2\frac{- \epsilon^{12}x - 3\epsilon^{12} - 6\epsilon^{10}x - 14\epsilon^{10} + \epsilon^8x^3 + 9\epsilon^8x^2 + 7\epsilon^8x + 4\epsilon^8 + 12\epsilon^6x^2 + 2\epsilon^6x - 2\epsilon^4x^3 + 3\epsilon^4x^2 + x^3}{\sqrt{(\epsilon^8-2\epsilon^4+1)x^2+2\epsilon^4(3\epsilon^4+4\epsilon^2+1)x+5\epsilon^8}}
% $$
Thus, the maximum of $f(x)$ is achieved at $x = 0$: there exists a sufficiently small $\epsilon_0  >0$ such that
$$
J_{\infty,1}^2 (\mK_\epsilon^{(1)}) =\max_x\; f(x) = f(0) = \frac{3\epsilon^4+\sqrt{5\epsilon^8}}{2\epsilon^4} = \left(\frac{1 + \sqrt{5}}{2}\right)^2, \qquad \text{if }\; \epsilon \leq \epsilon_0.
$$
% This proves \Cref{fact:hinf-example}.  \hfill $\qed$

% \vspace{50mm}

\vspace{2mm}

\textbf{Case 2}: We next consider another sequence of parameterized dynamic policies of the form
\begin{equation} \label{eq:appendix_ex_hinf_b}
\mK_\epsilon^{(2)} = \begin{bmatrix}0 & \epsilon+\epsilon^d \\ -\epsilon & \epsilon^2 \end{bmatrix} \in \mathcal{C}_1, \quad \forall  1 > \epsilon > 0
\end{equation}
where $d>0$ is a positive integer. When $d = 4$, this policy \cref{eq:appendix_ex_hinf_b} is the same as \cref{eq:hinf-boundary-2-maintext} in the main text. The corresponding closed-loop transfer function from $d(t)$ to $z(t)$ is
$$
    \mathbf{T}_{zd}(\mK_\epsilon^{(2)},s) = \frac{1}{s^2 + (1- \epsilon^2)s + \epsilon^{d+1}} \begin{bmatrix} s-\epsilon^2 & -\epsilon^2 -\epsilon^{d+1} \\ -\epsilon^2 -\epsilon^{d+1}  &  -(\epsilon^2 +\epsilon^{d+1}) (s + 1) \end{bmatrix}.
$$
Using \Cref{eq:sigma_max} we obtain its maximum singular value
$$
\sigma_\mathrm{max}(\mathbf{T}_{zd}(\mK_\epsilon^{(2)},j\omega))
=\left(\frac{\mathrm{tr}(\mathbf{T}_{zd}^\her\mathbf{T}_{zd})+\sqrt{\mathrm{tr}(\mathbf{T}_{zd}^\her\mathbf{T}_{zd})^2-4\mathrm{det}(\mathbf{T}_{zd}^\her\mathbf{T}_{zd})}}{2}\right)^{1/2},
$$
where
\begin{align*}
    &\mathrm{tr}(\mathbf{T}_{zd}^\her\mathbf{T}_{zd})=\frac{\epsilon^{2d+2}\omega^2+6\epsilon^{d+3}+2\epsilon^{d+3}\omega^2+4\epsilon^4+\omega^2+3\epsilon^{2d+2}+\epsilon^4\omega^2}{\omega^2-2\epsilon^{d+1}\omega^2+\omega^4+\epsilon^{2d+2}-2\epsilon^2\omega^2+\epsilon^4\omega^2},\\
    &\mathrm{det}(\mathbf{T}_{zd}^\her\mathbf{T}_{zd})=\frac{\epsilon^2(\epsilon+\epsilon^d)^2}{\omega^2-2\epsilon^{d+1}\omega^2+\omega^4+\epsilon^{2d+2}-2\epsilon^2\omega^2+\epsilon^4\omega^2}.
\end{align*}
% Alternatively, we can have the expression
% $$
% \sigma_\mathrm{max}(\mathbf{T}_{zd}(\mK_\epsilon^{(2)},j\omega))
% =\left(\frac{b+\sqrt{b^2-4ac}}{2a}\right)^{1/2},
% $$
% where
% \begin{align*}
%     a&=\omega^2-2\epsilon^{d+1}\omega^2+\omega^4+\epsilon^{2d+2}-2\epsilon^2\omega^2+\epsilon^4\omega^2\\
%     b&=\epsilon^{2d+2}\omega^2+6\epsilon^{d+3}+2\epsilon^{d+3}\omega^2+4\epsilon^4+\omega^2+3\epsilon^{2d+2}+\epsilon^4\omega^2,\\
%     c&=\epsilon^2(\epsilon+\epsilon^d)^2.
% \end{align*}
We find that it is tedious to evaluate the $\mathcal{H}_\infty$ norms $J_{\infty,1}(\mK_2(\epsilon)) =
\sup_{\omega\in\mathbb{R}}\, \sigma_\mathrm{max}(\mathbf{T}_{zd}(\mK_\epsilon^{(2)},j\omega))
$ analytically. Instead, our numerical simulation in \Cref{fig:Hinf_boundary_2_fitting} suggests that for $d \geq 3$, we have
\begin{equation} \label{eq:Hinf_boundary_2}
J_{\infty,1}(\mK_2(\epsilon)) = \sup_{\omega \in \mathbb{R}} \, \sigma_{\mathrm{max}}(\mathbf{T}_{zd}(j\omega,\mK_2(\epsilon))) \approx \frac{2}{\epsilon^{d-1}} \qquad \text{if}\;\; 0 < \epsilon \leq 10^{-1}.
\end{equation}
Thus, the corresponding $\mathcal{H}_\infty$ performance for \cref{eq:appendix_ex_hinf_b} will go to infinity as $\epsilon \to 0$ (i.e., the sequence of the policies goes to the boundary). Note that the dynamic policies in \cref{eq:appendix_ex_hinf_a} and \cref{eq:appendix_ex_hinf_b} go to the same boundary point when $\epsilon \to 0$.
\begin{figure}
    \centering
    \setlength{\abovecaptionskip}{3pt}
    \includegraphics[width = 0.46\textwidth]{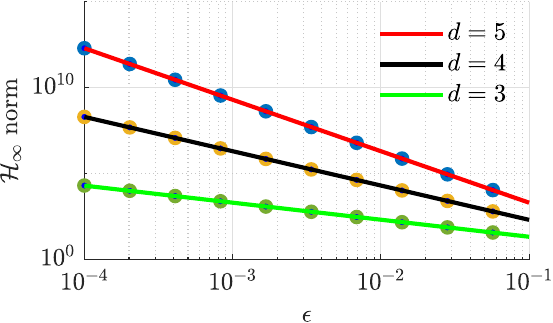}
    \caption{Numerical $\mathcal{H}_\infty$ norms for the policies \cref{eq:appendix_ex_hinf_b}: the dots are numerical values of $\mathcal{H}_\infty$ norms for each point $\epsilon$, and each line represents the least square fitting from those data points, which agree with \cref{eq:Hinf_boundary_2}.}
    \label{fig:Hinf_boundary_2_fitting}
\end{figure}

\section{Technical Proofs} \label{appendix:technical-proofs}

In this section, we provide some technical proofs that are missing in the main text.

\subsection{Auxiliary technical lemmas}
{
Define the set of stable matrices $\mathcal{S}_n = \{A \in \mathbb{R}^{n \times n}\mid \Re(\lambda_i(A)) < 0, i =1, \ldots, n\}$. It is not hard to see that $\partial\mathcal{S}_n$ consists of $n\times n$ real matrices whose set of eigenvalues has a nonempty intersection with the imaginary axis.

\begin{lemma} \label{lemma:Boundary-minimal-LTI-systems}
    Consider a sequence of stable LTI systems $\mathbf{G}_t = C_t(sI - A_t)^{-1}B_t + D_t$ with $A_t \in \mathcal{S}_n$, $B_t \in \mathbb{R}^{n \times m}$, $C_t \in \mathbb{R}^{p \times n}$ and $D_t\in\mathbb{R}^{p\times m}$. Suppose $\lim_{t \to \infty}(A_t, B_t,C_t,D_t) = (\hat{A}, \hat{B},\hat{C},\hat{D})$ with $\hat{A} \in \partial \mathcal{S}_n$, and
    the transfer matrix
    \[
    \hat{\mathbf{G}}(s) = \hat{C}(sI-\hat{A})^{-1}\hat{B} + \hat{D}
    \]
    has at least one pole that lies on the imaginary axis. Then
    \[
    \lim_{t\rightarrow\infty} \|\mathbf{G}_t\|_{\mathcal{H}_\infty} = +\infty.
    \]
    In addition, if $D_t=0$ for all $t=1,2,\ldots$, then
    \[
    \lim_{t\rightarrow\infty}
    \|\mathbf{G}_t\|_{\mathcal{H}_2}=+\infty.
    \]
\end{lemma}

\begin{proof}
Since one of the poles of the transfer matrix
$
\hat{\mathbf{G}}(s) = \hat{C}(sI-\hat{A})^{-1}\hat{B} + \hat{D}
$
lies on the imaginary axis, by using the fact that each pole of $\hat{\mathbf{G}}$ must also be a pole of some entry of $\hat{\mathbf{G}}$, we see that one of the entries of $\hat{\mathbf{G}}$, say $\hat{\mathbf{G}}_{ik}$, must have a pole on the imaginary axis. We can therefore write $\hat{\mathbf{G}}_{ik}$ in the following form:
\[
\hat{\mathbf{G}}_{ik}(s) = h(s)\prod_{\kappa=1}^K \frac{1}{(s-j\omega_\kappa)^{m_\kappa}},
\]
where $\omega_\kappa\in \mathbb{R}$ for $\kappa=1,\ldots,K$ are distinct, each $m_\kappa\geq 1$ is the multiplicity of $j\omega_\kappa$, and $h(s)$ is a rational function with no poles on the imaginary axis. Consequently,
\begin{equation}\label{eq:transfer_func_marginally_stable}
|\hat{\mathbf{G}}_{ik}(j\omega)|
=|h(j\omega)|\prod_{\kappa=1}^K \frac{1}{(\omega-\omega_\kappa)^{m_\kappa}}.
\end{equation}

Now let us consider the transfer function
\[
\mathbf{G}_{t,ik}(s)
=\left[C_t (sI-A_t)^{-1}B_t + D_{t} \right]_{ik},
\]
i.e., $\mathbf{G}_{t,ik}$ is the $(i,k)$'th entry of the transfer matrix $C_t (sI-A_t)^{-1}B_t + D_{t}$. Then it's not hard to verify by continuity that
\[
\lim_{t\rightarrow\infty} \mathbf{G}_{t,ik}(s) = \hat{\mathbf{G}}_{ik}(s)
\]
for every $s\in j\mathbb{R}$ that is not an eigenvalue of $A$. In other words, $\mathbf{G}_{t,ik}(j\omega)$ converges to $\hat{\mathbf{G}}_{ik}(j\omega)$ for almost every $\omega\in\mathbb{R}$.

We first show that $\lim_{t\rightarrow\infty} \|\mathbf{G}_t\|_{\mathcal{H}_\infty}=+\infty$. From the expression~\cref{eq:transfer_func_marginally_stable}, we can see that for any $M>0$, there exists $\omega_M\in\mathbb{R}$ such that $j\omega_M$ is not an eigenvalue of $A$ and $|\hat{\mathbf{G}}_{ik}(j\omega_M)|\geq 2M$. Then since $\lim_{t\rightarrow} \mathbf{G}_{t,ik}(j\omega_M)=\hat{\mathbf{G}}_{ik}(j\omega_M)$, we can find $T\in\mathbb{N}$ such that for all $t\geq T$, we have
\[
|\mathbf{G}_{t,ik}(j\omega_M)|\geq M.
\]
We also have
\[
\|\mathbf{G}_t\|_{\mathcal{H}_\infty}
\geq \sigma_{\max}(\mathbf{G}_{t}(j\omega_M))
\geq |\mathbf{G}_{t,ik}(j\omega_M)|.
\]
Summarizing the previous results, we see that for any $M>0$, there exists $T\in\mathbb{N}$ such that for all $t\geq T$, we have $\|\mathbf{G}_t\|_{\mathcal{H}_\infty}\geq M$. Therefore $\lim_{t\rightarrow\infty}\|\mathbf{G}_t\|_{\mathcal{H}_\infty}=+\infty$.

Now suppose $D_t=0$ for each $t$. To show that $\|\mathbf{G}_t\|_{\mathcal{H}_2}\rightarrow+\infty$, we note that
\begin{align*}
\|\mathbf{G}_t\|_{\mathcal{H}_2}^2
={} & \frac{1}{2\pi}\int_{-\infty}^{+\infty}
\operatorname{trace}(\mathbf{G}_t(j\omega)^\her\mathbf{G}_t(j\omega))\,d\omega
\geq{}
\frac{1}{2\pi}\int_{-\infty}^{+\infty}
|\mathbf{G}_{t,ik}(j\omega)|^2\,d\omega.
\end{align*}
Therefore
\begin{align*}
\liminf_{t\rightarrow\infty} \|\mathbf{G}_t\|_{\mathcal{H}_2}^2
\geq{} &
\frac{1}{2\pi}\liminf_{t\rightarrow\infty}
\int_{-\infty}^{+\infty} |\mathbf{G}_{t,ik}(j\omega)|^2\,d\omega \\
\geq{} &  \frac{1}{2\pi}\int_{-\infty}^{+\infty}
\liminf_{t\rightarrow\infty} |\mathbf{G}_{t,ik}(j\omega)|^2\,d\omega \\
={} & \frac{1}{2\pi}\int_{-\infty}^{+\infty}
|\hat{\mathbf{G}}_{ik}(j\omega)|^2\,d\omega \\
={} &
\frac{1}{2\pi}\int_{-\infty}^{+\infty}
|h(j\omega)|^2\prod_{\kappa=1}^K
\frac{1}{(\omega-\omega_\kappa)^{2m_\kappa}}\,d\omega
= +\infty,
\end{align*}
where in the second step we used Fatou's lemma~\cite[Theorem 1.17]{evans2015measure}. We can then conclude $\|\mathbf{G}_t\|_{\mathcal{H}_2}\rightarrow +\infty$.
\end{proof}
}

Combining \Cref{lemma:Boundary-minimal-LTI-systems} with \Cref{lemma:minimal-closed-loop-systems} and using the fact that the poles of the transfer matrix of a controllable and observable system $(\hat{A},\hat{B},\hat{C},\hat{D})$ are exactly the same as the eigenvalues of $\hat{A}$ (see e.g., \cite[Section 3.11]{zhou1996robust}), we get the following corollary:

\begin{corollary}\label{corollary:Boundary-minimal-LTI-systems}
Consider a sequence $(\mK_t)_{t=1}^\infty$ in $\mathcal{C}_n$. Suppose $\lim_{t \to \infty}\mK_t = \mK_\infty$ for some $\mK_\infty\in \partial\mathcal{C}_n$ that is controllable and observable, then
    \[
    \lim_{t\rightarrow\infty} J_{\infty,n}(\mK_t) = +\infty.
    \]
In addition, if $D_{\mK_t}=0$ for all $t=1,2,\ldots$, then
    \[
    \lim_{t\rightarrow\infty}
J_{\mathtt{LQG},n}(\mK_t)=+\infty.
    \]
\end{corollary}

\subsection{Proof of \Cref{proposition:informativity-non-degenerate}} \label{appendix:informative}

\paragraph{Part I. Sufficiency.}
We start our proof of sufficiency from the following fact about positive semidefinite matrices.
\begin{fact} \label{fact:PD-block-matrix}
    Given a positive semidefinite block matrix as
    $$
    \begin{bmatrix}
        A & B \\ B^\tr & C
    \end{bmatrix} \in \mathbb{S}^{n + n}_+,
    $$
    if $B \in \mathbb{R}^{n \times n}$ is full rank, then both $A$ and $C$ are strictly positive definite. %we have $A \in \mathbb{S}^n_{++}$ and $C \in \mathbb{S}^m_{++}$.
\end{fact}
\begin{proof}[Proof of \Cref{fact:PD-block-matrix}]
Suppose $A$ is only positive semidefinite with a zero eigenvalue. Let $x_1 \neq 0$ be the eigenvector, i.e., $Ax_1 = 0$. Then for all $x_2\in\mathbb{R}^n$, we have
\begin{equation} \label{eq:PD-block-matrix}
\begin{bmatrix}
    x_1 \\ x_2
\end{bmatrix}^\tr \begin{bmatrix}
        A & B \\ B^\tr & C
    \end{bmatrix}\begin{bmatrix}
    x_1 \\ x_2
\end{bmatrix} = 2 x_1^\tr B x_2 + x_2^\tr Cx_2 \geq 0.
\end{equation}
Since $B$ is full rank, we can find a vector $x_2$ such that $x_1^\tr B x_2 < 0$ (e.g., $x_2=-B^{-1}x_1/\|x_1\|$). The first term in \Cref{eq:PD-block-matrix} can be made arbitrarily below zero when we fix $x_2$ and scale $x_1$, making \Cref{eq:PD-block-matrix} invalid. Thus $A$ must be positive definite. Similarly, we can show $C$ is strictly positive definite. Note that, we cannot guarantee the whole matrix is positive definite, for example
$
\begin{bmatrix}
        1 & 1\\ 1& 1
    \end{bmatrix}.
$
\end{proof}

We are now ready to prove the sufficiency part of \Cref{proposition:informativity-non-degenerate}. As alluded to in the main text, the proof follows three steps:
\begin{enumerate}
    \item The full rank of $X_{\mK,12}$ implies that $(A_\mK,B_{\mK})$ is controllable, and thus $X_{\mK}$ from~\cref{eq:LyapunovX} is positive definite.
    \item $X_{\mK}$ is thus invertible, and we construct $P = \gamma X_{\mK}^{-1}$ with $\gamma = J_{\texttt{LQG},n}(\mK)$. We show this construction satisfies \Cref{eq:LQG-Bilinear-a,eq:LQG-Bilinear-b}.  %%that $X_{\mK}$ is positive definite, and thus it is invertible (seems not true, consider [1 1;1 1]. If $\mK$ is minimal then $X_\mK,Y_\mK\succ0$?).
\item Further $P_{12} =\gamma(X_{\mK}^{-1})_{12}$ is invertible since $X_{\mK,12}$ is.
\end{enumerate}
The three points above confirm that $\mK \in {\mathcal{C}}_{\mathrm{nd},0}$ by the definition \Cref{def:LQG-Cnd}. We here provide the details.

\begin{itemize}
    \item For point 1), let $\mK \in {\mathcal{C}}_{n,0}$ be an informative policy, and $X_\mK$ be the corresponding solution to \Cref{eq:LyapunovX}. Then, $X_\mK$ is positive semidefinite by \Cref{lemma:Lyapunov}. Upon partitioning $X_\mK$ as \cref{eq:co-variance}, {since $X_{\mK,12}$ is full rank and $X_{\mK} \succeq 0$, we must have $X_{\mK,22} \succ 0$\footnote{Note that we still cannot guarantee the whole matrix $X_{\mK} \succ 0$ at this stage. %Consider $A_{\mK} = %\begin{bmatrix}
    %    1 & 1\\ 1& 1
    %\end{bmatrix}$.
    } (cf. \Cref{fact:PD-block-matrix}).}

    Meanwhile, the bottom-right block of the equation \Cref{eq:LyapunovX} reads as
    $$
    B_\mK C X_{\mK,12} + A_{\mK} X_{\mK,22} +(B_\mK C X_{\mK,12} + A_{\mK} X_{\mK,22})^\tr + B_\mK V B_\mK^\tr = 0,
    $$
    which is the same as
    $$
    (A_{\mK} + B_\mK C X_{\mK,12}X_{\mK,22}^{-1})X_{\mK,22} + X_{\mK,22}(A_{\mK} + B_\mK C X_{\mK,12}X_{\mK,22}^{-1})^\tr + B_\mK V B_\mK^\tr = 0.
    $$
    By \Cref{lemma:converse-Lyapunov}, we know $A_{\mK} + B_\mK C X_{\mK,12}X_{\mK,22}^{-1}$ is stable. Since $ X_{\mK,22} \succ 0$,  \Cref{lemma:Lyapunov} further ensures that
    $$
    (A_{\mK} + B_\mK C X_{\mK,12}X_{\mK,22}^{-1}, B_\mK V^{1/2})
    $$
    is controllable, which is the same as $(A_\mK, B_\mK)$ being controllable (since controllability is invariant with state feedback; cf. \Cref{lemma:controllability}). Therefore, the solution $X_{\mK}$ to \Cref{eq:LyapunovX} is strictly positive definite, as confirmed in \cite[Lemma 4.5]{zheng2021analysis}.
    \item For point 2), since $X_\mK \succ 0$, the construction $P = \gamma X_{\mK}^{-1}$ with $\gamma = J_{\texttt{LQG},n}(\mK)$ is well-defined. Following the exactly same argument in \Cref{appendix:proof-H2-lemma} (especially the logic between \cref{eq:nonstrict-LMI-h2-norm} and \cref{eq:construction-X-L0}), this construction satisfies \Cref{eq:LQG-Bilinear-a,eq:LQG-Bilinear-b}.
    \item For point 3), we have the following observation
    \[PX_\mK=\begin{bmatrix}
    P_{11} & P_{12} \\ P_{12}^\tr & P_{22}
\end{bmatrix}\begin{bmatrix}
    X_{\mK,11} & X_{\mK,12} \\ X_{\mK,12}^\tr & X_{\mK,22}
\end{bmatrix}=\begin{bmatrix}
    \gamma I & 0 \\ 0 & \gamma I
\end{bmatrix} \quad \Rightarrow \quad P_{11}X_{\mK,12}+P_{12}X_{\mK,22}=0.\]
We know $P_{11}\succ0,X_{\mK,22}\succ0$ since $P\succ0,X_\mK\succ0$. Then $P_{12}$ being invertible follows from the fact that $X_{\mK,12}$ is invertible.
\end{itemize}
This completes our proof of the sufficiency part.

\paragraph{Part II. Necessity.} Suppose $\mK\in{\mathcal{C}}_{\mathrm{nd},0}$, and let $\gamma=J_{\LQG,n}(\mK)>0$. Then there exist $P\succ 0$ with $P_{12}\in \mathrm{GL}_n$ and $\Gamma\succeq 0$ such that
\[
\begin{bmatrix}
A_{\mathrm{cl}}(\mK)^\tr P + PA_{\mathrm{cl}}(\mK) & PB_{\mathrm{cl}}(\mK) \\
B_{\mathrm{cl}}(\mK)^\tr P & -\gamma I
\end{bmatrix}\preceq 0,
\quad
\begin{bmatrix}
P & C_{\mathrm{cl}}(\mK)^\tr \\
C_{\mathrm{cl}}(\mK) & \Gamma
\end{bmatrix}\succeq 0,
\qquad \operatorname{tr}(\Gamma)\leq \gamma.
\]
Let $\hat{X} = \gamma P^{-1}$. It's not hard to see that $P_{12}\in \mathrm{GL}_n$ implies $\hat{X}_{12}\in \mathrm{GL}_n$ (see point 3) in the proof of the sufficiency part). By \Cref{lemma:H2-norm-non-strict}, we also have
\[
A_{\mathrm{cl}}(\mK) \hat{X}
+ \hat{X} A_{\mathrm{cl}}(\mK)^\tr
+ B_{\mathrm{cl}}(\mK)B_{\mathrm{cl}}(\mK)^\tr \preceq 0
\]
and
\[
\operatorname{tr}(C_{\mathrm{cl}}(\mK) \hat{X} C_{\mathrm{cl}}(\mK)^\tr) \leq \gamma^2.
\]
As a result, there exists some $\tilde{B} \in \mathbb{R}^{{2n}\times k}$ such that
\[
A_{\mathrm{cl}}(\mK) \hat{X}
+ \hat{X} A_{\mathrm{cl}}(\mK)^\tr
+ B_{\mathrm{cl}}(\mK)B_{\mathrm{cl}}(\mK)^\tr
+ \tilde{B}\tilde{B}^\tr = 0.
\]
On the other hand, note that $X_\mK$ is the solution to the Lyapunov equation
\[
A_{\mathrm{cl}}(\mK) X_\mK
+ X_\mK A_{\mathrm{cl}}(\mK)^\tr
+ B_{\mathrm{cl}}(\mK)B_{\mathrm{cl}}(\mK) = 0,
\]
and satisfies $\operatorname{tr}(C_{\mathrm{cl}}(\mK)X_\mK C_{\mathrm{cl}}(\mK)^\tr ) = \gamma^2$. Now let $\tilde{X} = \hat{X}-X_\mK$, and we have
\[
A_{\mathrm{cl}}(\mK) \tilde{X}
+\tilde{X}A_{\mathrm{cl}}(\mK)
+\tilde{B}\tilde{B}^\tr=0,
\]
showing that $\tilde{X}\succ 0$ or $\hat{X}\succeq X_\mK$, while
\[
\gamma^2
=\operatorname{tr}(C_{\mathrm{cl}}(\mK)X_\mK C_{\mathrm{cl}}(\mK)^\tr )
\leq
\operatorname{tr}(C_{\mathrm{cl}}(\mK)\hat{X} C_{\mathrm{cl}}(\mK)^\tr )
\leq \gamma^2,
\]
showing that
$\operatorname{tr}(C_{\mathrm{cl}}(\mK)X_\mK C_{\mathrm{cl}}(\mK)^\tr )=\operatorname{tr}(C_{\mathrm{cl}}(\mK)\hat{X} C_{\mathrm{cl}}(\mK)^\tr )=\gamma^2$. Consequently,
\[
\operatorname{tr}(C_{\mathrm{cl}}(\mK)\tilde{X}C_{\mathrm{cl}}(\mK)^\tr) = 0.
\]
By plugging $\tilde{X} = \int_0^{+\infty} e^{A_{\mathrm{cl}}(\mK)\tau} \tilde{B}\tilde{B}^\tr e^{A_{\mathrm{cl}}(\mK)^\tr\tau}\,d\tau$ into the above equality, we get
\begin{align*}
0={} & \int_0^{+\infty}
\operatorname{tr}
\left(C_{\mathrm{cl}}(\mK)e^{A_{\mathrm{cl}}(\mK)\tau}\tilde{B}\tilde{B}^\tr
e^{A_{\mathrm{cl}}(\mK)^\tr\tau}C_{\mathrm{cl}}(\mK)^\tr\right)\,d\tau \\
={} &
\int_0^{+\infty}
\left\|C_{\mathrm{cl}}(\mK)e^{A_{\mathrm{cl}}(\mK)\tau}\tilde{B}\right\|_F^2\,d\tau
\end{align*}
which further gives
\[
C_{\mathrm{cl}}(\mK)e^{A_{\mathrm{cl}}(\mK)\tau}\tilde{B} = 0
\]
for all $\tau\in[0,+\infty)$.

The following lemma is essential for our subsequent derivation.
\begin{lemma}\label{lemma:informativity_intermediate}
Suppose $(Q^{1/2},A)$ is observable, and suppose $\tilde{B}\in\mathbb{R}^{{2n}\times k}$ satisfies
\[
C_{\mathrm{cl}}(\mK) e^{A_{\mathrm{cl}}(\mK)\tau}\tilde{B}=0
\]
for all $\tau\in[0,+\infty)$. Let $\tilde{X}\in\mathbb{R}^{{2n}\times {2n}}$ be the solution to the Lyapunov equation
\[
A_{\mathrm{cl}}(\mK)X+XA_{\mathrm{cl}}(\mK)^\tr + \tilde{B}\tilde{B}^\tr = 0.
\]
Then $\tilde{X}$ has the form
\[
\tilde{X} = \begin{bmatrix}
0_{n\times n} & 0_{n\times n} \\
0_{n\times n} & \tilde{X}_{22}
\end{bmatrix}
\]
for some $\tilde{X}_{22}\in\mathbb{R}^{n\times n}$.
\end{lemma}
\begin{proof}
By applying Laplace transform to $C_{\mathrm{cl}}(\mK) e^{A_{\mathrm{cl}}(\mK)\tau} \tilde{B} = 0$, we obtain
\begin{equation}
C_{\mathrm{cl}}(\mK) (sI-A_{\mathrm{cl}}(\mK))^{-1}\tilde{B} = 0
\end{equation}
for all $s\in\mathbb{C}$ with $\operatorname{Re}[s]\geq 0$. To further derive implications from this equality, we decompose $\tilde{B}$ as
\[
\tilde{B} = \begin{bmatrix}
\tilde{B}_1 \\ \tilde{B}_2
\end{bmatrix},
\qquad \text{where }\tilde{B}_1, \tilde{B}_2\in\mathbb{R}^{n\times k}.
\]
Note that $A_{\mathrm{cl}}(\mK) = \begin{bmatrix}
A & BC_\mK \\ B_\mK C & A_\mK
\end{bmatrix}$, which gives
\begin{align}
& (sI-A_{\mathrm{cl}}(\mK))^{-1}\tilde{B}
=
\begin{bmatrix}
sI-A & -BC_\mK \\ -B_\mK C & sI-A_\mK
\end{bmatrix}^{-1}\begin{bmatrix}
\tilde{B}_1 \\ \tilde{B}_2
\end{bmatrix} \nonumber \\
={} &
\left(
\begin{bmatrix}
I & 0 \\
-B_\mK C(sI\!-\!A)^{-1} & I
\end{bmatrix}
\!
\begin{bmatrix}
sI\!-\!A & 0 \\
0 & sI\!-\!A_\mK\!-\!B_\mK \mathbf{P}(s) C_\mK
\end{bmatrix}
\!
\begin{bmatrix}
I & -(sI\!-\!A)^{-1}BC_\mK \\ 0 & I
\end{bmatrix}
\right)^{\!\!-1}
\!\!\begin{bmatrix}
\tilde{B}_1 \\ \tilde{B}_2
\end{bmatrix} \nonumber \\
={} &
\begin{bmatrix}
I & (sI-A)^{-1}BC_\mK \\
0 & I
\end{bmatrix}
\begin{bmatrix}
sI\!-\!A & 0 \\
0 & sI\!-\!A_\mK\!-\!B_\mK \mathbf{P}(s) C_\mK
\end{bmatrix}^{-1}
\begin{bmatrix}
I & 0 \\
B_\mK C(sI\!-\!A)^{-1} & I
\end{bmatrix}
\begin{bmatrix}
\tilde{B}_1 \\ \tilde{B}_2
\end{bmatrix} \nonumber \\
={} &
\begin{bmatrix}
I & (sI-A)^{-1}BC_\mK \\
0 & I
\end{bmatrix}
\begin{bmatrix}
(sI-A)^{-1}\tilde{B}_1 \\
(sI-A_\mK-B_\mK \mathbf{P}(s) C_\mK)^{-1}
(B_\mK C(sI-A)^{-1}\tilde{B}_1+\tilde{B}_2)
\end{bmatrix}
\label{eq:proof_informative_temp2}
\end{align}
where we denote $\mathbf{P}(s)=C(sI-A)^{-1}B$. Consequently, by using $C_{\mathrm{cl}}(\mK) = \begin{bmatrix}
Q^{1/2} & 0 \\ 0 & R^{1/2}C_\mK
\end{bmatrix}$, we get
\begin{align}
0 ={} & C_{\mathrm{cl}}(\mK)(sI-A_{\mathrm{cl}}(\mK))^{-1}\tilde{B} \nonumber \\
={} &
\begin{bmatrix}
Q^{1/2} & 0 \\ 0 & R^{1/2}C_\mK
\end{bmatrix}
\begin{bmatrix}
I & (sI\!-\!A)^{-1}BC_\mK \\
0 & I
\end{bmatrix}
\begin{bmatrix}
(sI-A)^{-1}\tilde{B}_1 \\
(sI\!-\!A_\mK\!-\!B_\mK \mathbf{P}(s) C_\mK)^{-1}
(B_\mK C(sI\!-\!A)^{-1}\tilde{B}_1+\tilde{B}_2)
\end{bmatrix} \nonumber \\
={} &
\begin{bmatrix}
Q^{1/2} & Q^{1/2}(sI\!-\!A)^{-1}BC_\mK \\
0 & R^{1/2}C_\mK
\end{bmatrix}
\begin{bmatrix}
(sI-A)^{-1}\tilde{B}_1 \\
(sI\!-\!A_\mK\!-\!B_\mK \mathbf{P}(s) C_\mK)^{-1}
(B_\mK C(sI\!-\!A)^{-1}\tilde{B}_1+\tilde{B}_2)
\end{bmatrix} \nonumber \\
={} &
\begin{bmatrix}
Q^{1/2} & Q^{1/2}(sI\!-\!A)^{-1}B \\
0 & R^{1/2}
\end{bmatrix}
\begin{bmatrix}
(sI-A)^{-1}\tilde{B}_1 \\
C_\mK(sI\!-\!A_\mK\!-\!B_\mK \mathbf{P}(s) C_\mK)^{-1}
(B_\mK C(sI\!-\!A)^{-1}\tilde{B}_1+\tilde{B}_2)
\end{bmatrix}
\label{eq:proof_informative_temp1}
\end{align}
By inspecting the bottom $m\times k$ block of the right-hand side of~\Cref{eq:proof_informative_temp1} and using $R\succ 0$, we get
\[
0 = C_\mK(sI\!-\!A_\mK\!-\!B_\mK \mathbf{P}(s) C_\mK)^{-1}
(B_\mK C(sI\!-\!A)^{-1}\tilde{B}_1+\tilde{B}_2),
\]
and by plugging it back into~\Cref{eq:proof_informative_temp1} and calculating its top $n\times k$ block, we get
\[
0 = Q^{1/2}(sI-A)^{-1}\tilde{B}_1.
\]
We now use the assumption that $(Q^{1/2},A)$ is observable, which implies $\tilde{B}_1=0$. Furthermore,
\[
0 = C_\mK(sI\!-\!A_\mK\!-\!B_\mK \mathbf{P}(s) C_\mK)^{-1} \tilde{B}_2.
\]
We can now simplify~\Cref{eq:proof_informative_temp2} as
\begin{align*}
(sI-A_{\mathrm{cl}}(\mK))^{-1}\tilde{B}
={} & \begin{bmatrix}
I & (sI-A)^{-1}BC_\mK \\ 0 & I
\end{bmatrix}
\begin{bmatrix}
0 \\
(sI\!-\!A_\mK\!-\!B_\mK \mathbf{P}(s) C_\mK)^{-1} \tilde{B}_2
\end{bmatrix} \\
={} &
\begin{bmatrix}
(sI-A)^{-1}B\cdot C_\mK(sI\!-\!A_\mK\!-\!B_\mK \mathbf{P}(s) C_\mK)^{-1} \tilde{B}_2 \\
(sI\!-\!A_\mK\!-\!B_\mK \mathbf{P}(s) C_\mK)^{-1} \tilde{B}_2
\end{bmatrix} \\
={} &
\begin{bmatrix}
0 \\
(sI\!-\!A_\mK\!-\!B_\mK \mathbf{P}(s) C_\mK)^{-1} \tilde{B}_2
\end{bmatrix}.
\end{align*}
By applying the inverse Laplace transform, we see that
\[
e^{A_{\mathrm{cl}}(\mK)\tau}\tilde{B}
=\begin{bmatrix}
0_{n\times k} \\ *
\end{bmatrix},
\qquad\forall \tau\in[0,+\infty).
\]
Consequently, the solution to the Lyapunov equation $A_{\mathrm{cl}}(\mK)X+XA_{\mathrm{cl}}(\mK)^\tr + \tilde{B}\tilde{B}^\tr = 0$ is given by
\[
\int_0^{+\infty} e^{A_{\mathrm{cl}}(\mK)\tau}\tilde{B}\tilde{B}^\tr
e^{A_{\mathrm{cl}}(\mK)^\tr\tau}\,d\tau
=
\begin{bmatrix}
0_{n\times n} & 0_{n\times n} \\
0_{n\times n} & *
\end{bmatrix},
\]
which completes the proof of \Cref{lemma:informativity_intermediate}.
\end{proof}

We can now apply \Cref{lemma:informativity_intermediate} and see that
\[
X_\mK = \hat{X} - \tilde{X}
=\begin{bmatrix}
\hat{X}_{11} & \hat{X}_{12} \\
\hat{X}_{12}^\tr & \hat{X}_{22}-\tilde{X}_{22}
\end{bmatrix}.
\]
The invertibility of $\hat{X}_{12}$ then implies the invertibility of $X_{\mK,12}$, which completes the proof.

\subsection{Proof of \Cref{proposition:measure-zero-LQG}} \label{appendix:measure-zero-LQG}

%The idea of the proof is to construct a set $\mathcal{N}\subseteq \hat{\mathcal{C}}_{n}$ such that $\mathcal{N}$ has measure zero and ${\mathcal{C}}_{n,0}\backslash\mathcal{N}\subseteq {\mathcal{C}}_{\mathrm{nd},0}$. Then we have ${\mathcal{C}}_{n,0}\backslash {\mathcal{C}}_{\mathrm{nd},0}\subseteq \mathcal{N}$, and the conclusion follows by noting that any subset of $\mathcal{N}$ also has measure zero.

%The construction of the set $\mathcal{N}$ relies on the following lemma on the zeros of real analytic functions.

We first present the following lemma on the zeros of real analytic functions.

\begin{lemma}[\cite{mityagin2015zero}] \label{lemma:measure-zero-analytical}
Let $U\subseteq\mathbb{R}^d$ be an open connected set, and let $f:U\rightarrow\mathbb{R}$ be a real analytic function. If $f$ is not identically zero on $U$, then the zero set
\[
\{x\in U \mid f(x)=0\}
\]
has measure zero.
\end{lemma}

Next, for each $\mathcal{C}_{n,0}$, let $X_\mK$ be the solution to~\eqref{eq:LyapunovX}, i.e.,
\[
A_{\mathrm{cl}}(\mK) X_\mK
+ X_\mK A_{\mathrm{cl}}(\mK)^\tr
+ B_{\mathrm{cl}}(\mK)B_{\mathrm{cl}}(\mK)^\tr = 0.
\]
By the results presented in \Cref{appendix:preliminaries}, we have $X_\mK\succeq 0$ and each entry of $X_\mK$ is a real analytic function of $\mK$. Now, we consider the set
\[
\mathcal{N} = \{\mK\in \mathcal{C}_{n,0} \mid \operatorname{det}(X_{\mK,12})=0\},
\]
where $X_{\mK,12}$ denotes the submatrix of $X_\mK\in \mathbb{R}^{2n\times 2n}$ corresponding to the first $n$ rows and the last $n$ columns. It is not hard to see that $\mK\mapsto \operatorname{det}(X_{\mK,12})$ is a real analytic function on $\mathcal{C}_{n,0}$ which is not identically zero. Considering that the set $\mathcal{C}_{n,0}$ has at most two connected components, we can apply \Cref{lemma:measure-zero-analytical} to each connected component separately and then use the fact that a finite union of measure-zero sets still has measure zero. Consequently, the set $\mathcal{N}$ has measure zero.

On the other hand, by \Cref{proposition:informativity-non-degenerate}, we see that the complement of $\mathcal{N}$ in $\mathcal{C}_{n,0}$ is just the set $\mathcal{C}_{\mathrm{nd},0}$. The proof is now complete.

\subsection{Proof of \Cref{lemma:LQG-lower-order-stationary-point-to-high-order}} \label{appendix:degenerate-controllers}

% We here provide a few more basic properties about (non)-degenerate controllers.

% \subsubsection{LQG control}
The set of non-degenerate policies is invariant with respect to similarity transformations. For any fixed $T\in\mathrm{GL}_n$, recall that $\mathscr{T}_T:\mathcal{C}_n\rightarrow\mathcal{C}_n$ denotes the mapping given by
\begin{equation}
\mathscr{T}_T(\mK)
\coloneqq
 \begin{bmatrix}
I_m & 0 \\
0 & T
\end{bmatrix}\begin{bmatrix}
D_{\mK} & C_{\mK} \\
B_{\mK} & A_{\mK}
\end{bmatrix}\begin{bmatrix}
I_p & 0 \\
0 & T
\end{bmatrix}^{-1}
=\begin{bmatrix}
D_{\mK} & C_{\mK}T^{-1} \\
TB_{\mK} & TA_{\mK}T^{-1}
\end{bmatrix}.
\tag{\ref{eq:similarity-transformation}}
\end{equation}

\begin{lemma} \label{lemma:invariance-non-degenerate-controllers}
    Fix $T \in \mathrm{GL}_n$. For all $\mK \in {\mathcal{C}}_{\mathrm{nd},0}$, we have $\mathscr{T}_T(\mK) \in {\mathcal{C}}_{\mathrm{nd},0}$.
\end{lemma}

\begin{proof}
Let $\mK \in {\mathcal{C}}_{\mathrm{nd},0}$ and $\tilde{\mK} = \mathscr{T}_T(\mK)$. Clearly, $J_{\texttt{LQG},n}(\tilde{\mK}) = J_{\texttt{LQG},n}(\mK)$. By definition of ${\mathcal{C}}_{\mathrm{nd},0}$, there exist $P \in \mathbb{S}^{2n}_{++}\; \mathrm{with} \; P_{12} \in \mathrm{GL}_{n}$ and $\Gamma$ such that  \Cref{eq:LQG-Bilinear-a} and \Cref{eq:LQG-Bilinear-b} hold with $\gamma = J_{\texttt{LQG},n}(\mK)$.

It is straightforward to verify that
\begin{align} \label{eq:invar_cl_matrices}
    A_{\mathrm{cl}}(\tilde{\mK}) = \begin{bmatrix}
    I_n & 0 \\
    0 & T
    \end{bmatrix}A_{\mathrm{cl}}({\mK})\begin{bmatrix}
    I_n & 0 \\
    0 & T
    \end{bmatrix}^{-1},\quad
    B_{\mathrm{cl}}(\tilde{\mK}) = \begin{bmatrix}
    I_n & 0 \\
    0 & T
    \end{bmatrix} B_{\mathrm{cl}}(\mK),\quad
    C_{\mathrm{cl}}(\tilde{\mK}) = C_{\mathrm{cl}}(\mK)\begin{bmatrix}
    I_n & 0 \\
    0 & T
    \end{bmatrix}^{-1}.
\end{align}
Then, we can check that the following matrices
\begin{equation} \label{eq:invar_tilde_P}
    \tilde{P} = \begin{bmatrix}
    I_n & 0 \\
    0 & T
    \end{bmatrix}^{-\tr} P \begin{bmatrix}
    I_n & 0 \\
    0 & T
    \end{bmatrix}^{-1} \in \mathbb{S}^{2n}_{++}, \qquad \tilde{\Gamma} = \Gamma
\end{equation}
with $\tilde{P}_{12} \in \mathrm{GL}_n$ satisfy  \Cref{eq:LQG-Bilinear-a} and \Cref{eq:LQG-Bilinear-b} with $\gamma = J_{\texttt{LQG},n}(\tilde{\mK})$. Thus, $\tilde{\mK} \in {\mathcal{C}}_{\mathrm{nd},0}$.
%
% \begin{subequations} \label{eq:LQG-Bilinear}
%     \begin{align}
%         \min_{\mK, \gamma, P, \Gamma} \quad & \gamma \nonumber \\
%         \mathrm{subject to} \quad & \begin{bmatrix} A_{\mathrm{cl}}(\mK)^\tr P+PA_{\mathrm{cl}}(\mK) & PB_{\mathrm{cl}}(\mK) \\ B_{\mathrm{cl}}(\mK)^\tr P & -\gamma I \end{bmatrix} \preceq 0, \label{eq:LQG-Bilinear-a}\\
%         &\begin{bmatrix} P & C_{\mathrm{cl}}(\mK)^\tr \\ C_{\mathrm{cl}}(\mK) & \Gamma \end{bmatrix} \succeq 0,\;  P \succ 0, \; \mathrm{trace}(\Gamma) \leq \gamma, \;D_{\mK} = 0,  \label{eq:LQG-Bilinear-b}
%     \end{align}
% \end{subequations}
%
\end{proof}

\begin{lemma} \label{lemma:Hinf_invariance-non-degenerate-controllers}
    Fix $T \in \mathrm{GL}_n$. For all $\mK \in {\mathcal{C}}_{\mathrm{nd}}$, we have $\mathscr{T}_T(\mK) \in {\mathcal{C}}_{\mathrm{nd}}$.
\end{lemma}

\begin{proof}
The proof is similar to that of \Cref{lemma:invariance-non-degenerate-controllers}. Let $\mK \in {\mathcal{C}}_{\mathrm{nd}}$ and $\tilde{\mK} = \mathscr{T}_T(\mK)$. Clearly, $J_{\infty,n}(\tilde{\mK}) = J_{\infty,n}(\mK)$. By definition of ${\mathcal{C}}_{\mathrm{nd}}$, there exists $P \in \mathbb{S}^{2n}_{++}\; \mathrm{with} \; P_{12} \in \mathrm{GL}_{n}$ such that \Cref{eq:Hinf-Bilinear-a} holds with $\gamma = J_{\infty,n}(\mK)$. For the closed-loop matrices, we have $A_{\mathrm{cl}}(\tilde{\mK}), B_{\mathrm{cl}}(\tilde{\mK}), C_{\mathrm{cl}}(\tilde{\mK})$ as in \Cref{eq:invar_cl_matrices} and $D_{\mathrm{cl}}(\tilde{\mK}) = D_{\mathrm{cl}}(\mK)$. One can check that $\tilde{P}$ as in \Cref{eq:invar_tilde_P} satisfies \Cref{eq:Hinf-Bilinear-a} with $\gamma = J_{\infty,n}(\tilde{\mK})$ and hence $\tilde{\mK} \in {\mathcal{C}}_{\mathrm{nd}}$.
\end{proof}

We next identify a class of degenerate policies, which also finishes the proof of \Cref{lemma:LQG-lower-order-stationary-point-to-high-order}.
\begin{theorem} \label{theorem:degenerate-controllers-LQG}
    Consider a reduced-order policy $\hat{\mK} = \begin{bmatrix}
0 & \hat{C}_{\mK} \\ \hat{B}_{\mK} & \hat{A}_{\mK}
\end{bmatrix} \in {\mathcal{C}}_{q,0}$ where $1 \leq q < n$. Suppose $\hat{\mK}\in {\mathcal{C}}_{q,0}$ is minimal. %If $(A_\mK, B_\mK)$  is not controllable and $(C_\mK, A_\mK)$ is not observable, then this controller $\mK$ is degenerate.   {\color{red} To prove}
Then the following policy with any stable matrix $\Lambda \in \mathbb{R}^{(n-q) \times (n-q)}$
\begin{equation} \label{eq:stationary_nonglobally_K-augmented}
        \tilde{\mK}=\left[\begin{array}{c:cc}
    0 & \hat{C}_{\mK} &  0 \\[2pt]
    \hdashline
    \hat{B}_{\mK} & \hat{A}_{\mK} & 0 \\[-2pt]
    0 & 0 & \Lambda
    \end{array}\right] \in \hat{\mathcal{C}}_{n}
\end{equation}
is internally stabilizing but degenerate, i.e., $\tilde{\mK} \in {\mathcal{C}}_{n,0}\backslash{\mathcal{C}}_{\mathrm{nd},0}$.
\end{theorem}
\begin{proof}
    It is easy to see that $\tilde{\mK}$ is an internally stabilizing policy. It remains to check that all feasible solutions $P$ to the LMI \Cref{eq:LQG-Bilinear-a}  and \Cref{eq:LQG-Bilinear-b}  with $\gamma = J_{\texttt{LQG},n}(\tilde{\mK})$ have low-rank off-diagonal block $P_{12}$.

The proof is divided into three steps.

\begin{itemize}
    \item \textbf{Step 1} (another equivalent LMI): We know that \Cref{eq:LQG-Bilinear-a}  and \Cref{eq:LQG-Bilinear-b}  are feasible with $P\succ 0$, $\Gamma$, $\gamma = J_{\texttt{LQG},n}(\tilde{\mK})$ and $P_{12} \in \mathrm{GL}_n$ if and only if there exists an $X \in \mathbb{S}^{2n}_{++}$ with $X_{12} \in \mathrm{GL}_n$ such that
    \begin{equation} \label{eq:degenerate-Lyapunov-1}
    A_{\mathrm{cl}}(\tilde{\mK}) X+XA_{\mathrm{cl}}(\tilde{\mK})^\tr+B_{\mathrm{cl}}(\tilde{\mK}) B_{\mathrm{cl}}(\tilde{\mK})^\tr \! \preceq 0,\ \mathrm{tr}(C_{\mathrm{cl}}(\tilde{\mK}) XC_{\mathrm{cl}}(\tilde{\mK})^\tr)\leq \gamma^2.
    \end{equation}
    \item \textbf{Step 2} (unique solution $\hat{X}$ for the reduced-order policy $\hat{\mK}$): It is clear that $J_{\texttt{LQG},q}(\hat{\mK}) = J_{\texttt{LQG},n}(\tilde{\mK})$. By \Cref{lemma:minimal-closed-loop-systems}, the closed-loop system when applying $\hat{\mK}$ is minimal. Thus, \Cref{lemma:unique-solution-H2} ensures that the following LMI
    \begin{equation} \label{eq:degenerate-Lyapunov-2}
    A_{\mathrm{cl}}(\hat{\mK}) \hat{X}+\hat{X}A_{\mathrm{cl}}(\hat{\mK})^\tr+B_{\mathrm{cl}}(\hat{\mK}) B_{\mathrm{cl}}(\hat{\mK})^\tr \! \preceq 0,\ \mathrm{tr}(C_{\mathrm{cl}}(\hat{\mK}) \hat{X}C_{\mathrm{cl}}(\hat{\mK})^\tr)\leq \gamma^2,
    \end{equation}
    has a unique solution $\hat{X} \in \mathbb{S}^{n+q}_{++}$ and holds with equalities, i.e.,  \Cref{eq:degenerate-Lyapunov-2} becomes a Lyapunov equation.
    \item \textbf{Step 3} (the off-diagonal block $X_{12}$ is low rank): Next, let us expand the closed-loop matrices \cref{eq:closed-loop-matrices} for the full-order policy $\tilde{\mK}$, which reads as
    \begin{equation} \label{eq:closed-loop-matrix-big-K}
\begin{aligned}
A_{\mathrm{cl}}(\tilde{\mK})
\coloneqq\  &
\begin{bmatrix}
A & B\hat{C}_\mK & 0 \\ \hat{B}_\mK C & \hat{A}_\mK & 0 \\ 0 & 0 & \Lambda
\end{bmatrix} = \begin{bmatrix}
A_{\mathrm{cl}}(\hat{\mK}) & 0 \\  0 & \Lambda
\end{bmatrix}  , \\
B_{\mathrm{cl}}(\tilde{\mK})
\coloneqq\  &\begin{bmatrix} I & 0 \\ 0 & \hat{B}_{\mK} \\ 0 & 0  \end{bmatrix} = \begin{bmatrix} B_{\mathrm{cl}}(\hat{\mK}) \\ 0   \end{bmatrix}, \\
C_{\mathrm{cl}}(\tilde{\mK})
\coloneqq\  &
 \begin{bmatrix}
        Q^{1/2} & 0 &0\\
        & R^{1/2}\hat{C}_{\mK} & 0
    \end{bmatrix} = \begin{bmatrix}
       C_{\mathrm{cl}}(\hat{\mK})   & 0
    \end{bmatrix}.
\end{aligned}
    \end{equation}
We partition the matrix $X$ in \cref{eq:degenerate-Lyapunov-1} as
$$
X = \begin{bmatrix}
    \hat{X} & X_{12} \\ X_{12}^\tr & X_{22}
\end{bmatrix}, \quad \text{with} \;\hat{X} \in \mathbb{S}^{n+q}, \; X_{22} \in \mathbb{S}^{n-q}, \; X_{12} \in \mathbb{R}^{(n+q) \times (n-q)}.
$$
Some simple algebra finds that \cref{eq:degenerate-Lyapunov-1} is the same as
$$
\begin{aligned}
\begin{bmatrix}
    A_{\mathrm{cl}}(\hat{\mK}) \hat{X}+\hat{X}A_{\mathrm{cl}}(\hat{\mK})^\tr+B_{\mathrm{cl}}(\hat{\mK}) B_{\mathrm{cl}}(\hat{\mK})^\tr & A_{\mathrm{cl}}(\hat{\mK})X_{12} + X_{12}\Lambda^\tr \\ \Lambda X_{12}^\tr + X_{12}^\tr A_{\mathrm{cl}}(\hat{\mK})^\tr & \Lambda X_{22} + X_{22} \Lambda^\tr
\end{bmatrix} \preceq 0, \quad
\mathrm{tr}(C_{\mathrm{cl}}(\hat{\mK}) \hat{X}C_{\mathrm{cl}}(\hat{\mK})^\tr)\leq \gamma^2.
\end{aligned}
$$
The previous \textbf{Step 2} has confirmed that \cref{eq:degenerate-Lyapunov-2} has a unique solution $\hat{X}$ and holds with equalities. Thus, we must have
$$A_{\mathrm{cl}}(\hat{\mK})X_{12} + X_{12}\Lambda^\tr = 0_{(n+q)\times(n-q)}, % \quad  \Rightarrow \quad ( I_{n-q} \otimes A_{\mathrm{cl}}(\hat{\mK})  + \Lambda \otimes I_{n+q}) \mathrm{vec}(X_{12}) = 0,
$$
which ensures $X_{12} = 0$ since $A_{\mathrm{cl}}(\hat{\mK})$ and $\Lambda$ are stable \cite[Lemma 2.7]{zhou1996robust}. %(thus $I_{n-q} \otimes A_{\mathrm{cl}}(\hat{\mK})  + \Lambda \otimes I_{n+q}$ is invertible) {\color{red}[Ref]}.}
Therefore, any solution $X$ to \Cref{eq:degenerate-Lyapunov-1} has low-rank off-diagonal block $X_{12}$.
\end{itemize}

\noindent Therefore, we have established that $\tilde{\mK} \in {\mathcal{C}}_{n,0}\backslash{\mathcal{C}}_{\mathrm{nd},0}$. This completes the proof.
\end{proof}

\begin{remark}
    As discussed in \Cref{remark:degenerate-LQG-controller}, the question of whether or not a policy is non-degenerate is a property associated with its state-space realization. Although the policy $\tilde{\mK}$ in \cref{eq:stationary_nonglobally_K-augmented} is guaranteed to be degenerate, it is possible that another state-space realization as
    $$
    \tilde{\mK}_1=\left[\begin{array}{c:cc}
    0 & \hat{C}_{\mK} &  0 \\[2pt]
    \hdashline
    \hat{B}_{\mK} & \hat{A}_{\mK} & 0 \\[-2pt]
    B_1 & A_{21} & \Lambda
    \end{array}\right]
    $$
    with suitable $B_1 \in \mathbb{R}^{(n-q)\times m}$ and $A_{11} \in \mathbb{R}^{(n-q)\times q}$ is non-degenerate. We note that $\tilde{\mK}_1$, $\tilde{\mK}$ in \cref{eq:stationary_nonglobally_K-augmented} and $\hat{\mK}$ all correspond to the same transfer function in the frequency domain. It is interesting to investigate conditions under which such a state-space realization is possible. \hfill \qed
\end{remark}

% \vspace{20mm}

% \input{arXiv-Part I/app-ECL}

\subsection{Proof of \Cref{lemma:H_inf_some_local_Lipschitz}} \label{appendix:Hinf-cost-function}

Our proof will follow the idea sketched in~\cite[Proposition 3.1]{apkarian2006nonsmooth2}, i.e., the subdifferential regularity of $\|\mathbf{T}_{zd}\|_\infty$ follows from the convexity of $\|\cdot\|_\infty$ and the continuous differentiability of the mapping from $\mathcal{C}$ to $\mathcal{RH}_\infty$ given by $\mK\mapsto \mathbf{T}_{zd}$ in \cref{eq:transfer-function-zd}. But we will fill in the missing details of why the mapping $\mK\mapsto \mathbf{T}_{zd}$ is continuously differentiable. We will temporarily use $n$ to denote the dimension of $A_{\mathrm{cl}}(\mK)$, and denote $\mathcal{A} = \{A\in\mathbb{R}^{n\times n}:A\text{ is stable}\}$.

We first define the mapping $\mathscr{T}:\mathcal{A}\rightarrow\mathcal{RH}_\infty$ by
\[
(\mathscr{T}(A))(s) = (sI-A)^{-1}.
\]
Note that $A_{\mathrm{cl}}(\mK)$, $B_{\mathrm{cl}}(\mK)$, $C_{\mathrm{cl}}(\mK)$ and $D_{\mathrm{cl}}(\mK)$ are all affine functions of $\mK$ and thus are continuously differentiable. As a result, the continuous differentiability of $\mK\mapsto \mathbf{T}_{zd}$ will follow if we can show that $\mathscr{T}$ is continuously differentiable.

Let $A\in\mathcal{A}$ be an arbitrary stable matrix, and define the linear mapping $\xi_A:\mathbb{R}^{n\times n}\rightarrow \mathcal{RH}_\infty$ by
\[
(\xi_A(\Delta))(s) = (sI-A)^{-1}\Delta(sI-A)^{-1}.
\]
%Evidently $\xi_A$ is a linear mapping.
Let $\Delta\in\mathbb{R}^{n\times n}$ be an arbitrary matrix satisfying
$
 \|(sI-A)^{-1}\|_\infty\cdot\|\Delta\|_2<1
$, and we consider bounding the quantity
\begin{align*}
r_A(\Delta)\coloneqq\ &
\big\|\mathscr{T}(A+\Delta)-\mathscr{T}(A)-\xi_A(\Delta)
\big\|_\infty.
\end{align*}
Then, as long as $\frac{r_A(\Delta)}{\|\Delta\|_2}\rightarrow 0$ as $\|\Delta\|\rightarrow 0$, we can conclude that $\xi_A$ is the Fr\'{e}chet derivative of $\mathscr{T}$ at $A$. Indeed, we have, for any $s=\mathrm{j}\omega$ with $\omega\in\mathbb{R}$,
\begin{align*}
&(\mathscr{T}(A+\Delta)-\mathscr{T}(A)-\xi_A(\Delta))(s) \\
%=&
%(sI\!-\!(A+\Delta))^{-1}\!-\!(sI\!-\!A)^{-1}
% \!-\!(sI\!-\!A)^{-1}\Delta(sI\!-\!A)^{-1} \\
= &
\left((I-(sI\!-\!A)^{-1}\Delta)^{-1}-I-(sI\!-\!A)^{-1}\Delta\right)(sI\!-\!A)^{-1}.
\end{align*}
Since $\|(sI\!-\!A)^{-1}\Delta\|_\infty\leq \|(sI\!-\!A)^{-1}\|_\infty\|\Delta\|_2< 1$,
\begin{align*}
(I-(sI-A)^{-1}\Delta)^{-1}
=
\sum\nolimits_{k=0}^\infty \left((sI-A)^{-1}\Delta\right)^k,
\end{align*}
where the right-hand side converges absolutely. Therefore
\begin{align*}
r_A(\Delta) =\
%& \left\|\mathscr{T}(A+\Delta)-\mathscr{T}(A)-\xi_A(\Delta)\right\|_\infty \\
%=\
&
\left\|\sum\nolimits_{k=2}^\infty \left((sI-A)^{-1}\Delta\right)^k (sI-A)^{-1}
\right\|_\infty \\
\leq\ &
\sum\nolimits_{k=2}^\infty (\left\|(sI \!-\!A)^{-1}\right\|_\infty \|\Delta\|_2)^k \|(sI \!-\!A)^{-1}\|_\infty \\
=\ &
\frac{\|(sI-A)^{-1}\|_\infty^3 }{1-\|(sI-A)^{-1}\|_\infty \|\Delta\|_2} \|\Delta\|^2_2,
\end{align*}
from which we can easily check that $r_A(\Delta)/\|\Delta\|_2$ converges to $0$ as $\|\Delta\|_2\rightarrow 0 %\infty
$. Therefore $\mathscr{T}$ is differentiable at $A$, and its Fr\'{e}chet derivative is given by the linear mapping $\xi_A$.

Finally, we show that $\xi_A$ is continuous in $A$. By definition, we have $\xi_A(\Delta) =
\mathscr{T}(A)\Delta\mathscr{T}(A)$, which leads to
\begin{align*}
& \sup_{\|\Delta\|_2=1}
\left\|\xi_{A+H}(\Delta)-\xi_{A}(\Delta)\right\|_\infty \\
=\ &
\sup_{\|\Delta\|_2=1}
\left\|\mathscr{T}(A+H)\Delta\mathscr{T}(A+H)
-\mathscr{T}(A)\Delta\mathscr{T}(A)\right\|_\infty \\
\leq\ &
\sup_{\|\Delta\|_2=1}
\big(\| (\mathscr{T}(A+H)-\mathscr{T}(A))\Delta \mathscr{T}(A+H)\|_\infty \\
&\qquad\qquad + \| \mathscr{T}(A)\Delta (\mathscr{T}(A+H)-\mathscr{T}(A))\|_\infty\big) \\
\leq\ &
\| (\mathscr{T}(A\!+\!H) \!-\! \mathscr{T}(A))\|_\infty
\left(\|\mathscr{T}(A\!+\!H)\|_\infty \!+\! \|\mathscr{T}(A)\|_\infty\right).
\end{align*}
As $\|H\|\rightarrow 0$, the quantity on the left-hand side will then converge to $0$, implying that the mapping $\xi_A$ is continuous in $A$. Our proof is now complete.

\subsection{Proof of \Cref{theorem:hinf_non_globally_optimal_stationary_point}} \label{appendix:lower-order-stationary-point-to-high-order}

%{\color{red}@Rich, could you write a few more detailed steps for this proof.}
Our proof is based on the Clarke subdifferential formula in \Cref{lemma:subdifferential-Hinf}. Suppose $\mK=\begin{bmatrix}
D_{\mK} & C_{\mK} \\ B_{\mK} & A_{\mK}
\end{bmatrix}\in \mathcal{C}_q$ such that $0\in \partial J_{\infty,q}(\mK)$, and let
\[
\tilde{\mK}
    =\left[\begin{array}{c:cc}
    D_\mK & C_{\mK} &  0 \\[2pt]
    \hdashline
    B_{\mK} & A_{\mK} & 0 \\[-2pt]
    0 & 0 & \Lambda
    \end{array}\right]\]
for some $\Lambda\in\mathbb{R}^{q'\times q'}$ that is stable. We note that the closed-loop transfer matrices corresponding to $\mK$ and $\tilde{\mK}$ are the same, and so are their $\mathcal{H}_\infty$ norms:
\[
\mathbf{T}_{zd}(\mK,s) = \mathbf{T}_{zd}(\tilde{\mK},s),
\qquad J_{\infty,q}(\mK)=J_{\infty,q+q'}(\tilde{\mK}).
\]
Let
\[
\mathcal{Z} = \{s\in j\mathbb{R}\cup\{\infty\}\mid \sigma_{\max}(\mathbf{T}_{zd}(\mK,s)) = J_{\infty,q}(\mK)\}
\]
By \Cref{lemma:subdifferential-Hinf}, we see that there exist finitely many $s_1,\ldots,s_K\in\mathcal{K}$ and associated positive semidefinite matrices $Y_1,\ldots,Y_K$ with $\sum_\kappa Y_\kappa=1$ such that
\begin{align*}
0 = \frac{1}{J_{\infty,q}(\mK)}
\sum_{\kappa=1}^K
\operatorname{Re}\bigg\{\!
& \left(
\begin{bmatrix}
0_{p\times n} & V^{1/2} \\ 0_{q\times n} & 0_{q\times p}
\end{bmatrix}
+
\begin{bmatrix}
C & 0 \\
0 & I_q
\end{bmatrix}
(s_\kappa I-A_{\mathrm{cl}}(\mK))^{-1}B_{\mathrm{cl}}(\mK)
\right)
\mathbf{T}_{zd}(\mK,s_\kappa)^\her Q_{s_\kappa} Y_\kappa Q_{s_\kappa}^\her \\
&\quad \cdot
\left(
\begin{bmatrix}
0_{n\times m} & 0_{n\times q} \\ R^{1/2} & 0_{m\times q}
\end{bmatrix}
+C_{\mathrm{cl}}(\mK)(s_\kappa I-A_{\mathrm{cl}}(\mK))^{-1}
\begin{bmatrix}
B & 0 \\ 0 & I_q
\end{bmatrix}
\right)
\!\bigg\}^{\!\tr},
\end{align*}
where the dimensions of some matrices are explicitly given for clarity. We then note that
\[
A_{\mathrm{cl}}(\tilde{\mK})
=\begin{bmatrix}
A_{\mathrm{cl}}(\mK) & 0_{(n+q)\times q'} \\
0_{q'\times(n+q)} & \Lambda
\end{bmatrix},
\quad
B_{\mathrm{cl}}(\tilde{\mK}) = \begin{bmatrix}
B_{\mathrm{cl}}({\mK}) \\
0_{q'\times(n+p)}
\end{bmatrix},
\quad
C_{\mathrm{cl}}(\tilde{\mK}) =
\begin{bmatrix}
C_{\mathrm{cl}}({\mK}) & 0_{(n+m)\times q'}
\end{bmatrix}.
\]
As a result, we have
\begin{align*}
& \begin{bmatrix}
0_{p\times n} & V^{1/2} \\
0_{(q+q')\times n} & 0_{(q+q')\times p}
\end{bmatrix}
+
\begin{bmatrix}
C & 0 \\ 0 & I_{q+q'}
\end{bmatrix}
(s_\kappa I - A_{\mathrm{cl}}(\tilde{\mK}))^{-1} B_{\mathrm{cl}}(\tilde{\mK}) \\
={} &
\begin{bmatrix}
0_{p\times n} & V^{1/2} \\
0_{q\times n} & 0_{q\times p} \\
0_{q'\times n} & 0_{q'\times p}
\end{bmatrix}
+
\begin{bmatrix}
C & 0 & 0 \\
0 & I_q & 0 \\
0 & 0 & I_{q'}
\end{bmatrix}
\begin{bmatrix}
(s_\kappa I - A_{\mathrm{cl}}({\mK}))^{-1} & 0 \\
0 & (s_\kappa I  - \Lambda)^{-1}
\end{bmatrix}
\begin{bmatrix}
B_{\mathrm{cl}}(\mK) \\ 0_{q'\times(n+p)}
\end{bmatrix} \\
={} &
\begin{bmatrix}
\begin{bmatrix}
0_{p\times n} & V^{1/2} \\
0_{q\times n} & 0_{q\times p}
\end{bmatrix}
+ \begin{bmatrix}
C & 0 \\
0 & I_q
\end{bmatrix}
(s_\kappa I-A_{\mathrm{cl}}({\mK}))^{-1} B_{\mathrm{cl}}(\mK) \\
0_{q'\times (n+p)}
\end{bmatrix}.
\end{align*}
Similarly,
\begin{align*}
& \begin{bmatrix}
0_{n\times m} & 0_{n\times (q+q')} \\
R^{1/2} & 0_{m\times (q+q')}
\end{bmatrix}
+C_{\mathrm{cl}}(\tilde{\mK})(s_\kappa I-A_{\mathrm{cl}}(\tilde{\mK})
\begin{bmatrix}
B & 0 \\ 0 & I_{q+q'}
\end{bmatrix} \\
={} &
\begin{bmatrix}
\begin{bmatrix}
0_{n\times m} & 0_{n\times q} \\
R^{1/2} & 0_{m\times q}
\end{bmatrix}
+ C_{\mathrm{cl}}(\mK)(s_\kappa I - A_{\mathrm{cl}}(\mK))^{-1}\begin{bmatrix}
B & 0 \\ 0 & I_q
\end{bmatrix}
&
0_{(n+m)\times q'}
\end{bmatrix}.
\end{align*}
Straightforward calculation then leads to
\begin{align*}
& \frac{1}{J_{\infty,q+q'}(\tilde{\mK})}
\sum_{\kappa=1}^K
\operatorname{Re}\bigg\{\!
\left(
\begin{bmatrix}
0_{p\times n} & V^{1/2} \\ 0_{(q+q')\times n} & 0_{(q+q')\times p}
\end{bmatrix}
+
\begin{bmatrix}
C & 0 \\
0 & I_{q+q'}
\end{bmatrix}
(s_\kappa I \!-\! A_{\mathrm{cl}}(\tilde{\mK}))^{-1}B_{\mathrm{cl}}(\tilde{\mK})
\right) \\
& \qquad\qquad \cdot
\mathbf{T}_{zd}(\tilde{\mK},s_\kappa)^\her Q_{s_\kappa} Y_\kappa Q_{s_\kappa}^\her
\left(
\begin{bmatrix}
0_{n\times m} & 0_{n\times {q+q'}} \\ R^{1/2} & 0_{m\times {q+q'}}
\end{bmatrix}
+C_{\mathrm{cl}}(\mK)(s_\kappa I \!-\! A_{\mathrm{cl}}(\mK))^{-1}
\begin{bmatrix}
B & 0 \\ 0 & I_{q+q'}
\end{bmatrix}
\right)
\!\bigg\}^{\!\tr}
=0,
\end{align*}
and by applying \Cref{lemma:subdifferential-Hinf} again, we see that $0\in\partial J_{\infty,q+q'}(\tilde{\mK})$.

\begin{comment}
WLOG, assume that $\omega_0$ is the only frequency that achieves the $\mathcal{H}_\infty$ norm $J_{\infty, q}(\mK)$. Then $\mK$ being Clarke stationary implies that
\begin{equation} \label{eq:proof-thm-5.2}
    \|{\mathbf{T}_{zd}(\mK)}\|_\infty^{-1} \mathrm{Re} \left(\mathbf{G}_{21}(\mK,j\omega_0)\mathbf{T}_{zd}(\mK,j\omega_0)^\her Q_0Y_0Q_0^\her \mathbf{G}_{12}(\mK,j\omega_0)\right)^\tr=0.
\end{equation}
Using the equalities below
\[
\mathbf{T}_{zd}(\mK,j\omega_0)=\mathbf{T}_{zd}(\tilde{\mK},j\omega_0),\mathbf{G}_{21}(\tilde{\mK},j\omega_0)=\begin{bmatrix}
\mathbf{G}_{21}({\mK},j\omega_0) \\ 0 \end{bmatrix}, \mathbf{G}_{12}(\tilde{\mK},j\omega_0)=\begin{bmatrix}
\mathbf{G}_{12}({\mK},j\omega_0) & 0 \end{bmatrix},\]
we can show that the following subgradient $\Phi^{\mathrm{sg}}$ of $J_{\infty, q+q'}\big(\tilde{\mK}\big)$ satisfies
\begin{align*}
    \Phi^{\mathrm{sg}}&=\|\mathbf{T}_{zd}(\tilde{\mK})\|_\infty^{-1}\mathrm{Re} \left(\mathbf{G}_{21}(\tilde{\mK},j\omega_0)\mathbf{T}_{zd}(\tilde{\mK},j\omega_0)^HQ_{u_0}Y_{u_0}Q_{u_0}^HG_{12}(\tilde{\mK},j\omega_0)\right)^\tr \\
    &=\begin{bmatrix}
\|\mathbf{T}_{zd}(\mK)\|_\infty^{-1}\mathrm{Re} \left(\mathbf{G}_{21}(\mK,j\omega_0)\mathbf{T}_{zd}(\mK,j\omega_0)^HQ_{u_0}Y_{u_0}Q_{u_0}^H\mathbf{G}_{12}(\mK,j\omega_0)\right)^\tr \\ 0 \end{bmatrix}.
\end{align*}
Therefore \Cref{eq:proof-thm-5.2} implies $\Phi^{\mathrm{sg}}$=0.
\end{comment}

%\section{Analogous results in discrete time} \label{appendix:discrete_time} 

\end{document}